\documentclass{article}[11pt]
\usepackage{db_missing_data}
\usepackage{mathtools}
\usepackage{amssymb}
\usepackage{amsthm}
\usepackage{multirow} 
\usepackage[margin=1in]{geometry}
\usepackage{hyperref} 
\hypersetup{
    colorlinks,
    linkcolor={blue!80!black},
    citecolor={green!50!black},
}
\usepackage{xcolor}
\usepackage[shortlabels]{enumitem}
\usepackage[normalem]{ulem}

\usepackage{cleveref}

\usepackage{enumitem}

\newtheoremstyle{myremark} 
    {\topsep}                    
    {\topsep}                    
    {\rm}                        
    {}                           
    {\bf}                        
    {.}                          
    {.5em}                       
    {}  
\newtheorem{theorem}{Theorem}
\newtheorem{lemma}{Lemma}
\newtheorem{definition}{Definition}[section]
\newtheorem{corollary}{Corollary}
\newtheorem{proposition}{Proposition}
\theoremstyle{myremark}
\newtheorem{remark}{Remark}[section]

\makeatletter
\long\def\@makecaption#1#2{
        \vskip 0.8ex
        \setbox\@tempboxa\hbox{\small {\bf #1:} #2}
        \parindent 1.5em  
        \dimen0=\hsize
        \advance\dimen0 by -3em
        \ifdim \wd\@tempboxa >\dimen0
                \hbox to \hsize{
                        \parindent 0em
                        \hfil 
                        \parbox{\dimen0}{\def\baselinestretch{0.96}\small
                                {\bf #1.} #2
                                } 
                        \hfil}
        \else \hbox to \hsize{\hfil \box\@tempboxa \hfil}
        \fi
        }
\makeatother

\begin{document}

\begin{center}
  {\bf{\Large{Challenges of the inconsistency regime: \\
        Novel debiasing methods for missing data models}}} \\

\vspace*{0.1in}

{\large{
\begin{tabular}{cc}
Michael Celentano$^{\star}$ & Martin J. Wainwright$^{\star,
  \ddagger, \dagger, \circ}$
\end{tabular}
}}

\vspace*{.2in}

\begin{tabular}{c}
Department of Statistics$^\star$, and \\ Department of Electrical
Engineering and Computer Sciences$^\ddagger$ \\ UC Berkeley, Berkeley,
CA \\
\end{tabular}

\vspace*{.2in}

\begin{tabular}{c}
  Department of Electrical Engineering and Computer
  Sciences$^\dagger$\\
  Department of Mathematics$^\circ$ \\
  LIDS, and Statistics and Data Science Center \\
  Massachusetts Institute of Technology, Cambridge, MA
\end{tabular}

\vspace*{0.1in}

\today

\vspace*{0.1in}

\begin{abstract}
  We study semi-parametric estimation of the population mean when data
  is observed missing at random (MAR) in the $n < p$ ``inconsistency
  regime'', in which neither the outcome model nor the
  propensity/missingness model can be estimated consistently.
  Consider a high-dimensional linear-GLM specification in which the
  number of confounders is proportional to the sample size.  In the
  case $n > p$, past work has developed theory for the classical AIPW
  estimator in this model and established its variance inflation and
  asymptotic normality when the outcome model is fit by ordinary least
  squares. Ordinary least squares is no longer feasible in the case $n
  < p$ studied here, and we also demonstrate that a number of
  classical debiasing procedures become inconsistent.  This challenge
  motivates our development and analysis of a novel procedure: we
  establish that it is consistent for the population mean under
  proportional asymptotics allowing for $n < p$, and also provide
  confidence intervals for the linear model coefficients. Providing
  such guarantees in the inconsistency regime requires a new debiasing
  approach that combines penalized $M$-estimates of both the outcome
  and propensity/missingness models in a non-standard way.
\end{abstract}

\end{center}

\section{Introduction}

In semi-parametric problems, the goal is to estimate a target
parameter in the presence of one or more ``nuisance'' components.
This paper studies the estimation of a population mean with data
missing at random (MAR) in a setting where the outcome and propensity
nuisance models cannot be estimated consistently.  We receive iid
observations $(y_i\action_i,\action_i,\bx_i)$, $i \in [n] \defn \{1,
\ldots, \numobs \}$, where $y_i$ is the outcome of interest,
$\action_i \in \{0,1\}$ is a binary missingness indicator, and $\bx_i
\in \reals^p$ are covariates.  The MAR assumption is that the outcome
and missingness indicator are conditionally independent given
covariates: $y_i \indep \action_i \mid \bx_i$.  The estimand of
interest is the population mean $\mu_{\out} \defn \E[y_i]$.

Many methods for estimating the population mean rely on
access to sufficiently accurate estimates of either the \emph{outcome
model} $\mu: \reals^p \rightarrow \reals$ or the \emph{propensity
model} $\pi: \reals^p \rightarrow (0,1)$, or both, where
\begin{equation}
    \mu(\bx) \defn \E[y\mid \bx], \quad \mbox{and} \quad \pi(\bx)
    \defn \P(\action = 1 \mid \bx).
\end{equation}
These methods, and their accompanying statistical guarantees, can be
roughly categorized as either \emph{model-agnostic} or
\emph{model-aware}.  Model-agnostic approaches assume that the
functions $\mu$ and/or $\pi$ can be estimated consistently, and at
sufficiently fast rates, but are agnostic to the particular structural
assumptions under which such guarantees hold, or to the particular
methods that achieve them.  There are at least three standard
approaches---outcome regression (a special case of the
$G$-formula~\cite{ROBINS19861393}), Inverse Probability Weighting
(IPW) (also known as the Horvitz-Thompson
estimator)~\cite{Hahn1998,hirano2003,horvitzThompson1952}, and
Augmented Inverse Probability Weighting
(AIPW)~\cite{bangRobins2005}---each of which can be wrapped around any
black-box method for estimating $\mu$ or $\pi$.  Letting $\hmu$
(respectively $\hpi$) be some estimate of $\mu$ (respectively of
$\pi$), these approaches compute the following quantities:
\begin{subequations}
\label{eq:classical-estimates}  
\begin{align}
\mbox{$\sG$-approach:} \quad & \hmu_{\out}^{\sG} \defn \frac{1}{n}
\sum_{i=1}^n \hmu(\bx_i) \\
  \mbox{$\IPW$-approach:} \quad &  
  \hmu_{\out}^{\IPW} \defn \frac{1}{n} \sum_{i=1}^n
  \frac{\action_i y_i}{\hpi(\bx_i)}, \quad \mbox{and} \\
  \mbox{$\AIPW$-approach:} \quad & \quad \hmu_{\out}^{\AIPW} \defn
  \frac{1}{n} \sum_{i=1}^n \hmu(\bx_i) + \frac{1}{n} \sum_{i=1}^n
  \frac{\action_i}{\hpi(\bx_i)} (y_i - \hmu(\bx_i)).
\end{align}
\end{subequations}
These procedures can be combined with sample splitting strategies that
reduce bias and allow for improved statistical
guarantees~\cite{chernozhukov2018,newey2018crossfitting}.  All three
enjoy consistency or $\sqrt{n}$-consistency guarantees under
appropriate consistency conditions on $\hmu$ or $\hpi$.  A notable
property of the AIPW estimator, which involves estimates of both
nuisance parameters, is that it achieves $\sqrt{n}$-consistency if
$\E[(\hmu(\bx)-\mu(\bx))^2]^{1/2}\E[(\hpi(\bx)-\pi(\bx))^2]^{1/2} =
o(n^{-1/2})$ and appropriate overlap conditions are satisfied
\cite{bickelKlassenRitovWellner1998,bangRobins2005,chernozhukov2018,kennedy2023semiparametric}.
More generally, the error in estimating the nuisance parameters
induces an error in the estimate of $\mu_\out$ that is, up to a
constant, the product of the size of the errors of the two nuisance
parameters individually.

When the model is known to satisfy additional structural conditions,
model-agnostic methods and their associated guarantees can become
suboptimal.  For example, when $\mu$ and $\pi$ are assumed to belong
to H\"older smoothness classes, estimators that are tailored to
particular function classes can, in some regimes, achieve errors
smaller than those given by optimal model-agnostic guarantees
(e.g.,~\cite{mukherjee2017semiparametric,robins2008higher,robinsLiMukherjeeTchetgen2017,vanDerVaart2014}).
Such modified methods and analyses can be termed as being
``model-aware''.  Recent work has established that the improvements
enjoyed by model-aware approaches relative to model-agnostic
approaches are fundamental: Balakrishnan et
al.~\cite{Balakrishnan:2023aa} show that model-agnostic guarantees
cannot be improved without additional structural assumptions.


\subsection{The inconsistency regime}

Both of these previously described lines of work involve methods that
are based on consistent estimators of one or both of the nuisance
functions.  In contrast, the focus of this paper is a fundamentally
different and more challenging setting---known as the
\emph{inconsistency regime}---in which neither of the nuisance
functions $\mu$ nor $\pi$ can be estimated consistently.  Focusing on
the regime in which the sample size $n$ is less than the dimension
$p$, we ask: is it possible (and if so, how) to obtain consistent
estimates of the population mean when consistent estimation of the
nuisance components is no longer possible?

\subsubsection{Investigation for GLM-based missing data models}

In order to bring focus to the previously posed question, we study a
model-aware procedure for a particular instantiation of the missing
data problem.  Suppose that the outcome variable $y_i \in \reals$ is
related to the covariate vector $\bx_i \in \reals^p$ via a linear
model, whereas the missingness indicators are related via a
generalized linear model (GLM).  These assumptions are formalized via
the equations
\begin{align}
\label{eq:model}
    y_i = \theta_{\out,0} + \< \bx_i , \btheta_\out \> + \eps_i,
    \qquad \mbox{and} \qquad \P(\action_i = 1 \mid \bx_i) =
    \pi(\theta_{\prop,0} + \< \bx_i , \btheta_\prop \>),
\end{align}
where $\eps_i \sim \normal(0,\sigma^2)$ and $\pi : \reals \rightarrow
(0,1)$ is a known link function satisfying certain regularity
conditions (see Assumption A1 below).  We also assume that
the features $\bx_i$ are uncentered jointly Gaussian $\bx_i \sim
\normal(\bmu_\sx,\bSigma)$, where the feature covariance $\bSigma$ is
either known to the statistician or can be estimated consistently in
operator norm.

We study this model in an asymptotic regime in which model-agnostic
guarantees fail dramatically: the sample size $n$ scales
proportionally with the dimensionality of the covariates $p$, and the
sparsity or other assumptions we make about $\btheta_\out$ and
$\btheta_\prop$ are fundamentally insufficient to guarantee consistent
estimation of either the outcome model or the propensity
model~\cite{Barbier_2019}.  We call such a regime an
\emph{inconsistency regime}.  Despite the inconsistency of nuisance
parameter estimates, we develop methods that are provably consistent
for the population mean.  When $n\E[\pi(\bx)] > (1+\epsilon) p$ for
some $\epsilon > 0$, previous work~\cite{yadlowskyYun2021,sur2022} has
established that both outcome regression and AIPW are
$\sqrt{n}$-consistent for the population mean, provided that the
outcome model is fit with ordinary least squares.  On the other hand,
in the complementary regime $n \E[\pi(\bx)] < p$, ordinary least
squares is infeasible, and, as discussed in the sequel, neither the
outcome regression estimate nor the AIPW estimate are consistent.  In
this challenging regime, we develop alternative estimators that are
provably consistent for the population mean, and we provide numerical
evidence that these estimators are in fact $\sqrt{n}$-consistent.
Like the AIPW estimator, the estimators we propose involve estimates
of both the outcome and propensity models, but we combine these
estimates in a non-standard way.  Unlike the previous works
\cite{yadlowskyYun2021,sur2022}, our estimate does not require sample
splitting or cross fitting.  Some previous work has studied conditions
under which sample splitting is not required, but is restricted to
consistency regimes \cite{chenSyrgkanisAustern2022}.


\subsubsection{Classical models fail in the inconsistency regime}

We begin by presenting the results of a simulation study that
recapitulates three known phenomena that motivate the present work.

\begin{figure}[h!]
  \begin{center}
    \begin{tabular}{ccc}
\includegraphics[width =
  0.45\linewidth]{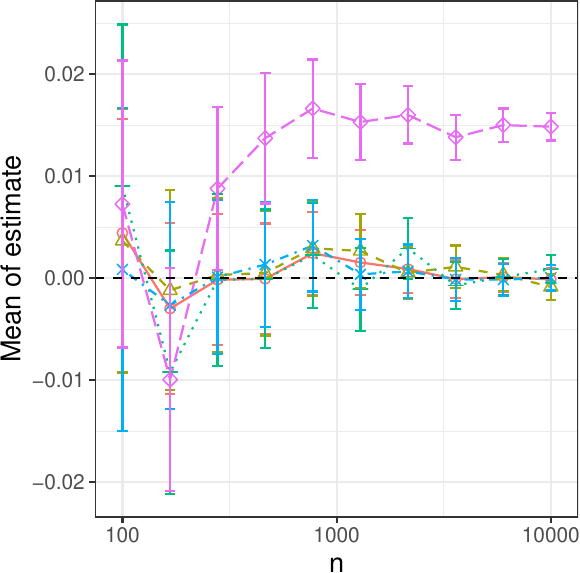}
&& 
\includegraphics[width =
  0.45\linewidth]{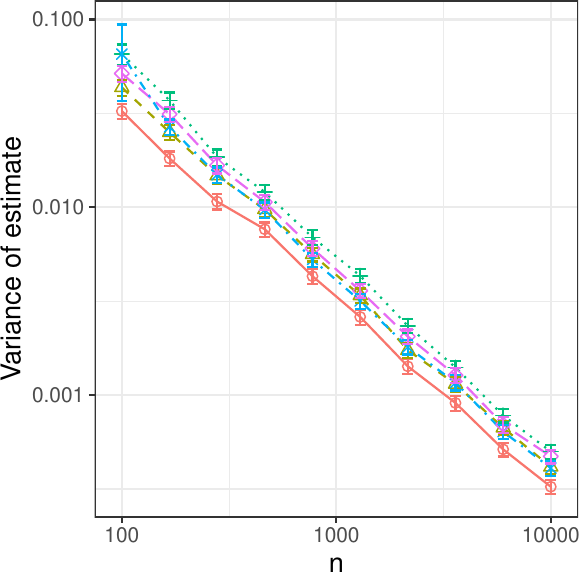}
    \end{tabular}
\includegraphics[width =
  0.8\linewidth]{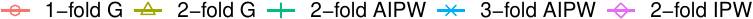}\\
\vspace*{0.05in}
\caption{Comparison of G-computation, AIPW, and IPW. We take
  $\theta_{\prop,0} = \theta_{\out,0} = 0$, $\btheta_\prop =
  \btheta_\out = \be_1$, $\bSigma = \id_p$, $\epsilon_i \iid
  \normal(0,1)$, $\mu_{\sx} = \bzero$, and link $\pi(\eta) =
  \tfrac{1}{10} + \tfrac{9}{10} \logit^{-1}(\eta)$.  The outcome model
  is fit by ordinary least squares and the propenisty model by
  logistic regression.  In this case, the population mean outcome is
  $\mu_\out \defn \E[y] = 0$.  For each value of $n$, $p$ is taken as
  the closest integer to $0.07n$. The mean and variance of these
  estimates are computed across 1000 replicates, with 95\% confidence
  intervals shown. See text for detailed description of cross-fitting
  strategies.}
\label{fig:variance-well-specified-outcome}
\end{center}
\end{figure}

First, as described in the papers~\cite{yadlowsky2022,sur2022}, when
the outcome model is fit with ordinary least squares, both
G-computation and AIPW are $\sqrt{n}$-consistent under proportional
asymptotics.  \Cref{fig:variance-well-specified-outcome} displays the
results of a simulation study in which
\begin{align*}
\theta_{\prop,0} = \theta_{\out,0} = 0, \quad \btheta_\prop =
\btheta_\out = \be_1, \quad \bSigma = \id_p, \quad \epsilon_i \iid
\normal(0,1), \quad \mu_{\sx} = \bzero, \quad \mbox{and} \quad
\pi(\eta) = \tfrac{1}{10} + \tfrac{9}{10} \logit^{-1}(\eta),
\end{align*}
whence the population mean outcome is $\mu_\out \defn \E[y] = 0$.
Because $\btheta_\prop = \btheta_\out$, the observed outcomes are
confounded, and the mean of $y$ conditional on its being observed is
biased.  For 10 values of the sample size $n$ spanning from $n = 100$
to $n = 1000$, we simulate from the model \eqref{eq:model} in which $p
= 0.07 n$ (or the closest integer), $\htheta_{\prop,0},\hbtheta_\prop$
is fit using logistic regression, and $\htheta_{\out,0},\hbtheta_\out$
is fit using ordinary least squares.  We then estimate $\mu_\out$
using either outcome regression ($G$-computation), AIPW, or IPW, as
in~\cref{eq:classical-estimates}, where
\begin{align*}
\hmu(\bx_i) = \htheta_{\out,0} + \< \bx_i , \hbtheta_\out \>, \quad
\mbox{and} \quad \hpi(\bx_i) = \tfrac{1}{10} + \tfrac{9}{10}
\logit^{-1}(\htheta_{\prop,0} + \< \bx_i , \hbtheta_\prop\>).
\end{align*}
We employ three cross-fitting strategies.  A $1$-fold cross-fit fits
the functions $\hmu$ and $\hpi$ and computes the averages in
display~\eqref{eq:classical-estimates} using all the data.  A 2-fold
cross-fit splits the data into disjoint sets $\cI_1$ and $\cI_2$ each
of size $n/2$, fits the functions $\hmu$ and $\hpi$ on $\cI_1$, and
computes the averages in display~\eqref{eq:classical-estimates} on
$\cI_2$.  It then does the same with the roles of $\cI_1$ and $\cI_2$
reversed and finally takes the averages of the two estimates of
$\mu_\out$.  A 3-fold cross fit splits the data into disjoint sets
$\cI_1$, $\cI_2$, $\cI_3$ of size $n/3$, fits $\hmu$ on $\cI_1$,
$\hpi$ on $\cI_2$, and computes the averages in
display~\eqref{eq:classical-estimates} on $\cI_3$.  It then does the
same for all 6 permutations of the roles $\cI_1$, $\cI_2$, and $\cI_3$
and finally takes the average of the estimates of $\mu_\out$ across
all permutations.  The 3-fold cross-fitting strategy was considered in
the paper~\cite{sur2022}.  Taking $p = 0.07 \, n$ guarantees that the
least-squares estimate and logistic regression estimate are
well-defined on all folds with high
probability~\cite{candesSur2020,wenpinYe2020}. The left panel
of~\Cref{fig:variance-well-specified-outcome} shows the average value
of each estimate computed over 1000 replicates at several sample
sizes, and the right panel in
\Cref{fig:variance-well-specified-outcome} displays the empirical
variance of these estimates.  In both cases, we plot error bars with
half-width $1.96$ times the standard error of the estimated mean and
variance across the $1000$ replicates.  Recalling that the true value
of the population mean is $\mu_\out= 0$, one can see that the IPW
estimate has a noticeable bias that does not decay as $n \rightarrow
\infty$.  On the other hand, one can see that both the $G$-computation
and AIPW estimates have no significant bias, and their variance decays
as $1/\numobs$.  The goal of this paper is to develop an estimator
with these same properties when $\numobs \E[\pi(\bx)] < p$.

Second, the success of $G$-computation and AIPW in the inconsistency
regime relies crucially on the use of ordinary least squares to fit
the outcome model.  In~\Cref{fig:variance-regularization}, we simulate
from the same model that underlies the results shown
in~\Cref{fig:variance-well-specified-outcome}, except that we use
ridge regression to fit the outcome model parameter $\btheta_\out$.
We see that 2-fold and 3-fold AIPW estimators have substantially
smaller bias than 1-fold and 2-fold $G$-computation estimators.
Nevertheless, all estimates are biased with a bias that does not decay
as $n \rightarrow \infty$.  Thus, the main challenge in the regime
$n\E[\pi(\bx)] < p$ is that ordinary least squares is not feasible.
\begin{figure}[h!]
  \begin{center}
    \begin{tabular}{cc}
\includegraphics[width =
  0.45\linewidth]{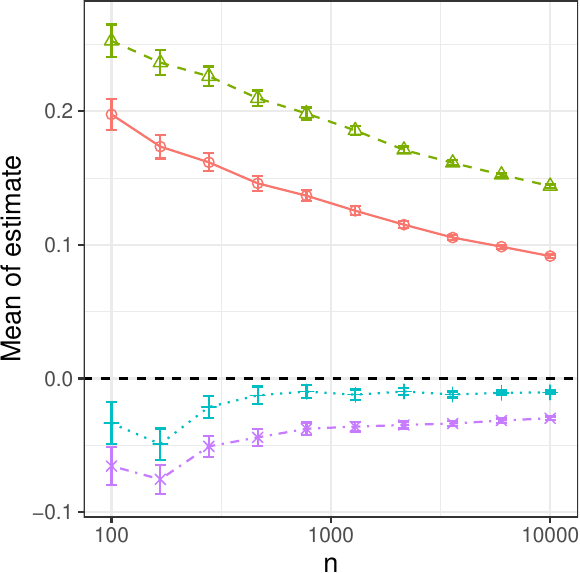} &
\includegraphics[width =
  0.45\linewidth]{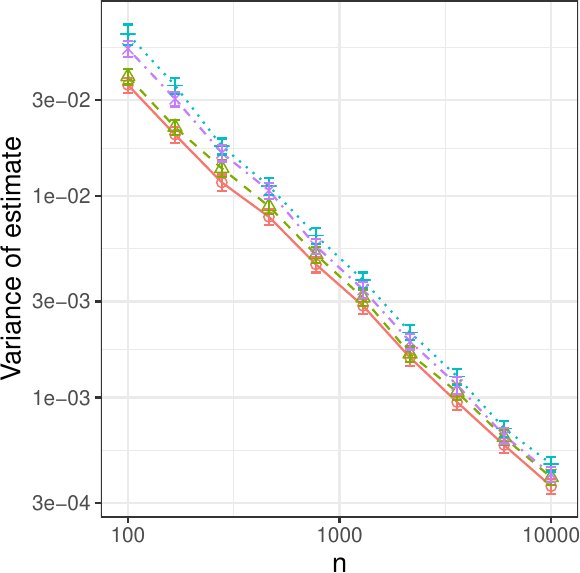}
    \end{tabular}
\includegraphics[width =
  0.5\linewidth]{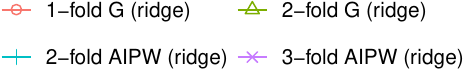} \\
\vspace*{0.1in}
\caption{Does AIPW require the outcome to be fit using Ordinary Least
  Squares?  Comparison of $G$-computation and AIPW with a
  well-specified linear outcome model.  We simulate from exactly the
  same setting as in \Cref{fig:variance-well-specified-outcome}.
  The outcome model is fit by ridge regression with regularization
  parameter chosen by cross-validation using the default parameters in
  \texttt{cv.glmnet} in \texttt{glmnet version 4.1.7}, \texttt{R
    version 4.2.2}, and the propensity model fit by logistic
  regression (without penalty).  The true value of the population mean
  is $\mu_\out \defn \E[y] = 0$.}
\label{fig:variance-regularization}
\end{center}
\end{figure}

Third, in the inconsistency regime, AIPW is less efficient than 1-fold
$G$-computation and is not protected against misspecification of the
outcome model.  Indeed, in the right-hand plot
of~\Cref{fig:variance-well-specified-outcome}, we see that 1-fold
$G$-computation has the smallest variance among all simulated
estimators, which is consistent with the results in the
paper~\cite{yadlowsky2022}.  In fact, even in the classical
semi-parametric setting, AIPW is often less efficient than
$G$-computation~\cite{kangShafer2007,robinsSuedLeiGomezRotnitzky2007}.
Classically, AIPW may be preferred in order to protect against
misspecification of the outcome model.  Unsurprisingly, because in the
current setting the propensity model is not estimated consistently,
this protection is not achieved.
In~\Cref{fig:outcome-misspecification} we simulate from the same
setting as
in~\Cref{fig:variance-well-specified-outcome,fig:variance-regularization},
except that the outcome model includes a term quadratic in the
features---viz.
\begin{align*}
y_i = \< \bx_i , \btheta_\out \> + (\< \bx_i , \btheta_\out \>^2 - 1)
+ \eps_i.
\end{align*}
Although both the 2-fold and 3-fold AIPW estimators have substantially
smaller bias than than those based on 1-fold and 2-fold outcome
regression, their bias does not decay as $n \rightarrow \infty$, so
that they are not actually consistent.  Thus, when $n\E[\pi(\bx)] >
(1+\epsilon)p$ in the inconsistency regime, it is not clear that
propensity modeling plays any useful role: the IPW estimator is
inconsistent (\Cref{fig:variance-well-specified-outcome}), and the
AIPW estimator incurs inflated variance relative to outcome regression
without protecting against outcome model misspecification.  In
contrast, the estimator that we develop for the regime $n\E[\pi(\bx)]
< p$ relies crucially on estimating the propensity model; we are
unaware of any consistent estimator that avoids propensity estimation
altogether.

\begin{figure}[h!]
\begin{center}
\includegraphics[width =
  0.45\linewidth]{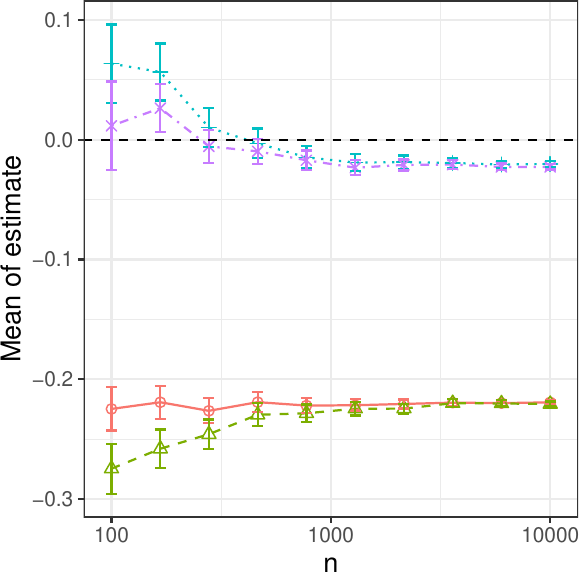}
\includegraphics[width =
  0.45\linewidth]{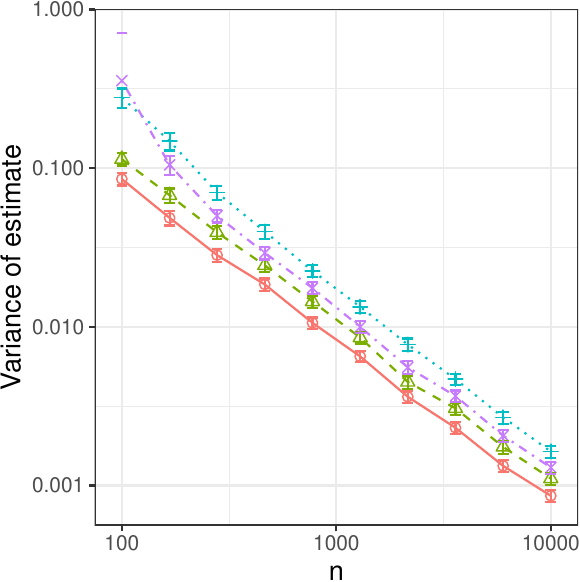}
\\ \includegraphics[width =
  0.675\linewidth]{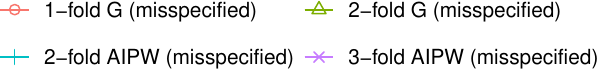}
\caption{Does AIPW protect against outcome misspecification?
  Comparison of $G$-computation and AIPW with a misspecified outcome
  model.  We simulate from the same setting as in
  \Cref{fig:variance-well-specified-outcome}, except generate outcomes
  according to $y_i = \< \bx_i , \btheta_\out \> + (\< \bx_i ,
  \btheta_\out \>^2 - 1) + \eps_i$.  The outcome model is fit by
  ordinary least squares and the propenisty model by logistic
  regression.  The true value of the population mean is $\mu_\out
  \defn \E[y] = 0$.}
\label{fig:outcome-misspecification}
\end{center} 
\end{figure}

\subsection{Related literature}

The problem of estimating population means from data missing at random
is a now classical problem in statistics, with the AIPW approach
incorporating both outcome and propensity modeling going back to the
1990s
(e.g,~\cite{robinsRotnitzky1995,robinsRotnitzkyZhou1994,robinsRotnitzky1995,scharfsteinRotnitzkyRobins1999}).
The double-robustness property of this and related estimators has been
explored in various
papers~\cite{robinsRotnitzky2001,laan2003unified,bangRobins2005,kangShafer2007,robinsSuedLeiGomezRotnitzky2007}.
There is also a long line of work that studies approaches based on
outcome modeling and propensity modeling
(e.g.,~\cite{ROBINS19861393,snowdenRoseMortimer2011,vansteelandtKeiding2011,horvitzThompson1952,rosenbaumRubin1983,Hahn1998,hirano2003,su2023estimated}).
Recently, there has been increased interest in the double-robustness
property and greater emphasis on the ability to achieve estimation
error guarantees with black-box methods
(e.g.,~\cite{chernozhukov2018,chernozhukov2021automatic,chernozhukovChetverikovDemirerDufloHansenNewey2017}).
Model-aware approaches in the context of H\"older smoothness classes
have also been
studied~\cite{robins2008higher,robinsLiMukherjeeTchetgen2017,newey2018crossfitting}.
Model-aware analysis of $G$-computation, AIPW, and IPW estimators in
the context of random design, linear outcome, generalized linear
propensity models was studied by~\cite{yadlowsky2022,sur2022}.
$G$-computation in the proportional regime with a linear outcome model
and potential heteroscedasticity was studied in the
papers~\cite{cattaneoJanssonNewey2018,cattaneo_jansson_newey_2018}.
There is a literature on minimax lower bounds, including of the
model-aware
variety~\cite{robinsTchetgenLiVanderVaart2009,mou2022offpolicy,mou2023kernelbased}
as well as the model-agnostic variety~\cite{Balakrishnan:2023aa}.
Barbier et al.~\cite{Barbier_2019} established the impossibility of
consistent nuisance parameter estimation under proportional
asymptotics.

Our methods build on a large literature developing exact asymptotic
characterization of convex procedures in generalized linear models;
for instance, see the
papers~\cite{bayatiMontanari2012,stojnic2013framework,Donoho:2016aa,Karoui2015OnTP,pmlr-v40-Thrampoulidis15,thrampoulidisAbbasiHassibi2018,surCandes2019,zhou2022,zhaoSurCandes2022,miolane2021,celentanoMontanariWei2020,montanari2023generalization},
as well as references therein.  Specifically, our methods are based on
debiasing constructions that were first introduced in the
papers~\cite{zhang2014confidence,buhlmann2013statistical,javanmard2014confidence},
and more recently, developed under proportional asymptotics in the
papers~\cite{javanmard2014hypothesis,javanmard2018debiasing,miolane2021,celentanoMontanariWei2020,bellec2019biasing,bellec2019debiasingIntervals}.
Our development builds upon our own past work in Celentano and
Montanari~\cite{celentano2021cad}, where a subset of the current
authors studied estimation of conditional covariances for
random-design linear regression models, and introduced ``correlation
adjustment'' for accounting for joint errors in estimating nuisance
parameters.  The methods in this paper require similarly motivated but
technically distinct forms of correlation adjustment.  Our proof is
based on the Convex Gaussian Min-Max theorem, developed by Stojnic
\cite{stojnic2013framework} as a sharpening of Gordon's Gaussian
comparison inequality \cite{gordon1985,gordon1988} for convex-concave
objectives.

\subsection{Outline of paper}

Let us provide an outline of the remainder of this paper.
In~\Cref{SecProblem}, we provide the construction of our proposed
estimator for outcome model, propensity model, and the population
mean.  \Cref{SecTheory} is devoted to our two main results:
\Cref{thm:pop-mean} establishes consistency of our proposed estimator
for the population mean, whereas~\Cref{thm:db-normality} establishes,
in a certain sense, the asymptotic normality and unbiasedness of
certain estimates we construct for the linear-model parameter
coefficients $\theta_{\out,j}$.  Our estimator requires four
estimating four adjustment factors, and our consistency and asymptotic
normality guarantees hold under a general consistency guarantee on
these adjustment factors.  In~\Cref{sec:consistency}, we provide one
such construction of these adjustment factors, and show they satisfy
the necessary consistency guarantee
(\Cref{prop:adjustment-concentration}).  In~\Cref{sec:simulations}, we
provide simulations that demonstrate the success of our estimator.
In~\Cref{SecProofs}, we provide a high-level overview of our proof
techniques.  Our proofs rely on an exact asymptotic characterization
for missing data models, which we also provide there.  This section is
intended only to provide a preview of our proof techniques.  Technical
details are deferred to the appendices.  We end
in~\Cref{SecDiscussion} with a concluding discussion.


\subsection{Notation}

For a square matrix $\bA \in \reals^{K \times K}$, we denote by
$\bA_k$ the $k\times k$ upper-left sub-matrix of $\bA$, by
$\bA_{\,\cdot\,,k}$ the $k^\text{th}$ column of $\bA$, and by
$\bA_{k,\,\cdot\,}$ the $k^\text{th}$ row of $\bA$.  For two vectors
$\ba,\bb \in \reals^N$, we denote the inner product between $\ba$ and
$\bb$ by $\< \ba , \bb \>$ and the Euclidean norm by $\|\ba\|$.  For
two matrices $\bA,\bB \in \reals^{N \times k}$, we denote $\llangle
\bA , \bB \rrangle = \bA^\top \bB \in \reals^{k \times k}$.  Note that
$\llangle \bA , \bB \rrangle_{\ell\ell'} = \< \ba_\ell ,
\bb_{\ell'}\>$, where $\ba_\ell$ is the $\ell^\text{th}$ column of
$\bA$, and $\bb_\ell$ is the $\ell^\text{th}$ column of $\bB$.  If
$\ba,\bb \in \reals^N$ are random vectors with finite second moments,
we denote $\< \ba , \bb \>_{\Ltwo} = \E[\< \ba, \bb \>]$ and $\| \ba
\|_{\Ltwo}^2 = \E[\| \ba \|^2]$.  Likewise, if $\bA,\bB \in \reals^{N
  \times k}$ are random matrices with finite second moments, we denote
$\llangle \bA , \bB \rrangle_{\Ltwo} = \E[\llangle \bA , \bB \rrangle]$.
Thus, in these cases, $\< \ba , \bb \>$ and $\llangle \bA , \bB
\rrangle$ are random quantities, and $\< \ba , \bb \>_{\Ltwo}$ and
$\llangle \bA , \bB \rrangle_{\Ltwo}$ are deterministic quantities.  The
operator norm of a matrix is denotes $\|\bA\|_{\op}$.  For a
positive-definite symmetric matrix $\bSigma$, we denote the
$\bSigma$-inner product $\< \ba , \bb \>_{\bSigma} = \< \ba , \bSigma
\bb \>$ and the $\bSigma$-norm $\| \ba \|_{\bSigma} = \sqrt{\< \ba ,
  \ba \>_{\bSigma}}$.

For a random vector with iid coordinates, we sometimes denote a scalar
random variable distributed from the distribution of each coordinate
by a non-boldface letter without a subscript. Thus, $d$ denotes a
random variable distributed according to $\action_i$, $y$ a random
variable distributed according to $y_i$.  We also denote a random
variable distributed according to $\bx_i$ by $\bx$.

Constants $C,c,C',c' > 0$ or with additional subscripts or
superscripts are reserved to refer to positive constants that depend
only on the constants appearing in Assumption A1 below.  They do not
depend on $n,p$ and may change at each appearance unless otherwise
specified.  We use $A \lesssim B$ to mean $A \leq C \, B$ for
some constant $C$ depending on Assumption A1, and define ``$\gtrsim$''
similarly.


\section{Debiasing in the inconsistency regime}
\label{SecProblem}

In this section, we describe the debiasing estimators that we analyze
in this paper.  Recall that the observed data set consists of an
i.i.d.\ collection of triples $\{(\action_i y_i, y_i, \bx_i)
\}_{i=1}^n$ obeying~\cref{eq:model}, where $\eps_i \sim
\normal(0,\sigma^2)$, $\pi : \reals \rightarrow (0, 1)$ is a known
link function, and $\bx_i \iid \normal(\bmu_{\sx}, \bSigma)$ for some
unknown mean vector $\bmu_{\sx} \in \reals^p$, and a known covariance
matrix $\bSigma \in \reals^{p\times p}$.

Our goal is to estimate the outcome mean $\mu_\out = \E[y]$.  Given
the assumed structure of our model, the tower property guarantees the
representation
\begin{align}
\label{EqnPopRep}  
  \mu_\out = \E \Big[ \E[y \mid \bx] \Big] \; = \; \E \Big[
    \theta_{\out,0} + \<\bx , \btheta_\out \> \Big],
\end{align}
which we exploit in developing our estimators.

\subsection{High-level overview}

We begin with a high-level description of the estimators analyzed in
this paper, and the ingredients in their construction.

\subsubsection{Three-stage approach}

A standard approach to estimating $\mu_\out$---known either as outcome
regression or as a special case of $G$-computation---is to replace the
population coefficients $(\theta_{\out,0}, \btheta_\out) \in \reals
\times \reals^p$ in~\cref{EqnPopRep} with empirical estimates
$(\htheta_{\out, 0}, \hbtheta_\out)$, and to replace the population
expectation over $\bx$ with an empirical expectation.  In our work,
due to the challenges of the inconsistency regime, an additional step
is required: we need to construct \emph{suitably debiased}
$(\htheta_{\out,0}^\de, \hbtheta_\out^\de)$ versions of the estimates,
and use them in the plug-in estimator
\begin{align}
\label{eq:db-G-computation}
    \hmu_\out^\de \defn \frac{1}{n} \sum_{i=1}^n \big(
    \htheta_{\out,0}^\de + \<\bx_i , \hbtheta_\out^\de \> \big),
\end{align}
\noindent All of the methods analyzed here involve the following three
steps:
\begin{itemize}
\item In all cases, we begin with \emph{base estimates}
  $(\htheta_{\out, 0}, \hbtheta_\out)$ obtained from penalized
  regression methods, as described in~\Cref{SecBaseEstimates}.
\item
  Let $\hbias_{\out, 0}$ and $\hbbias_\out$ be estimates of the biases
  of $\hbtheta_{\out,0}$ and $\hbtheta_\out$, respectively. Using
  these bias estimates, we form the \emph{debiased coefficient
  estimates}
\begin{align}
\label{eq:outcome-bias-correction}
  \htheta_{\out,0}^\de \defn \htheta_{\out,0} - \hbias_{\out, 0},
  \qquad \mbox{and} \qquad \hbtheta_\out^\de \defn \hbtheta_\out -
  \hbbias_\out.
\end{align}
Among a number of other ingredients, our bias estimates are based
upon a degrees-of-freedom adjustment, as described in~\Cref{sec:dof-adjustment}.
\item Finally, using these debiased estimates, we compute the
  \emph{plug-in estimator}~\eqref{eq:db-G-computation}.  To be clear,
  our methods do \emph{not} require sample splitting: the same data is
  used to compute the initial estimates, perform debiasing, and in the
  empirical expectation in the final
  estimate~\eqref{eq:db-G-computation}.
\end{itemize}

The main difficulty associated with these three step methods lies in
devising suitable estimates of the bias, and in analyzing them. In the
consistency regime, the biases of $\hbtheta_{\out,0}$ and
$\btheta_\out$ are sufficiently small that debiasing is not required
for consistency of $G$-computation.  As we discuss
in~\Cref{sec:procedure-definitions}, a na\"{i}ve application of
standard bias estimates--- even those designed for the inconsistency
regime when data is not missing---produce inconsistent estimates of
$\mu_\out$ when used as plug-ins. Accordingly, in this paper, we
propose and analyze a novel method for estimating the bias, as
described in~\Cref{SecUnknownProp}.  This method does not assume
knowledge of the propensity score.  This estimate uses the base
estimates described in~\Cref{SecBaseEstimates} and the
degrees-of-freedom adjustment described in~\Cref{sec:dof-adjustment}.

\subsubsection{Regularized estimators}
\label{SecBaseEstimates}

Both the outcome and propensity models are generalized linear models,
and we denote the linear predictors in these models by
\begin{equation}
\eta_{\out,i} \defn \theta_{\out,0} + \< \bx_i , \btheta_\out \>,
\qquad \mbox{and} \qquad \eta_{\prop,0} \defn \theta_{\prop,i} + \<
\bx_i , \btheta_\prop \>.
\end{equation}
Similarly, we use $\eta_\out \equiv \eta_\out(\bx)$ and $\eta_\prop
\equiv \eta_\prop(\bx)$ to denote the generic form of these linear
predictors with a generic covariate vector $\bx$.  Note that the
population means of these quantities are given by $\mu_\out \defn
\E[\eta_\out(\bx)]$ and $\mu_\prop \defn \E[\eta_\prop(\bx)]$.

\paragraph{Estimating the outcome model:}

Letting $w: \reals \rightarrow \reals_{>0}$ be a weight function,
define (with a slight abuse of notation) the weights $w_i = w(\theta_{\prop,0} + \<
\bx_i, \btheta_\prop\>)$ for $i \in [n]$.  Using these weights, we
estimate the pair $(\theta_{\out,0},\btheta_\out)$ using a penalized
form of weighted-least squares
\begin{equation}
\label{eq:outcome-fit}
    (\htheta_{\out,0},\hbtheta_\out) \defn \argmin_{(v_0,\bv) \in
  \reals \times \reals^p} \Big\{ \frac{1}{2n} \sum_{i=1}^n \action_i
w_i \big(y_i - v_0 - \< \bx_i , \bv \>\big)^2 + \Omega_\out(\bv)
\Big\},
\end{equation}
where $\Omega_\out: \reals^p \rightarrow \reals$ is a convex penalty
function.  For example, taking $\Omega_\out(\bv) = \lambda \| \bv
\|^2/2$ leads to a weighted form of ridge regression.  We only
establish results for fixed regularization parameter, which we
therefore absorb into the penalty $\Omega_\out$.  We require
$\Omega_\out$ to be strongly convex and smooth with constants of order
one; see~\Cref{SecAssumptions} below for more detail. In the case of
ridge regression, this corresponds to taking an order one $\lambda$.
In the inconsistency regime and under our choice of feature
normalization, the regularization parameter that minimizes prediction
error and is chosen by cross validation is of this order (see, for
example, the
papers~\cite{karoui2013asymptotic,miolane2021,hastieMontanariRossetTibshirani2022}).

Typically, we take $w \equiv 1$ (i.e., $w(\eta) = 1$ for all $\eta$),
so that this is just an unweighted penalized-least squares estimate.
However, we also consider oracle inverse probability weighting and
adjusted modified inverse probability weighting schemes, which
motivates writing the estimator in a more general form.  (Note that we
are ultimately interested in methods that have no prior knowledge of
the propensity model, in which case we always take $w \equiv 1$. Other
choices of the function $w$ are presented primarily for motivation and
illustrative purposes.)

\paragraph{Estimating the propensity model:}
We estimate the propensity model parameter using the penalized
M-estimator
\begin{align}
\label{eq:propensity-fit}
(\htheta_{\prop,0},\hbtheta_{\prop}) \defn \argmin_{(v_0,\bv) \in
  \reals \times \reals^p} \Big\{ \frac{1}{2 n} \sum_{i=1}^n
\ell_{\prop}\big(v_0 + \< \bx_i , \bv \> ; \action_i\big) +
\Omega_\prop(\bv) \Big\},
\end{align}
where $\ell_\prop: \reals \times \{0,1\}$ is a loss function that is
convex in its first argument, and $\Omega_\prop: \reals^p \rightarrow
\reals$ is another convex penalty function.  As above, the
regularization parameter is absorbed into the penalty $\Omega_\prop$,
and, in the case of ridge regularization, corresponds to taking
$\lambda = \Theta(1)$.  The Gaussian-GLM structure also allows us to
estimate the propensity model parameter by a moment method
\begin{equation}
    \hbtheta_\prop
        =
        \frac{1}{n_1}\sum_{i=1}^n \action_i \bx_i
        -
        \frac{1}{n}\sum_{i=1}^n \bx_i,
\end{equation}
where $n_1 = \sum_{i=1}^n \action_i$, with $\htheta_{\prop,0}$ then
determined by solving moment equations described
in~\Cref{sec:summary-stat-estimates}.  We establish results for both
the penalized M-estimator and moment-method estimator for the
propensity model.


\subsubsection{Degrees of freedom adjustments}
\label{sec:dof-adjustment}

Our estimates of the bias terms $\bias_{\out,0}$ and $\bias_\out$ rely
on adjustments for the degrees-of-freedom of the
estimators~\eqref{eq:outcome-fit} and~\eqref{eq:propensity-fit}.
Recall the loss function $\ell_{\out}(\eta;\eta_{\prop}, \action, y)
\defn \action w(\eta_\prop) (y - \eta)^2/2$ that underlies the outcome
estimator~\eqref{eq:outcome-fit} and, similarly, the loss
$\ell_\prop(\eta; \action)$ for the propensity
estimate~\eqref{eq:propensity-fit}.  The degrees-of-freedom
adjustments depend on the real number solutions $\hzeta_\out^\theta$
and $\hzeta_\out^\eta$ to the estimating equations
\begin{subequations}
\label{eq:dof-emp-def}
  \begin{equation}
\label{eq:dof-adjust-emp}    
\begin{gathered}
    \hzeta_\out^\theta = \frac{1}{n} \sum_{i=1}^n
    \frac{\ddot\ell_\out(\heta_{\out,i};\eta_{\prop,i},\action_i,y_i)}{\hzeta_\out^\eta
      \ddot \ell_\out(\heta_{\out,i};\eta_{\prop,i},\action_i,y_i)+1},
    \quad \mbox{and} \quad \hzeta_\out^\eta = \frac{1}{n} \Tr\Big(
    \bSigma \big(\hzeta_\out^\theta \bSigma + \nabla^2
    \Omega_\out(\hbtheta_\out)\big)^{-1} \Big),
\end{gathered}
\end{equation}
along with the real number solutions $\hzeta_\prop^\theta$ and
$\hzeta_\prop^\eta$ to the estimating equations
\begin{equation}
  \begin{gathered}
\label{eq:dof-adjust-action}        
    \hzeta_\prop^\theta = \frac{1}{n} \sum_{i=1}^n
    \frac{\ddot\ell_\prop(\heta_{\prop,i};\action_i)}{\hzeta_\prop^\eta
      \ddot \ell_\prop(\heta_{\prop,i};\action_i)+1}, \quad \mbox{and}
    \quad \hzeta_\prop^\eta = \frac{1}{n} \Tr\Big( \bSigma
    \big(\hzeta_\prop^\theta \bSigma + \nabla^2
    \Omega_\prop(\hbtheta_\prop)\big)^{-1} \Big).
\end{gathered}
\end{equation}
\end{subequations}
As shown in~\Cref{AppDOF}, these equations have unique positive
solutions.

The quantities $\hzeta_\out^\theta$ and $\hzeta_\prop^\theta$
correspond to degrees-of-freedom adjustment factors that have appeared
elsewhere in the debiasing
literature~\cite{javanmard2014hypothesis,miolane2021,celentanoMontanariWei2020,bellec2019biasing,bellec2019debiasingIntervals,bellec2022}.
For example, consider~\cref{eq:dof-adjust-emp} in the case that
$\action_i = 1$ for all $i$ and we fit the model with unpenalized
least squares ($n > p$) or the Lasso with least-squares loss.  In
these cases, $\ddot\ell_\out(\eta;\eta_{\prop,i},\action_i,y_i) = 1$,
whence the first equation in~\cref{eq:dof-adjust-emp} states
$\hzeta_\out^\theta = 1 / (\hzeta_\out^\eta + 1)$.  For least-squares,
the second equation in~\cref{eq:dof-adjust-emp} reads
$\hzeta_\out^\eta = p/(n\hzeta_\out^\theta)$.  Solving for
$\hzeta_\out^\eta$ and $\hzeta_\out^\theta$, we have $\hzeta_\out^\eta
= (p/n)/(1-p/n)$ and $\hzeta_\out^\theta = 1-p/n$.  In the case of the
Lasso, we interpret $\nabla^2 \lambda \| \hbtheta_\out \|_1 =
\diag(s_j)$, with $s_j = \infty$ if $\htheta_{\out,j} = 0$ and $0$
otherwise, and get $\hzeta_\out^\eta = \| \htheta_\out
\|_0/(n\hzeta_\out^\theta)$.  Solving gives $\hzeta_\out^\eta =
(\|\hbtheta_\out\|_0/n)/(1-\|\hbtheta_\out\|_0/n)$ and
$\hzeta_\out^\theta = 1-\|\hbtheta_\out\|_0/n$.  In both these cases
$\hzeta_\out^\theta = 1 - \dfhat_\out/n$, where $\dfhat_\out =
\Tr\Big(\frac{\de}{\de \by} \bX \hbtheta_\out \Big)$ is the
degrees-of-freedom in the Steinian sense~\cite{stein1981}.
Equation~\eqref{eq:dof-adjust-emp} generalizes this definition to a
wide class of loss functions and penalties and to linear models with
missing outcomes and binary GLMs.  To be clear, alternative
generalizations are possible~\cite{bellec2022}.  We make no claims
regarding the relative merits of the choice~\eqref{eq:dof-emp-def} and
these alternatives.  We have adopted the
definition~\eqref{eq:dof-emp-def} because is best suited to our proof
techniques.


\subsection{Two classes of debiasing procedures}
\label{sec:procedure-definitions}

With this high-level overview in place, we are equipped to describe
two classes of debiasing procedures, depending on whether or not the
propensity score is known.  To be clear, our primary contribution is
to develop an estimator that does \emph{not} require knowledge of the
propensity function, which we describe
in~\Cref{SecUnknownProp}. Nonetheless, in order to develop intuition,
it is helpful to analyze an estimator in which the bias estimates are
based on this knowledge, which we do in~\Cref{SecKnownProp} to follow.

Before presenting these debiasing procedures, it is worth emphasizing
the difficulty of achieving consistency for $\mu_\out$ using plug-ins
to the G-computation \eqref{eq:db-G-computation}.  There are various
natural estimators that we have found to fail.  For example, we
studied an unmodified version of the debiasing procedures presented in
the
papers~\cite{javanmard2014hypothesis,miolane2021,celentanoMontanariWei2020,bellec2019biasing,bellec2019debiasingIntervals,bellec2022}
(which, by necessity, is computed using only the units with observed
outcomes):
\begin{subequations}
\label{eq:db-naive}
  \begin{align}
    \hbias_{\out,0} & \defn \bmu_{\sx}^\top
    \frac{\bSigma^{-1}\bX^\top\big(\ba \odot \bw \odot (\by -
      \htheta_{\out,0}\ones - \bX
      \hbtheta_\out)\big)}{n\hzeta_\out^\theta}, \\
    \hbbias_{\out} & \defn - \frac{\bSigma^{-1}\bX^\top\big (\ba
      \odot \bw \odot (\by - \htheta_{\out,0}\ones - \bX
      \hbtheta_\out)\big)}{n\hzeta_\out^\theta}.
  \end{align}
\end{subequations}
We considered these bias estimates under two weighting schemes
in~\cref{eq:outcome-fit}.  The first takes weights $w_i = 1$ and the
second takes weights $w_i = 1 / \pi(\theta_{\prop,0} + \<
\bx_i,\btheta_\prop\>)$.  In Appendix~\ref{SecInconsistency}, we show
that under both weighting schemes and certain additional conditions,
these bias estimates lead to provably inconsistent estimates of
$\mu_\out$.  The former weighting scheme is our primary interest in
this paper because we assume no knowledge of the propensity model.
The latter is an oracle procedure inspired by M-estimation approaches
that reweight the loss by the inverse propensity score
(e.g.,~\cite{arkhangelsky2021doublerobust}).  It fails despite using
oracle knowledge of the propensity score in a natural way.  We also
considered using the base estimates~\eqref{eq:outcome-fit} as
plug-ins.  In Appendix~\ref{SecInconsistency}, we show that these lead
to inconsistent estimates of $\mu_\out$ as well.

In fact, the conditions under which we prove inconsistency are simple
enough to be summarized here.  It occurs as soon as the signal
strength $\| \btheta_\prop \|_2$ and its overlap with the outcome
model parameter $\< \btheta_\prop , \btheta_\out \>$ are positive and
bounded away from zero.  The signal strength lower bound states that
the missingness indicator is not independent (or nearly independent)
of the covariates.  The overlap lower bound states that the
missingness indicator is sufficiently confounded with the outcome.  Of
course, this type of confounding is exactly what motivates the
analysis of the current paper.  For simplicity, we only prove
inconsistency when ridge regression is used and features are
independent ($\bSigma = \id_p$), although we expect it to occur more
generally.  Precise statements can be found in
Appendix~\ref{SecInconsistency}.  


\subsubsection{Oracle augmented shifted-confounder weighting}
\label{SecKnownProp}

The bias of the base estimates can be estimated by a weighted sample
averages as follows:
\begin{equation}
\label{eq:db-with-offset-explicit-intro}
  \begin{aligned}
    \hbias_{\out,0} & \defn \frac{1}{\numobs} \sum_{i=1}^n \omega_{0,i}^\orc a_i
    (\htheta_{\out,0} + \< \bx_i , \hbtheta_\out \> - y_i), \qquad
    & \hbias_{\out} & \defn \frac{1}{\numobs} \sum_{i=1}^n \bomega_i^\orc a_i
    (\htheta_{\out,0} + \< \bx_i , \hbtheta_\out \> - y_i),
  \end{aligned}
\end{equation}
where 
\begin{equation}
    \omega_{0,i}^\orc
        \defn
        - \frac{\bmu_{\sx,\cfd}^\top \bSigma_{\cfd}^{-1} \bx_i }{\hzeta_\out^\theta},
    \qquad
    \bomega_i^\orc
        \defn
        \frac{\bSigma_{\cfd}^{-1} \bx_i }{\hzeta_\out^\theta},
\end{equation}
and $\bmu_{\sx,\cfd} = \E[\bx \mid \action = 1]$ and $\bSigma_{\cfd} =
\Var(\bx \mid d=1)$ are the mean and variance of the features
conditional on the outcome being observed, and $\hzeta_\out^\theta$ is
the previously described adjustment~\eqref{eq:dof-adjust-emp} for the
degrees-of-freedom.  In matrix form, these estimates take the form
\begin{equation}
\label{eq:orc-ASCW-matrix-form}
  \begin{aligned}
    \hbias_{\out,0} & = \bmu_{\sx,\cfd}^\top
    \frac{\bSigma_{\cfd}^{-1}\bX^\top\big(\ba \odot (\by -
      \htheta_{\out,0}\ones - \bX
      \hbtheta_\out)\big)}{n\hzeta_\out^\theta},
    \quad&
    \hbbias_{\out} & = - \frac{\bSigma_{\cfd}^{-1}\bX^\top\big (\ba
      \odot (\by - \htheta_{\out,0}\ones - \bX
      \hbtheta_\out)\big)}{n\hzeta_\out^\theta}.
  \end{aligned}
\end{equation}
Even if the confounder distribution is known, the weights
$\omega_{i,0}^\orc$, $\bomega_i^\orc$ depend on oracle knowledge of
the propensity model.  Indeed, an elementary computation shows that
\begin{equation}
\label{eq:cfd-mean-variance}
\begin{gathered}
    \bmu_{\sx,\cfd} = \bmu_{\sx} + \alpha_1
    \bSigma^{1/2}\btheta_\prop, \qquad \bSigma_{\cfd} = \bSigma +
    (\alpha_2 - \alpha_1^2) \bSigma^{1/2} \btheta_\prop
    \btheta_\prop^\top \bSigma^{1/2},
\end{gathered}
\end{equation}
where
\begin{equation}
\label{eq:alpha-12}
    \alpha_1 \defn \frac{\E[\pi'(\eta_\prop)]}{\E[\pi(\eta_\prop)]},
    \qquad \alpha_2 \defn
    \frac{\E[\pi''(\eta_\prop)]}{\E[\pi(\eta_\prop)]}.
\end{equation}
Thus, the conditional first and second moments of confounder
distribution are shifted by the missingness mechanism.  The weights
$\omega_{i,0}^\orc$, $\bomega_i^\orc$ account for this confounder
shift.  We refer to the correction of the base estimates with the bias
estimates \cref{eq:db-with-offset-explicit-intro} as the
\textit{oracle augmented shifted-confounder weighting}, or more
succinctly, \textit{oracle ASCW}.

The estimates \eqref{eq:db-with-offset-explicit-intro} are equivalent
to the estimates \eqref{eq:db-naive} with $\bmu_{\sx,\cfd}$ and
$\bSigma_{\cfd}$ in place of $\bmu_{\sx} = \E[\bx]$ and $\bSigma =
\Var(\bx)$, respectively.  Because the data used to fit
$(\htheta_{\out,0},\hbtheta_\out)$ only includes units with $\action =
1$, it is natural to replace the unconditional mean and variance of
the features with their conditional mean and variance.  For this
reason, equation~\eqref{eq:db-with-offset-explicit-intro} rather than
equation~\eqref{eq:db-naive} is arguably the natural generalization of
the debiasing methods in the
papers~\cite{javanmard2014hypothesis,bellec2019debiasingIntervals,celentanoMontanariWei2020}.


\subsubsection{Empirical shifted-confounder augmentation}
\label{SecUnknownProp}

When the confounder mean $\bmu_{\sx}$ and propensity model parameters
$\btheta_\prop, \alpha_1, \alpha_2$ are not known, one might consider
computing the shifted-confounder weights with plug-in estimates for
$\btheta_\prop, \alpha_1, \alpha_2$.  We shall see below that it is
possible to consistently estimate $\alpha_1$ and $\alpha_2$ in the
present setting.  Nevertheless, as we have already stated, consistent
estimation of $\btheta_\prop$ in the $\ell_2$-norm is
impossible~\cite{Barbier_2019}. Under proportional asymptotics with
inconsistent estimates of $\btheta_\out$, all existing guarantees (of
which we are aware) for the debiased estimate require that $\bSigma$
be estimated consistently in operator
norm~\cite{bellec2019debiasingIntervals}; such operator-norm
consistent estimation is possible only under very strong structural
assumptions on $\bSigma$ (e.g., see the
papers~\cite{bickelLevina2008,karoui2008,caiZhangZhou2010}).  In the
present setting, because $\ell_2$-norm consistent estimation of
$\btheta_\prop$ is not possible, operator norm consistent estimation
of $\bSigma_{\cfd}$ is not possible.  Without knowledge of
$\btheta_\prop$, a new approach is required.

Our approach is based on an expansion of the
estimates~\eqref{eq:orc-ASCW-matrix-form}.  Using the
Sherman-Morrison-Woodbury formula, we can write $\bSigma_{\cfd}^{-1} =
\bSigma^{-1} -
\sc_{\bSigma}\bSigma^{-1/2}\btheta_\prop\btheta_\prop^\top
\bSigma^{-1/2}$, where $\sc_{\bSigma} \defn (\alpha_2-\alpha_1^2)/(1 +
(\alpha_2-\alpha_1^2)\|\btheta_\prop\|_{\bSigma}^2)$.  With this
notation, the bias estimates
in~\cref{eq:db-with-offset-explicit-intro} can be expanded as
\begin{equation}
\label{eq:db-cfd-expanded}
\begin{gathered}
    \hbias_{\out,0} = \dbAdj_{01} +
    \dbAdj_{02}, \qquad 
    \hbias_\out =
    -\frac{\bSigma^{-1}\bX^\top\big (\ba \odot (\by -
      \htheta_{\out,0}\ones - \bX
      \hbtheta_\out)\big)}{n\hzeta_\out^\theta} 
    + 
    \dbAdj_1\,
    \bSigma^{-1/2}\btheta_\prop,
\end{gathered}
\end{equation}
where
\begin{equation}
\label{eq:dbAdj}
\begin{gathered}
    \dbAdj_{01} \defn  \frac{\big\< \bmu_{\sx,\cfd} ,
      \bSigma^{-1}\bX^\top\big(\ba \odot (\by - \htheta_{\out,0}\ones
      - \bX \hbtheta_\out)\big)\big\>}{n\hzeta_\out^\theta},
    \\ \dbAdj_{02} \defn -\sc_{\bSigma}
    \frac{\big\<\bmu_{\sx,\cfd},\bSigma^{-1/2}\btheta_\prop\big\>\big\<\bSigma^{-1/2}\btheta_\prop,\bX^\top\big(\ba
      \odot (\by - \htheta_{\out,0}\ones - \bX
      \hbtheta_\out)\big)\big\>}{n\hzeta_\out^\theta}, \\ \dbAdj_1
    \defn \sc_{\bSigma}
    \frac{\big\<\bSigma^{-1/2}\btheta_\prop,\bX^\top\big(\ba \odot
      (\by - \htheta_{\out,0}\ones - \bX
      \hbtheta_\out)\big)\big\>}{n\hzeta_\out^\theta}.
\end{gathered}
\end{equation}
Our main insight is that the debiasing procedure is successful
provided that we use as plug-ins: \textit{(i)} any consistent
estimates of the scalar quantities $\dbAdj_{01}, \dbAdj_{02},
\dbAdj_1$, and \textit{(ii)} an appropriately debiased (but not
necessarily consistent) estimate of the high-dimensional parameter
$\btheta_\prop$.  Explicitly, we use
\begin{equation}
\label{eq:SCA-bias-hat}
\begin{gathered}
    \hbias_{\out,0} \defn \widehat{\dbAdj}_{01} +
    \widehat{\dbAdj}_{02}, \\ 
    \hbias_\out =
    -\frac{\bSigma^{-1}\bX^\top\big (\ba \odot (\by -
      \htheta_{\out,0}\ones - \bX
      \hbtheta_\out)\big)}{n\hzeta_\out^\theta} 
      + 
      \widehat{\dbAdj}_1\,
    \bSigma^{-1/2}\hbtheta_\prop^\de.
\end{gathered}
\end{equation}
We provide consistency guarantees for $\hmu_\out$ and approximate
normality guarantees for $\hbtheta_\out^\de$ for two constructions of
the debiased estimate $\hbtheta_\prop^\de$ and any consistent
estimates $\widehat{\dbAdj}_{01}$, $\widehat{\dbAdj}_{02}$,
$\widehat{\dbAdj}_1$ of $\dbAdj_{01}$, $\dbAdj_{02}$, $\dbAdj_1$.
In~\Cref{sec:adjustment-construction}, we give particular
constructions of $\widehat{\dbAdj}_{01}$, $\widehat{\dbAdj}_{02}$, and
$\widehat{\dbAdj}_1$ that are consistent, although any construction
achieving consistency suffices.

We consider the following two estimates of $\hbtheta_\prop^\de$:
\begin{equation}
\label{eq:prop-db-option}
\begin{aligned}
    &\textbf{Moment method:} \qquad& \hbtheta_\prop^\de &=
  \frac{1}{\hbeta} \bSigma^{-1/2} \big( \hbmu_{\sx,\cfd}-\hbmu_{\sx}
  \big), \quad \mbox{and} \\
& \textbf{M-estimation:} \qquad& \hbtheta_\prop^\de &= \frac{1}{\hbeta}
  \Big( \hbtheta_\prop + \frac{\bSigma^{-1}\bX^\top \nabla_{\bmeta}
    \ell_\prop (\htheta_{\prop,0}\ones + \bX
    \hbtheta_\prop;\ba)}{n\hzeta_\prop^\theta} \Big),
\end{aligned}
\end{equation}
where $\hbtheta_\prop$ is the base estimate~\eqref{eq:propensity-fit}, $\hbmu_{\sx,\cfd} \defn \frac{1}{n_1} \sum_{i=1}^n a_i \bx_i$,
$\hbmu_{\sx} \defn \frac{1}{n} \sum_{i=1}^n \bx_i$,
and
$n_1 = \sum_{i=1}^n a_i$.
The factor $\hbeta$ corrects for the shrinkage that occurs in high-dimensional binary regression models \cite{surCandes2019,yadlowskyYun2021},
and its definition depends on whether the moment method or M-estimation is used to construct $\hbtheta_\prop^\de$ in \eqref{eq:prop-db-option}.

The first construction of $\hbtheta_\prop^\de$ is based on the
identity $\bmu_{\sx,\cfd} - \bmu_{\sx} = \alpha_1 \bSigma^{1/2}
\btheta_\prop$, which is a straightforward consequence of Gaussian
integration by parts.  In this case, we set $\hbeta$ to be any
consistent estimate of $\alpha_1$.  The construction of one such
estimate is given in \Cref{sec:adjustment-construction}.

The second construction of $\hbtheta_\prop^\de$ is based on prior work
on debiasing penalized binary-outcome regression developed
\cite{surCandes2019,yadlowskyYun2021,bellec2022}.  This prior work
establishes that the debiased estimate
$\hbtheta_\prop+\frac{\bSigma^{-1}\bX^\top \nabla_{\bmeta} \ell_\prop
  (\htheta_{\prop,0}\ones + \bX
  \hbtheta_\prop;\ba)}{n\hzeta_\prop^\theta}$ is centered not on the
true parameter $\btheta_\prop$, as it would be in linear outcome
models, but on a shrunken version of the true parameter $\beta
\btheta_\prop$.  These works also provide several methods for
consistently estimating this shrinkage factor.  Our method succeeds if
$\hbeta$ is any consistent estimate of this shrinkage factor.  The
construction of one such estimate is given
in~\Cref{sec:adjustment-construction}.


\section{Theoretical guarantees}
\label{SecTheory}

In this section, we state and discuss a number of theoretical results
associated with our estimators.  In \Cref{SecAssumptions}, we state
the assumptions under which our results hold.
In~\Cref{sec:consistency}, we state our consistency result for
estimation of the population mean $\mu_\out$.  In
\Cref{sec:consistency}, we state our result regarding the approximate
normality (in a certain sense) of the estimate $\hbtheta_\out^\de$.


\subsection{Assumptions}
\label{SecAssumptions}

Our consistency result holds as $n,p \rightarrow \infty$
simultaneously.  It holds for any sequence of models that satisfy
the following constraints, each of which are defined by a pair
of constants $C, c > 0$ that do not depend on $(n,p)$.
\begin{description}
    \item{\textbf{Assumption A1}}

    \begin{enumerate}[label=(\alph*)]
        \item The sampling rate $n/p$ lies in the interval $(c, C)$.
        \item The noise variance $\sigma^2$ in the linear model lies
          in the interval $(c, C)$.

        \item The feature covariance is strictly positive definite:
          $\bSigma \succ 0$.

        \item The link function $\pi: \reals \rightarrow (0,1)$ is
          twice-differentiable, strictly increasing, with first and
          second derivatives bounded in absolute value by $C$.
          Moreover, $1-c > \E[\pi(\theta_{\prop,0} + \< \bx ,
            \btheta_{\prop}\>)] > c > 0$.  Moreover, for $\eta > 0$,
          $1 - Ce^{-c\eta} > \pi(\eta)$ and $\pi(-\eta) >
          Ce^{-c\eta}$.

        \item The mean of the features has bounded
          $\bSigma^{-1}$-norm: $\| \bmu_{\sx} \|_{\bSigma^{-1}} < C$.

        \item The size of the signals and offsets are bounded above in
          the $\bSigma$-norm: $|\theta_{\out,0}|,|\theta_{\prop,0}| <
          C$, $\| \btheta_{\out} \|_{\bSigma}, \| \btheta_{\prop}
          \|_{\bSigma} < C$.

        \item The penalties are differentiable, strongly smooth, and
          strongly convex in the $\bSigma$-metric: $\nabla
          \Omega_{\out}(\bv)$ and $\nabla \Omega_{\prop}(\bv)$ exist
          for all $\bv$, and $\bv \mapsto
          \Omega_{\out}(\bSigma^{-1/2}\bv)$ and $\bv \mapsto
          \Omega_{\prop}(\bSigma^{-1/2}\bv)$ are $c$-strongly convex
          and $C$-strongly smooth.  The penalties have minimizers
          $\bv_\out$ and $\bv_\prop$ bounded by $\| \bv_\out
          \|_{\bSigma},\|\bv_\prop\|_{\bSigma} < C$.  Their Hessian is
          $C\sqrt{n}$-Lipschitz in Frobenius norm:

        \item The weight function is uniformly bounded above and
          below: $0 < c < w(h) < C < \infty$ for all $h$.  Moreover,
          $1/w(h)$ is differentiable and $C$-Lipschitz.

        \item The loss function $\ell_{\prop}$ satisfies the following
          properties:

        \begin{itemize}

            \item It is $c$-strongly convex and three-times
              differentiable in the linear predictor, with its second
              and third derivatives bounded in absolute value by $C$.

            \item It is non-negative.

            \item It is informative in the sense that $\partial_\eta
              \ell_\prop(\eta;0) - \partial_\eta \ell_\prop(\eta;1)
              \geq c > 0$ for all $\eta \in \reals$.

            \item It is bounded and has bounded derivatives at $\eta =
              0$: $\ell_\prop(0;0),\ell_\prop(0;1) < C$, $-C <
              \partial_\eta \ell_\prop(0;1) < -c < 0$ and $C >
              \partial_\eta \ell_\prop(0;0) \geq c > 0$.

        \end{itemize}

    \end{enumerate}

\end{description}
We sometimes refer to the list of constants appearing in Assumption A1
by $\cPmodel$.  The approximate normality result also holds for models
satisfying assumption A1, with approximation errors that vanish as
$n,p \rightarrow \infty$.

Assumption A1(a) requires $n$ and $p$ grow proportionally.  We expect
only a lower bound on $n/p$ is necessary, but our current proof
requires an upper bound as well.  Importantly, we allow $n/p < 1$, and
in fact, allow it to be arbitrarily small as long is it does not
vanish in the high dimensional limit.  The accuracy of $\hmu_\out^\de$
necessarily diverges as $\sigma^2$ diverges, justifying the upper
bound on $\sigma^2$.  The lower bound on $\sigma^2$ is likely an
artifact of the proof.  Assumption A1(d) imposes strict overlap on
average over the covariates $\bx$.  It implies that a strictly
positive fraction of the units has outcomes observed, and a strictly
positive fraction has outcomes unobserved.  It does not require strict
overlap conditional on the covariates $\bx$, but does constrain the
rate at which overlap decays as the linear predictor grows.  Together
with assumptions A1(e) and (f), it implies that the ``typical'' unit
satisfies a form a strict overlap.  The decay rate is satisfied by
popular link functions, including the logistic-link.  The upper bounds
on $|\theta_{\out,0}|,|\theta_{\prop,0}|,\| \btheta_{\out}
\|_{\bSigma},\| \btheta_{\prop}\|_{\bSigma}$ together with the lower
bound on $\sigma^2$ imply that the data is noisy: the fraction of
variation in the outcome explained by the covariates (i.e., $R^2$) is
necessarily bounded away from 1.  Our constraints permit (but do not
require) that it also be bounded away from 0.  Assumption A1(g)
requires strong convexity and strong smoothness of $\Omega_\out$ and
$\Omega_\prop$ in the $\bSigma$-metric.  Generalizations to penalties
that are not strongly smooth (e.g.,~elastic net) or strongly convex
(e.g.,~Lasso) are of substantial interest.  It is likely that
specialized analyses can generalize our results to these cases, as in
the
papers~\cite{miolane2021,celentanoMontanariWei2020,bellec2019debiasingIntervals}.
We make the current assumptions to facilitate our proofs.  When the
singular values of $\bSigma$ are all of order $O(1)$, assumption A1(g)
includes ridge regression with the choice of regularization parameter
that minimizes prediction error and which is chosen by cross
validation (see, for example,
\cite{karoui2013asymptotic,miolane2021,hastieMontanariRossetTibshirani2022}).
Bounds on the weight $w(h)$ is required to make Assumption A1(g)
meaningful.  If strict overlap is satisfied ($\pi(\eta) > c$ for all
$\eta$), then assumption A1(h) is satisfied by the inverse propensity
weights $w(h) = 1/\pi(h)$.  Assumption A1(i) includes the
least-squares loss.  Because of the strong convexity constraint, it
does not include the logistic loss.  We expect a generalization to
logistic loss is possible at the cost of additional technical
difficulty in the proofs.  Our primary interest is to show that
constructing a consistent estimate of $\mu_\out$ is possible, so we
did not pursue extending the result to logistic loss.

Throughout the paper, we always use $C,c > 0$, possibly with
subscripts or superscripts, to denote constants that depend only on
the constants in $\cPmodel$, even though we do not explicitly state
this fact on each occurrence.


\subsection{Consistency under proportional asymptotics}
\label{sec:consistency}
  
We can now state our main consistency result for the population mean.
Within our framework, we say that an estimate $\widehat{\dbAdj}$ is
\emph{consistent} for $\dbAdj$ if $\widehat{\dbAdj} - \dbAdj \gotop 0$
\mbox{as $(n, p) \rightarrow \infty$.}
\begin{theorem}
\label{thm:pop-mean}
Under Assumption A1, outcome regression \eqref{eq:db-G-computation} with empirical SCA plug-ins $(\htheta_{\out,0}^\de,\hbtheta_\out^\de)$ achieves consistency for the population mean
\begin{align}
|\hmu_\out^\de - \mu_\out | & \gotop 0 \qquad \mbox{as $(n, p)
  \rightarrow \infty$,}
\end{align}
provided the bias corrections \eqref{eq:SCA-bias-hat} are computed using consistent estimates of  $\dbAdj_{01}$, $\dbAdj_{02}$, 
$\dbAdj_1$ and either the moment method or M-estimation approach \eqref{eq:prop-db-option} for the debiased propensity parameter.
\end{theorem}

\noindent Theorem \ref{thm:pop-mean} also applies to the oracle ASCW bias corrections \eqref{eq:db-cfd-expanded} because they use exact knowledge of the $\dbAdj$ quantities and propensity model parameter.
In~\Cref{sec:adjustment-construction}, we provide
constructions of $\widehat{\dbAdj}_{01}$, $\widehat{\dbAdj}_{02}$,
$\widehat{\dbAdj}_1$, and $\hbeta$ that are consistent under
assumption A1.  Thus, \Cref{thm:pop-mean} together with these
constructions provides a consistent estimate of the population mean
under a proportional asymptotics with $n/p$ possibly (but not
necessarily) less than 1.  Other constructions
$\widehat{\dbAdj}_{01}$, $\widehat{\dbAdj}_{02}$,
$\widehat{\dbAdj}_1$, and $\hbeta$ are likely possible.  We
prove~\Cref{thm:pop-mean} in~\Cref{sec:proof-main-theorems}.

The propensity model parameter plays a crucial role in the oracle ASCW
debiased estimate because it appears in $\bmu_{\sx,\cfd}$ and
$\bSigma_{\cfd}$ (see~\cref{eq:cfd-mean-variance}).  Estimation of the
propensity model parameter plays a crucial role in the oracle SCA
debiased estimate via $\hbtheta_\prop^\de$ and, as to be seen
in~\Cref{sec:adjustment-construction}, via the estimates
$\widehat{\dbAdj}_{01}$, $\widehat{\dbAdj}_{02}$ and
$\widehat{\dbAdj}_1$.  As we have summarized
in~\Cref{sec:procedure-definitions} and will state precisely in
Appendix~\ref{SecInconsistency}, some natural approaches that avoid
propensity modeling are provably inconsistent in relevant settings.
We are unaware of successful approaches that avoid propensity
modeling altogether.  This is in stark contrast to the setting $n >
p$, where G-computation with OLS plug-ins achieve consistency without
propensity modeling~\cite{yadlowsky2022,sur2022}.

An important feature of~\Cref{thm:pop-mean} is that, unlike the
approaches in past work~\cite{yadlowsky2022,sur2022}, it avoids sample
splitting: the G-computation~\eqref{eq:db-G-computation} averages over
the same data used to fit $\htheta_{\out,0}^\de$ and
$\hbtheta_\out^\de$ as well as the propensity estimate
$\hbtheta_\prop^\de$ that enters into the construction of the
empirical SCA estimate.  Cross-fitting, which cuts in half the amount
of data used to fit each nuisance parameter, can degrade the nuisance
error by a constant factor.  In consistency regimes in which the
classical semi-parametric efficiency lower bound can be achieved, the
nuisance errors have a negligible impact on the estimation error of
the population mean, whence this potential constant inflation has no
impact on asymptotic
variance~\cite{chernozhukov2018,newey2018crossfitting}.  In contrast,
in the present setting the nuisance errors may have a non-negligible
impact on the estimation error of the population mean, whence
cross-fitting may potentially impact efficiency.  In this and related
settings, estimators that avoid sample splitting sometimes outperform
estimators that use sample splitting in simulation
(see~\Cref{fig:variance-well-specified-outcome,fig:variance-regularization,fig:outcome-misspecification},
or, for example, the papers~\cite{guoWangCaiLi2019,maCaiLi2021}).  We
do not perform an efficiency comparison of various sample splitting
strategies (indeed, our current theory does not precisely quantify the
fluctuations of $\hmu_\out^\de$), but we believe it is important to
develop methods and theory for which sample splitting is not required.
Some previous work has studied conditions under which sample splitting
is not required, but this analysis is restricted to consistency
regimes~\cite{chenSyrgkanisAustern2022}.

We in fact prove a more quantitative form of consistency than stated
in~\Cref{thm:pop-mean}.  For the constructions of
$\widehat{\dbAdj}_{01}$, $\widehat{\dbAdj}_{02}$,
$\widehat{\dbAdj}_1$, and $\hbeta$ provided
in~\Cref{sec:adjustment-construction}, the proof
of~\Cref{thm:pop-mean} establishes a non-asymptotic exponential tail
bound on $\hmu_\out^\de - \mu_\out$ depending only on the constants
$C,c > 0$ appearing in Assumption A1 (see
Appendix~\ref{sec:db-proofs}).  Although this bound provides a more
quantitative statement than in~\Cref{thm:pop-mean}, we do not expect
that it is tight in the rate of convergence $\hmu_\out^\de - \mu_\out
\gotop 0$ or in its dependence on the constants in Assumption A1.
Identifying the correct rate of convergence or dependence on these
constants may be beyond the capacity of the Gaussian comparison proof
technique used here, which often fails to identify optimal
rates~\cite{celentanoMontanariWei2020}.  The simulations
in~\Cref{sec:simulations} suggest the rate of convergence is
$\hmu_\out^\de - \mu_\out = O_p(n^{-1/2})$ (ignoring dependence on
other constants).


\subsubsection{Normality of debiased estimates}
\label{sec:db-review}

A main focus of the debiasing literature is normal inference on the
linear model coefficients
(e.g.,~\cite{zhang2014confidence,buhlmann2013statistical,javanmard2014confidence,javanmard2014hypothesis,javanmard2018debiasing,miolane2021,celentanoMontanariWei2020,bellec2019biasing,bellec2019debiasingIntervals}).
We augment~\Cref{thm:pop-mean} by establishing the approximate
normality and unbiasedness (in a certain sense)
 of
$\hbtheta_\out^\de$
with a fully empirical standard error.
The coordinate-wise standard error is given by $\htau \bSigma_{j|-j}^{-1/2} $,
where $\htau^2 \defn \shat_\out^2 - 2\, \widehat{\dbAdj}_1\,
\shat_{\out\circ} + (\widehat{\dbAdj}_1)^2\, \shat_\circ^2$,
and $\shat_\out^2$, $\shat_{\out\circ}$, and $\shat_{\circ}^2$ are given by the correlation between certain empirical influence functions:
$
\shat_\out^2 \defn \frac{1}{\numobs} \| \hbi_\out \|^2 $, $
\shat_{\out\circ}^2 \defn \frac{1}{n\hbeta} \< \hbi_\out , \hbi_\circ \>
$, and $ \shat_{\circ}^2 \defn \frac{1}{n\hbeta^2} \| \hbi_\circ \|^2 $, where
$\hbi_\out = \frac{\ba \odot (\by - \htheta_{\out,0}\ones - \bX
  \hbtheta_\out)}{\hzeta_\out^\theta}$ and
    \begin{equation}
        \hbi_\circ =
            \begin{cases}
                \frac{\ba}{\hpi} - \ones \quad&\text{if the moment
                  method in equation~\eqref{eq:prop-db-option} is used},
                \\[5pt]
                -\frac{\nabla_{\bmeta}\ell_\prop(\htheta_{\prop,0}\ones
                  + \bX \hbtheta_\prop;\ba)}{\hzeta_\prop^\theta}
                \quad&\text{if M-estimation in
                  equation~\eqref{eq:prop-db-option} is used},
            \end{cases}
    \end{equation}
    with $\hpi = \frac{1}{\numobs} \sum_{i=1}^n \action_i$.
\begin{theorem}
\label{thm:db-normality}
Under Assumption A1, 
the empirical SCA 
estimates $(\htheta_{\out,0}^\de,\hbtheta_\out^\de)$ satisfy
\begin{equation}
  \begin{gathered}
    \htheta_{\out,0}^\de - \theta_{\out,0} \gotop 0 \quad
    \text{and} \quad \frac{1}{p} \sum_{j=1}^p \indic\Bigg\{ \frac{
      \sqrt{n} \big( \htheta_{\out,j}^\de - \theta_{\out,j} \big)
    }{ \htau \bSigma_{j|-j}^{-1/2} } \leq t \Bigg\} \gotop
    \P\big(\normal(0,1) \leq t\big) \qquad \mbox{as $(n, p)
      \rightarrow \infty$,}
  \end{gathered}
\end{equation}
provided the bias corrections \eqref{eq:SCA-bias-hat} are computed using consistent estimates of  $\dbAdj_{01}$, $\dbAdj_{02}$, 
$\dbAdj_1$ and either the moment method or M-estimation approach \eqref{eq:prop-db-option} for the debiased propensity parameter.
\end{theorem}
\noindent As in the case of Theorem \ref{thm:pop-mean},
the oracle ASCW bias corrections \eqref{eq:db-cfd-expanded} satisfy the criteria of Theorem \ref{thm:db-normality} because they use exact knowledge of the $\dbAdj$ quantities.
Because they use the true propensity model parameter rather than propensity model estimates, the oracle ASCW standard error instead must take $\htau^2 = \shat_\out^2$.
We prove Theorem \ref{thm:db-normality} and describe the oracle ASCW standard error in~\Cref{sec:proof-main-theorems}. 

\Cref{thm:db-normality} establishes a form of
$\numobs^{-1/2}$-asymptotic normality that holds on average across
coordinates.  We note that these types of average-coordinate
guarantees are common in the literature on high-dimensional
asymptotics
(e.g.,~\cite{javanmard2014hypothesis,miolane2021,celentanoMontanariWei2020}).
The standard errors $\htau\bSigma_{j|-j}^{-1/2}/\sqrt{n}$ are fully
empirical (assuming knowledge of the feature variance $\bSigma$).
They agree with the standard errors one would get from OLS in a
low-dimensional asymptotics with outcome noise $\htau^2$.  In general,
this noise variance is larger than $\sigma^2$ by a constant factor
that does not decay as $n,p \rightarrow \infty$.
\Cref{thm:db-normality} implies that if one constructs confidence
intervals for each coordinates of the unknown parameter vector
$\btheta_\out$ based on normality theory and the empirical effective
noise level $\htau^2$, then the empirical coverage of these intervals
across coordinates is, with high probability, close to the nominal
level.  Usually, the statistician is also interested in a guarantee
that holds for a single, prespecified coordinate $\theta_{\out,j}$.
Using alternative proof techniques to those used here, Bellec and
Zhang~\cite{bellec2019debiasingIntervals} establish coordinate-wise
normality for debiased estimates in linear models with fully-observed
outcomes under a large class of penalties.  Whether these techniques
are able to establish similar results for the oracle ASCW or empirical
SCA debiased estimates, or appropriate modifications of them, is a
promising direction for future work.

\Cref{thm:db-normality} states that the oracle ASCW and empirical SCA
debiased estimates behave, in a certain sense, like OLS estimates.
Given the success of G-computation with OLS estimates as a plug-in, as
studied by Yadlowsky~\cite{yadlowsky2022} and Jiang et
al.~\cite{sur2022}, it is perhaps not surprising that $G$-computation
with these debiased estimates as a plug-in is consistent for the
population mean.

As with the convergence $\hmu_\out^\de - \mu_\out \gotop 0$
in~\Cref{thm:pop-mean}, we prove a more quantitative form of the
consistency and coverage guarantees of~\Cref{thm:db-normality}.  For
the constructions of $\widehat{\dbAdj}_{01}$, $\widehat{\dbAdj}_{02}$,
$\widehat{\dbAdj}_1$, and $\hbeta$ provided
in~\Cref{sec:adjustment-construction}, the proof
of~\Cref{thm:db-normality} establishes non-asymptotic exponential tail
bounds depending only on the constants $C,c > 0$ appearing in
Assumption A1 (see~\Cref{sec:proof-main-theorems}).  As before, we do
identifying optimal rates and dependence on constants is left to
future work.


\section{Explicit constructions of empirical adjustments}
\label{sec:adjustment-construction}

\Cref{thm:pop-mean,thm:db-normality} hold for any consistent estimates
of the $\dbAdj$-quantities, along with any consistent estimate of
$\alpha_1$ or of a shrinkage factor that we define below.  In this
section, we explicitly provide one such construction of such
consistent estimates.

\subsection{Estimating the signal strength, offset, and covariance spike}
\label{sec:summary-stat-estimates}

Our construction of consistent estimates for the $\dbAdj$-quantities
relies on estimates of certain model parameters, including:
\begin{itemize}
\item the probability of observing the outcome marginalized over the
  covariates $\barpi \defn \E[\pi(\bx)]$
\item the norm of the feature mean $\gamma_{\bmu} \defn \| \bmu_{\sx}
  \|_{\bSigma^{-1}}$
\item the propensity model signal strength $\gamma_\prop \defn \|
  \btheta_\prop \|_{\bSigma}$
\item the propensity model linear signal strength $\gamma_{\prop*}
  \defn \alpha_1\| \btheta_\prop \|_{\bSigma}$.
\item the propensity model offset $\mu_\prop \defn \< \bmu_\sx ,
  \btheta_\prop \>$
\item and the covariance spike prefactor $\sc_{\bSigma}$. (Recall that
  $\bSigma_{\cfd}^{-1} = \bSigma^{-1} -
  \sc_{\bSigma}\bSigma^{-1/2}\btheta_\prop\btheta_\prop^\top
  \bSigma^{-1/2}$, where $\sc_{\bSigma} \defn (\alpha_2-\alpha_1^2)/(1
  + (\alpha_2-\alpha_1^2)\|\btheta_\prop\|_{\bSigma}^2)$).
\end{itemize}
  The best linear predictor of $a$ given $\bx$ is $\< \bx , \alpha_1
  \btheta_\prop \>$, which is why we refer to $\gamma_{\prop*}$ as the
  ``propensity model linear signal strength.''  We estimate these
  parameters as follows.

First, we define the estimates
\begin{equation}
\begin{gathered}
\hpi \defn \frac{1}{n} \sum_{i=1}^n \action_i, \qquad \hgamma_{\bmu}^2
\defn \Big( \| \hbmu_\sx \|^2 - \frac{p}{n} \Big)_+, \quad \mbox{and}
\quad \hgamma_{\prop*}^2 \defn \Big( \| \hbmu_{\sx,\cfd} - \hbmu_\sx
\|^2 - \frac{p}{n} \frac{1-\hpi}{\hpi} \Big)_+.
\end{gathered}
\end{equation}
Then, we let $\hmu_\prop$ and $\hgamma_{\prop}$ be the unique
solutions (cf.~\Cref{lem:propensity-summary-stats}) to the fixed point
equations
\begin{equation}
\label{eq:prop-offset-and-signal-strenght-estimates}
    \E_{G \sim \normal(0,1)}[\pi(\hmu_\prop + \hgamma_\prop G)] =
    \hpi, \quad \mbox{and} \quad \E_{G \sim
      \normal(0,1)}[G\pi(\hmu_\prop + \hgamma_\prop G)] = \hpi
    \hgamma_{\prop*}.
\end{equation}
Finally, we define the estimates
\begin{equation}
\begin{gathered}
    \halpha_1 = \frac{\E[\pi'(\hmu_\prop + \hgamma_\prop
        G)]}{\E[\pi(\hmu_\prop + \hgamma_\prop G)]}, \qquad \halpha_2
    = \frac{\E[\pi''(\hmu_\prop + \hgamma_\prop G)]}{\E[\pi(\hmu_\prop
        + \hgamma_\prop G)]}, \quad \mbox{and} \quad \hsc_{\bSigma} =
    \frac{\halpha_2 - \halpha_1^2}{1 + (\halpha_2 -
      \halpha_2^2)\hgamma_\prop^2}.
\end{gathered}
\end{equation}
With these definitions, we have the following auxiliary result:
\begin{lemma}
\label{lem:propensity-summary-stats}
Under Assumption A1, with exponentially high probability, there are
unique solutions to the fixed point
relation~\eqref{eq:prop-offset-and-signal-strenght-estimates}.
Moreover, as $n,p \rightarrow \infty$, we have the consistency
relations:
    \begin{equation}
    \begin{gathered}
        \hpi - \barpi \gotop 0, \qquad \hgamma_{\bmu} - \| \bmu_\sx \|
        \gotop 0, \qquad \hgamma_{\prop*} - \alpha_1 \| \btheta_\prop
        \| \gotop 0, \qquad \hmu_\prop - \mu_\prop \gotop 0, \qquad
        \hgamma_\prop - \| \btheta_\prop \| \gotop 0, \\ \halpha_1 -
        \alpha_1 \gotop 0, \qquad \halpha_2 - \alpha_2 \gotop 0,
        \qquad \hsc_{\bSigma} - \frac{\alpha_2 - \alpha_1^2}{1 +
          (\alpha_2 - \alpha_1^2)\| \btheta_\prop\|^2} \gotop 0.
    \end{gathered}
    \end{equation}
\end{lemma}
\noindent We prove~\Cref{lem:propensity-summary-stats}
in~\Cref{sec:hbeta-dbAdj-consistency}.


\subsection{Construction of propensity renormalization factor}
\label{sec:prop-debiasing}

Both the moment method and M-estimation approach for constructing
$\hbtheta_\prop^\de$ in equation~\eqref{eq:prop-db-option} rely on a
renormalization factor $\hbeta$.  For the moment method approach,
\Cref{thm:pop-mean,thm:db-normality} hold provided $\hbeta - \alpha_1
\gotop 0$.  In this case, we take $\hbeta = \halpha_1$ as defined in
equation~\eqref{eq:prop-offset-and-signal-strenght-estimates}, with
consistency given by~\Cref{lem:propensity-summary-stats}.

For the M-estimation approach, we must take $\hbeta$ to be a
consistent estimate of the shrinkage factor $\beta$ describing the
shrinkage of $\hbtheta_\prop+\frac{\bSigma^{-1}\bX^\top
  \nabla_{\bmeta} \ell_\prop (\htheta_{\prop,0}\ones + \bX
  \hbtheta_\prop;\ba)}{n\hzeta_\prop^\theta}$ as an estimate of
$\hbtheta_\prop$.  This shrinkage factor is given by the solution to a
complicated set of non-linear equations that describes the estimator
\eqref{eq:propensity-fit}.  We describe these equations, which we call
\mbox{``the fixed point equations'',}
in~\Cref{sec:exact-asymptotics-body}.  In the unpenalized case,
multiple approaches have been developed to estimate this shrinkage
factor~\cite{surCandes2019,yadlowskyYun2021}.  Our argument is
inspired by the strategy given in the paper~\cite{yadlowskyYun2021},
which solves an empirical version of the fixed point equations.  It is
difficult to build intuition for this construction because the fixed
point equations themselves are difficult to understand.  Nevertheless,
the construction is fully explicit and easy to state.  We begin by
defining
\begin{align*}
    \hbmeta_\prop^{\loo} \defn \hbmeta_\prop + \hzeta_\prop^\eta
    \nabla_{\bmeta} \ell_\prop(\hbmeta_\prop;\ba), \qquad
    \bareta_\prop^\loo = \frac{1}{n} \sum_{i=1}^n
    \heta_{\prop,i}^{\loo}, \quad \mbox{and} \quad
    \shat_{\prop\widehat{\prop}} \defn \frac{1}{n \hpi \halpha_1}
    \sum_{i=1}^n \big(\heta_{\prop,i}^\loo - \bareta_\prop^\loo \big)
    \action_i,
\end{align*}
along with
\begin{align}
\label{eq:pen-logit-shrinage-est}
 \hbeta & \defn \E\Big[ \pi'(G_\prop) \Big(
   \ell_\prop'(\prox\big[\hzeta_\prop^\eta
     \ell_\prop(\,\cdot\,;1)(G_\prop^\loo);1) -
     \ell_\prop'(\prox\big[\hzeta_\prop^\eta
       \ell_\prop(\,\cdot\,;)(G_\prop^\loo);0) \big]) \Big) \Big],
\end{align}
where the expectation is taken over the pair
\begin{equation*}
  (G_\prop, G_\prop^\loo) \sim \normal \left(
  \begin{pmatrix}
    \hmu_\prop \\ \bareta_\prop^\loo
  \end{pmatrix},
  \begin{pmatrix}
    \hgamma_\prop^2 & \shat_{\prop\widehat{\prop}}
    \\ \shat_{\prop\widehat{\prop}} & \| \hbtheta_\prop \|^2
  \end{pmatrix}
  \right).
\end{equation*}

For both the moment method and $M$-estimation approaches for
constructing $\hbtheta_\prop^\de$ (cf.~\cref{eq:prop-db-option}),
asymptotic normality holds in a sense similar
to~\Cref{thm:db-normality}.  This is not the focus of the present
paper, but is proved in the course of establishing our main results.
See~\Cref{thm:exact-asymptotics-body}
in~\Cref{sec:exact-asymptotics-body} for precise statements.

\subsection{Construction of debiasing adjustments}
\label{sec:pop-cor-with-adjust}

We construct $\widehat{\dbAdj}_{01}$, $\widehat{\dbAdj}_{02}$, and
$\widehat{\dbAdj}_1$ using $\hbmu_{\sx,\cfd} = \frac{1}{n_1}\sum_{i=1}^n
\action_i \bx_i$, $\hbtheta_\prop^\de$, and $\hsc_{\bSigma}$ as
plug-ins in~\cref{eq:dbAdj}.  It should be noted that the estimates
$\hbmu_{\sx,\cfd}$ and $\hbtheta_\prop^\de$ are not
$\ell_2$-consistent for $\bmu_\sx$ and $\btheta_\prop$.  Moreover,
because they are fit using the same data, their errors are correlated
with each other and with $\bX^\top \big( \ba\odot (\by -
\htheta_{\out,0}\ones - \bX\hbtheta_\out)\big)$.  Thus, the plug-in
estimates have an order-one bias, which must itself be estimated and
removed.  We refer to such a correction as a ``correlation
adjustment,'' in keeping with the terminology of Celentano and
Montanari~\cite{celentano2021cad}.  Because the correlation between
$\hbtheta_\prop^\de$ and $\hbmu_{\sx,\cfd}$ and $\bX^\top
\big(\ba\odot (\by - \htheta_{\out,0}\ones - \bX\hbtheta_\out)\big)$
depends on the way $\hbtheta_\prop^\de$ is constructed, the
correlation adjustment depends on whether the moment method or
M-estimation approach~\eqref{eq:prop-db-option} is used to construct
$\hbtheta_\prop$.

In particular, we set 
\begin{equation}
\label{eq:hat-dbAjd}
\begin{gathered}
    \widehat{\dbAdj}_{01}
        \defn
        \frac{\big\< \hbmu_{\sx,\cfd} , \bSigma^{-1}\bX^\top\big(\ba \odot (\by - \htheta_{\out,0}\ones - \bX \hbtheta_\out)\big)\big\>}{n\hzeta_\out^\theta},
    \\
    \widehat{\dbAdj}_{02}
        \defn
        -
        \hsc_{\bSigma}
        \Big(
            \<\hbmu_{\sx,\cfd},\bSigma^{-1/2}\hbtheta_\prop^\de\>
            -
            \frac{\shat_{\sx\circ}p}{n}
        \Big)
        \Big(
            \frac{\<\bSigma^{-1/2}\hbtheta_\prop^\de,
            \bX^\top\big(\ba \odot (\by - \htheta_{\out,0}\ones - \bX \hbtheta_\out)\big)\>}{n\hzeta_\out^\theta}
            -
            \frac{\shat_{\out\circ}(p - n\hzeta_\out^\eta\hzeta_\out^\theta)}{n}
        \Big),
    \\
    \widehat{\dbAdj}_1
        \defn
        \hsc_{\bSigma}
        \Big(
            \frac{\<\bSigma^{-1/2}\hbtheta_\prop^\de,
            \bX^\top\big(\ba \odot (\by - \htheta_{\out,0}\ones - \bX \hbtheta_\out)\big)\>}{n\hzeta_\out^\theta}
            -
            \frac{\shat_{\out\circ}(p - n\hzeta_\out^\eta\hzeta_\out^\theta)}{n}
        \Big),
\end{gathered}
\end{equation}
where $\shat_{\out\circ}$ is as in \Cref{thm:db-normality} and $\shat_{\sx\circ} = \frac{1}{\numobs} \< \hbi_{\sx,\cfd}, \hbi_\circ\>$ where $\hbi_{\sx,\cfd} = \ba/\hpi$ and $\hpi,\hbi_\circ$ are as in \Cref{thm:db-normality}. 
Note that $\widehat{\dbAdj}_{01}$ does not include a correlation adjustment term.
This is because,
as our proof shows,
$\hbmu_{\sx,\cfd}$ is approximately uncorrelated with $\bSigma^{-1}\bX^\top\big(\ba \odot (\by -
\htheta_{\out,0}\ones - \bX \hbtheta_\out)\big)$ despite being fit on the same data.
This is a consequence of the
fact that we fit the outcome model with an offset.

The following auxiliary result certifies consistency of these
estimates:
\begin{proposition}
\label{prop:adjustment-concentration}
    The estimates $\widehat{\dbAdj}_{01}$, $\widehat{\dbAdj}_{02}$,
    $\widehat{\dbAdj}_1 $, and $\hbeta $ constructed
    in~\cref{eq:pen-logit-shrinage-est,eq:hat-dbAjd} are all
    consistent in probability.
\end{proposition}
\noindent We prove~\Cref{prop:adjustment-concentration}
in~\Cref{sec:hbeta-dbAdj-consistency}.


\section{Simulations}
\label{sec:simulations}

In~\Cref{fig:debiasing}, we display the performance of outcome
regression~\eqref{eq:db-G-computation} as a function of sample size
for several choices of plug-in.  For 10 values of $n$ ranging from 100
to 1000 equidistant on a log-scale, and set $p = 1.25 n$ (or the
closest integer) at each value.  We set $\theta_{\out,0} =
\theta_{\prop,0} = 0$, $\btheta_\out = \btheta_\prop = \be_1$, $\sigma
= .2$, $\pi(\eta) = \tfrac{1}{10} + \tfrac{9}{10}\logit^{-1}(\eta)$,
$\bmu_\sx = \bzero$, and $\bSigma = \id_p$.  At each value of $n$, we
generate 1000 replicates of the data and compute each of 5 different
plug-ins.  \textsf{Ridge} uses the solution to~\cref{eq:outcome-fit}
as a plug-in, with $w_i = 1$ and $\Omega_\out(\bv) = \|\bv\|^2/2$.
\textsf{Debiased ridge (naive)} uses the debiased ridge estimate with
bias estimates as given by the relation~\eqref{eq:db-naive} and $w_i =
1$ in equation~\eqref{eq:outcome-fit}.  \textsf{Debiased ridge (IPW)} uses
the debiased ridge estimate with bias estimates as given by
\eqref{eq:db-naive} and weights $w_i = 1/\pi(\bx_i)$ in
equation~\eqref{eq:outcome-fit}.  \textsf{Oracle ASCW} uses the oracle ASCW
debiased ridge estimates of~\Cref{SecKnownProp}.  \textsf{Empirical
  SCA} uses the empirical SCA debiased ridge estimates
of~\Cref{SecUnknownProp}.  We use the debiased estimate of the
propensity model parameter given by moment method in
~\eqref{eq:prop-db-option}; that is, $\hbtheta_\prop^\de =
(\hbmu_{\sx,\cfd}-\hbmu_\sx)/\hbeta$, with the corresponding
definitions of $\shat_\out$, $\shat_\prop$, $\shat_{\out\prop}$, and
$\shat_{\sx\prop}$.  In the left plot of~\Cref{fig:debiasing}, we show
the mean value of each outcome regression estimate across 1000
replicates with $95\%$ confidence bars based on normal approximation
for the empirical mean.  In the right plot of~\Cref{fig:debiasing}, we
show the empirical variance of each outcome regression estimate across
1000 replicates with $95\%$ confidence bars based on normal
approximation for the empirical variance.

Consistent with our theory, we see that \textsf{Ridge},
\textsf{Debiased ridge (naive)}, and \textsf{Debiased ridge (IPW)}
have non-negligible bias that does not appear to decay as $n$ grows.
On the other hand, both \textsf{Oracle ASCW} and the \textsf{Empirical
  SCA} have no apparent bias across the entire range of sample sizes.
We see that the variance of each estimate decays with $n$, apparently
at a $1/n$ rate.

We emphasize that oracle ASCW depends on oracle knowledge of the
propensity model, whereas empirical SCA relies on an inconsistent
estimate of the propensity model.  An interesting feature of the
variance plot on the right is that empirical SCA has smaller variance
than oracle ASCW.  We note that a similar phenomenon has been observed
the context of inverse propensity weighted
estimates~\cite{hirano2003}.

In~\Cref{fig:PE-w-lam,fig:debiasing}, we consider a fixed sample size
($n = 1000$) and several values of the regularization parameter.
In~\Cref{fig:PE-w-lam}, we plot the prediction error on a new sample
as a function of the regularization parameter for the ridge estimator
with $w_i = 1$ (\textsf{Ridge}) and $w_i = 1/\pi(\bx_i)$
(\textsf{Ridge (IPW)}). (We do not perform the debiasing correction).
The purpose of this plot is to demonstrate that the range of $\lambda$
we consider corresponds to the range in which the prediction error is
optimized.  It is also the range within which the cross-validated
choice of $\lambda$ is expected to lie.  In the bottom two plots, we
see that across this range of $\lambda$, \textsf{Ridge},
\textsf{Debiased ridge (naive)}, and \textsf{Debiased ridge (IPW)} are
strongly biased.  On the other hand, \textsf{Oracle ASCW} and
\textsf{Empirical SCA} have no apparent bias. Across this range of
regularization parameters, \textsf{Empirical SCA} has lower variance
than \textsf{Oracle ASCW}.

\begin{figure}[h!]
\centering
\includegraphics[width =
  0.45\linewidth]{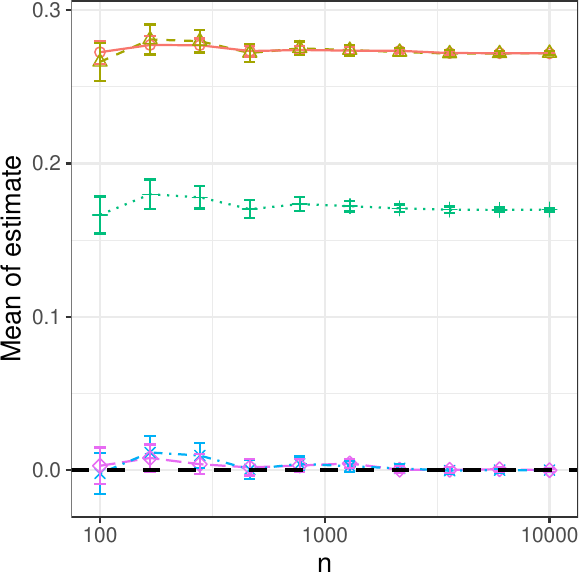}
\includegraphics[width =
  0.45\linewidth]{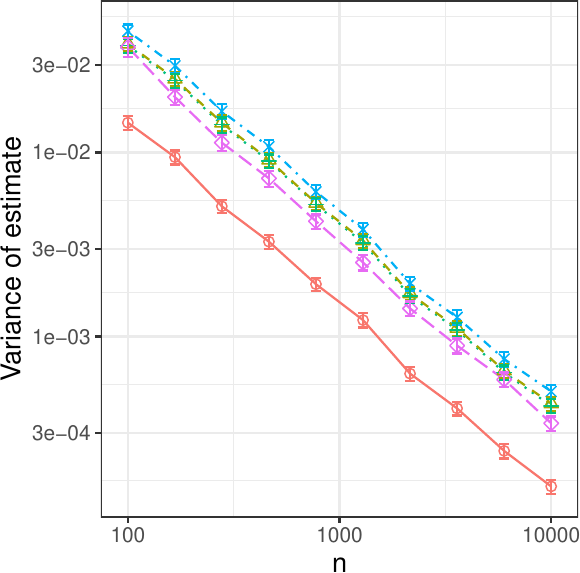}
\\
\includegraphics[width =
  0.675\linewidth]{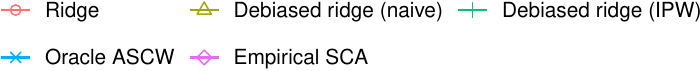}
\caption{Comparison of debiasing methods.  We set $\theta_{\out,0} =
  \theta_{\prop,0} = 0$, $\btheta_\out = \btheta_\prop = \be_1$,
  $\sigma = \frac{1}{5}$, $\pi(\eta) = \tfrac{1}{10} +
  \tfrac{9}{10}\logit^{-1}(\eta)$, $\bmu_\sx = \bzero$, and $\bSigma =
  \id_p$.  The outcome model is fit by ridge regression with
  regularization parameter $\lambda = 1$, and the propensity model by
  the moment method.  In this case, the population mean outcome is
  $\mu_\out \defn \E[y] = 0$.  For each value of $n$, $p$ is taken as
  the closest integer to $1.25n$.  The mean and variance of these
  estimates are computed across 1000 replicates, with 95\% confidence
  intervals shown. See the text for further details.}
\label{fig:debiasing}
\end{figure}

\begin{figure}[h!]
\centering \includegraphics[width =
  0.5\linewidth]{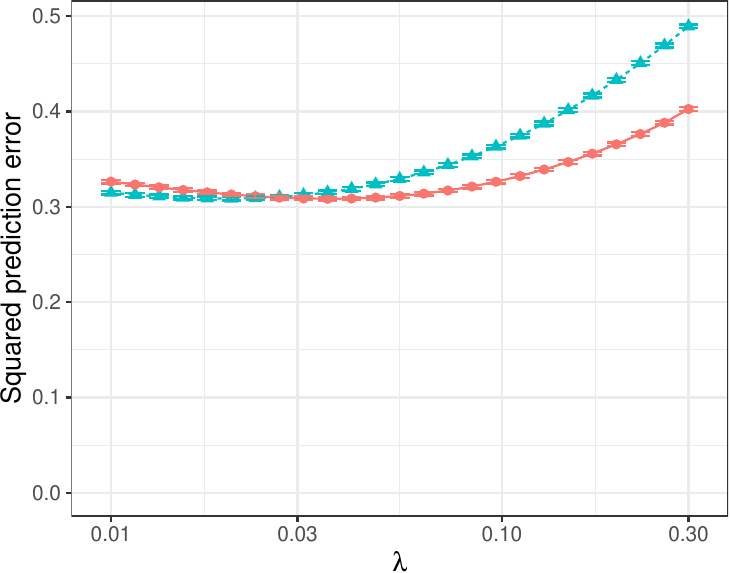}\\ \includegraphics[width
  = 0.25\linewidth]{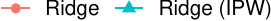}
\caption{Prediction error of ridge regression with $w_i = 1$
  (\textsf{Ridge}) and with $w_i = 1/\pi(\bx_i)$ (\textsf{Ridge
    (IPW)}) across a range of regularization parameters $\lambda$.  $n
  = 1000$ and we simulate from exactly the same setting as
  in~\Cref{fig:variance-well-specified-outcome}, except that the
  outcome model is fit by ridge regression for a range of
  regularization parameters as displayed in the figure.  For each
  replicate, the prediction error is $\| \hbtheta_\out - \btheta_\out
  \|^2$ because $\bx_i \sim \normal(\bzero,\id_p)$.}
\label{fig:PE-w-lam}
\end{figure}

\begin{figure}[h!]
\centering \includegraphics[width =
  0.45\linewidth]{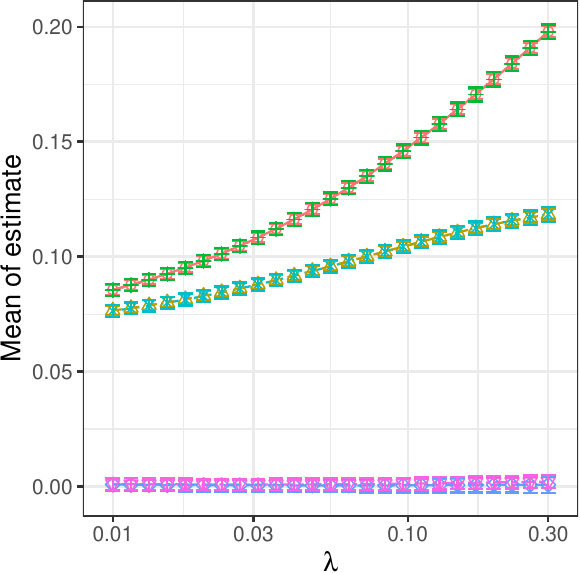}
\includegraphics[width =
  0.45\linewidth]{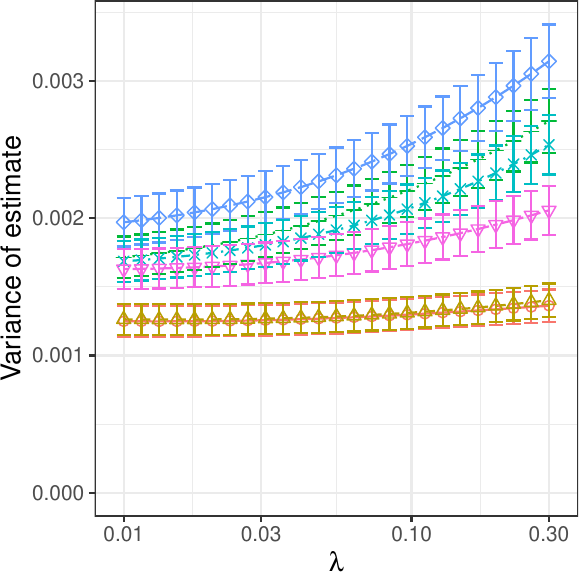}
\\ \includegraphics[width =
  0.6\linewidth]{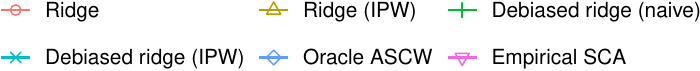}
\caption{Comparison of debiasing methods at $n = 1000$ across a range
  of regularization parameter $\lambda$.  We simulate from exactly the
  same setting as in~\Cref{fig:variance-well-specified-outcome},
  except that the outcome model is fit by ridge regression for a range
  of regularization parameters as displayed in the figure.  For each
  replicate, the prediction error is $\| \hbtheta_\out - \btheta_\out
  \|^2$ because $\bx_i \sim \normal(\bzero,\id_p)$.}
\label{fig:debiasing-w-lam}
\end{figure}


\section{Proofs}
\label{SecProofs}

The proofs of~\Cref{thm:pop-mean,thm:db-normality}, as well as that
of~\Cref{prop:adjustment-concentration}, are all based on an exact
asymptotic characterization of the joint behavior of the estimates
$(\htheta_{\out,0}, \hbtheta_\out)$ and $(\theta_{\prop,0},
\hbtheta_\prop)$, along with the mean estimates \mbox{$\hbmu_\sx =
  \frac{1}{\numobs} \sum_{i=1}^n \bx_i$} and \mbox{$\hbmu_{\sx,\cfd} =
  \frac{1}{n_1} \sum_{i=1}^n \action_i \bx_i$.}  We state the exact
asymptotic characterization here, and then provide a high-level
summary of its utility in proving these three results.  We defer many
technical aspects of the argument to the supplementary material.

\subsection{Exact asymptotics for missing data models}
\label{sec:exact-asymptotics-body}

The exact asymptotic characterization involves a comparison of the
estimates $\hbtheta_\out$ and $\hbtheta_\prop$ to analogous quantities
in what we call the \emph{fixed-design model}.  We first introduce
some notation and define some important quantities.  Then we describe
the fixed-design models.  The parameters of these models must be
chosen to solve a certain system of equations called \emph{the
fixed point equations}, which we present next.  Finally, we present
the exact asymptotic characterization.

\subsubsection{The random-design model}

We refer to the model~\eqref{eq:model} as the \textit{random-design
  model}.  The goal of exact asymptotic theory is to describe the joint
distribution of the following random-design quantities:
\begin{subequations}
\label{eq:rnd-des-objects}
\begin{align}
    \label{eq:rnd-des-lin-pred}
\mbox{\emph{Estimates:}} \qquad & \hbtheta_\prop \defn
\text{see~\cref{eq:outcome-fit}}, \qquad \hbtheta_\out \defn
\text{see~\cref{eq:propensity-fit}}.  \\
\mbox{\emph{True linear predictors:}} \qquad & \bmeta_\prop \defn
\ones \theta_{\prop,0} + \bX \btheta_\prop, \qquad \bmeta_\out \defn
\ones \theta_{\out,0} + \bX \btheta_\out 
\\
\mbox{\emph{Empirical covariate means:}} \qquad & \hbmu_{\sx,\cfd}
\defn \frac{1}{n_1} \sum_{i=1}^n \action_i\bx_i, \qquad \hbmu_\sx =
\frac{1}{\numobs} \sum_{i=1}^n \bx_i, \\
\mbox{\emph{Estimated linear predictors:}} \qquad & \hbmeta_\prop
\defn \ones \htheta_{\prop,0} + \bX \hbtheta_\prop, \qquad
\hbmeta_\out \defn \ones \htheta_{\out,0} + \bX \hbtheta_\out \\
\mbox{\emph{Empirical scores:}} \qquad & \hbpsi_\prop \defn
\nabla_{\bmeta} \ell_\prop\big(\hbmeta_\prop;\ba\big), \qquad
\hbpsi_\out \defn \underbrace{\nabla_{\bmeta}
  \ell_\out\big(\hbmeta_\out;\by \odot \ba,\ba,\bmeta_\prop\big)}_{ -
  \ba \odot \bw \odot (\by - \hbmeta_\out)},
\end{align}
\end{subequations}
where $n_1 \defn \sum_{i=1}^n a_i$.  The characterization of the
distribution of these quantities involves:
\begin{enumerate}[label=$\bullet$,itemsep=0em]
\item the \emph{empirical influence functions}
\begin{subequations}
  \begin{align}
    \hbi_{\sx} = \ones, \qquad \hbi_{\sx,\cfd} = \frac{1}{\hpi} \ba,
    \qquad \hbi_\prop = -\frac{\hbpsi_\prop}{\hzeta_\prop^\theta},
    \qquad \hbi_\out = -\frac{\hbpsi_\out}{\hzeta_\out^\theta},
  \end{align}
  where $\hpi \defn \frac{1}{n} \sum_{i=1}^n \action_i$.
\item the \emph{effective leave-one-out estimates}
\begin{align}
\label{eq:eff-loo-est}
    \hbtheta_\prop^\loo \defn \hbtheta_\prop -
    \frac{\bSigma^{-1}\bX^\top \nabla_{\bmeta}
      \ell_\prop(\hbmeta_\prop;\ba)}{n\hzeta_\prop^\theta}, \qquad
    \hbtheta_\out^\loo \defn \hbtheta_\out +
    \frac{\bSigma^{-1}\bX^\top \big(\ba \odot \bw \odot (\by -
      \hbmeta_\out)\big)}{n \hzeta_\out^\theta};
\end{align}
\item and the \emph{effective leave-one-out linear predictors}
\begin{align}
    \hbmeta_\prop^\loo \defn \hbmeta_\prop + \hzeta_\prop^\eta \nabla
    \ell_\prop\big(\hbmeta_\prop;\ba\big), \qquad \hbmeta_\out^\loo
    \defn \hbmeta_\out + \hzeta_\out^\eta \ba \odot \bw \odot (\by -
    \hbmeta_\out).
\end{align}
\end{subequations}
\end{enumerate}

\subsubsection{The fixed-design models}

The exact asymptotic characterization states that, in a certain sense,
the estimates \eqref{eq:rnd-des-objects} in the random-design model
\eqref{eq:model} behave like analogous estimates in two related
statistical models.  These models involve linear observations of
unknown parameters with a deterministic design matrix, so that we call
them the \textit{fixed-design models}.  The two fixed-design models
are defined on separate probability spaces, which we now define.

\paragraph{Fixed-design model for parameter estimation:}  This model contains
Gaussian observations of the feature means and outcome linear effects
\begin{equation}
\label{eq:fixed-design-param-outcomes}
\begin{gathered}
    \by_{\sx}^f = \bmu_{\sx} + \bg_{\sx}^f, \qquad \by_{\sx,\cfd}^f =
    \bmu_{\sx,\cfd} + \bg_{\sx,\cfd}^f, \\
    \by_\prop^f = \bSigma^{1/2} \big(\beta_{\prop\prop}
    \btheta_{\prop}\big) + \bg_{\prop}^f, \quad \mbox{and} \quad
    \by_\out^f = \bSigma^{1/2} \big( \beta_{\out\prop}\btheta_{\prop}
    + \btheta_{\out}\big) + \bg_{\out}^f,
\end{gathered}
\end{equation}
where $(g_{\sx,i}^f, g_{\sx,\cfd,i}^f,
g_{\prop,i}^f,g_{\out,i}^f)\stackrel{\mathrm{iid}}\sim
\normal(0,\bS/n)$, and $\bS \in \S_{\geq0}^4$, $\beta_{\prop\prop}$,
and $\beta_{\out\prop}$ are parameters to be determined below.  In
this model, we consider the estimates
\begin{equation}
\label{eq:param-fit-f}
\begin{gathered}
    \hbtheta_{\prop}^f = \argmin_{\bv} \Big\{
    \frac{\zeta_{\prop}^\theta}2 \big\| \by_{\prop}^f - \bSigma^{1/2}
    \bv \big\|^2 + \Omega_{\prop}(\bv) \Big\}, \quad \mbox{and} \quad
    \hbtheta_{\out}^f = \argmin_{\bv} \Big\{ \frac{\zeta_{\out}^\theta
    }2 \big\| \by_{\out}^f - \bSigma^{1/2} \bv \big\|^2 +
    \Omega_{\out}(\bv) \Big\},
\end{gathered}
\end{equation}
where $\zeta_{\prop}^\theta$, $\zeta_{\out}^\theta$ are parameters to
be determined below.


\paragraph{Fixed-design model for linear-predictor estimation:} It contains
linear predictors $(\bmeta_{\prop}^f, \bmeta_{\out}^f)$ with the same
distribution as the linear predictors in the random design
model~\eqref{eq:rnd-des-lin-pred}: $(\eta_{\prop,i}^f,
\eta_{\out,i}^f) \iid \normal\big((\mu_\prop, \mu_\out)^\top, \llangle
\bTheta \rrangle_{\bSigma}\big)$, where $\bTheta =
(\btheta_{\prop},\btheta_{\out})$, $\mu_\prop = \theta_{\prop,0} + \<
\bmu_{\sx},\btheta_\prop\>$ and $\theta_{\out,0} + \<
\bmu_{\sx},\btheta_\out\>$.  It also contains missingness indicator
and outcome variables with the same distribution as in the random
design model:
\begin{equation}
\label{eq:fixed-design-outcomes}
\begin{gathered}
    \ba^f 
    \text{ where }
    \P\big(\action_i^f = 1 \mid \eta_{\prop,i}^f\big)
        =
        \pi(\eta_{\prop,i}^f)
        =
        \pi_i^f
    \text{ independently for each $i$},
    \\
    \by^f
        =
        \bmeta_{\out}^f + \beps_\out^f,
    \qquad
    \bw^f
        =
        w(\bmeta_{\prop}^f),
    \qquad
    \beps_\out^f \eqnd \beps_\out
    \text{ independent of everything else.}
\end{gathered}
\end{equation}
Further,
it contains observations
$\hbmeta_\prop^{f,\loo},\hbmeta_\out^{f,\loo}$ correlated with the
linear predictors $(\bmeta_{\prop}^f,\bmeta_{\out}^f)$ via
\begin{equation}
\label{eq:fixed-design-linear-predictor-dist}
    \begin{pmatrix}
        \eta_{\prop,i}^f \\[3pt]
        \eta_{\out,i}^f \\[3pt]
        \heta_{\prop,i}^{f,\loo} \\[3pt]
        \heta_{\out,i}^{f,\loo} 
    \end{pmatrix}
        \stackrel{\mathrm{iid}}\sim
        \normal
        \left(
            \begin{pmatrix}
                \mu_{\prop}\\[3pt]
                \mu_{\out}\\[3pt]
                \hmu_{\prop}^f\\[3pt]
                \hmu_{\out}^f
            \end{pmatrix}
            ,\;
            \begin{pmatrix}
                \llangle \bTheta \rrangle_{\bSigma} & \llangle \bTheta , \hbTheta^f \rrangle_{\bSigma,L^2} \\[3pt]
                \llangle \hbTheta^f , \bTheta \rrangle_{\bSigma,L^2} & \llangle \hbTheta^f \rrangle_{\bSigma,L^2}\\[3pt]
            \end{pmatrix}
        \right),
\end{equation}
 and $\hbTheta^f = (\hbtheta_{\prop}^f,\hbtheta_{\out}^f)$,
where $\hmu_{\prop}^f,\hmu_{\out}^f$ are deterministic parameters to be determined below.
The covariance matrix depends on expectations in the fixed-design model for parameter estimation, 
so is a function of the parameters $\bS$, $\beta_{\prop\prop}$ and $\beta_{\out\prop}$, in addition to $\hmu_{\prop}^f$, $\hmu_{\out}^f$.
The \emph{estimated linear predictors} are
\begin{equation}
\label{eq:lin-predict-f}
\begin{gathered}
    \hbmeta_{\prop}^f
        =
        \argmin_{\bu}
        \Big\{
            \frac{1}2
            \big\|
                \hbmeta_{\prop}^{f,\loo}
                -
                \bu
            \big\|^2
            +
            \zeta_{\prop}^\eta
            \sum_{i=1} \ell_{\prop}(u_i;\action_i^f)
        \Big\},
    \\
    \hbmeta_{\out}^f
        =
        \argmin_{\bu}
        \Big\{
            \frac{1}2
            \big\|
                \hbmeta_{\out}^{f,\loo}
                -
                \bu
            \big\|^2
            +
            \frac{\zeta_{\out}^\eta}2
            \sum_{i=1}^n \action_i^f w_i^f ( y_i^f - u_i )^2
        \Big\},
\end{gathered}
\end{equation}
for parameters $\zeta_\out^\eta,\zeta_\out^\theta$ to be determined below.
The \emph{empirical scores} are
\begin{equation}
\label{eq:fixed-design-score}
    \hbpsi_{\prop}^f = \nabla_{\bmeta} \ell_{\prop} \big(\hbmeta_{\prop}^f;\ba^f\big),
    \qquad
    \hbpsi_{\out}^f = \nabla_{\bmeta} \ell_{\out} \big(\hbmeta_{\out}^f;\by^f,\ba^f,\bmeta_{\out}^f\big)
        =
        -
        \ba^f \odot \bw^f \odot (\by^f - \hbmeta_\out^f).
\end{equation}
The \emph{empirical influence functions} are
\begin{equation}
\label{eq:empirical-influence-function}
\begin{gathered}
    \hbi_{\sx}^f
        =
        \ones,
    \qquad
    \hbi_{\sx,\cfd}^f
        =
        \frac{1}{\barpi}\ba^f,
    \qquad
    \hbi_\prop^f
        =
        -\frac{\hbpsi_\prop^f}{\zeta_\prop^\theta},
    \qquad
    \hbi_\out^f
        =
        -\frac{\hbpsi_\out^f}{\zeta_\out^\theta},
\end{gathered}
\end{equation}
where $\barpi = \E[d]/n = \E\big[\pi(\eta_\prop^f)\big]$.
We denote by $\IFhat^f \in \reals^{n \times 4}$ the matrix with columns $\hbi_{\sx}^f$, $\hbi_{\sx,\cfd}^f$, $\hbi_{\prop}^f$ and $\hbi_{\out}^f$.


\subsubsection{The fixed point equations}

By construction, the two fixed-design models are determined by the
collection of parameters
\begin{align*}
\big \{ \bS, \; \beta_{\prop\prop}, \; \beta_{\out\prop}, \;
\zeta_{\prop}^\theta, \; \zeta_{\prop}^\eta, \; \zeta_{\out}^\theta,
\; \zeta_{\prop}^\theta, \; \hmu_\prop^f, \; \hmu_\out^f \big \}.
\end{align*}
The fixed-design models characterize the behavior of the random-design
model when these parameters are the unique solution to \textit{the
  fixed point equations}:
\begin{equation}
\label{eq:regr-fixed-pt}
\tag{FD-fixpt}
\begin{gathered}
    \< \ones , \hbi_{\prop}^f \>_{\Ltwo} = \< \ones , \hbi_{\out}^f
    \>_{\Ltwo} = 0, \\ \beta_{\prop\prop} = \frac{1}{n} \E\Big[
      \Tr\Big( \frac{\de \hbi_{\prop}^f}{\de \bmeta_{\prop}^f} \Big)
      \Big], \quad \beta_{\out\prop} = \frac{1}{n} \E\Big[ \Tr\Big(
      \frac{\de \hbi_{\out}^f}{\de \bmeta_{\prop}^f} \Big) \Big],
    \\ \bS = \frac{1}{n}\llangle \IFhat^f \rrangle_{\Ltwo}.
    \\ \zeta_{\prop}^\theta = \frac{1}{n} \E\Big[ \Tr \Big( \frac{\de
        \hbpsi_{\prop}^f}{\de \hbmeta_{\prop}^{f,\loo}} \Big) \Big],
    \qquad \zeta_{\out}^\theta = \frac{1}{n} \E\Big[ \Tr \Big(
      \frac{\de \hbpsi_{\out}^f}{\de \hbmeta_{\out}^{f,\loo}} \Big)
      \Big], \\ \zeta_{\prop}^\eta\zeta_{\prop}^\theta = \frac{1}{n}
    \E\Big[ \Tr\Big( \bSigma^{1/2} \frac{\de \hbtheta_{\prop}^f }{\de
        \by_{\prop}^f} \Big) \Big], \qquad
    \zeta_{\out}^\eta\zeta_{\out}^\theta = \frac{1}{n} \E\Big[
      \Tr\Big( \bSigma^{1/2} \frac{\de \hbtheta_{\out}^f }{\de
        \by_{\out}^f} \Big) \Big], \\ \text{and indices 3, 4 are
      innovative with respect to both $\bS$ and the covariance
      in~\cref{eq:fixed-design-linear-predictor-dist}.}
\end{gathered}
\end{equation}
The right-hand sides of these equations involve expectations taken
with respect to the distributions of the fixed-design models
determined by the parameters $\big \{ \bS, \; \beta_{\prop\prop}, \;
\beta_{\out\prop}, \; \zeta_{\prop}^\theta, \; \zeta_{\prop}^\eta, \;
\zeta_{\out}^\theta, \; \zeta_{\prop}^\theta, \; \hmu_\prop^f, \;
\hmu_\out^f \big \}$.  Thus, they are functions of these parameters.
In all cases, the derivatives do not exist in the classical sense, and
so need to be interpreted carefully.
(See~\Cref{sec:exact-asymptotics} for the correct interpretation).
There are 18 degrees of freedom in choosing the fixed point
parameters, and 18 independent scalar equations
in~\cref{eq:regr-fixed-pt}, accounting for symmetry of $\bS$.
In~\Cref{sec:exact-asymptotics}, we state a lemma asserting that the
fixed point equations have a unique solution.  It also provides bounds
on this solution that are useful in our proofs, and may be of
independent interest because they control the operating
characteristics of the estimators $\hbtheta_\out$ and
$\hbtheta_\prop$.


\subsubsection{Characterization}

We call a function $\phi: (\reals^N)^k \rightarrow \reals$ is order-$2$ pseudo-Lipschitz with constant $L$ if 
\begin{equation}
\begin{gathered}
  |\phi(\bzero,\ldots,\bzero)| \leq L, \\ \Big|
  \phi(\bx_1,\ldots,\bx_k) - \phi(\bx_1',\ldots,\bx_k') \Big| \leq L
  \Big( 1 + \sum_{\ell=1}^k \| \bx_\ell \| + \sum_{\ell=1}^k \|
  \bx_\ell' \| \Big) \sum_{\ell=1}^k \| \bx_\ell - \bx_\ell' \|.
\end{gathered}
\end{equation}
If the constant $L$ is not stated, then it should be taken as $L = 1$.
The following result guarantees that order-$2$ pseudo-Lipschitz
functions of the quantities~\eqref{eq:rnd-des-objects} concentrate on
the expectation of the analogous quantities in the fixed-design models
determined by the solution to the fixed point
equations~\eqref{eq:regr-fixed-pt}:
\begin{theorem}
\label{thm:exact-asymptotics-body}
    Under Assumption A1 and as $(n,p)\rightarrow \infty$, we have the
    following:
    \begin{enumerate}[(a)]
        \item 
        The estimated population means concentrate:
        \begin{equation}
            \htheta_{\prop,0} + \< \bmu_{\sx} , \hbtheta_\prop \>
                -
                \hmu_\prop^f
                \gotop 0,
            \qquad
            \htheta_{\out,0} + \< \bmu_{\sx} , \hbtheta_\out \>
                -
                \hmu_\out^f
                \gotop 0.
        \end{equation}

        \item 
        The effective-regularization terms concentrate:
        \begin{equation}
            \hzeta_{\prop}^\eta
                -
                \zeta_{\prop}^\eta
                \gotop 0,
            \qquad
            \hzeta_{\prop}^\theta
                -
                \zeta_{\prop}^\theta
                \gotop 0,
            \qquad
            \hzeta_{\out}^\eta
                -
                \zeta_{\out}^\eta
                \gotop 0,
            \qquad
            \hzeta_{\out}^\theta
                -
                \zeta_{\out}^\theta
                \gotop 0.
        \end{equation}

        \item 
        For any sequence of order-2 pseudo-Lipschitz function $\phi:(\reals^p)^6\rightarrow \reals$,
        \begin{equation}
            \phi\big(
                \hbtheta_{\prop},
                \hbtheta_{\out},
                \hbmu_{\sx},
                \hbmu_{\sx,\cfd},
                \hbtheta_{\prop}^{\loo},
                \hbtheta_{\out}^{\loo}
            \big)
            -
            \E\Big[
                \phi\Big(
                    \hbtheta_{\prop}^f,
                    \hbtheta_{\out}^f,
                    \bSigma^{-1/2} \by_{\sx}^f,
                    \bSigma^{-1/2} \by_{\sx,\cfd}^f,
                    \bSigma^{-1/2} \by_{\prop}^f,
                    \bSigma^{-1/2} \by_{\out}^f
                \Big)
            \Big]
            \gotop 0.
        \end{equation}
        
        \item 
        For any sequence of order-2 pseudo-Lipschitz functions $\phi:\reals^4 \rightarrow \reals$,
        \begin{equation}
            \frac{1}{n}
            \sum_{i=1}^n
            \phi\big(
                \eta_{\prop,i},
                \eta_{\out,i},
                \hpsi_{\prop,i},
                \hpsi_{\out,i}
            \big)
            -
            \E\Big[
                \phi\Big(
                    \eta_{\prop,i}^f,
                    \eta_{\out,i}^f,
                    \hpsi_{\prop,i}^f,
                    \hpsi_{\out,i}^f
                \Big)
            \Big]
            \gotop 0.
        \end{equation}

    \end{enumerate}

\end{theorem}
\noindent \Cref{thm:exact-asymptotics-body} is similar to statements
made elsewhere in the exact asymptotics
literature~\cite{bayatiMontanari2012,karoui2013asymptotic,thrampoulidisAbbasiHassibi2018,miolane2021,celentano2021cad},
but differs in several important respects.  First, rather than
characterize the behavior of a single regression estimate
$\hbtheta_\out$ or $\hbtheta_\prop$ marginally, it describes their
distribution jointly with each other and with the estimates of the
marginal and confounded means of the covariates $\hbmu_{\sx}$ and
$\hbmu_{\sx,\cfd}$.  This is essential to the analysis of the
empirical SCA estimate because non-negligible biases can result from
correlations between these quantities.  Indeed, estimating the
$\dbAdj$ terms requires correcting for these biases.  A similar joint
characterization was used in Celentano and
Montanari~\cite{celentano2021cad}, but was restricted to two linear
models rather than a linear and binary outcome model.  Substantial
technical novelty is required then to establish the existence and
uniqueness of and bounds on the solutions to the fixed-point equations
\eqref{eq:regr-fixed-pt}, as we see in Appendix
\ref{sec:fixed-point-parameter-analysis}.  Further, in the context of
binary outcome models, as far as we are aware, exact asymptotics has
been limited either to the case of separable penalties with iid
designs \cite{salehiAbbasiHassibi2019} or unpenalized regression with
correlated designs \cite{surCandes2019,zhaoSurCandes2022}.  We handle
non-separable penalties with arbitrary covariance (although, we should
note, Assumption A1 does not hold for the logistic loss, which is
beyond our current scope).  Further, we make no assumption that the
coordinates of $\btheta_\out$ or $\btheta_\prop$ are distributed iid
from some prior, which is common in the exact asymptotic literature
\cite{salehiAbbasiHassibi2019}.  Instead, as we will see in the
Appendix \ref{sec:exact-asymptotics}, our results take the form of
non-asymptotic concentration bounds which hold uniformly over the
class of parameter vectors defined by Assumption A1.  This implies
that the limits in Theorems \ref{thm:pop-mean} and
\ref{thm:db-normality} hold uniformly over this model class.  The
proof of~\Cref{thm:exact-asymptotics} is based on the Convex-Gaussian
min-max theorem in a conditional form.  This conditional form was also
used in the paper~\cite{celentano2021cad}, but our application of it
here must address several technical difficulties due to the novelties
listed above.  In~\Cref{sec:exact-asymptotics}, we provide a slightly
more quantitative bound for the concentration
in~\Cref{thm:exact-asymptotics-body}.


\subsection{Debiasing theorems from exact asymptotics}

\Cref{thm:exact-asymptotics-body} is the basis for the proofs
of~\Cref{thm:pop-mean,thm:db-normality} along
with~\Cref{prop:adjustment-concentration}.  It allows us to reduce the
consistency and coverage guarantees in the random design model to
corresponding statements in the fixed-design model.  In the
fixed-design model, these statements either hold by definition or as
algebraic relationships guaranteed by the fixed point
equations~\eqref{eq:regr-fixed-pt}.

Here we provide a very high-level and heuristic summary of some of
these arguments;
see~\Cref{sec:proof-main-theorems,sec:hbeta-dbAdj-consistency} for the
full technical details.  Throughout, we assume that $\bSigma = \id_p$.
(As we describe in~\Cref{sec:independent-covariates}, this condition
can be imposed without loss of generality.)

By~\cref{eq:orc-ASCW-matrix-form}, the oracle ASCW estimate can be written
as
\begin{align*}
\hbtheta_\out^\de & = \hbtheta_\out +
\frac{\bSigma_{\cfd}^{-1}\bX^\top \big(\ba \odot (\by -
  \htheta_{\out,0} \ones - \bX \hbtheta_\out)\big)}{n
  \hzeta_\out^\theta}
\end{align*}
Moreover, by the KKT conditions for~\cref{eq:outcome-fit}, we have the
relation
\begin{align*}
  \frac{1}{\numobs} \bX^\top \big(\ba \odot (\by - \htheta_{\out,0}
  \ones - \bX \hbtheta_\out)\big) & = \nabla \Omega_\out(\hbtheta_\out),
\end{align*}
and by the Sherman-Morrison-Woodbury identity, we have
$\bSigma_{\cfd}^{-1} = \id_p - \sc_{\bSigma}
\btheta_\prop\btheta_\prop^\top$.  Thus, as indicated
by~\Cref{thm:exact-asymptotics-body}, the object in the fixed-design
model corresponding to $\hbtheta_\out^\de$ is given by
\begin{align*}
\hbtheta_\out^f + \frac{1}{ \zeta_\out^\theta} \bSigma_{\cfd}^{-1}
\nabla \Omega_\out(\hbtheta_\out^f).
\end{align*}
Similarly, by the KKT conditions for~\cref{eq:param-fit-f} in the
fixed-design model, we have \mbox{$\frac{1}{\zeta_\out^\theta} \nabla
  \Omega_\out(\hbtheta_\out^f) = \by_\out^f - \hbtheta_\out^f$,} and
hence the equivalence
\begin{align*}
  \hbtheta_\out^f + \frac{1}{\zeta_\out^\theta} \bSigma_{\cfd}^{-1}
  \nabla \Omega_\out(\hbtheta_\out^f) & = \hbtheta_\out^f + (\id_p -
  \sc_{\bSigma}\btheta_\out\btheta_\out^\top)(\by_\out^f -
  \hbtheta_\out^f) = \beta_{\out\prop}\btheta_\prop + \btheta_\out +
  \bg_\out^f - \sc_{\bSigma}\<\btheta_\out , \by_\out^f -
  \hbtheta_\out^f \>.
\end{align*}
By Gaussian concentration of measure, the random variable
$\<\btheta_\out , \by_\out^f - \hbtheta_\out^f \>$ concentrates on its
expectation $\<\btheta_\out , \by_\out^f - \hbtheta_\out^f
\>_{\Ltwo}$.

Moreover, one can show (cf.~\Cref{sec:db-proofs}) that the fixed point
equations~\eqref{eq:regr-fixed-pt} imply the relation
\begin{align*}
  \beta_{\out\prop} = \sc_{\bSigma} \<\btheta_\out,\by_\out^f -
  \hbtheta_\out^f\>_{\Ltwo}.
\end{align*}
Thus, the quantity $\btheta_\out + \bg_\out^f$ in the fixed-design
model corresponds to $\hbtheta_\out^\de$.  This is an unbiased
Gaussian estimate of $\btheta_\out$ with coordinate-wise variance $\|
\hbi_\out^f \|_{\Ltwo}^2/n^2$.  \Cref{thm:exact-asymptotics-body}
allows us to argue that the empirical standard error $\htau^2/n = \|
\hbi_\out \|^2/n^2$ concentrates on $\| \hbi_\out^f \|_{\Ltwo}^2/n^2$.
Thus, the coverage guarantee of \Cref{thm:db-normality} holds for the
corresponding object in the fixed-design models.  Using a Lipschitz
approximations of the indicator functions,
\Cref{thm:exact-asymptotics-body} can be used to establish the same
holds for oracle ASCW debiased estimator. \\

We next observe that the population mean estimate can be written as
\begin{align*}
\hmu_{\out}^\de = \htheta_{\out,0}^\de + \< \hbmu_\sx ,
\hbtheta_\out^\de \>.
\end{align*}
Following the reasoning of the previous paragraph, the quantity $\<
\bmu_\sx + \bg_\sx^f , \btheta_\out + \bg_\out^f \>$ is the the
fixed-design object corresponding to $\< \hbmu_\sx , \hbtheta_\out^\de
\>$.  Using the fixed point equations~\eqref{eq:regr-fixed-pt}, we can
show that $\bg_\sx^f$ is uncorrelated with $\bg_\out^f$, so that the
fixed-design quantity $\< \bmu_\sx + \bg_\sx^f , \btheta_\out +
\bg_\out^f \>$ concentrates on $\< \bmu_\sx , \btheta_\out \>$.
\Cref{thm:exact-asymptotics-body} can be used that the same
holds for the random-design quantity $\< \hbmu_\sx , \hbtheta_\out^\de
\>$.  We also argue that $\htheta_{\out,0}^f$ concentrates on
$\theta_{\out,0}$, so that $\hmu_{\out}^\de$ concentrates on
$\theta_{\out,0} + \< \bmu_\sx , \btheta_\out \> = \mu_\out$.  This
argument leads to the guarantee in~\Cref{thm:pop-mean} for the oracle
ASCW estimates.

The analysis of the empirical SCA estimates and the $\dbAdj$-estimates
given in~\cref{eq:hat-dbAjd} are based on a similar strategy.
Estimating the $\dbAdj$-quantities requires correlation adjustments;
let us provide the basic intuition here.  For example, in order to
consistently estimate the term $\< \bmu_{\sx,\cfd} , \btheta_\prop \>$
in the second line of~\cref{eq:dbAdj}, we might consider the plug-in
estimate $\< \hbmu_{\sx,\cfd} , \hbtheta_\prop^\de \>$.  When we use
the moment-method to construct $\hbtheta_\prop^\de$
in~\cref{eq:prop-db-option}, this corresponds to the objective $\<
\bmu_{\sx,\cfd} + \bg_{\sx,\cfd}^f , \btheta_\prop + (\bg_{\sx,\cfd}^f
- \bg_\sx^f ) / \alpha_1 \>$ in the fixed-design models.  This
quantity is biased for $\< \bmu_{\sx,\cfd} , \btheta_\prop \>$ because
$\bg_{\sx,\cfd}^f$ and $\bg_{\sx,\cfd}^f - \bg_\sx^f$ are correlated.
The fixed point equations~\eqref{eq:regr-fixed-pt} give us an estimate
of this correlation in terms of the correlation of empirical influence
functions.  This allows us to estimate this bias consistently and
remove it, as in~\cref{eq:dbAdj}.  We refer to this procedure as a
correlation adjustment. \\

\noindent Complete technical details are given
in~\Cref{sec:proof-main-theorems,sec:hbeta-dbAdj-consistency}.


\section{Discussion}
\label{SecDiscussion}

This paper studied consistent estimation of the population mean of an outcome missing at random in an ``inconsistency regime;'' that is, in a regime in which consistent estimation of both the outcome and the propensity model are impossible.
We focused on a linear-GLM outcome-propensity specification.
G-computation, AIPW, and IPW enjoy model-agnostic guarantees that rely on consistent estimation of either the outcome model or propensity model, but they fail to be consistent for the population mean in our regime.
When $n > p$ by a constant fraction, both the G-computation and AIPW estimators are $\sqrt{n}$-consistent and asymptotically normal provided OLS is used to estimate the outcome model \cite{yadlowsky2022,sur2022}.
The primary contribution of this paper is to develop a method---\textit{empirical shifted-confounder augmentation (SCA)}---which is consistent for the outcome mean in the challenging regime that $n < p$ and hence least-squares is no longer possible.
The method can be viewed as G-computation with a carefully constructed plug-in estimate for the outcome model.
Our construction of the plug-in builds on previous methods for debiasing penalized regression estimators,
but, unlike these previous methods, must deal with bias introduced not only by regularization but also by the missingness mechanism.
In particular, missingness induces a distribution shift between the features conditional on missingness and unconditionally.
This distribution shift biases parameter estimates.
We show how, in the linear-GLM setting, an estimate of the propensity model parameter can be used to correct for this bias despite the estimate's inconsistency.
We moreover show that ``on average across coordinates'' the debiased outcome model parameter is asymptotically normal for the true parameter at a $\sqrt{n}$-rate,
and we provide an empirical standard error.
In the course of studying our method, we develop exact asymptotic theory for the joint distribution of the outcome and propensity model estimates, which may be of independent interest.

In the context of semiparametric estimation and inference more broadly,
the empirical SCA method is a step towards designing methodology that eliminates or reduces bias in moderate signal-to-noise regimes in which nuisance parameters cannot be estimated well and in which classical approaches fail.
There are several promising avenues for future work to address its current shortcomings and to determine whether these are fundamental.
\begin{itemize}

    \item We have not studied the rate of convergence of
      $\hmu_\out^\de$, although Figure \ref{fig:debiasing} indicates
      it may converge at a $\sqrt{n}$-rate.  Moreover, Theorem
      \ref{thm:db-normality} establishes normality of
      $\hbtheta_{\out}^\de$ ``on average across coordinates,'' but
      does not provide a normality guarantee for any fixed coordinate
      $\theta_{\out,j}$.  Similar shortcomings have been present in
      earlier work on exact asymptotics.  Bayati et
      al.~\cite{bayatiErdogduMontanari} and Celentano and
      Montanari~\cite{celentano2021cad} do not establish a rate for
      their estimates of the conditional variance or covariance in
      high-dimensional linear models.  Likewise, Miolane and
      Montanari~\cite{miolane2021} and Celentano et
      al.~\cite{celentanoMontanariWei2020} do not establish
      coordinate-wise normality for the debiased Lasso.  In both these
      cases, this is due to their reliance on approximate message
      passing or Convex Gaussian Min-Max Theorem based proof
      techniques.  Recently, Tan et al.~\cite{tan2022noise} and Bellec
      and Zhang~\cite{bellec2019debiasingIntervals} established
      $\sqrt{n}$-consistent estimation for conditional covariances and
      coordinate-wise normality of debiased estimates in
      high-dimensional linear models.  Their methods are based on
      Stein's formula and Gaussian Poincar\'e inequalities, and it is
      possible similar techniques could be applied fruitfully in the
      present context.

    \item Our results require that the unconditional covariance of the
      confounders is known.  It is likely straightforward to
      generalize our results to cases in which this covariance can be
      estimated consistently in operator norm.  In the
      high-dimensional regime, operator-norm consistency would require
      either strong structural assumptions
      \cite{bickelLevina2008,karoui2008,caiZhangZhou2010} or access to
      a much larger unlabeled data set.  Although this is a strong
      requirement, a large body of work assumes feature distributions
      are known, and in some applications this may be a reasonable
      assumption \cite{candesFanJansonLv2018}.  In contrast, when $n >
      (1+\epsilon)p$ and under mild conditions, $\sqrt{n}$-consistency
      is possible without any knowledge of the feature covariance
      \cite{yadlowsky2022}.  Thus, an important open research
      direction is to establish tight minimax lower bounds for
      estimation that are sensitive to varying levels of prior
      knowledge of the feature covariance $\bSigma$, and to determine
      whether these exhibit a phase transition at $n = p$.  It is
      interesting to note that a similar phenomenon occurs in the
      problem of conditional covariance estimation across two linear
      models.  In this setting, consistent estimation is possible when
      $n > (1+\epsilon)p$ without knowledge of the feature covariance,
      but existing approaches require operator-norm consistent
      estimation of $\bSigma$ when $n < p$ \cite{celentano2021cad}.

    \item Our results require that the features be jointly Gaussian.
    There is a growing body of work studying universality for regression models under proportional asymptotics \cite{bayatiLelargeMontanari2015,abbasiSalehiHassibi2019,montanariSaeed2022,huLu2023}.
    In particular, in the proportional regime, one can often generalize exact asymptotic results from Gaussian features to sub-Gaussian features or features of the form $\bx = \bSigma^{1/2} \bz$ where $\bz$ has independent sub-Gaussian entries.
    We thus expect that the Gaussianity assumption can be relaxed substantially.
    Rigorously investigating this possibility is left to future work.

\end{itemize}


\subsection*{Acknowledgements}
This work was partially supported by National Science Foundation
grants NSF-CCF-1955450 and NSF-DMS-2311072, as well as DOD-ONR Office
of Naval Research N00014-21-1-2842 to MJW.  MC is supported by the Miller Institute for Basic Research in Science, University of
California, Berkeley.

\bibliographystyle{my_alpha} 
\bibliography{db_missing_data}


\newpage

\appendix

\tableofcontents

\section{High probability approximations: conventions}
\label{sec:high-probability}

The concentration bounds that we establish in this paper hold with
exponentially high probability, with constants depending on
$\cPmodel$.  Let us introduce some shorthand notation for the the type
of concentration results that we provide.  For two scalar random
variables $A, B$, we use $A \mydoteq B$ to mean there exist constants
$C,c,c',r > 0$ depending only on $\cPmodel$ such that, for all
$\epsilon < c'$, $\P(|A - B| \geq \epsilon) \leq C e^{-cn\epsilon^r}$.
We use $A \lessdot B$ to mean there exist constants $C,c,c',r > 0$
depending only on $\cPmodel$ such that, for all $\epsilon < c'$, $\P(A
- B > \epsilon) \leq C e^{-cn\epsilon^r}$.  We equivalently write $B
\gtrdot A$.  Moreover, for two random vectors $\ba,\bb \in \reals^N$,
we use $\ba \mydoteq \bb$ to denote $\| \ba - \bb \| \mydoteq 0$.
Moreover, for two random matrices $\bA,\bB \in \reals^{N \times M}$ we
use $\bA \mydoteq \bB$ to denote $\| \bA - \bB \|_{\op} \mydoteq 0$.
It is straightforward to check that these relation have the following
properties: $A \mydoteq B$ if and only if $A \lessdot B$ and $B
\lessdot A$; if $A \mydoteq B$ and $B \mydoteq C$, then $A \mydoteq
C$; if $A \lessdot B$ and either $B \lessdot C$, then $A \lessdot C$.
We use these relationships freely and without comment throughout the
paper.  We also say an event occurs with ``exponentially high
probability'' if there exist constants $C,c > 0$ depending on
$\cPmodel$ such that the event occurs with probability at least $1 -
Ce^{-cn}$.  Note that an event that occurs with exponentially high
probability occurs with probability approaching 1 along any sequence
of models and estimators each satisfying assumption A1 and $n
\rightarrow \infty$.  Likewise, along any such sequence, vectors
$\ba_n$, $\bb_n$ satisfying $\ba_n \mydoteq \bb_n$ have $\| \ba_n -
\bb_n \| \gotop 0$.  Because Assumption A1 requires $C > n/p > c > 0$,
these high probability bounds are meaningful in a proportional
asymptotics.  We typically use this notation and terminology to
describe our concentration bounds, making explicit probability
statements only when required for clarity or rigor.

\section{Exact asymptotics for missing data models}
\label{sec:exact-asymptotics}

In this appendix, we provide additional details on the exact
asymptotic characterization that we sketched out
in~\Cref{sec:exact-asymptotics-body}.


\subsection{Definition of derivatives}

In all cases, the derivatives in the fixed point
equations~\eqref{eq:regr-fixed-pt} do not exist in the classical
sense.  Here we describe how they are to be interpreted.  We begin by
writing the scores explicitly as
\begin{equation}
\label{eq:score-explicit}
\begin{gathered}
    \hpsi_{\prop,i}^f = \ell_\prop'\big( \prox\big[
      \zeta_\prop^\eta\ell_\prop(\,\cdot\,;\action_i^f)\big]'(\heta_{\prop,i}^{f,\loo});
    \action_i^f\big), \qquad \hpsi_{\out,i}^f = -\frac{\zeta_\out^\eta
      w(\eta_{\prop,i}^f) \action_i^f}{1 + \zeta_\out^\eta
      w(\eta_{\prop,i}^f)} (y_i^f - \heta_{\out,i}^{f,\loo}).
\end{gathered}
\end{equation}
Here the reader should recall that for a lower semi-continuous,
proper, convex function $f: \reals^N \rightarrow \reals \cup
\{\infty\}$, the proximal operator is given by
\begin{equation}
\label{eq:prox-def}
    \prox[f](\by) \defn \argmin_{\bv \in \reals^N} \Big\{ \frac{1}2 \|
    \by - \bv \|^2 + f(\bv) \Big\}.
\end{equation}
(See the sources~\cite{bauschke2011convex,parikhBoyd2014} for more
background on proximal operators.)

The derivatives of the score are interpreted as
\begin{equation}
\label{eq:score-deriv-explicit}
\begin{gathered}
    \frac{\de \hpsi_{\prop,i}^f}{\de \eta_{\prop,j}^f}
        =
        \delta_{ij}
        \pi'(\eta_{\prop,j}^f)
        \Big(
            \ell_\prop'\big(\prox\big[\zeta_\prop^\eta\ell_\prop(\,\cdot\,;1)\big](\heta_{\prop,i}^{f,\loo});1\big)
            -
            \ell_\prop'\big(\prox\big[\zeta_\prop^\eta\ell_\prop(\,\cdot\,;0)\big](\heta_{\prop,i}^{f,\loo});0\big)
        \Big),
    \\
    \frac{\de \hpsi_{\prop,i}^f}{\de \heta_{\prop,j}^{f,\loo}}
        =
        \delta_{ij}
        \prox\big[\zeta_\prop^\eta\ell_\prop(\,\cdot\,;\action_i^f)\big]'(\heta_{\prop,i}^{f,\loo})
        =
        \delta_{ij}
        \frac{1}{\zeta_\prop^\eta \ddot\ell_\prop(\heta_{\prop,i}^f;\action_i^f)+1},
    \\
    \frac{\de \hpsi_{\out,i}^f}{\de \eta_{\prop,j}^f}
        =
        -\delta_{ij}
        \Big(
            \frac{\de}{\de \eta_{\prop,i}^f}
            \frac{\zeta_\out^\eta w(\eta_{\prop,i}^f)\pi(\eta_{\prop,i}^f)}{1 + \zeta_\out^\eta w(\eta_{\prop,i}^f)}
        \Big)
        (y_i^f - \heta_{\out,i}^{f,\loo}),
    \qquad
    \frac{\de \hpsi_{\out,i}^f}{\de \heta_{\out,j}^{f,\loo}}
        =
        \action_i^f
        \frac{\zeta_\out^\eta w(\eta_{\prop,i}^f)}{1 + \zeta_\out^\eta w(\eta_{\prop,i}^f)},
\end{gathered}
\end{equation}
Since the functions $\Omega_\prop$ and $\Omega_\out$ are twice
differentiable, both of the functions $\hbtheta_\prop^f$ and
$\hbtheta_\out^f$ are (classically) differentiable in $\by_\prop^f$
and $\by_\out^f$, with derivatives
\begin{equation}
\label{eq:param-est-deriv-fixed-design}
    \bSigma^{1/2} \frac{\de \hbtheta_{\prop}^f }{\de \by_{\prop}^f} =
    \zeta_\prop^\theta\bSigma^{1/2} \Big( \zeta_\prop^\theta \bSigma +
    \nabla^2 \Omega_\prop(\hbtheta_\prop^f) \Big)^{-1} \bSigma^{1/2},
    \quad \mbox{and} \quad \bSigma^{1/2} \frac{\de \hbtheta_{\out}^f
    }{\de \by_{\out}^f} = \zeta_\out^\theta\bSigma^{1/2} \Big(
    \zeta_\out^\theta \bSigma + \nabla^2 \Omega_\out(\hbtheta_\out^f)
    \Big)^{-1} \bSigma^{1/2}.
\end{equation}


\subsection{Existence/uniqueness of fixed points}

The following result certifies uniqueness of the fixed points defined
by the relation~\eqref{eq:regr-fixed-pt}, along with some of their
properties:
\begin{lemma}
\label{lem:regr-fixpt-exist-and-bounds}
    Under Assumption A1, equations \eqref{eq:regr-fixed-pt} have a
    unique solution.  Moreover, these solutions satisfy the following
    properties:
    \begin{itemize}
        \item \textbf{Bounded standard errors.} $0 < c < S_{\ell\ell}
          < C$ for $1 \leq \ell \leq 4$.

        \item \textbf{Bounded effective regularization.} $0 < c <
          \zeta_\out^\theta,\zeta_\out^\eta,\zeta_\prop^\theta,\zeta_\prop^\theta
          < C$.
        \item \textbf{Bounded bias.}
          $|\beta_{\prop\prop}|,|\beta_{\out\prop}| <C$, and
          $\beta_{\prop\prop} > c > 0$.
        \item \textbf{Bounded estimates.} $\| \hbtheta_\prop^f
          \|_{\Ltwo}$, $\|\hbtheta_\out^f \|_{\Ltwo} < C$.
    \end{itemize}
\end{lemma}
\noindent See~\Cref{sec:regr-fixpt-exist-unique-bound}.  for the
proof.

\subsection{Exact asymptotics}

Here we give a more quantitatively refined statement
of~\Cref{thm:exact-asymptotics-body}, in particular using the
high-probability approximation notions introduced
in~\Cref{sec:high-probability}.

\begin{theorem}
\label{thm:exact-asymptotics}
Under Assumption A1, we have:
\begin{enumerate}[(a)]
\item
The estimated population means concentrate:
\begin{equation}
  \htheta_{\prop,0} + \< \bmu_{\sx} , \hbtheta_\prop \> \mydoteq
  \hmu_\prop^f, \qquad \htheta_{\out,0} + \< \bmu_{\sx} ,
  \hbtheta_\out \> \mydoteq \hmu_\out^f.
\end{equation}

\item 
  The effective-regularization terms concentrate:
  \begin{equation}
    \hzeta_{\prop}^\eta \mydoteq \zeta_{\prop}^\eta, \qquad
    \hzeta_{\prop}^\theta \mydoteq \zeta_{\prop}^\theta, \qquad
    \hzeta_{\out}^\eta \mydoteq \zeta_{\out}^\eta, \qquad
    \hzeta_{\out}^\theta \mydoteq \zeta_{\out}^\theta.
  \end{equation}
  
\item 
For any order-2 pseudo-Lipschitz function $\phi: (\reals^p)^6
\rightarrow \reals$,
\begin{equation}
  \phi \big( \hbtheta_{\prop}, \hbtheta_{\out}, \hbmu_{\sx},
  \hbmu_{\sx,\cfd}, \hbtheta_{\prop}^{\loo}, \hbtheta_{\out}^{\loo}
  \big) \mydoteq \E\Big[ \phi\Big( \hbtheta_{\prop}^f,
    \hbtheta_{\out}^f, \bSigma^{-1/2} \by_{\sx}^f, \bSigma^{-1/2}
    \by_{\sx,\cfd}^f, \bSigma^{-1/2} \by_{\prop}^f, \bSigma^{-1/2}
    \by_{\out}^f \Big) \Big].
\end{equation}
        
\item 
  For any order-2 pseudo-Lipschitz function $\phi:\reals^4
  \rightarrow \reals$,
  \begin{equation}
    \frac{1}{n} \sum_{i=1}^n \phi\big( \eta_{\prop,i}, \eta_{\out,i},
    \hpsi_{\prop,i}, \hpsi_{\out,i} \big) \mydoteq \E\Big[ \phi\Big(
      \eta_{\prop,i}^f, \eta_{\out,i}^f, \hpsi_{\prop,i}^f,
      \hpsi_{\out,i}^f \Big) \Big].
  \end{equation}
\end{enumerate}
where in both locations, $\phi$ is an order-2 pseudo-Lipschitz
function (with the appropriate number of arguments).
\end{theorem}
\noindent See~\Cref{sec:exact-asymptotics-proof} for the proof.


\subsection{Reduction to independent covariates}
\label{sec:independent-covariates}

Without loss of generality, we can assume that $\bSigma = \id_p$.
Indeed, if one replaces $\bX$ by $\bX \bSigma^{-1/2}$, $\btheta_\prop$
with $\bSigma^{1/2} \btheta_\prop$, $\btheta_\out$ with $\bSigma^{1/2}
\btheta_\out$, $\bmu_{\sx}$ with $\bSigma^{-1/2} \btheta_\out$, $\bv
\mapsto \Omega_\out(\bv)$ with $\bv \mapsto
\Omega_\out(\bSigma^{-1/2}\bv)$, and $\bv \mapsto \Omega_\prop(\bv)$
with $\bv \mapsto \Omega_\prop(\bSigma^{-1/2}\bv)$, one can see that
Assumption A1 holding before these replacements is equivalent to
Assumption A1 holding after these replacements if $\bSigma$ is
replaced with $\id_p$.  Moreover, the conclusion
of~\Cref{thm:exact-asymptotics} as well as all other results in the
paper (\Cref{thm:pop-mean,thm:db-normality}, as well as
\Cref{lem:propensity-summary-stats}) are equivalent to the same
results applied with $\bSigma = \id_p$ to the model arrived at after
these replacements.  Thus in the remainder of remainder of the
Appendix, we assume that $\bSigma = \id_p$.


\section{The Cholesky decomposition}
\label{sec:cholesky}

The proof of~\Cref{thm:exact-asymptotics} is greatly facilitated by
the Cholesky decomposition.  The main ideas are most easily understood
if one assumes all matrices are invertible, and all Cholesky
decompositions are defined in the standard way \cite[Corollary
  7.2.9]{horn13}.  This section is devoted to the technical details
required to relax this invertibility assumption.  Accordingly, for the
first reading, we suggest the reader may skip this section and operate
under the simplifying assumption of invertibility.

In order to define the Cholesky decomposition for possibly
non-invertible matrices, we need to adopt certain conventions so as to
resolve its non-uniqueness, which we describe in this section.

\subsection{Definition of the decomposition}

Consider $\bK \in \S_{\geq0}^k$ and $(X_1, \ldots, X_k) \sim
\normal(0, \bK)$.  For each index $\ell \in [k]$, we have the
decomposition
\begin{equation}
\label{eq:seq-regr}
X_\ell = X_\ell^\| + \alpha_\ell X_\ell^\perp,
\end{equation}
where $X_\ell^\|$ is a linear combination of $\{X_{\ell'}\}_{\ell' < \ell}$
and
$\alpha_\ell X_\ell^\perp$ is independent of $\{X_{\ell'}\}_{\ell' < \ell}$.
If $\rank(\bK_\ell) = \rank(\bK_{\ell-1}) + 1$, we call the index $\ell$ \textit{innovative}.
In this case, $\alpha_\ell X_\ell^\perp \neq 0$,
and we take $\alpha_\ell > 0$ and $\| X_\ell^\perp \|_{\Ltwo} = 1$.
If $\rank(\bK_\ell) = \rank(\bK_{\ell-1})$, we call the index $\ell$ \textit{predictable}.
In the case, $\alpha_\ell X_\ell^\perp = 0$,
and we take $\alpha_\ell = 0$ and $ X_\ell^\perp = 0$.
One of these two cases always occurs.
With these conventions, the decomposition in the previous display is unique.

This construction uniquely determines a matrix $\bL \in \reals^{k \times k}$ defined by
\begin{equation}
    L_{\ell \ell'} = \< X_\ell , X_{\ell'}^\perp \>_{\Ltwo}
    \quad
    \text{for $\ell,\ell' \in [k]$.}
\end{equation}
We show that $\bK = \bL\bL^\top$, and $\bL$ is the unique such choice
that is lower-triangular and ``innovation-compatible'' with $\bK$:
\begin{definition}
\label{def:predictable-innovative}
 A lower-triangular matrix $\bL \in \reals^{k \times k}$
 \emph{innovation-compatible} with $\bK$ if the $\ell^\text{th}$
 column of $\bL$ is $\bzero$ for all indices $\ell$ predictable with
 respect to $\bK$.
\end{definition}
\begin{lemma}[Cholesky decomposition]
\label{lem:cholesky}
    The matrix $\bL$ is lower-triangular, has non-negative entries on
    its diagonal, is innovation-compatible with respect to $\bK$, and
    satisfies $\bK = \bL\bL^\top$.  The $\ell^\text{th}$ entry on its
    diagonal is 0 if and only if $\ell$ is predictable with respect to
    $\bK$.  The sub-matrix $\bL_\ell$ depends on $\bK$ only via
    $\bK_\ell$.  Moreover,
    \begin{equation}
    \label{eq:original-from-perp}
        X_\ell = \sum_{\ell'=1}^\ell L_{\ell \ell'} X_{\ell'}^\perp.
    \end{equation}
\end{lemma}

\begin{proof}[Proof of~\Cref{lem:cholesky}]
 Because the quantity $X_{\ell'}^\perp$ is independent of $X_\ell$ for
 $\ell' > \ell$, we have $L_{\ell \ell'} = \< X_\ell , X_{\ell'}^\perp
 \>_{\Ltwo} = 0$ for $\ell' > \ell$, so $\bL$ is lower-triangular.
 Because $X_\ell^\perp$ is independent of $X_\ell^\|$, $L_{\ell \ell}
 = \<X_\ell, X_\ell^{\perp}\>_{\Ltwo} = \alpha_\ell \| X_\ell^\perp
 \|_{\Ltwo}^2 \geq 0$ because we have chosen that $\alpha_\ell \geq
 0$.  Thus, $\bL$ has non-negative entries on its diagonal.  If
 $\ell'$ is predictable, then $X_{\ell'}^\perp = 0$, so that $L_{\ell
   \ell'} = \< X_\ell , X_{\ell'}^\perp \>_{\Ltwo} = 0$.  Thus, $\bL$
 is innovation compatible with respect to $\bK$.  $\alpha_\ell \|
 X_\ell^\perp \|_{\Ltwo}^2 = 0$ if and only if $\ell$ is predictable,
 whence the $\ell^\text{th}$ diagonal entry of $\bL$ is 0 if and only
 if $\ell$ is predictable.  The sub-matrix $\bL_\ell$ depends on $\bK$
 only via $\bK_\ell$ because its construction relies only on the
 random variables $(X_1, \ldots, X_\ell) \sim \normal(0,\bK_\ell)$.

 Using the fact that $X_\ell^\|$ can be written as a linear
 combination of $\{X_{\ell'}\}_{\ell' < \ell}$ and using $X_1 =
 \alpha_{11} X_1^\perp$ as a base case, by induction each $X_\ell$ can
 be written as a linear combination
 \begin{equation}
   X_\ell = \sum_{\ell'=1}^\ell \alpha_{\ell \ell'} X_{\ell'}^\perp.
    \end{equation}
    Taking the inner product with $X_{\ell''}^\perp$ on both sides and
    using that $\< X_{\ell'}^\perp , X_{\ell''}^\perp \>_{\Ltwo} = 0$
    when $\ell' \neq \ell''$, we have $\alpha_{\ell \ell'} \|
    X_{\ell'}^\perp \|^2 = \<X_{\ell},X_{\ell'}^\perp\> = L_{\ell
      \ell'}$.  If $\ell'$ is innovative, this gives $\alpha_{\ell
      \ell'} = L_{\ell \ell'}$.  If $\ell'$ is predictable, then
    $X_{\ell'} = 0$, so that we can set $\alpha_{\ell \ell'} =
    L_{\ell\ell'}$ without affecting the sum.
    equation~\eqref{eq:original-from-perp} follows.  It implies
    $K_{\ell\ell'} = \< X_\ell , X_{\ell'} \>_{\Ltwo} = \sum_{j=1}^k
    L_{\ell j} L_{\ell' j} = (\bL\bL^\top)_{\ell \ell'}$, whence $\bK
    = \bL\bL^\top$.
\end{proof}

We next define the Cholesky pseudo-inverse $\bL^\ddagger$, which
behaves, in certain ways, like the inverse of $\bL$ even when $\bL$ is
singular.  By equation~\eqref{eq:seq-regr}, we have the inclusion
$X_\ell^\perp \in \spn\{X_{\ell'}\}_{\ell' \leq \ell}$.  Because
$X_{\ell'+1}$ is linearly independent of $\{ X_{\ell''} \}_{\ell''
  \leq \ell'}$ if and only if $\ell'+1$ is innovative, $\{X_{\ell'}:
\ell' \leq \ell,\, \text{$\ell'$ innovative}\}$ is a basis for
$\spn\{X_{\ell'}\}_{\ell' \leq \ell}$.  We can thus uniquely write
$X_\ell^\perp$ as a linear combination
\begin{equation}
\label{eq:perp-from-innovatives}
    X_\ell^\perp = \sum_{\substack{\ell' \leq \ell \\ \text{$\ell'$
          innovative}}} L_{\ell \ell'}^\ddagger X_{\ell'},
\end{equation}
which we take to define $L_{\ell \ell'}^\ddagger$ for innovative
$\ell' \leq \ell$.  Define $L_{\ell \ell'}^\ddagger = 0$ for $\ell'$
predictable or $\ell' > \ell$.  Then $\bL^\ddagger$ is
lower-triangular, innovation-compatible with $\bK$, and
\begin{equation}
\label{eq:perp-from-cholinv}
    X_\ell^\perp = \sum_{\ell' = 1}^k L_{\ell \ell'}^\ddagger
    X_{\ell'}.
\end{equation}
The Cholesky pseudo-inverse has the following properties.
\begin{lemma}[Cholesky pseudo-inverse]
\label{lem:cholesky-pseudo-inverse}
    The matrix $\bL^\ddagger \in \reals^{k \times k}$
    is lower-triangular, has non-negative entries on its diagonal, and is innovation-compatible with $\bK$.
    If $\ell$ is predictable, then the $\ell^{\text{th}}$ row of $\bL^\ddagger$ is 0.
    The sub-matrix $\bL_\ell^\ddagger$ depends on $\bK$ only via $\bK_\ell$.
\end{lemma}

\begin{proof}[Proof of~\Cref{lem:cholesky-pseudo-inverse}]
We have already shown that $\bL^\ddagger$ is lower triangular and
innovation-compatible with $\bK$.  Because
$\<X_\ell^\perp,X_{\ell'}\>_{\Ltwo} = 0$ for $\ell' < \ell$ and
$L_{\ell \ell'}^\ddagger = 0$ for $\ell' > \ell$, taking the inner
product of equation~\eqref{eq:perp-from-cholinv} with $X_\ell^\perp$
gives $\| X_\ell^\perp \|_{\Ltwo}^2 = L_{\ell \ell}^\ddagger \< X_\ell
, X_\ell^\perp \>_{\Ltwo} = L_{\ell \ell}^\ddagger L_{\ell \ell} \|
X_\ell^\perp \|_{\Ltwo}^2$.  If $\ell$ is predictable, we have
$L_{\ell\ell}^\ddagger = 0$, and if $\ell$ is innovative, we must have
$L_{\ell\ell}^\ddagger > 0$ because $\| X_\ell^\perp \|_{\Ltwo}^2 = 1$
and $L_{\ell\ell} \geq 0$.  Thus, $\bL^\ddagger$ has non-negative
entries on its diagonal.  If $\ell$ is predictable, then $X_\ell^\perp
= 0$, whence equation~\eqref{eq:perp-from-innovatives} is solved by
$L_{\ell \ell'}^\ddagger = 0$ for all $\ell' \leq \ell$ and $\ell'$
innovative.  Because the decomposition in
equation~\eqref{eq:perp-from-innovatives} is unique, the
$\ell^\text{th}$ row of $\bL^\ddagger$ is 0 if $\ell$ is predictable.
The sub-matrix $\bL_\ell^\ddagger$ depends on $\bK$ only via
$\bK_\ell$ because its construction relies only on the random
variables $(X_1,\ldots,X_\ell) \sim \normal(0,\bK_\ell)$.
\end{proof}

Define the pseudo-inverse $\bK^\ddagger \defn (\bL^\ddagger)^\top
\bL^\ddagger$, as well as the $k$-dimensional matrices
\begin{equation}
\begin{gathered}
    \id_{\bK} \be_\ell = \id_{\bK}^\perp \be_\ell = \bzero \quad \text{if $\ell$ is predictable},
    \\
    \id_{\bK} \bx = \bx \quad \text{if $\bx \in \range(\bK)$}
    \qquad\text{and}\qquad
    \id_{\bK}^\perp \be_\ell = \be_\ell \quad \text{if $\ell$ is innovative}.
\end{gathered}
\end{equation}
The next lemma establishes the sense in which Cholesky pseudo-inverse
$\bL^\ddagger$ behaves like an inverse of $\bL$, and the
pseudo-inverse $\bK^\ddagger$ behaves like an inverse of $\bK$.

\begin{lemma}[Pseudo-inverse identities]
\label{lem:cholesky-inverse-identities}
    The Cholesky pseudo-inverse of $\bL$ satisfies
    \begin{equation}
        \bL\bL^{\ddagger} = \id_{\bK}, 
        \qquad
        \bL^{\ddagger} \bL = \id_{\bK}^\perp,
        \qquad
        \bL \id_{\bK}^\perp = \id_{\bK} \bL = \bL,
        \qquad
        \bL^\ddagger \id_{\bK} = \id_{\bK}^\perp \bL^\ddagger = \bL^\ddagger.
    \end{equation} 
    The pseudo-inverse of $\bK$ satisfies
    \begin{equation}
    \label{eq:K-pseudo-inverse}
        \bK\bK^\ddagger = \id_{\bK},
        \qquad
        \bK^\ddagger\bK = \id_{\bK}^\top.
    \end{equation}
    The matrix $\id_{\bK}^\perp$ is diagonal, so $\id_{\bK}^\perp = (\id_{\bK}^\perp)^\top$.
    The matrix $\id_{\bK}$ is the unique matrix with the properties that $\id_{\bK}^\top \be_\ell = \be_\ell$ if $\ell$ is innovative,
    and $\id_{\bK}^\top \bx = \bzero$ if $\bx \in \range(\bK)^\perp$.
\end{lemma}

\begin{proof}[Proof of~\Cref{lem:cholesky-inverse-identities}]
    If $\ell$ is predictable, then the $\ell^{\text{th}}$ column of
    both $\bL$ and $\bL^\ddagger$ are 0 by innovation compatibility,
    whence $\bL\bL^{\ddagger}\be_\ell = \bL \bzero = \bzero$ and
    $\bL^{\ddagger}\bL\be_\ell = \bL^\ddagger \bzero = \bzero$.  By
    equations~\eqref{eq:original-from-perp} and
    \eqref{eq:perp-from-cholinv}, almost surely
    \begin{equation}
        X_\ell
            =
            \sum_{\ell' = 1}^k L_{\ell\ell'} X_{\ell'}^\perp
            =
            \sum_{\ell' = 1}^k L_{\ell\ell'} \sum_{\ell'' = 1}^k L_{\ell' \ell''}^\ddagger X_{\ell''},
        \quad
        \text{whence }
        \begin{pmatrix}
            X_1 \\[1pt]
            \vdots \\[1pt]
            X_k
        \end{pmatrix}
        =
        \bL\bL^\ddagger
        \begin{pmatrix}
            X_1 \\[1pt]
            \vdots \\[1pt]
            X_k
        \end{pmatrix}
        =
        \bL\bL^\ddagger\bL
        \begin{pmatrix}
            X_1^\perp \\[1pt]
            \vdots \\[1pt]
            X_k^\perp
        \end{pmatrix}.
    \end{equation}
    The support of $(X_1,\ldots,X_k)^\top$ is exactly $\range(\bK)$, whence $\bL\bL^\ddagger \bx = \bx$ for all $\bx \in \range(\bK)$.
    Note that $\{ \be_\ell : \text{$\ell$ predictable} \} \cup \range(\bK)$ span $\reals^k$ because $\range(\bK) = \range(\bL)$,
    and $\bL$ is lower triangular with 0's on the diagonal only in predictable indices.
    Thus, $\bL\bL^\ddagger$ and $\id_{\bK}$ act by matrix multiplication equivalently on a set spanning $\reals^k$, so are the same.
    Likewise,
    almost surely
    \begin{equation}
        X_\ell^\perp
            =
            \sum_{\ell' = 1}^k L_{\ell\ell'}^\ddagger X_{\ell'}
            =
            \sum_{\ell' = 1}^k L_{\ell\ell'}^\ddagger \sum_{\ell'' = 1}^k L_{\ell' \ell''} X_{\ell''}^\perp,
        \quad
        \text{whence }
        \begin{pmatrix}
            X_1^\perp \\[1pt]
            \vdots \\[1pt]
            X_k^\perp
        \end{pmatrix}
        =
        \bL^\ddagger\bL
        \begin{pmatrix}
            X_1^\perp \\[1pt]
            \vdots \\[1pt]
            X_k^\perp
        \end{pmatrix}
        =
        \bL^\ddagger\bL\bL^\ddagger
        \begin{pmatrix}
            X_1 \\[1pt]
            \vdots \\[1pt]
            X_k
        \end{pmatrix}.
    \end{equation}
    The support of $(X_1^\perp,\ldots,X_k^\perp)^\top$ is exactly
    $\spn\{\be_\ell : \text{$\ell$ innovative}\}$, whence
    $\bL^\ddagger\bL\be_\ell = \be_\ell$ for all $\ell$ innovative.
    Thus, $\bL^\ddagger\bL$ and $\id_{\bK}^\perp$ act by matrix
    multiplication equivalently on a set spanning $\reals^k$, so are
    the same.  For all predictable $\ell$, $\bL\id_{\bK}^\perp
    \be_\ell = \id_{\bK} \bL \be_\ell = \bL \be_\ell = \bzero$ because
    $\id_{\bK}^\perp \be_\ell = \bL \be_\ell = \bzero$.  For
    innovative $\ell$, $\bL \id_{\bK}^\perp \be_\ell = \bL \be_\ell$.
    Thus $\bL\id_{\bK}^\perp = \bL$.  Further, $\bL \be_\ell \in
    \range(\bK)$, whence $\id_{\bK} \bL \be_\ell = \bL \be_\ell$.
    Thus, $\bL = \id_{\bK} \bL$.  For all predicable $\ell$,
    $\bL^\ddagger \id_{\bK} \be_\ell = \id_{\bK}^\perp \bL^\ddagger
    \be_\ell = \bL^\ddagger \be_\ell = \bzero$ because $\id_{\bK}
    \be_\ell = \bL^\ddagger \be_\ell = \bzero$.  For $\bx \in
    \range(\bK)$, $\bL^\ddagger \id_{\bK} \bx = \bL^\ddagger \bx$.
    Thus, $\bL^\ddagger \id_{\bK}$ and $\bL^\ddagger$ act equivalently
    by matrix multiplication on a basis, so are the same Further,
    $\bL^\ddagger \bx \in \spn\{\be_\ell : \text{$\ell$
      innovative}\}$, whence $\id_{\bK}^\perp \bL^\ddagger \bx =
    \bL^\ddagger \bx$.  Thus $\bL^\ddagger$ and $\id_{\bK}^\perp
    \bL^\ddagger$ act equivalently by matrix multiplication on a
    basis, so are the same.  The identities $\bK\bK^\ddagger =
    \id_{\bK}$ and $\bK^\ddagger\bK = \id_{\bK}^\top$ are consequences
    of the above identities.

    By its definition, $\id_{\bK}^\perp$ is the diagonal matrix that
    has entries $(\id_{\bK}^\perp)_{\ell\ell} = 1$ for innovative
    $\ell$ and $(\id_{\bK}^\perp)_{\ell\ell} = 0$ for predictable
    $\ell$.  Note that for any $\ell$ innovative, $\ell'$ predictable,
    $\bx \in \range(\bK)$, and $\bx' \in \range(\bK)^\perp$, we have
    $\be_{\ell'}^\top \id_{\bK}^\top \be_\ell = \bzero^\top \be_\ell =
    0$, $\bx^\top \id_{\bK}^\top \be_\ell = \bx^\top \be_\ell =
    x_\ell$, $\be_{\ell'}^\top \id_{\bK}^\top \bx' = \bzero^\top \bx'
    = 0$, $\bx^\top \id_{\bK}^\top \bx' = \bx^\top \bx' = 0$.  Because
    $\{ \be_{\ell'} : \text{$\ell'$ predictable} \} \cup \range(\bK)$
    spans $\reals^k$, we conclude $\id_{\bK}^\top \be_\ell = \be_\ell$
    and $\id_{\bK}^\top \bx = \bzero$.
\end{proof}

\subsection{Explicit construction of the Cholesky decomposition}
\label{sec:chol-explicit}

At times we will need to quantify the continuity of $\bL$, $\bL^\ddagger$, and $\bK^\ddagger$ in $\bK$.
This is most easily done by explicit computation,
which is carried out here for reference.

Let $\sL: \S_{\geq 0}^k \rightarrow \reals^{k \times k}$,
$\sL^\ddagger:\S_{\geq 0}^k \rightarrow \reals^{k \times k}$,
and $\sK^\ddagger:\S_{\geq 0}^k \rightarrow \S_{\geq 0}^k$
be the function which takes $\bK \in \S_{\geq 0}^k$ to the Cholesky matrix $\bL$,
Cholesky pseudo-inverse $\bL^\ddagger$,
and pseudo-inverse $\bK^\ddagger$.
\begin{lemma}[Explicit Cholesky decomposition]
\label{lem:cholesky-explicit}
    The functions $\sL$ and $\sL^\ddagger$ are defined inductively in $\ell$ by
    \begin{equation}
    \label{eq:L-explicit}
    \begin{gathered}
        \sL(\bK)_{\ell \ell'}
            =
            \begin{cases}
                \sum_{\ell''=1}^{\ell'} \sL^\ddagger(\bK)_{\ell' \ell''} K_{\ell\ell''}
                \quad &\text{if $\ell' < \ell$},
            \\[3pt]
                \sqrt{
                    K_{\ell\ell} - \sum_{j = 1}^{\ell-1} \sL(\bK)_{\ell \ell'}^2
                }
                \quad &\text{if $\ell' = \ell$},                        
            \end{cases}
        \\[5pt]
        \sL^\ddagger(\bK)_{\ell\ell'}
            =
            \begin{cases}
                -\frac{1}{\sL(\bK)_{\ell\ell}} \sum_{\ell''=\ell'}^{\ell-1} \sL(\bK)_{\ell \ell''}\sL^\ddagger(\bK)_{\ell''\ell'}
                \quad &\text{if $\ell' < \ell$},
            \\[3pt]
                \frac{1}{\sL(\bK)_{\ell\ell}} \quad&\text{if $\ell' = \ell$ and $\ell$ is innovative},
            \\[3pt]
                0\quad &\text{if $\ell$ is predictable}.
            \end{cases}
    \end{gathered}
    \end{equation}
    Moreover, $\sK^\ddagger(\bK) = \sL^\ddagger(\bK)^\top \sL^\ddagger(\bK)$.
\end{lemma}

\begin{proof}[Proof of~\Cref{lem:cholesky-explicit}]
    First, the base case $\ell = 1$.
    We have $\sL(\bK)_{11} = \| X_1^\perp \|_{\Ltwo} = \| X_1 \|_{\Ltwo} = K_{11}^{1/2}$.
    Because $X_1^\perp = X_1 / \| X_1 \|_{\Ltwo}$ if $X_1 \neq 0$ and $L_{11}^\ddagger = 0$ otherwise,
    we have
    $\sL^\ddagger(\bK)_{11} = \frac{1}{\sL_{11}(\bK)}$ if 1 is innovative and 0 if 1 is predictable.

    Next, assume we have defined $\sL(\bK)_{\ell'\ell''}$ and
    $\sL^\ddagger(\bK)_{\ell'\ell''}$ for $\ell'' \leq \ell' < \ell$.
    Then, for $\ell' < \ell$, $ \sL(\bK)_{\ell\ell'} = \< X_\ell ,
    X_{\ell'}^\perp \>_{\Ltwo} = \Big\< X_\ell ,
    \sum_{\ell''=1}^{\ell'} \sL^\ddagger(\bK)_{\ell' \ell''}
    X_{\ell''} \Big\>_{\Ltwo} = \sum_{\ell''=1}^{\ell'}
    \sL^\ddagger(\bK)_{\ell' \ell''} K_{\ell\ell''} $, and $
    \sL(\bK)_{\ell\ell} = \Big\| X_\ell - \sum_{\ell'=1}^{\ell-1}
    \<X_\ell,X_{\ell'}^\perp\>_{\Ltwo}X_{\ell'}^{\perp} \Big\|_{\Ltwo}
    = \sqrt{ K_{\ell\ell} - \sum_{j = 1}^{\ell-1} \sL(\bK)_{\ell
        \ell'}^2 } $.  Moreover,
 \begin{equation}
   X_k^\perp = \frac{X_\ell - \sum_{\ell'=1}^{\ell-1}
     \<X_\ell,X_{\ell'}^\perp\>_{\Ltwo}X_{\ell'}^{\perp}}{\|X_\ell
     - \sum_{\ell'=1}^{\ell-1}
     \<X_\ell,X_{\ell'}^\perp\>_{\Ltwo}X_{\ell'}^{\perp}\|_{\Ltwo}}
   = \frac{X_\ell - \sum_{\ell'=1}^{\ell-1}
     \sL(\bK)_{\ell\ell'}\sum_{\ell''=1}^{\ell'}
     \sL^\ddagger(\bK)_{\ell'\ell''}
     X_{\ell''}}{\sL(\bK)_{\ell\ell}},
 \end{equation}
 whence the form given for $\sL^\ddagger(\bK)_{\ell\ell'}$ holds by
 comparison with equation~\eqref{eq:perp-from-cholinv}.
\end{proof}


\section{Reduction to Gordon standard form}

The first step in proving~\Cref{thm:exact-asymptotics} is to rewrite
the optimization problems~\eqref{eq:propensity-fit}
and~\eqref{eq:outcome-fit} in a min-max form that is more immediately
amenable to the application of the sequential Gordon inequality.  We
likewise rewrite the fixed point equations~\eqref{eq:regr-fixed-pt} in
an equivalent form that is easier to work with when applying the
sequential Gordon inequality.  Without loss of generality, we assume
that $\bSigma = \id_p$ throughout our proof.
(See~\Cref{sec:independent-covariates} for justification.)

\subsection{Reindexing and simplification to independent covariates}

From Assumption~A1, the function $\pi$ is differentiable and strictly
increasing.  Define the limits \mbox{$\pi_\infty \defn \lim_{\eta
    \rightarrow \infty} \pi(\eta)$} and \mbox{$\pi_{-\infty} =
  \lim_{\eta \rightarrow -\infty} \pi(\eta)$.}  There is a probability
distribution on $\reals$ with density with respect to Lebesgue measure
given by $\pi'(\,\cdot\,)/(\pi_\infty - \pi_{-\infty})$.  Let
$\beps_\prop$ have components $\eps_{\prop,i}$ distributed iid from
this measure.  Let $\beps_\prop'$ be independent of $\beps_\prop$ have
iid components $\eps_{\prop,i}' \in \{0,1,\circ\}$ with probabilities
$\P(\eps_{\prop,i}' = 0) = \pi_{-\infty}$, $\P(\eps_{\prop,i}' = 1) =
1 - \pi_\infty$, and $\P(\eps_{\prop,i}' = \circ) = \pi_\infty -
\pi_{-\infty}$.  We can realize the model~\eqref{eq:model} by setting
\begin{equation}
\label{eq:prop-functional-form}
 \action_i = \indic\{ \eps_{\prop,i}' = 1\} + \indic\{ \eps_{\prop,i}'
 = \circ\} \indic\{ \eps_{1,i} \leq \theta_{\prop,0} + \< \bx_i ,
 \btheta_{\prop} \> , \}.
\end{equation}

The proof of~\Cref{thm:exact-asymptotics} has an inductive
structure. For the purposes of the proof, it is convenient to use an
alternative notation in which random-design quantities are indexed by
the order in which they appear in this induction. Define
\begin{equation}
\begin{aligned}
    (\theta_{1,0},\btheta_1) &\defn (\theta_{\prop,0},\btheta_\prop),
    \quad
    &
    (\beps_1,\beps_1') &\defn (\beps_\prop,\beps_\prop'),
    &
    (\theta_{2,0},\btheta_2) &\defn (\theta_{\out,0},\btheta_\out),
    \quad
    & 
    \beps_2 &\defn \beps_{\out},
  \\ 
  \by_1 &\defn \ba, 
  \quad
  &
  \bmeta_1 &\defn \bmeta_\prop,
  \quad
  & \by_2 &\defn \by, 
  &\quad
    \bmeta_2 &\defn \bmeta_\out.
\end{aligned}
\end{equation}
We will sometimes go back and forth between this notation and that in equation~\eqref{eq:rnd-des-objects}, and will attempt to do so
in a way that minimizes potential confusion.  


\subsection{Min-max formulation of primary optimization}

Our
first step is to rewrite the optimizations in
equations~\eqref{eq:propensity-fit} and \eqref{eq:outcome-fit} as saddle point problems of the
form:
\begin{equation}
\label{eq:min-max}
\begin{gathered}
    \min_{\substack{v_0 \in \reals \\ \bv \in \reals^p}} 
    \max_{\bu \in \reals^n}
        \Big\{
            \bu^\top \bA \bv 
                + 
                \phi_k(\bu;v_0,\bv;)
        \Big\},
\end{gathered}
\end{equation}
where $\bA = \bX - \ones \bmu_{\sx}^\top$ is the mean-centered version of $\bX$,
with entries $A_{ij} \iid \normal(0,1)$.
It is also convenient to define the true parameters $(\theta_{\prop,0},\btheta_\prop)$ and $(\theta_{\out,0},\btheta_\out)$ as well as the mean and confounded mean estimates $\hbmu_{\sx}$, $\hbmu_{\sx,\cfd}$ as solutions to saddle point problems.
Thus,
define
\begin{equation}
\label{eq:objectives}
\begin{gathered}
    \phi_1(\bu;v_0,\bv) \defn
        -
        \indic\{\bu = \bzero \}
        + 
        \indic\{ (v_0,\bv) = (\theta_{1,0},\btheta_1)\},
    \\
    \phi_2(\bu;v_0,\bv) \defn
        -
        \indic\{ \bu = \bzero \}
        + 
        \indic\{ (v_0,\bv) = (\theta_{2,0},\btheta_2) \},
    \\
    \phi_3(\bu;v_0,\bv) \defn
        -\indic\{ \bu = -\ones /n \}
        +\indic\{ (v_0,\bv) = (0,\bzero) \},
    \\
    \phi_4(\bu;v_0,\bv) \defn
        -\indic\{ \bu = -\ba /n \}
        +\indic\{ (v_0,\bv) = (0,\bzero) \},
    \\
    \phi_5(\bu;v_0,\bv) \defn
        \< \bu , \ones \>(v_0 + \< \bmu_{\sx},\bv\>)
        - \ell_5^*(\bu;\bw,\by_1,\by_2)
        + \Omega_5(\bv),
    \\
    \phi_6(\bu;v_0,\bv) \defn
        \<\bu , \ones \>(v_0 + \< \bmu_{\sx},\bv\>)
        -\ell_6^*(\bu;\bw,\by_1,\by_2)
        + \Omega_6(\bv),
\end{gathered}
\end{equation}
where 
\begin{equation}
\label{eq:ellk*-def}
\begin{gathered}
    \ell_5^*(\bu;\bw,\by_1,\by_2)
        \defn
        \frac{1}{n}
        \sum_{i=1}^n \ell_{\prop}^*(nu_i;y_{1,i}),
    \qquad
    \ell_6^*(\bu;\bw,\by_1,\by_2)
        \defn
        \< \bu , \by_2 \> 
        +
        \frac{n}2
        \sum_{i=1}^n w_i^{-1} \frac{u_i^2}{y_{1,i}},
    \\
    \Omega_5(\bv) \defn \Omega_{\prop}(\bv),
    \qquad
    \Omega_6(\bv) \defn \Omega_{\out}(\bv),
\end{gathered}
\end{equation}
and we adopt the convention that $u\mapsto u^2/0$ is the convex indicator function $\indic\{u = 0\}$.
Here $\ell_\prop^*$ is the Fenchel-Legendre conjugate of $\ell_\prop$ in its first argument :
\begin{equation}
    \ell_\prop^*(nu_i;y_{1,i})
        \defn
        \sup_{\eta \in \reals}
        \Big\{
            nu_i\eta - \ell_\prop(\eta;y_{1,i})
        \Big\}.
\end{equation}
The min-max problems in equation~\eqref{eq:min-max} are referred to as the \emph{primary optimizations},
and the objectives as the \emph{primary objectives}.
We denote the saddle point of the $k^\text{th}$ primary optimization by $(\bu_k^{\PO};v_{k,0}^{\PO},\bv_k^{\PO})$,
where ``$\PO$'' denotes ``primary optimization.''
By Fenchel-Legendre duality,
\begin{equation}
\label{eq:po-to-regr-cov}
\begin{aligned}
    (v_{1,0}^{\PO},\bv_1^{\PO}) &= (\theta_{\prop,0},\btheta_\prop),
        \qquad&
        (v_{2,0}^{\PO},\bv_2^{\PO}) &= (\theta_{\prop,0},\btheta_\out),
    \\
    (v_{3,0}^{\PO},\bv_3^{\PO}) &= (0,\bzero),
        \qquad&
        (v_{4,0}^{\PO},\bv_4^{\PO}) &= (0,\bzero),
    \\
    (v_{5,0}^{\PO},\bv_5^{\PO}) &= (\htheta_{\prop,0},\hbtheta_\prop),
    \qquad&
    (v_{6,0}^{\PO},\bv_6^{\PO}) &= (\htheta_{\out,0},\hbtheta_\out),
\end{aligned}
\end{equation}
and
\begin{equation}
\label{eq:uPO-identity}
\begin{aligned}
    \bu_1^{\PO} &= \bzero,
        \quad&
        \bu_2^{\PO} &= \bzero,
    \\
    \bu_2^{\PO} &= -\ones/n,
        \quad&
        \bu_2^{\PO} &= -\ba/n = -\by_1/n,
    \\
    \bu_5^{\PO} &= \frac{1}{n} \ell_{\prop}'\big(\hbmeta_\prop;\ba\big)
        = \frac{1}{n} \hbpsi_{\prop},
    \quad&
    \bu_6^{\PO} &= -\frac{1}{n} \ba \odot \bw \odot \big(\by - \hbmeta_\out \big)
        = 
        \frac{1}{n} \hbpsi_{\out},
\end{aligned}
\end{equation}
where $\ell_\prop'$ is applied coordinate-wise.
\Cref{thm:exact-asymptotics} follows from an exact asymptotic
characterization of $(v_{1,0}^{\PO},\bv_1^{\PO})$,
$(v_{2,0}^{\PO},\bv_2^{\PO})$, $(v_{3,0}^{\PO},\bv_3^{\PO})$,
$(v_{4,0}^{\PO},\bv_4^{\PO})$, $(v_{5,0}^{\PO},\bv_5^{\PO})$,
$(v_{6,0}^{\PO},\bv_6^{\PO})$ and $\bu_3^{\PO}$, $\bu_4^{\PO}$,
$\bu_5^{\PO}$, $\bu_6^{\PO}$.

\subsection{Definition of state evolution}

Similar to~\Cref{thm:exact-asymptotics}, the exact asymptotic
characterization of $(v_{1,0}^{\PO},\bv_1^{\PO})$,
$(v_{2,0}^{\PO},\bv_2^{\PO})$, $(v_{3,0}^{\PO},\bv_3^{\PO})$,
$(v_{4,0}^{\PO},\bv_4^{\PO})$, $(v_{5,0}^{\PO},\bv_5^{\PO})$,
$(v_{6,0}^{\PO},\bv_6^{\PO})$ and $\bu_3^{\PO}$, $\bu_4^{\PO}$,
$\bu_5^{\PO}$, $\bu_6^{\PO}$ states that these behave like analogous
quantities in two models on distinct probability spaces.  We call
these models \textit{state evolution}.  One probability space contains
Gaussian randomness $\bG^{\SE} \in \reals^{p \times 6}$ distributed
$\bG^{\SE} \sim \normal(0,\bK_g \otimes \id_p)$.  The other contains
Gaussian randomness $\bH^{\SE} \in \reals^{n \times 6}$ distributed
$\bH^{\SE} \sim \normal(0,\bK_h \otimes \id_n)$ and randomness
$\beps_1^\SE$, ${\beps_1^\SE}'$, and $\beps_2^\SE$ with the same
distribution as $\beps_1$, $\beps_1'$, and $\beps_2$, respectively.
The covariances $\bK_g \in \S_{\geq 0}^6$ and $\bK_h \in \S_{\geq
  0}^6$ are defined below.  The superscript ``$\SE$'' denotes ``state
evolution.''  On these probability spaces, we define $\bu_i^{\SE}$,
$\bv_k^{\SE}$, $k \leq 6$ by
\begin{align}
\label{eq:SE-opt}
\begin{gathered}
    \bu_k^{\SE}
        =
        \argmin_{\bu}
        \Big\{
            \frac{\zeta_{kk}^u}2\|\bu\|^2
            +
            \sum_{\ell = 1}^{k-1} \zeta_{k\ell}^u \<\bu_\ell^{\SE},\bu\>
            -
            \<\bh_k^{\SE},\bu\>
            +
            \phi_{k,u}(\bu;\bH_{k-1}^{\SE};\beps_1^{\SE},{\beps_1^{\SE}}',\beps_2^{\SE})
        \Big\},
    \\
    \bv_k^{\SE}
        =
        \argmin_{\bv}
        \Big\{
            \frac{\zeta_{kk}^v}2\|\bv\|^2
            +
            \sum_{\ell = 1}^{k-1} \zeta_{k\ell}^v \<\bv_\ell^{\SE},\bv\>
            -
            \<\bg_k^{\SE},\bv\>
            +
            \phi_{k,v}(\bv;\bG_{k-1}^{\SE})
        \Big\},
\end{gathered}
\end{align}
where $\phi_{k,u}$ and $\phi_{k,v}$ are given by
\begin{subequations}
\label{eq:SE-penalties}
\begin{equation}
\begin{gathered}
\begin{aligned}
    \phi_{1,u}(\bu)
        &= 
        \indic\{\bu = \bzero\},
    \qquad&
    \phi_{2,u}(\bu;\bH_1)
        &= 
        \indic\{\bu = \bzero\},
    \\
    \phi_{3,u}(\bu;\bH_2)
        &= 
        \indic\{\bu = -\ones/n\},
    \qquad&
    \phi_{4,u}(\bu;\bH_3)
        &= 
        \indic\{\bu = -\ba/n\},
\end{aligned}
\\
\begin{aligned}
    \phi_{5,u}(\bu;\bH_4;\beps_1,\beps_2)
        &= 
        -
        \<\bu,\ones\>(\nu_{5,0}+\nu_{5,\sx})
        +
        \ell_5^*(\bu;\bw,\by_1,\by_2),
    \\
    \phi_{6,u}(\bu;\bH_5;\beps_1,\beps_2)
        &= 
        - 
        \< \bu,\ones\>(\nu_{6,0} + \nu_{6  ,\sx})
        +
        \ell_6^*(\bu;\bw,\by_1,\by_2),
\end{aligned}
\end{gathered}
\end{equation}
and
\begin{align}
\begin{gathered}
    \phi_{1,v}(\bv) 
        = 
        \indic\{ \bv = \btheta_1 \},
    \quad
    \phi_{2,v}(\bv;\bG_1) 
        =
        \indic\{ \bv = \btheta_2 \},
    \quad
    \phi_{3,v}(\bv) 
        =
        \indic\{ \bv = \bzero \},
    \quad
    \phi_{4,v}(\bv;\bG_1) 
        =
        \indic\{ \bv = \bzero \},
    \\
    \phi_{5,v}(\bv;\bG_2)
        =
        \Omega_5(\bv),
    \qquad
    \phi_{6,v}(\bv;\bG_3)
        =
        \Omega_6(\bv),
\end{gathered}
\end{align}
\end{subequations}
where $\by_1,\by_2,\bw$ on the right-hand sides of the previous displays are interpreted as functions of $\bH_k,\beps_1,{\beps_1}',\beps_2$ given by
\begin{equation}
\label{eq:yw-func-of-Heps}
\begin{gathered}
    y_{1,i} = 
        \indic\{ \eps_{1,i}' = 1\} 
        +
        \indic\{ \eps_{1,i}' = \circ\}
        \indic\{ \eps_{1,i} \leq \theta_{1,0} + \< \bmu_{\sx},\btheta_1\> + h_{1,i}\},
    \qquad
    \by_2 = (\theta_{2,0} + \< \bmu_{\sx} , \btheta_2 \>) \ones + \bh_2 + \beps_2, 
    \\
    \bw \defn w\big((\theta_{1,0} + \< \bmu_{\sx},\btheta_1\>)\ones + \bh_1\big),
\end{gathered}
\end{equation}
and the parameters $\{\zeta_{k\ell}^u\}_{1 \leq \ell \leq k \leq 6}$,
$\{ \zeta_{k\ell}^v \}_{1 \leq \ell \leq k \leq 6}$,
$\{\nu_{k,0}\}_{5\leq k \leq 6}$, $\{\nu_{k,\sx}\}_{5\leq k \leq 6}$
are defined below.  The objective
$\phi_{k,u}(\bu;\bH_{k-1}^{\SE};\beps_1^{\SE},{\beps_1^{\SE}}',\beps_2^{\SE})$
depends on the auxiliary noise
$\beps_1^{\SE},{\beps_1^{\SE}}',\beps_2^{\SE}$ and the parameters
listed in the previous sentence, but for notational simplicity this
dependence will often be left implicit.  We denote by $\bZ_u$ and
$\bZ_v$ the lower triangular matrices with $[\bZ_u]_{k \ell} =
\zeta_{k\ell}^u$ and $[\bZ_v]_{k \ell} = \zeta_{k\ell}^v$.

The state evolution distribution is determined by the parameters $\bK_g$, $\bK_h$, $\bZ_u$, $\bZ_v$, $\{\nu_{k,0}\}_{5\leq k \leq 6}$, $\{\nu_{k,\sx}\}_{5\leq k \leq 6}$.
The state evolution characterizes the distribution of
$(v_{1,0}^{\PO},\bv_1^{\PO})$,
$(v_{2,0}^{\PO},\bv_2^{\PO})$,
$(v_{3,0}^{\PO},\bv_3^{\PO})$,
$(v_{4,0}^{\PO},\bv_4^{\PO})$,
$(v_{5,0}^{\PO},\bv_5^{\PO})$, 
$(v_{6,0}^{\PO},\bv_6^{\PO})$
and
$\bu_3^{\PO}$, $\bu_4^{\PO}$, $\bu_5^{\PO}$, $\bu_6^{\PO}$ when these parameters are the unique solution to the fixed point equations:
\begin{equation}
\label{eq:fixpt-general}
\tag{SE-fixpt}
\begin{gathered}
    \bK_g 
        = 
        \llangle \bU^{\SE} \rrangle_{\Ltwo},
        \qquad 
        \bK_h 
        = 
        \llangle \bV^{\SE} \rrangle_{\Ltwo},
    \\
    \bK_h \bZ_v^\top
        = 
        \llangle \bH^{\SE} , \bU^{\SE} \rrangle_{\Ltwo},
        \qquad 
        \bK_g \bZ_u^\top
        = 
        \llangle \bG^{\SE} , \bV^{\SE} \rrangle_{\Ltwo},
    \\
        \nu_{5,\sx} = \<\bmu_{\sx},\bv_5^{\SE}\>_{\Ltwo},
        \qquad 
        \nu_{6,\sx} = \<\bmu_{\sx},\bv_6^{\SE}\>_{\Ltwo},
    \\
        \<\ones,\bu_5^{\SE}\>_{\Ltwo} = \<\ones,\bu_6^{\SE}\>_{\Ltwo} = 0,
    \\
    \text{where $\bZ_v$ is lower-triangular and innovation-compatible with $\bK_h$} 
    \\
    \text{and $\bZ_u$ is lower-triangular and innovation-compatible with $\bK_g$}.
\end{gathered}
\end{equation}
The right-hand sides of these equations involve expectations taken with respect to the state evolution distribution,
so they are functions of these parameters $\bK_g$, $\bK_h$, $\bZ_u$, $\bZ_v$, $\{\nu_{k,0}\}_{5\leq k \leq 6}$, $\{\nu_{k,\sx}\}_{5\leq k \leq 6}$.
The next two lemmas shows that the fixed point equations have a unique solution
and establishes bounds on their solution.
\begin{lemma}[Existence and uniqueness of fixed point parameters]
\label{lem:fixed-pt-soln}
    Under Assumption A1,
    the fixed point equations \eqref{eq:fixpt-general} have a unique solution.
\end{lemma}

\begin{lemma}[Bounds on fixed point parameters]
\label{lem:fixed-pt-bound}
    Under Assumption A1,
    the solution to the fixed point equations \eqref{eq:fixpt-general} satisfies the following bounds.

    \begin{itemize}

        \item Covariance bounds: $|L_{g,k\ell}| \lesssim 1/\sqrt{n}$, $|L_{h,k\ell}| \lesssim 1$ for all $\ell \leq k$, $|K_{g,k\ell}| \lesssim 1 / n$, and $|K_{h,k\ell}| \lesssim 1$.

        \item Non-trivial innovations: 
        $L_{g,kk} \asymp 1/\sqrt{n}$ for $k = 3,4,5,6$, and and $L_{h,kk} \asymp 1$ for $k = 5,6$.

        \item Bounds on effective regularization: 
        $\zeta_{k\ell}^v \lesssim 1$ and $\zeta_{k\ell}^u \lesssim n$ for all $\ell \leq k$ and all $k$,
        and $\zeta_{kk}^v \asymp 1$, $\zeta_{kk}^u \asymp n$ for $k = 5,6$.
        Moreover, $-\zeta_{51}^v \gtrsim 1$.

        \item Bounds on Cholesky pseudo-inverse: $(\bL_g^\ddagger)_{k\ell} \lesssim \sqrt{n}$ and $(\bL_h^\ddagger)_{k\ell} \lesssim 1$ for all $\ell \leq k$.

        \item Bounds on offset: $|\nu_{5,0}| \leq C$ and $|\nu_{6,0}| \leq C$.

        \item Bounds on mean-effects: $|\nu_{k,\sx}| \leq C$.

    \end{itemize}
    Moreover,
    the solutions have the following form:
    \begin{equation}
    \label{eq:Z-form}
        \bZ_v
            =
            \begin{pmatrix}
                0 & 0 & 0 & 0 & 0 & 0 \\[5pt]
                0 & 0 & 0 & 0 & 0 & 0 \\[5pt]
                0 & 0 & 0 & 0 & 0 & 0 \\[5pt]
                -\barpi \alpha_1 \indic\{\btheta_1 \neq \bzero\} & 0 & 0 & 0 & 0 & 0 \\[5pt]
                \zeta_{51}^v & 0 & 0 & 0 & \zeta_{55}^v & 0 \\[5pt]
                \zeta_{61}^v & \zeta_{62}^v & 0 & 0 & 0 & \zeta_{66}^v
            \end{pmatrix},
        \qquad
        \bZ_u
            =
            \begin{pmatrix}
                0 & 0 & 0 & 0 & 0 & 0 \\[5pt]
                0 & 0 & 0 & 0 & 0 & 0 \\[5pt]
                0 & 0 & 0 & 0 & 0 & 0 \\[5pt]
                0 & 0 & 0 & 0 & 0 & 0 \\[5pt]
                0 & 0 & 0 & 0 & \zeta_{55}^u & 0 \\[5pt]
                0 & 0 & 0 & 0 & 0 & \zeta_{66}^u
            \end{pmatrix},
    \end{equation}
    and
    \begin{equation}
    \label{eq:K-form}
        \bK_g
            =
            \begin{pmatrix}
                0 & 0 & 0 & 0 & 0 & 0 \\[5pt]
                0 & 0 & 0 & 0 & 0 & 0 \\[5pt]
                0 & 0 & 1/n & \barpi/n & K_{g,35} & K_{g,36} \\[5pt]
                0 & 0 & \barpi/n & \barpi/n & K_{g,45} & K_{g,46} \\[5pt]
                0 & 0 & K_{g,53} & K_{g,54} & K_{g,55} & K_{g,56} \\[5pt]
                0 & 0 & K_{g,63} & K_{g,64} & K_{g,65} & K_{g,66}
            \end{pmatrix},
        \qquad
        \bK_h
            =
            \begin{pmatrix}
                \llangle \bTheta \rrangle & \bzero_{2 \times 2} & \bK_{h,1:2,5:6} \\[5pt]
                \bzero_{2 \times 2} & \bzero_{2 \times 2} & \bzero_{2 \times 2} \\[5pt]
                \bK_{h,5:6,1:2} & \bzero_{2 \times 2} & \bK_{h,5:6,5:6}
            \end{pmatrix}.
    \end{equation}
\end{lemma}   
\noindent We prove~\Cref{lem:fixed-pt-soln,lem:fixed-pt-bound}
in~\Cref{sec:fix-pt-exist-proof,sec:fix-pt-bound-proof} respectively.
Unless otherwise specified, we will henceforth assume that the
parameters $\bK_g$, $\bK_h$, $\bZ_u$, $\bZ_v$, $\{\nu_{k,0}\}_{5\leq k
  \leq 6}$, $\{\nu_{k,\sx}\}_{5\leq k \leq 6}$ are taken to be the
solution to equations~\eqref{eq:fixpt-general}, and that
$(\bu_1^{\SE},\bu_2^{\SE},\bu_3^{\SE},\bu_4^{\SE},\bu_5^{\SE},\bu_6^{\SE})$,
$(\bv_1^{\SE},\bv_2^{\SE},\bv_3^{\SE},\bv_4^{\SE},\bv_5^{\SE},\bv_6^{\SE})$
have the state evolution distribution corresponding to these
parameters.

\begin{remark}[Fenchel-Legendre dual of state evolution optimization]
    We will frequently work with the Fenchel-Legendre dual of equation~\eqref{eq:SE-opt},
    which we state here for later reference.
    In particular,
    for $k = 5,6$, the first line of equation~\eqref{eq:SE-opt} is equivalent to
    \begin{equation}
    \label{eq:fenchel-legendre}
    \begin{gathered}
        n\bu_5^{\SE}
        =
        \nabla \ell_a\Big((\nu_{5,0} + \nu_{5,\sx})\ones + \bh_5^\SE - \sum_{\ell=1}^5 \zeta_{5\ell}^u \bu_\ell^\SE ;\by_1^{\SE}\Big),
        \\
        n\bu_6^{\SE}
        =
        -\by_1^{\SE} 
        \odot
        w(\bh_1^\SE)
        \odot
        \Big(
            \bh_2^\SE + \beps_2^\SE
            -
            \Big(
                (\nu_{6,0} + \nu_{6,\sx})\ones + \bh_6^{\SE} - \sum_{\ell=1}^6 \zeta_{5\ell}^u \bu_\ell^{\SE}
            \Big)
        \Big).
    \end{gathered}
    \end{equation}
\end{remark}

\subsection{State evolution describes the primary optimization}

Define the \emph{debiasing noise} for each optimization in equation~\eqref{eq:min-max} via:
\begin{equation}
\label{eq:loo-noise-def}
    \bg_k^{\PO} = \sum_{\ell=1}^k \zeta_{k\ell}^v \bv_\ell^{\PO} - \bA^\top \bu_k^{\PO},
    \qquad 
    \bh_k^{\PO} = \sum_{\ell=1}^k \zeta_{k\ell}^u \bu_\ell^{\PO} + \bA \bv_k^{\PO}.
\end{equation}
We define
\begin{equation}
\label{eq:perp-from-orig}
\begin{gathered}
    \bU^{\PO,\perp} = \bU^{\PO}\bL_g^{\ddagger\top},
        \qquad
        \bV^{\PO,\perp} = \bV^{\PO}\bL_h^{\ddagger\top},
        \qquad
        \bG^{\PO,\perp} = \bG^{\PO}\bL_g^{\ddagger\top},
        \qquad
        \bH^{\PO,\perp} = \bH^{\PO}\bL_h^{\ddagger\top},
    \\
    \bU^{\SE,\perp} = \bU^{\SE}\bL_g^{\ddagger\top},
        \qquad
        \bV^{\SE,\perp} = \bV^{\SE}\bL_h^{\ddagger\top},
        \qquad
        \bG^{\SE,\perp} = \bG^{\SE}\bL_g^{\ddagger\top},
        \qquad
        \bH^{\SE,\perp} = \bH^{\SE}\bL_h^{\ddagger\top}.
\end{gathered}
\end{equation}
whence we also have
\begin{equation}
\label{eq:orig-from-perp}
\begin{gathered}
    \bU^{\PO} = \bU^{\PO,\perp}\bL_g^{\top},
        \qquad
        \bV^{\PO} = \bV^{\PO,\perp}\bL_h^{\top},
        \qquad
        \bG^{\PO} = \bG^{\PO,\perp}\bL_g^{\top},
        \qquad
        \bH^{\PO} = \bH^{\PO,\perp}\bL_h^{\top},
    \\
    \bU^{\SE} = \bU^{\SE,\perp}\bL_g^{\top},
        \qquad
        \bV^{\SE} = \bV^{\SE,\perp}\bL_h^{\top},
        \qquad
        \bG^{\SE} = \bG^{\SE,\perp}\bL_g^{\top},
        \qquad
        \bH^{\SE} = \bH^{\SE,\perp}\bL_h^{\top}.
\end{gathered}
\end{equation}
Order-$2$ pseudo-Lipschitz functions of primary-optimization quantities concentrate on the expectation of the analogous state evolution quantities:
\begin{theorem}
\label{thm:se}
    Under Assumption A1, we have the following.

    \begin{itemize}

        \item \textbf{Concentration of parameter estimates.}
        For any function $\phi: (\reals^p)^{6 + 6} \rightarrow \reals$ which is order-$2$ pseudo-Lipschitz,
        we have
        \begin{equation}
        \label{eq:se-conc}
        \begin{gathered}
            \phi\Big(
                    \bV^{\PO,\perp},
                    \frac{1}{\sqrt{n}}\bG^{\PO,\perp}
                \Big)
                \mydoteq 
                \E\Big[
                \phi\Big(
                    \bV^{\SE,\perp},
                    \frac{1}{\sqrt{n}}\bG^{\SE,\perp}
                \Big)
            \Big].
        \end{gathered}
        \end{equation}

        \item \textbf{Concentration of score estimates.}
        For any function $\phi: \reals^{6 + 6 +1} \rightarrow \reals$ which is order-2 pseudo-Lipschitz,
        \begin{equation}
        \label{eq:se-conc-u}
            \frac{1}{n}
            \sum_{i=1}^n
                \phi\Big(
                    (nu_{\ell,i}^{\PO})_{\ell=1}^6,
                    (h_{\ell,i}^{\PO})_{\ell=1}^5,
                    \eps_{2,i}
                \Big)
            \mydoteq
            \E\Big[
                \phi\Big(
                    (nu_{\ell,i}^{\SE})_{\ell=1}^6,
                    (h_{\ell,i}^{\SE})_{\ell=1}^5,
                    \eps_{2,i}^{\SE}
                \Big)
            \Big].
        \end{equation}

        \item \textbf{Concentration of offset.} For $\ell = 5,6$,
        \begin{equation}
            v_{\ell,0}^{\PO} \mydoteq \nu_{\ell,0}.
        \end{equation}

    \end{itemize}
\end{theorem}

\begin{remark}
We expect that the relation~\eqref{eq:se-conc-u} holds also if
$(h_{\ell,i}^{\PO})_{\ell=1}^5$ is replaced by
$(h_{\ell,i}^{\PO})_{\ell=1}^6$ and $(h_{\ell,i}^{\SE})_{\ell=1}^5$ is
replaced by $(h_{\ell,i}^{\SE})_{\ell=1}^6$.  Due to the lack of
differentiability of $\ell_6^*$ when $y_{1,i} = 0$ (see
equation~\eqref{eq:ellk*-def}), establishing this involves technical
arguments which we do not pursue because they are not necessary for
the debiasing results.
\end{remark}

\section{State evolution identities}
\label{sec:reg-explicit}

Using Gaussian integration by parts,
we get explicit expressions for the effective regularization parameters $\zeta_{k\ell}^u$ and $\zeta_{k\ell}^v$ for $k = 5,6$.
\begin{lemma}
\label{lem:eff-reg-explicit}

    \begin{enumerate}[(a)]

    \item 
        Consider a state evolution $(\bu_k^{\SE})_k$, $(\bv_k^{\SE})_k$, $(\bg_k^{\SE})_k$, $(\bv_k^{\SE})_k$, $\beps_1^\SE$, ${\beps_1^{\SE}}'$, $\beps_2^\SE$ 
        at parameters
        $(\bK_g,\allowbreak \bK_h,\allowbreak \bZ_u,\allowbreak \bZ_v,\allowbreak \{\nu_{k,0}\}_{5\leq k \leq 6},\allowbreak \{\nu_{k,\sx}\}_{5\leq k \leq 6})$ (not necessarily a solution to the fixed point equations \eqref{eq:fixpt-general}).
        That is, we assume $\bG^{\SE} \sim \normal(0,\bK_g \otimes \id_p)$, $\bH^{\SE} \sim \normal(0,\bK_h \otimes \id_n)$, and equation~\eqref{eq:SE-opt}.
        For $\action \in \{0,1\}$, 
        let
        $u_{5,i}^{\SE,a} \defn \argmin_{u \in \reals}\Big\{\frac{\zeta_{55}^u}2u^2 - (\nu_{5,0} + \nu_{5,\sx} + h_{5,i}^{\SE}) + \ell_\prop^*(nu;a)\Big\}$
        and
        $\heta_{5,i}^{\SE,a} \defn \nu_{5,0} + \nu_{5,\sx}+h_{5,i}^{\SE}-\zeta_{55}^u u_{5,i}^{\SE,a}$.
        Let $\hbZ_u,\hbZ_v \in \reals^{6\times6}$ with $\hzeta_{41}^v = -\barpi \alpha_1$,
        \begin{equation}
        \label{eq:zeta-hat-56-u}
        \begin{gathered}
            \hzeta_{55}^u
                \defn
                \E\Big[
                    \Tr\Big(
                    \Big(
                        \zeta_{55}^v \id_p
                        +
                        \nabla^2 \Omega_5(\bv_5^{\SE})
                    \Big)^{-1}
                    \Big)
                \Big],
                \quad
                \hzeta_{66}^u
                \defn
                \E\Big[
                    \Tr\Big(
                    \Big(
                        \zeta_{66}^v \id_p
                        +
                        \nabla^2 \Omega_6(\bv_6^{\SE})
                    \Big)^{-1}
                    \Big)
                \Big],
            \\
            \hzeta_{65}^u
                \defn
                \zeta_{65}^v
                \E\Big[
                \Tr\Big(
                    \Big(
                        \zeta_{66}^v \id_p
                        +
                        \nabla^2 \Omega_6(\bv_6^{\SE})
                    \Big)^{-1}
                    \Big(
                        \zeta_{55}^v \id_p
                        +
                        \nabla^2 \Omega_5(\bv_5^{\SE})
                    \Big)^{-1}
                \Big)
                \Big],
        \end{gathered}
        \end{equation}
        and
        \begin{equation}
        \label{eq:zeta-hat-56-v}
        \begin{gathered}
            \hzeta_{51}^v
                \defn
                \E[\pi'(h_{1,i}^{\SE})(u_{5,i}^{\SE,1} - u_{5,i}^{\SE,0})],
            \qquad
            \hzeta_{55}^v
                \defn
                \E\Big[
                    \frac{1-\pi(h_{1,i}^{\SE})}{\zeta_{55}^u + n/\ell_\prop''(\heta_{5,i}^{\SE,0};0)}
                    +
                    \frac{\pi(h_{1,i}^{\SE})}{\zeta_{55}^u + n/\ell_\prop''(\heta_{5,i}^{\SE,1};1)}
                \Big],
            \\
            \hzeta_{61}^v
                \defn
                n\E\Big[
                    \Big(
                        \frac{\de}{\de h_{1,i}^{\SE}}
                        \frac{\pi(h_{1,i}^{\SE})w(h_{1,i}^{\SE})}{\zeta_{66}^u w(h_{1,i}^{\SE}) + n}
                    \Big)
                    \big(\nu_{6,0} + \nu_{6,\sx} - \mu_\out + h_{6,i}^{\SE} - h_{2,i}^{\SE} - \eps_{2,i}^{\SE} - \zeta_{65}^u u_{5,i}^{\SE,1}\big)
                \Big],
            \\
            \hzeta_{66}^v
                \defn
                -\hzeta_{62}^v
                \defn
                n\E\Big[
                    \frac{\pi(h_{1,i}^{\SE})}{\zeta_{66}^u + n/w(h_{1,i}^{\SE})}            
                \Big],
            \quad
            \hzeta_{65}^v
                \defn
                -n\zeta_{65}^u
                \E\Big[
                    \frac{\pi(h_{1,i}^{\SE})}{(\zeta_{55}^u + n/\ell_\prop''(\heta_{5,i}^{\SE,1};1))(\zeta_{66}^u + n/w(h_{1,i}^{\SE}))}
                \Big],
        \end{gathered}
        \end{equation}
        and all other entries are set to 0.
        Then
        \begin{equation}
        \label{eq:IBP-general}
            \llangle \bG^\SE , \bV^\SE \rrangle = \bK_g \hbZ_u^\top,
            \qquad
            \text{and if $\zeta_{k\ell}^u = 0$ for $\ell \leq 4$, then}
            \quad
            \llangle \bH^\SE , \bU^\SE \rrangle = \bK_h \hbZ_v^\top.
        \end{equation}

        \item 
        If $(\bK_g,\allowbreak \bK_h,\allowbreak \bZ_u,\allowbreak \bZ_v,\allowbreak \{\nu_{k,0}\}_{5\leq k \leq 6},\allowbreak \{\nu_{k,\sx}\}_{5\leq k \leq 6})$ is a solution to the fixed point equations \eqref{eq:fixpt-general},
        we further have 
        $\bZ_u^\top = \id_{\bK_g}^\top \hbZ_u^\top$ and $\bZ_v^\top = \id_{\bK_h}^\top \hbZ_v^\top$,
        which we display for future reference:
        \begin{equation}
        \label{eq:zeta56-uv-explicit}
            \bZ_u^\top
                =
                \id_{\bK_g}^\top
                \begin{pmatrix} 
                    0 & 0 & 0 & 0 & 0 & 0
                    \\ 
                    0 & 0 & 0 & 0 & 0 & 0
                    \\ 
                    0 & 0 & 0 & 0 & 0 & 0
                    \\ 
                    0 & 0 & 0 & 0 & 0 & 0
                    \\ 
                    0 & 0 & 0 & 0 & \hzeta_{55}^u & \hzeta_{65}^u
                    \\ 
                    0 & 0 & 0 & 0 & 0 & \hzeta_{66}^u 
                \end{pmatrix},
            \qquad
            \bZ_v^\top
                =
                \id_{\bK_h}^\top
                \begin{pmatrix} 
                    0 & 0 & 0 & -\barpi \alpha_1 & \hzeta_{51}^v & \hzeta_{61}^v
                    \\ 
                    0 & 0 & 0 & 0 & 0 & \hzeta_{62}^v
                    \\ 
                    0 & 0 & 0 & 0 & 0 & 0
                    \\ 
                    0 & 0 & 0 & 0 & 0 & 0
                    \\ 
                    0 & 0 & 0 & 0 & \hzeta_{55}^v & \hzeta_{65}^v
                    \\ 
                    0 & 0 & 0 & 0 & 0 & \hzeta_{66}^v 
                \end{pmatrix}.
        \end{equation}
        For $\zeta_{51}^v,\zeta_{61}^v,\zeta_{62}^v$,
        we have the explicit expressions
        \begin{equation}
            (\zeta_{51}^v,\zeta_{61}^v,\zeta_{62}^v)
            =
            \begin{cases}
                (0,0,0) \quad&\text{if } \btheta_1 = \btheta_2 = \bzero,
                \\
                \big(\hzeta_{51}^v,\hzeta_{61}^v + \frac{\<\btheta_2,\btheta_1\>}{\|\btheta_1\|^2}\hzeta_{62},0\big) \quad&\text{if } \btheta_1 \neq \bzero,\, \btheta_2 \propto \btheta_1,
                \\
                (\indic_{\btheta_1\neq \bzero}\hzeta_{51}^v,\indic_{\btheta_1\neq \bzero}\hzeta_{61}^v,\hzeta_{62}^v) \quad&\text{if } \btheta_2 \not\propto \btheta_1.
            \end{cases}
        \end{equation}
        In all cases, 
        $\zeta_{51}^v \btheta_1 = \hzeta_{51}^v \btheta_1$
        and
        $
        \zeta_{61}^v \btheta_1 + \zeta_{62}^v \btheta_2
        =
        \hzeta_{61}^v \btheta_1 + \hzeta_{62}^v \btheta_2
        =
        \hzeta_{61}^v \btheta_1 - \hzeta_{66}^v \btheta_2
        $.

    \end{enumerate}   
    
\end{lemma}

\begin{remark}
    The proof of~\Cref{lem:eff-reg-explicit} uses only the fixed point
    equation~\eqref{eq:fixpt-general} and the definition of the state
    evolution.  It assumes neither~\Cref{lem:fixed-pt-soln}
    nor~\Cref{lem:fixed-pt-bound}.  This fact is germane
    because~\Cref{lem:eff-reg-explicit} is used in the proof
    of~\Cref{lem:fixed-pt-bound}.
\end{remark}

\begin{proof}[Proof of~\Cref{lem:eff-reg-explicit}(a)]
 Using the relations $\bv_3^{\SE} = \bv_4^{\SE} = \bzero$, the KKT
 conditions for the optimization problem~\eqref{eq:SE-opt} imply that
 $\bv_5^{\SE}$ is the unique solution to $\zeta_{55}^v \bv_5^{\SE} +
 \zeta_{51}^v \btheta_1 + \zeta_{52}^v \btheta_2 - \bg_5^{\SE} +
 \nabla \Omega_5(\bv_5^{\SE}), $ whence $\bv_5^{\SE}$ can be written
 as a function only of $\bg_5^{\SE}$, which we denote by
 $f_5^v(\bg_5^{\SE})$.  By Assumption A1, $\Omega_5$ is
 twice-differentiable and $c$-strongly convex.  Thus, taking the
 derivative of the preceding display, we see that $f_5^v$ is
 differentiable, and
 \begin{align*}
\frac{\de f_5^v}{\de \bg_5^{\SE}} (\bg_5^{\SE}) = \Big( \zeta_{55}^v
\id_p + \nabla^2 \Omega_5(\bv_5^{\SE}) \Big)^{-1}.
 \end{align*}
Likewise, the KKT conditions for the problem~\eqref{eq:SE-opt} imply
that $\bv_6^{\SE}$ is the unique solution to $ \zeta_{66}^v
\bv_5^{\SE} + \zeta_{61}^v \btheta_1 + \zeta_{62}^v \btheta_2 +
\zeta_{65}^v f_5^v(\bg_5^{\SE}) - \bg_6^{\SE} + \nabla
\Omega_6(\bv_6^{\SE}), $ whence $\bv_6^{\SE}$ can be written as a
function only of $\bg_5^{\SE},\bg_6^{\SE}$, which we denote by
$f_6^v(\bg_5^{\SE},\bg_6^{\SE})$.  As before, $f_6^v$ is
differentiable.  Its derivatives are given by
\begin{align*}
\frac{\de f_6^v}{\de \bg_6^{\SE}} (\bg_5^{\SE},\bg_6^{\SE}) & = \Big(
\zeta_{66}^v \id_p + \nabla^2 \Omega_6(\bv_6^{\SE}) \Big)^{-1}, and \\
\frac{\de f_6^v}{\de \bg_5^{\SE}} (\bg_5^{\SE},\bg_6^{\SE}) & =
\zeta_{65}^v \Big( \zeta_{66}^v \id_p + \nabla^2 \Omega_6(\bv_6^{\SE})
\Big)^{-1} \Big( \zeta_{55}^v \id_p + \nabla^2 \Omega_5(\bv_5^{\SE})
\Big)^{-1}.
\end{align*}
The fixed point equations~\eqref{eq:fixpt-general} (in the first
equality) and Gaussian integration by parts (in the second equality)
establish\footnote{This formula is valid even if $K_{g,55} = 0$ or
$K_{g,66} = 0$, in which case $\bg_5^{\SE} = \bzero$ or $\bg_6^{\SE} =
\bzero$} the first equation in the relation~\eqref{eq:IBP-general}.

From~\cref{eq:zeta56-uv-explicit}, we have $\zeta_{5\ell}^u = 0$ for
$\ell \leq 4$.  Moreover, using the definition of $\bu_5^{\SE}$
(see~\cref{eq:SE-opt}), we can write
\begin{align*}
u_{5,i}^{\SE} = (1 - y_{1,i}^{\SE})u_{5,i}^{\SE,0} + y_{1,i}^{\SE}
u_{5,i}^{\SE,1}.
\end{align*}
Note that $u_{5,i}^{\SE,a}$ is the unique solution to the subgradient
inclusion
\begin{align*}
\nu_{5,0} + \nu_{5,\sx} + h_{5,i}^{\SE} - \zeta_{55}^u u_{5,i}^{\SE,a}
\in \partial \ell_\prop^*(nu_{5,i}^{\SE};a).
\end{align*}
Thus, the quantity $u_{5,i}^{\SE,a}$ can be written only as a function
of $h_{5,i}^{\SE}$, which we denote by $f_5^{u,a}(\bh_5^{\SE})$.  By
the relation~\eqref{eq:fenchel-legendre} and the fact that the
function $\ell_\prop$ is twice-differentiable and strongly convex in
its first argument (Assumption A1), we have
\begin{equation}
  \frac{\de f_5^{u,a}(\bh_5^{\SE})_i}{\de h_{5,i}^{\SE}} =
  \frac{1}{\zeta_{55}^u + n/\ell_\prop''(\heta_{5,i}^{\SE,a};a)}.
\end{equation}
Then, for any $\ell$,
\begin{equation}
  \label{eq:hl-corr-u5}
  \begin{aligned}
    \< \bh_\ell^{\SE} , \bu_5^{\SE} \>_{\Ltwo} &= \< \bh_\ell^{\SE} ,
    (1-y_{1,i}^{\SE})\bu_5^{\SE,0} + y_{1,i}^{\SE}\bu_5^{\SE,1}
    \>_{\Ltwo} = \E[ \E[\< \bh_\ell^{\SE} ,
        (1-y_{1,i}^{\SE})\bu_5^{\SE,0} + y_{1,i}^{\SE}\bu_5^{\SE,1} \>
        \mid \bh_1^{\SE} , \bh_5^{\SE}] ] \\ &= \E[\< \bh_\ell^{\SE} ,
      (1-\pi(\bh_1^{\SE}))\bu_5^{\SE,0} +
      \pi(\bh_1^{\SE})\bu_5^{\SE,1} \>] = K_{h,\ell1} \hzeta_{51}^v +
    K_{h,\ell5} \hzeta_{55}^v,
  \end{aligned}
\end{equation}
where the last equality uses Gaussian integration by parts.

Now using $\zeta_{6\ell}^u = 0$ for $\ell \leq 4$ and that
$w(\,\cdot\,)$ is bounded above by $C$ (Assumption A1), the KKT
conditions for~\cref{eq:SE-opt} imply that
\begin{equation}
  \begin{aligned}
    u_{6,i}^{\SE} &= y_{1,i}^{\SE}\cdot \frac{\nu_{6,0} + \nu_{6,\sx}
      + h_{6,i}^{\SE} - y_{2,i}^{\SE} - \zeta_{65}^u
      u_{5,i}^{\SE}}{\zeta_{66}^u + n(w_i^{\SE})^{-1}} =
    y_{1,i}^{\SE}\cdot \frac{\nu_{6,0} + \nu_{6,\sx} - \mu_\out +
      h_{6,i}^{\SE} - h_{2,i}^{\SE} - \eps_{2,i}^{\SE} - \zeta_{65}^u
      u_{5,i}^{\SE,1}}{\zeta_{66}^u + n/w(h_{1,i}^{\SE})}.
  \end{aligned}
\end{equation}
Then, for any $\ell$
\begin{equation}
  \label{eq:hl-corr-u6}
  \begin{aligned}
    \< \bh_\ell^{\SE} , \bu_6^{\SE} \>_{\Ltwo} & = \E[ \E[\<
        \bh_\ell^{\SE} , \bu_6^{\SE} \> | \bh_1^{\SE} , \bh_2^{\SE} ,
        \bh_5^{\SE} , \bh_6^{\SE}] ] = n \E\Big[ h_{\ell,i}^{\SE}
      \pi(h_{1,i}^{\SE}) \frac{\nu_{6,0} + \nu_{6,\sx} - \mu_\out +
        h_{6,i}^{\SE} - h_{2,i}^{\SE} - \eps_{2,i}^{\SE} -
        \zeta_{65}^u u_{5,i}^{\SE,1}}{\zeta_{66}^u +
        n/w(h_{1,i}^{\SE})} \Big] \\ &= K_{h,\ell1} \hzeta_{61}^v +
    K_{h,\ell2} \hzeta_{62}^v + K_{h,\ell5} \hzeta_{65}^v +
    K_{h,\ell6} \hzeta_{66}^v,
  \end{aligned}
\end{equation}
where the last equality uses Gaussian integration by parts.  Using
these computations and Gaussian integration by parts for the remaining
terms leads to the second term in the relation~\eqref{eq:IBP-general}.
\end{proof}

\begin{proof}[Proof of~\Cref{lem:eff-reg-explicit}(b)]
 When $\bZ_u$ is innovation compatible with $\bK_g$ and $\bZ_v$ is
 innovation compatible with $\bK_h$, we get $\id_{\bK_g}^\top
 \bZ_u^\top = \bZ_u^\top$ and $\id_{\bK_h}^\top \bZ_v^\top =
 \bZ_v^\top$ (see~\Cref{lem:cholesky-inverse-identities}).  The first
 identity in equation~\eqref{eq:zeta56-uv-explicit} follows then from part
 (a).  This implies $\zeta_{k\ell}^u = 0$ for $\ell \leq 4$, whence
 the second equation in equation~\eqref{eq:zeta56-uv-explicit} also follows
 from part (a).  If $\btheta_1 = \btheta_2 = \bzero$, then both
 indices 1 and 2 are predictable, $[\id_{\bK_h}^\top]_{1:2,1:2} =
 \bzero$, and $\zeta_{51}^v = \zeta_{61}^v = \zeta_{62}^v = 0$.  If
 $\btheta_1 \neq \bzero$ and $\btheta_2 \propto \btheta_1$, then index
 1 is innovative, index 2 is predictable,
 $[\id_{\bK_h}^\top]_{1:2,1:2} = \begin{pmatrix} 1 & 0
   \\ L_{h,11}^{-1}L_{h,12} & 0 \end{pmatrix}$, and $\zeta_{51} =
 \hzeta_{51}^v$, $\zeta_{61}^v = \hzeta_{61}^v + L_{h,11}^{-1}L_{h,12}
 \hzeta_{62}^v$, and $\zeta_{62}^v = 0$.  Otherwise, if $\btheta_2
 \not\propto \btheta_1$, then index 2 is innovative,
 $[\id_{\bK_h}^\top]_{1:2,1:2} = \diag(\indic_{\btheta_1 \neq
   \bzero},1)$, and $\zeta_{51}^v = \indic_{\btheta_1 \neq \bzero}
 \hzeta_{51}^v$, $\zeta_{61}^v = \indic_{\btheta_1 \neq \bzero}
 \hzeta_{61}^v$, and $\zeta_{62}^v = \hzeta_{62}^v$.  The identities $
 \zeta_{61}^v \btheta_1 + \zeta_{62}^v \btheta_2 = \hzeta_{61}^v
 \btheta_1 + \hzeta_{62}^v \btheta_2 = \hzeta_{61}^v \btheta_1 -
 \hzeta_{66}^v \btheta_2 $ and $\zeta_{51}^v \btheta_1 = \hzeta_{51}^v
 \btheta_1$ are verified by checking each case individually.
\end{proof}

\section{Regression exact asymptotics follows from state evolution}
\label{sec:exact-asymptotics-proof}

In this section, we show that~\Cref{lem:regr-fixpt-exist-and-bounds}
and~\Cref{thm:exact-asymptotics}
imply~\Cref{lem:fixed-pt-soln,lem:fixed-pt-bound} along
with~\Cref{thm:se}.

\subsection{A construction of fixed-design models from state evolution}

Our first step is to show that any state evolution that solves
fixed point equations~\eqref{eq:fixpt-general} has embedded on the
same probability spaces a solution the fixed point
equations~\eqref{eq:regr-fixed-pt}.

\begin{lemma}
\label{lem:se-to-fixed-des}
Consider any solution $(\bK_g, \bK_h, \bZ_u, \bZ_v,
\{\nu_{k,0}\}_{5\leq k \leq 6}, \{\nu_{k,\sx}\}_{5\leq k \leq 6})$ to
the fixed point equations~\eqref{eq:fixpt-general}
satisfying~\Cref{lem:fixed-pt-bound}, and let $(\bu_k^{\SE})_k$,
$(\bv_k^{\SE})_k$, $(\bg_k^{\SE})_k$, $(\bv_k^{\SE})_k$,
$\beps_1^\SE$, ${\beps_1^{\SE}}'$, $\beps_2^\SE$ have distribution
given by the state evolution corresponding to these parameters.

Then there exists on the same probability spaces random variables
$(\bg_{\sx}^f, \bg_{\sx,\cfd}^f, \bg_{\prop}^f, \bg_\out^f)$,
$(\by_{\sx}^f, \by_{\sx,\cfd}^f, \by_{\prop}^f, \by_\out^f)$,
$(\hbtheta_\prop^f, \hbtheta_\out^f)$, $(\bmeta_\prop^f,
\bmeta_\out^f, \hbmeta_\out^{f,\loo}, \hbmeta_\out^{f,\loo},
\hbmeta_\prop^f, \hbmeta_\out^f)$ with distribution given by the
fixed-design models with parameters $(\bS,\allowbreak
\beta_{\prop\prop}, \allowbreak \beta_{\out\prop}, \allowbreak
\zeta_\prop^\theta, \allowbreak \zeta_\prop^\eta, \allowbreak
\zeta_\out^\theta, \allowbreak \zeta_\out^\theta, \allowbreak
\hmu_\prop^f, \allowbreak \hmu_\out^f)$ that solve the fixed point
equations~\eqref{eq:regr-fixed-pt}.
\end{lemma}

\begin{proof}[Proof of~\Cref{lem:se-to-fixed-des}]
We first provide the construction of the solutions to the fixed-point equations \eqref{eq:regr-fixed-pt} from the solutions to \eqref{eq:fixpt-general}.
With $\hzeta_{51}^v,\hzeta_{61}^v,\hzeta_{62}^v$ as
in~\Cref{lem:eff-reg-explicit} and $ \ba^f = \by_1^\SE $, $ \by^f =
\by_2^\SE $, $ \bw^f = w\big((\theta_{\prop,0} + \< \bmu_{\sx} ,
\btheta_{\prop} \>)\ones + \bh_1^{\SE}\big) $, we can set
    \begin{equation}
    \label{eq:from-uv-to-theta-eta-psi}
    \begin{gathered}
            \bS = n
                \begin{pmatrix}
                    1 & 0 & 0 & 0 
                    \\
                    0 & \barpi & 0 & 0
                    \\
                    0 & 0 & \zeta_{55}^v & 0
                    \\
                    0 & 0 & 0 & \zeta_{66}^v
                \end{pmatrix}^{-1}
                \bK_{g,3:6,3:6}
                \begin{pmatrix}
                    1 & 0 & 0 & 0 
                    \\
                    0 & \barpi & 0 & 0
                    \\
                    0 & 0 & \zeta_{55}^v & 0
                    \\
                    0 & 0 & 0 & \zeta_{66}^v
                \end{pmatrix}^{-1},
        \\
            \beta_{\prop\prop}
                =
                -\hzeta_{51}^v / \zeta_{55}^v,
            \qquad
            \beta_{\out\prop}
                =
                -\hzeta_{61}^v / \zeta_{66}^v,
        \\
            \zeta_\prop^\eta
                =
                \zeta_{55}^u/n,
            \qquad
            \zeta_\out^\eta
                =
                \zeta_{66}^u/n,
            \qquad
            \zeta_\prop^\theta
                =
                \zeta_{55}^v,
            \qquad
            \zeta_\out^\theta
                =
                \zeta_{66}^v,
        \\
            \hmu_\prop^f
                =
                \nu_{5,0} + \< \bmu_{\sx} , \bv_5^{\SE}\>_{\Ltwo},
            \qquad
            \hmu_\prop^f
                =
                \nu_{6,0} + \< \bmu_{\sx} , \bv_6^{\SE} \>_{\Ltwo},
    \end{gathered}
    \end{equation}
    and
    \begin{equation}
    \label{eq:fix-des-from-se}
    \begin{aligned}
    &
    \begin{aligned}
            \bg_{\sx}^f 
                &=
                \bg_3^{\SE},
            \quad&
            \bg_{\sx,\cfd}^f
                &=
                \bg_4^{\SE}/\barpi,
            \\
            \by_{\sx}^f
                &=
                \bmu_{\sx} + \bg_3^{\SE},
            \quad&
            \by_{\sx,\cfd}^f
                &=
                \bmu_{\sx,\cfd} + \bg_4^{\SE}/\barpi,
            \\
            \hbtheta_\prop^f
                &=
                \bv_5^{\SE},
            \quad&
            \hbtheta_\out^f
                &=
                \bv_6^{\SE},
    \end{aligned}
    \quad&&
    \begin{aligned}
            \bg_\prop^f
                &=
                \bg_5^{\SE}/\zeta_{33}^v,
            \quad&
            \bg_\out^f
                &=
                \bg_6^{\SE}/\zeta_{66}^v,
        \\
            \by_\prop^f
                &=
                \frac{-\zeta_{51}^v\btheta_\prop + \bg_5^{\SE}}{\zeta_{55}^v},
            \quad&
            \by_\out^f
                &=
                \frac{-\hzeta_{61}^v\btheta_\prop - \hzeta_{62}^v \btheta_\out + \bg_6^{\SE}}{\zeta_{66}^v},
        \\
            \bmeta_\prop^f
                &=
                \mu_\prop\ones + \bh_1^{\SE},
            \quad&
            \bmeta_\out^f
                &=
                \mu_\out\ones + \bh_2^{\SE},
    \end{aligned}
    \\
    &
        \hbmeta_\prop^{f,\loo}
            =
            (\nu_{5,0} + \nu_{5,\sx})\ones + \bh_5^{\SE},
        \qquad
        &&\hbmeta_\out^{f,\loo}
            =
            (\nu_{6,0} + \nu_{6,\sx})\ones + \bh_6^{\SE},
    \\
        &\hbmeta_\prop^f
            =
                (\nu_{5,0} + \nu_{5,\sx})\ones + \bh_5^{\SE}
                -
                \zeta_{55}^u\bu_5^{\SE},
        \qquad&&
        \hbmeta_\out^f
            =
                (\nu_{6,0} + \nu_{6,\sx})\ones + \bh_6^{\SE}
                -
                \zeta_{66}^u\bu_6^{\SE}.
    \end{aligned}
    \end{equation}
    We prove that this construction solves \eqref{eq:regr-fixed-pt} in two steps.  \\

    \noindent \textbf{Step 1: The random variables in
      equation~\eqref{eq:fix-des-from-se} are distributed from the
      fixed-design models with
      parameters~\eqref{eq:from-uv-to-theta-eta-psi}.} That is, we
    check equations~\eqref{eq:fixed-design-param-outcomes},
    \eqref{eq:param-fit-f}, \eqref{eq:fixed-design-outcomes},
    \eqref{eq:fixed-design-linear-predictor-dist}, and
    \eqref{eq:lin-predict-f}, and that
    $(g_{\sx,i}^f,g_{\sx,\cfd,i}^f,g_{\prop,i}^f,g_{\out,i}^f) \iid
    \normal(\bzero,\bS/n)$.

    equation~\eqref{eq:fixed-design-param-outcomes} is true by construction.
    The KKT conditions for the second line of equation~\eqref{eq:SE-opt} and for equation~\eqref{eq:param-fit-f} are
    \begin{equation}
    \label{eq:se-v-KKT}
    \begin{gathered}
        \zeta_{55}^v(\bg_5^{\SE}/\zeta_{55}^v - \zeta_{51}^v / \zeta_{55}^v\bv_1^{\SE} - \hbtheta_{\prop}^f) \in \partial \Omega_\prop(\bv_5^{\SE}),
        \\
        \zeta_{66}^v(\bg_6^{\SE}/\zeta_{66}^v - \zeta_{61}^v / \zeta_{66}^v\bv_1^{\SE} - \zeta_{62}^v / \zeta_{66}^v \bv_2^{\SE} - \hbtheta_{\out}^f) \in \partial \Omega_\out(\bv_6^{\SE}),
    \end{gathered}
    \qquad\text{and}\qquad
    \begin{gathered}
        \zeta_\prop^\theta(\by_\prop^f - \hbtheta_\prop^f) \in \partial \Omega_\prop(\hbtheta_\prop^f),
        \\
        \zeta_\out^\theta(\by_\out^f - \hbtheta_\out^f) \in \partial \Omega_\out(\hbtheta_\out^f),
    \end{gathered}
    \end{equation}
respectively.  By~\Cref{lem:se-to-fixed-des}, we have
\begin{align*}
\zeta_{61}^v \bv_1^\SE + \zeta_{62}^v \bv_2^\SE = \hzeta_{61}^v
\btheta_1 - \hzeta_{66}^v \btheta_2, \quad \mbox{and} \quad
\zeta_{51}^v \bv_1^\SE = \hzeta_{51}^v \btheta_1.
 \end{align*}
    
By~\Cref{lem:fixed-pt-bound}, index 6 is innovative with respect
to $\bK_h$, whence equation~\eqref{eq:zeta56-uv-explicit} implies
$\hzeta_{66}^v = \zeta_{66}^v$.  Thus, under assignments
\eqref{eq:from-uv-to-theta-eta-psi} and \eqref{eq:fix-des-from-se},
the two sets of KKT conditions in the previous display are equivalent,
whence equation~\eqref{eq:param-fit-f} is satisfied.
Equation~\eqref{eq:fixed-design-outcomes} follows from our
construction of $\ba^f,\by^f,\bw^f$ above.  Because $\bK_h = \llangle
\bV^{\SE} \rrangle_{\Ltwo}$ and $\bv_1^{\SE} = \btheta_\prop$,
$\bv_2^{\SE} = \btheta_\out$, $\bv_5^{\SE} = \hbtheta_\prop^f$,
$\bv_6^{\SE} = \hbtheta_\out^f$, we conclude that $\bmeta_\prop^f$,
$\bmeta_\out^f$, $\hbmeta_\prop^{f,\loo}$, $\hbmeta_\out^{f,\loo}$,
$\hmu_\prop^f$, $\hmu_\out^f$ as we have defined them satisfy
equation~\eqref{eq:fixed-design-linear-predictor-dist}.  Using that
$\zeta_{k\ell} = 0 $ for $\ell \leq 4$, under assignments
\eqref{eq:from-uv-to-theta-eta-psi} and \eqref{eq:fix-des-from-se},
the KKT conditions for equation~\eqref{eq:lin-predict-f} are
equivalent to the KKT conditions \eqref{eq:fenchel-legendre}.
equation~\eqref{eq:lin-predict-f} follows.  That
$(g_{\sx,i}^f,g_{\sx,\cfd,i}^f,g_{\prop,i}^f,g_{\out,i}^f) \iid
\normal(\bzero,\bS/n)$ holds by the definition of $\bg_{\sx}^f$,
$\bg_{\sx,\cfd}^f$, $\bg_\prop^f$, $\bg_\out^f$, $\bS$, and the fact
that $(g_{3,i}^{\SE},g_{4,i}^{\SE},g_{5,i}^{\SE},g_{6,i}^{\SE}) \iid
\normal(\bzero,\bK_{g,3:6,3:6})$.  \\

    \noindent \textbf{Step 2: The fixed point equations~\eqref{eq:regr-fixed-pt} are satisfied by parameters \eqref{eq:from-uv-to-theta-eta-psi}.} 

    By equations~\eqref{eq:fixed-design-score} and
    \eqref{eq:empirical-influence-function}, $\hbi_\sx^f = \ones =
    -n\bu_3^{\SE}$, $\hbi_{\sx,\cfd}^f = -n\bu_4^{\SE}/\barpi$,
    $\hbi_\prop^f = -\hbpsi_\prop^f/\zeta_\prop^\theta =
    -n\bu_5^{\SE}/\zeta_\prop^\theta$, and $\hbi_\out^f =
    -\hbpsi_\out^f/\zeta_\out^\theta =
    -n\bu_6^{\SE}/\zeta_\out^\theta$.  Thus, the first line of
    equation~\eqref{eq:regr-fixed-pt} is equivalent to the fourth line
    of equation~\eqref{eq:fixpt-general}.  Comparing the derivative
    identities \eqref{eq:score-deriv-explicit} with the expressions
    for $\hzeta_{k\ell}^v$ in~\Cref{lem:eff-reg-explicit}, the second
    line of equation~\eqref{eq:regr-fixed-pt} follows.  The third line
    of equation~\eqref{eq:regr-fixed-pt} is equivalent to $\bK_g =
    \llangle \bU^{\SE} \rrangle$ using the expressions for
    $\hbi_{\sx}^f$, $\hbi_{\sx,\cfd}^f$, $\hbi_\prop^f$, $\hbi_\out^f$
    above.  Because $5,6$ are innovative with respect to both
    $\bK_g,\bK_h$ (\Cref{lem:fixed-pt-bound});
    \Cref{lem:eff-reg-explicit} implies $\zeta_{kk}^u = \hzeta_{kk}^u$
    and $\zeta_{kk}^v = \hzeta_{kk}^v$ for $k=5,6$.  Then, comparing
    the derivative identities~\eqref{eq:score-deriv-explicit} with the
    expressions for $\hzeta_{k\ell}^v,\hzeta_{k\ell}^u$
    in~\Cref{lem:eff-reg-explicit} implies the fourth and fifth lines
    of~\cref{eq:regr-fixed-pt}.  Indices 3 and 4 are innovative with
    respect to $\bS$ and the covariance
    in~\cref{eq:fixed-design-linear-predictor-dist} because indices 5
    and 6 are innovative with respect to $\bK_g$, $\bK_h$
    by~\Cref{lem:fixed-pt-bound}.
\end{proof}


\subsection{A construction of state evolution from fixed-design models}

Similarly, any fixed-design model that solves the fixed point
equations~\eqref{eq:fixpt-general} has embedded on the same
probability spaces a solution the fixed point
equations~\eqref{eq:regr-fixed-pt}.

\begin{lemma}
\label{lem:fixed-des-to-se}
    Consider any solution
    $(\bS,\allowbreak\beta_{\prop\prop},\allowbreak\beta_{\out\prop},\allowbreak\zeta_\prop^\theta,\allowbreak\zeta_\prop^\eta,\allowbreak\zeta_\out^\theta,\allowbreak\zeta_\out^\theta,\allowbreak\hmu_\prop^f,\allowbreak\hmu_\out^f)$
    to the fixed point equations \eqref{eq:regr-fixed-pt}, and let
    $(\bg_{\sx}^f,\bg_{\sx,\cfd}^f,\bg_{\prop}^f,\bg_\out^f)$,
    $(\by_{\sx}^f,\by_{\sx,\cfd}^f,\by_{\prop}^f,\by_\out^f)$,
    $(\hbtheta_\prop^f,\hbtheta_\out^f)$,
    $(\bmeta_\prop^f,\bmeta_\out^f,\hbmeta_\out^{f,\loo},\hbmeta_\out^{f,\loo},\hbmeta_\prop^f,\hbmeta_\out^f)$
    have distribution given by the fixed-design models corresponding
    to these parameters.

    Then there exists on the same probability spaces random variables
    $(\bu_k^{\SE})_k$, $(\bv_k^{\SE})_k$, $(\bg_k^{\SE})_k$,
    $(\bv_k^{\SE})_k$, $\beps_1^\SE$, ${\beps_1^{\SE}}'$,
    $\beps_2^\SE$ with distribution given by the state evolution with
    parameters $(\bK_g, \bK_h, \bZ_u, \bZ_v, \{\nu_{k,0}\}_{5\leq k
      \leq 6}, \{\nu_{k,\sx}\}_{5\leq k \leq 6})$ that solve the
    fixed point equations \eqref{eq:regr-fixed-pt}.
    
    When the mapping of~\Cref{lem:se-to-fixed-des} is applied to these
    random variables and parameters, we obtain
    $(\bS,\allowbreak\beta_{\prop\prop},\allowbreak\beta_{\out\prop},\allowbreak\zeta_\prop^\theta,\allowbreak\zeta_\prop^\eta,\allowbreak\zeta_\out^\theta,\allowbreak\zeta_\out^\theta,\allowbreak\hmu_\prop^f,\allowbreak\hmu_\out^f)$
    and $(\bg_{\sx}^f,\bg_{\sx,\cfd}^f,\bg_{\prop}^f,\bg_\out^f)$,
    $(\by_{\sx}^f,\by_{\sx,\cfd}^f,\by_{\prop}^f,\by_\out^f)$,
    $(\hbtheta_\prop^f,\hbtheta_\out^f)$,
    $(\bmeta_\prop^f,\bmeta_\out^f,\hbmeta_\out^{f,\loo},\hbmeta_\out^{f,\loo},\hbmeta_\prop^f,\hbmeta_\out^f)$.
\end{lemma}


\subsubsection{Proof of~\Cref{lem:fixed-des-to-se}}

The construction is as follows.  Let $\bD =
\diag(1,\barpi,\zeta_\prop^\theta,\zeta_\out^\theta)$, $\hzeta_{51}^v
= - \beta_{\prop\prop} \zeta_\prop^\theta$, $\hzeta_{61}^v = -
\beta_{\out\prop} \zeta_\out^\theta$, and $\hzeta_{62}^v = -
\zeta_\out^\theta$.  Then set
    \begin{equation}
    \begin{gathered}
        \bK_g
            =
            \begin{pmatrix}
                \bzero_{2\times2}
                &
                \bzero_{2\times4}
            \\
                \bzero_{4\times2}
                &  
                \frac{1}{n}
                \bD
                \bS
                \bD
            \end{pmatrix},
        \qquad
        \bK_h
            =
            \begin{pmatrix}
                \llangle \bTheta \rrangle & \bzero_{2 \times 2} & \llangle \bTheta , \hbTheta^f \rrangle_{\Ltwo} \\[5pt]
                \bzero_{2 \times 2} & \bzero_{2 \times 2} & \bzero_{2 \times 2} \\[5pt]
                \llangle \hbTheta^f , \bTheta \rrangle_{\Ltwo} & \bzero_{2 \times 2} & \llangle \hbTheta^f \rrangle_{\Ltwo}
            \end{pmatrix},
        \\
        \nu_{5,0} = \hmu_\prop^f - \< \bmu_{\sx} , \hbtheta_\prop^f\>_{\Ltwo},
        \qquad
        \nu_{6,0} = \hmu_\out^f - \< \bmu_{\sx} , \hbtheta_\out^f\>_{\Ltwo},
        \qquad
        \nu_{5,\sx} = \< \bmu_{\sx} , \hbtheta_\prop^f\>_{\Ltwo},
        \qquad
        \nu_{6,\sx} = \< \bmu_{\sx} , \hbtheta_\out^f\>_{\Ltwo}.
    \end{gathered}
    \end{equation}
    Then, with these definitions, set $\bZ_v^\top=\id_{\bK_h}^\top\hbZ_v^\top$ and 
    $\bZ_u=\diag(0,0,0,0,n\zeta_\prop^\eta,n\zeta_\out^\eta)$, where
    \begin{equation}
    \begin{gathered}
        \hbZ_v
        =
        \begin{pmatrix}
            0 & 0 & 0 & 0 & 0 & 0 \\[5pt]
            0 & 0 & 0 & 0 & 0 & 0 \\[5pt]
            0 & 0 & 0 & 0 & 0 & 0 \\[5pt]
            -\barpi \alpha_1 & 0 & 0 & 0 & 0 & 0 \\[5pt]
            -\beta_{\prop\prop}\zeta_{\prop}^\theta & 0 & 0 & 0 & \zeta_{\prop}^\theta & 0 \\[5pt]
            -\beta_{\out\prop}^\theta\zeta_\out^\theta & -\zeta_\out^\theta & 0 & 0 & 0 & \zeta_{\out}^\theta
        \end{pmatrix}.
    \end{gathered}
    \end{equation}
We construct a state evolution with these parameters by setting
$\by_1^{\SE} = \ba^f$, $\by_2^\SE = \by^f$, $\bw^\SE =
w(\bmeta_\prop^f)$, and
\begin{equation}
  \begin{gathered}
    \bg_1^{\SE} = \bzero, \quad \bg_2^{\SE} = \bzero, \quad
   \bg_3^{\SE} = \bg_{\sx}^f, \quad \bg_4^{\SE} =
    \barpi\bg_{\sx,\cfd}^f, \quad \bg_5^{\SE} =
    \zeta_{\prop}^\theta\bg_\prop^f, \quad \bg_6^{\SE} =
    \zeta_{\out}^\theta \bg_\out^f, \\
\bh_1^{\SE} = \bmeta_\prop^f -
\mu_\prop \ones, \quad \bh_2^{\SE} = \bmeta_\out^f - \mu_\out
\ones, \quad \bh_3^{\SE} = \bzero, \quad \bh_4^{\SE} = \bzero,
\quad \bh_5^{\SE} = \hbmeta_\prop^{f,\loo} - \hmu_\prop^f \ones,
\quad \bh_6^{\SE} = \hbmeta_\out^{f,\loo} - \hmu_\out^f \ones,
    \\
\bu_1^{\SE} = \bzero, \quad \bu_2^{\SE} = \bzero, \quad \bu_3^{\SE} =
\frac{\ones}{n}, \quad \bu_4^{\SE} = \frac{\ba^f}{n}, \quad
\bu_5^{\SE} = \frac{\hbpsi_\prop^f}{n}, \quad \bu_6^{\SE} =
\frac{\hbpsi_\out^f}{n}, \\
\bv_1^{\SE} = \btheta_\prop, \quad \bv_2^{\SE} = \btheta_\out,
\quad \bv_3^{\SE} = \bzero, \quad \bv_4^{\SE} = \bzero, \quad
\bv_5^{\SE} = \hbtheta_\prop^f, \quad \bv_6^{\SE} =
\hbtheta_\out^f.
\end{gathered}
\end{equation}

    As for~\Cref{lem:se-to-fixed-des}, we prove the lemma in two
    steps.  \\

    \noindent \textbf{Step 1: The random variables just defined are
      distributed from the state evolution with parameters
      $(\bS,\allowbreak\beta_{\prop\prop},\allowbreak\beta_{\out\prop},\allowbreak\zeta_\prop^\theta,\allowbreak\zeta_\prop^\eta,\allowbreak\zeta_\out^\theta,\allowbreak\zeta_\out^\theta,\allowbreak\hmu_\prop^f,\allowbreak\hmu_\out^f)$.}
    That is, we check $\bG^{\SE} \sim \normal(0,\bK_g \otimes \id_p)$,
    $\bH^{\SE} \sim \normal(0,\bK_h \otimes \id_n)$,
    and~\cref{eq:SE-opt}.

    The fact that $\bG^{\SE} \sim \normal(0,\bK_g \otimes \id_p)$ and
    $\bH^{\SE} \sim \normal(0,\bK_h \otimes \id_n)$ follows from the
    fact that
    $(g_{\sx,i}^f,g_{\sx,\cfd,i}^f,g_{\prop,i}^f,g_{\out,i}^f) \iid
    \normal(0, \bS/n)$ and~\cref{eq:fixpt-general} as well as out
    definition of $\bK_g,\bK_h$ above.  That $\bu_1^{\SE}$,
    $\bu_2^{\SE}$, $\bu_3^{\SE}$, and $\bu_4^{\SE}$
    satisfy~\cref{eq:SE-opt} follows from the definition of
    $\phi_{k,u}$ for $k = 1,2,3,4$.  By the KKT conditions
    for~\cref{eq:lin-predict-f}, we have $\hbmeta_\prop^f =
    \hbmeta_\prop^{f,\loo} - \zeta_\prop^\eta \hbpsi_\prop^f =
    \hmu_\prop^f \ones + \bh_5^\SE - n \zeta_\prop^\eta \bu_5^{\SE}$
    and, likewise, $\hbmeta_\out^f = \hmu_\out^f \ones + \bh_6^{\SE} -
    n \zeta_\out^\eta \bu_6^{\SE}$.  Thus, under the change of
    variables above, the KKT conditions for~\cref{eq:lin-predict-f}
    are equivalent to the KKT conditions~\eqref{eq:fenchel-legendre}.
    The first line of equation~\eqref{eq:SE-opt} follows.  That
    $\bv_1^{\SE}$, $\bv_2^{\SE}$, $\bv_3^{\SE}$, and $\bv_4^{\SE}$
    satisfy equation~\eqref{eq:SE-opt} follows from the definition of
    $\phi_{k,v}$ for $k = 1,2,3,4$.  Because indices 3 and 4 are
    innovative with respect to the covariance in
    equation~\eqref{eq:fixed-design-linear-predictor-dist}, we have
    that indices 5 and 6 are innovative with respect to $\bK_h$, so
    $\zeta_{55}^v = \hzeta_{55}^v$ and $\zeta_{66}^v = \hzeta_{66}^v$.
    Further, by the same computation from the proof
    of~\Cref{lem:eff-reg-explicit}, we have that $\zeta_{51}^v
    \btheta_1 = \hzeta_{51}^v \btheta_1$ and $ \zeta_{61}^v \btheta_1
    + \zeta_{62}^v \btheta_2 = \hzeta_{61}^v \btheta_1 + \hzeta_{62}^v
    \btheta_2 = \hzeta_{61}^v \btheta_1 - \hzeta_{66}^v \btheta_2 $.
    Thus, we see the KKT conditions for
    equation~\eqref{eq:param-fit-f} are equivalent to the KKT
    conditions for the the second line of equation~\eqref{eq:SE-opt}.
    The second line of equation~\eqref{eq:SE-opt} follows.  \\

    \noindent \textbf{Step 2: The fixed point
      equations~\eqref{eq:fixpt-general} are satisfied by these
      parameters.}
    
    The first line of \eqref{eq:fixpt-general} follows from the
    covariance in equation~\eqref{eq:fixed-design-score} and the third
    line of \eqref{eq:regr-fixed-pt}.  Using the derivative identities
    \eqref{eq:score-deriv-explicit} and the final two lines of
    equation~\eqref{eq:regr-fixed-pt}, \Cref{lem:eff-reg-explicit}(a)
    implies that $\llangle \bG^\SE , \bV^\SE \rrangle = \bK_g
    \bZ_u^\top$ and $\llangle \bH^\SE , \bU^\SE \rrangle = \bK_h
    \hbZ_v^\top$ for $\hbZ_v$ and $\bZ_u$ as defined above.  Because
    $\bK_h \id_{\bK_h}^\top = \bK_h$
    (\Cref{lem:cholesky-inverse-identities}), we have $\llangle
    \bH^\SE , \bU^\SE \rrangle = \bK_h \bZ_v^\top$.  The identities
    $\nu_{5,\sx} = \< \bmu_{\sx} , \bv_5^{\SE} \>_{\Ltwo}$ and
    $\nu_{6,\sx} = \< \bmu_{\sx},\bv_6^{\SE}\>_{\Ltwo}$ hold by
    construction..  Because $\bu_5^{\SE}$ is a constant times
    $\hbpsi_\prop^f$ and $\bu_6^{\SE}$ is a constant times
    $\hbpsi_\out^f$, the equations $\< \ones , \bu_5^\SE \>_{\Ltwo} =
    \< \ones , \bu_6^{\SE} \>_{\Ltwo} = 0$ follow from the equations
    $\< \ones ,\hbi_\prop^f \>_{\Ltwo} = \< \ones , \hbi_\out^f
    \>_{\Ltwo} = 0$.  The innovation compatibility of $\bZ_v$ with
    $\bK_h$ holds by construction (because we multiply by
    $\id_{\bK_h}^\top$).  Indices 5 and 6 are innovative with respect
    to $\bK_g$ because indices 3 and 4 are innovative with respect to
    $\bS$, whence $\bZ_u$ is innovation compatible with $\bK_g$.
    Thus, we have verified the fixed point
    relations~\eqref{eq:fixpt-general}.

Using the relations
\begin{align*}
  \hbmeta_\prop^f = \hmu_\prop^f \ones + \bh_5^\SE - n
  \zeta_\prop^\eta \bu_5^{\SE}, \quad \mbox{and} \quad \hbmeta_\out^f
  = \hmu_\out^f \ones + \bh_6^{\SE} - n \zeta_\out^\eta \bu_6^{\SE}
\end{align*}
straightforward but tedious algebra shows that when the mapping
of~\Cref{lem:se-to-fixed-des} is applied to these random variables and
parameters, we obtain $(\bS, \allowbreak \beta_{\prop\prop},
\allowbreak \beta_{\out\prop}, \allowbreak \zeta_\prop^\theta,
\allowbreak \zeta_\prop^\eta, \allowbreak \zeta_\out^\theta,
\allowbreak\zeta_\out^\theta,\allowbreak
\hmu_\prop^f,\allowbreak\hmu_\out^f)$ and $(\bg_{\sx}^f,
\bg_{\sx,\cfd}^f, \bg_{\prop}^f, \bg_\out^f)$, $(\by_{\sx}^f,
\by_{\sx,\cfd}^f, \by_{\prop}^f, \by_\out^f)$, $(\hbtheta_\prop^f,
\hbtheta_\out^f)$, $(\bmeta_\prop^f, \bmeta_\out^f,
\hbmeta_\out^{f,\loo},\hbmeta_\out^{f,\loo},\hbmeta_\prop^f,\hbmeta_\out^f)$.

\subsection{Proof of~\Cref{lem:regr-fixpt-exist-and-bounds}}
\label{sec:regr-fixpt-exist-unique-bound}

\Cref{lem:fixed-pt-soln} guarantees existence of solutions to the
fixed point equations~\eqref{eq:fixpt-general}.  Moreover, these
satisfy the bounds of~\Cref{lem:fixed-pt-bound}, and
via~\Cref{lem:se-to-fixed-des}, we can use them to construct a
solution to the fixed point equations~\eqref{eq:regr-fixed-pt}.  Thus,
we have established existence.

For any solution so constructed, the asserted bounds on the standard
errors, effective regularization, bias, and estimates follow from the
bounds in~\Cref{lem:fixed-pt-bound} pushed through the construction
equations~\eqref{eq:from-uv-to-theta-eta-psi}
and~\eqref{eq:fix-des-from-se}.

\Cref{lem:fixed-des-to-se} ensures uniqueness, because the its mapping
inverts the mapping of~\Cref{lem:fixed-des-to-se}, distinct solutions
to the fixed point equations~\eqref{eq:regr-fixed-pt} must map to
distinct solutions to the fixed point
equations~\eqref{eq:fixpt-general}.  But~\Cref{lem:fixed-pt-soln}
ensures that this cannot happen.


\subsection{Proof of~\Cref{thm:exact-asymptotics}}

As justified in~\Cref{sec:independent-covariates}, we can assume
without loss of generality that $\bSigma = \id_p$.  \\

\noindent \textbf{Proof of part (a).}  By the change of variables in
equations~\eqref{eq:from-uv-to-theta-eta-psi}
and~\eqref{eq:po-to-regr-cov}, we have that this concentration is
equivalent to $v_{k,0}^{\PO} + \< \bmu_{\sx},\bv_k^{\PO}\> \mydoteq
\nu_{k,0} + \< \bmu_{\sx},\bv_k^{\SE} \>_{\Ltwo}$, which follows
from~\Cref{thm:se}.  \\

\noindent \textbf{Proof of part (b).}  We defer the proof of part (b)
to~\Cref{AppDOF}.  \\

\noindent \textbf{Proof of part (c).}  Recall that $\bX = \ones
\bmu_\sx^\top + \bA$, whence
\begin{align*}
\hbmu_\sx & = \frac{1}{n} \sum_{i=1}^n \bx_i = \bX^\top \ones / n =
\bmu_\sx + \bA^\top \ones / n = \bmu_\sx - \bA^\top \bu_3^{\PO} =
\bmu_\sx + \bg_3^{\PO}, \quad \mbox{and} \\
\hbmu_{\sx,\cfd} & = \frac{1}{n_1} \sum_{i=1}^n \action_i \bx_i =
\bX^\top \ba / \< \ones ,\ba \> = \bmu_\sx - \bA^\top \bu_4^{\PO} /
(-\< \ones , \bu_4^{\PO} \>) \mydoteq \bmu_\sx + \alpha_1 \btheta_1 +
\bg_4^{\PO}/\barpi,
\end{align*}
where we have used equations~\eqref{eq:uPO-identity}
and~\eqref{eq:loo-noise-def} and the fact that $\< \ones ,\bu_4^{\PO}
\> \mydoteq -\barpi$ by~\Cref{thm:se}.  By~\cref{eq:po-to-regr-cov},
we have $\hbtheta_\prop = \bv_5^{\PO}$ and $\hbtheta_\out =
\bv_6^{\PO}$.  By equations~\eqref{eq:from-uv-to-theta-eta-psi}
and~\eqref{eq:eff-loo-est} and the concentration of the effective
regularization terms (part (b)), we have
\begin{equation}
  \hbtheta_\prop^{\loo} \mydoteq \bv_5^{\PO} + \frac{\nabla
    \Omega_5(\bv_5^{\PO})}{\zeta_{55}^v}, \qquad \hbtheta_\out^{\loo}
  \mydoteq \bv_6^{\PO} + \frac{\nabla
    \Omega_6(\bv_6^{\PO})}{\zeta_{66}^v}.
    \end{equation}
Thus,~\Cref{thm:se} together with~\cref{eq:Z-form} and using the KKT
conditions for equations~\eqref{eq:min-max} and \eqref{eq:SE-opt},
\begin{equation}
\begin{aligned}
  \phi\Big( \hbtheta_{\prop}, \hbtheta_{\out}, &\hbmu_{\sx},
  \hbmu_{\sx,\cfd}, \hbtheta_{\prop}^{\loo}, \hbtheta_{\out}^{\loo}
  \Big) \mydoteq \phi\Big( \bv_5^{\PO}, \bv_6^{\PO}, \bmu_\sx +
  \bg_3^\PO, \bmu_{\sx,\cfd} + \bg_4^{\PO}/\barpi, \bv_5^{\PO} +
  \frac{\nabla \Omega_5(\bv_5^{\PO})}{\zeta_{55}^v}, \bv_6^{\PO} +
  \frac{\nabla \Omega_6(\bv_6^{\PO})}{\zeta_{66}^v} \Big) \\
& \mydoteq \E\Big[ \phi\Big( \bv_5^{\SE}, \bv_6^{\SE}, \bmu_\sx +
    \bg_3^\SE, \bmu_{\sx,\cfd} + \bg_4^{\SE}/\barpi, \bv_5^{\SE} +
    \frac{\nabla \Omega_5(\bv_5^{\SE})}{\zeta_{55}^v}, \bv_6^{\SE} +
    \frac{\nabla \Omega_6(\bv_6^{\SE})}{\zeta_{66}^v} \Big) \Big]
  \\
& = \E\Big[ \phi\Big( \bv_5^{\SE}, \bv_6^{\SE}, \bmu_\sx + \bg_3^\SE,
    \bmu_{\sx,\cfd} + \bg_4^{\SE}/\barpi, \bv_5^{\SE} -
    \frac{\zeta_{51}^v \bv_1^{\SE} + \zeta_{55}^v\bv_5^{\SE} -
      \bg_5^{\SE}}{\zeta_{55}^v}, \bv_6^{\SE} - \frac{\zeta_{61}^v
      \bv_1^{\SE} - \zeta_{66}^v\bv_2^{\SE} + \zeta_{66}^v \bv_6^{\SE}
      - \bg_6^{\SE}}{\zeta_{66}^v} \Big) \Big] \\
& = \E\Big[ \phi\Big( \hbtheta_\prop^f, \hbtheta_\out^f, \by_\sx^f,
    \by_{\sx,\cfd}^f, \by_\out^f, \by_\prop^f \Big) \Big],
\end{aligned}
\end{equation}
where in the last line we have used the change of variables in
equations~\eqref{eq:from-uv-to-theta-eta-psi}
and~\eqref{eq:fix-des-from-se}.  \\

\noindent \textbf{Proof of part (d).}  By~\cref{eq:loo-noise-def}, we
have
\begin{align*}
\bmeta_\prop & = \mu_\prop\ones + \bA \bv_1^{\PO} = \mu_\prop \ones +
\bh_1^{\PO} - \zeta_{11}^u \bu_1^{\PO} = \mu_\prop \ones +
\bh_1^{\PO}, \quad \mbox{and} \\
\bmeta_\out &  = \mu_\out\ones + \bA \bv_2^{\PO} =
\mu_\out \ones + \bh_2^{\PO} - \zeta_{21}^u \bu_1^{\PO} - \zeta_{22}^u
\bu_2^{\PO} = \mu_\out \ones + \bh_2^{\PO}
\end{align*}
As we noted in~\cref{eq:uPO-identity}, we have $\hbpsi_\prop = n
\bu_5^{\PO}$ and $\hbpsi_\out = n \bu_6^{\PO}$.  Further, by the
change of variables \eqref{eq:fix-des-from-se}, $\bmeta_\prop^f =
\mu_\prop\ones + \bh_1^{\SE}$, $\bmeta_\out^f = \mu_\out \ones +
\bh_2^{\SE}$, $\hbpsi_\prop^f = n \bu_5^{\SE}$, and $\hbpsi_\out^f = n
\bu_6^{\SE}$.  Thus, part (d) is a consequence of~\Cref{thm:se}.

\section{Proof of~\Cref{thm:se}}

We prove the state evolution in \Cref{thm:se} used an inductive
argument based on the induction hypothesis:
\begin{description}
    \item[Hypothesis (k)] \Cref{thm:se} holds for the first $k$
      optimization problems; that is, with $\bV^{\PO,\perp}$ replaced
      by $\bV_k^{\PO,\perp}$, $\bG^{\PO,\perp}$ replaced by
      $\bG_k^{\PO,\perp}$, $(nu_{\ell,i}^{\PO})_{\ell=1}^6$,
      $(nu_{\ell,i}^{\SE})_{\ell=1}^6$,
      $(h_{\ell,i}^{\PO})_{\ell=1}^5$, and
      $(h_{\ell,i}^{\SE})_{\ell=1}^5$ replaced by
      $(nu_{\ell,i}^{\PO})_{\ell=1}^k$,
      $(nu_{\ell,i}^{\SE})_{\ell=1}^k$,
      $(h_{\ell,i}^{\PO})_{\ell=1}^{k \wedge 5}$, and
      $(h_{\ell,i}^{\SE})_{\ell=1}^{k \wedge 5}$, and with $\ell \in
      \{5,6\} \cap [k]$ in the third bullet point.
\end{description}

\noindent We take as a base case is $k = 4$, because it does not
require Gaussian comparison techniques to prove.

\begin{lemma}[Base case]
\label{lem:base-case}
    Hypothesis (k) holds for $k = 4$.
\end{lemma}

\begin{proof}[Proof of~\Cref{lem:base-case}]
By equations~\eqref{eq:objectives} and~\eqref{eq:SE-penalties}, we
have the relations $\bu_1^{\PO} = \bu_2^{\PO} = \bzero $, $
\bu_3^{\PO} = \frac{\ones}{n} $, $ \bu_4^{\PO} = \frac{\ba}{n} =
\frac{\by_1}{n} $, $ \bv_1^{\PO} = \btheta_1,\; \bv_2^{\PO} =
\btheta_2 $, $ \bv_3^{\PO} = \bv_4^{\PO} = \bzero $, and $ \bu_1^{\SE}
= \bu_2^{\SE} = \bzero $, $ \bu_3^{\SE} = \frac{\ones}{n} $, $
\bu_4^{\SE} = \frac{\ba^{\SE}}{n} $, $ \bv_1^{\SE} = \btheta_1,\;
\bv_2^{\SE} = \btheta_2 $, $ \bv_3^{\SE} = \bv_4^{\SE} = \bzero $.
Using the definition of the leave-one-out noise and the constraints on
$\bZ_v,\bZ_u$ in~\cref{eq:Z-form},
  \begin{equation}
    \label{eq:gh-1to4}
    \begin{gathered}
      \bg_1^{\PO} = - \bA^\top \bu_1^{\PO} = \bzero, \;\; \bg_2^{\PO}
      = - \bA^\top \bu_1^{\PO} = \bzero, \;\; \bg_3^{\PO} = \bA^\top
      \ones/n, \;\; \bg_4^{\PO} = -\barpi \alpha_1 \btheta_1 +
      \bA^\top \by_1/n, \\ \bh_1^{\PO} = \bA \bv_1^{\PO} = \bA
      \btheta_1, \;\; \bh_2^{\PO} = \bA \bv_2^{\PO} = \bA \btheta_2,
      \;\; \bh_3^{\PO} = \bA \bv_3^{\PO} = \bzero, \;\; \bh_4^{\PO} =
      \bA \bv_4^{\PO} = \bzero.
    \end{gathered}
  \end{equation}
 We see that $(\bh_1^{\PO}, \bh_2^{\PO}, \bh_3^{\PO}, \bh_4^{\PO})
 \sim \normal(0,\bK_{h,4} \otimes \id_n)$, with $\bK_h$ as
 in~\cref{eq:K-form}.  Moreover, by their definitions,
 $(u_{\ell,i}^{\PO,\perp})_{\ell=1}^4$,
 $(h_{\ell,i}^{\PO,\perp})_{\ell=1}^4$, $\eps_{1,i}$, $\eps_{1,i}'$,
 $\eps_{2,i}$ has exactly the same distribution as
 $(u_{\ell,i}^{\SE,\perp})_{\ell=1}^4$,
 $(h_{\ell,i}^{\SE,\perp})_{\ell=1}^4$, $\eps_{1,i}^{\SE}$,
 ${\eps_{1,i}^{\SE}}'$, $\eps_{2,i}^{\SE}$.  They are independent
 across $i$ and $C$-sub-Gaussian by Assumption A1.  Thus, because
 $\phi$ is order-2 pseudo-Lipschitz, the left-hand side of
 equation~\eqref{eq:se-conc-u} has independent terms which are sub-Gamma
 with variance and scale parameters bounded by $C$ (see, for example,
 \cite[Proposition G.5]{miolane2021}).  The asserted concentration
 holds by Bernstein's inequality \cite[Section
   2.8]{boucheronLugosiMassart2013}.

 We write $ \bA = \bA \proj_{\btheta_1} + \bA \proj_{\btheta_1}^\perp
 = \frac{\bh_1^{\PO} \btheta_1^\top}{\|\btheta_1\|^2} + \bA
 \proj_{\btheta_1}^\perp $.  Note that $\ba$ depends on $\bA$ only via
 $\bh_1^{\PO}$, and $\bA \proj_{\btheta_1}^\perp$ is independent of
 $\bh_1^{\PO}$.  We can write
 \begin{equation}
   \begin{gathered}
     \bg_3^{\PO} = \frac{\btheta_1}{\|\btheta_1\|^2}
     \frac{\<\bh_1^{\PO},\ones\>}{n} + \proj_{\btheta_1}^\perp
     \bA^\top \ones/n, \qquad \bg_4^{\PO} = -\barpi\alpha_1\btheta_1 +
     \frac{\btheta_1}{\|\btheta_1\|^2} \frac{\<\bh_1^{\PO},\by_1\>}{n}
     + \frac{\proj_{\btheta_1}^\perp \bA^\top \by_1}{n}.
    \end{gathered}
    \end{equation}
 Because $\bh_1^{\PO} \sim \normal(0,\|\btheta_1\|^2\id_n)$, we have
 that $\< \bh_1^{\PO} , \ones \>/n \sim \normal(0,\|\btheta_1\|^2/n)$.
 Thus, with probability at least $Ce^{-cn\epsilon^2}$,
 $\Big\|\frac{\btheta_1}{\|\btheta_1\|^2}
 \frac{\<\bh_1^{\PO},\ones\>}{n}\Big\| \leq \epsilon$.  Moreover,
 $h_{1,i}^{\PO}y_{1,i}$ is $\|\btheta_1\|^2$-sub-Gaussian with, by
 equations~\eqref{eq:model}, \eqref{eq:alpha-12}, and Gaussian integration
 by parts, expectation $\E[h_{1,i}^{\PO}y_{1,i}] = \| \btheta_1\|^2
 \barpi \alpha_1$.  Thus, for $Z_2 \sim \normal(0,1)$, we have by
 sub-Gaussian concentration that with probability at least
 $Ce^{-cn\epsilon^2}$,
 $\Big\|\frac{\btheta_1}{\|\btheta_1\|^2}\frac{\<\bh_1^{\PO},\by_1\>}{n}
 - \barpi \alpha_1 \btheta_1 - \frac{Z_2 \btheta_1}{\sqrt{n}} \Big\|
 \leq \epsilon$.  Thus, for any order-2 pseudo-Lipschitz function
 $\phi$,
\begin{equation}
\phi \big( \bg_3^{\PO},\bg_4^{\PO} \big) \mydoteq \phi\Big( \frac{Z_1
  \btheta_1}{\sqrt{n}} + \frac{\proj_{\btheta_1}^\perp \bA^\top
  \ones}{n}, \frac{Z_2\btheta_1}{\sqrt{n}} +
\frac{\proj_{\btheta_1}^\perp \bA^\top \by_1}{n} \Big),
\end{equation}
where, conditionally on the realization of $\by_1$, $Z_1, Z_2$ are
drawn from $\normal\Big(0, \hbK \Big)$ independently of everything
else (except for $\by_1$, with dependence through the conditional
covariance as shown), where
\begin{equation}
  \hbK \defn \frac{1}{n} \begin{pmatrix} 1 & \< \ones ,\by_1 \>/n
    \\ \< \ones ,\by_1 \>/n & \< \ones ,\by_1 \>/n \end{pmatrix}.
\end{equation}
We see that, conditionally on $\by_1^{\PO}$, the arguments to the
function on the right-hand side are Gaussian vectors with covariance
$\hbK \otimes \id_p$.  By Gaussian concentration of Lipschitz
functions and the fact that with exponentially high probability the
arguments to $\phi$ on the right-hand sides of $\ell_2$ norm bounded
by $C$,
\begin{equation}
  \phi\Big( \frac{Z_1 \btheta_1}{\sqrt{n}} +
  \frac{\proj_{\btheta_1}^\perp \bA^\top \ones}{n},
  \frac{Z_2\btheta_1}{\sqrt{n}} + \frac{\proj_{\btheta_1}^\perp
    \bA^\top \by_1}{n} \Big) \mydoteq \E_{(\tilde\bg_3,\tilde\bg_4)\sim
    \normal(0,\hbK)} \big[ \phi(\tilde\bg_3,\tilde\bg_4) \big].
\end{equation}
Note the right-hand side is random because it depends on $\hbK$.  But,
because by sub-Gaussian concentration $\frac{1}{n}\sum_{i=1}^n y_{1,i}
\mydoteq \barpi$, we in fact have by~\cref{eq:K-form}, that
$\E_{(\tilde\bg_3, \tilde\bg_4) \sim \normal(0,\hbK)} \big[\phi(
  \tilde\bg_3, \tilde\bg_4)\big] \mydoteq \E\big[\phi(\bg_3^{\SE},
  \bg_4^{\SE})\big]$.  Because $\bg_3^{\PO}$, $\bg_4^{\PO}$ are the
only random parts of $(\bv_\ell)_{\ell=1}^4$, $(\bg_\ell)_{\ell=1}^4$,
we have established~\cref{eq:se-conc} in the case $k = 4$.  This
completes the proof of the base case.
\end{proof}

\noindent Our next lemma captures the induction step required to
complete the proof of~\Cref{thm:exact-asymptotics}.

\begin{lemma}[Induction step]
\label{lem:inductive-step}
For $k = 5, 6$, Hypothesis (k-1) implies Hypothesis (k).
\end{lemma}


\section{The induction step: Proof of~\Cref{lem:inductive-step}}

The proof of~\Cref{lem:inductive-step} is based upon the sequential
Gordon inequality, introduced in past work~\cite{celentano2021cad} due
to a subset of the current authors. This involves introducing a
certain auxiliary objective whose saddle points behave, under the
induction hypothesis, like the saddle points of the
problem~\eqref{eq:min-max}.  The sequential Gordon inequality
establishes the essential relationship between the auxiliary objective
and the min-max problem~\eqref{eq:min-max}.

We first introduce the auxiliary objective objective.  Then,
in~\Cref{sec:aux-vectors}, we introduce some random vectors whose
behavior is central to our proof.  In~\Cref{sec:conc-aux-vectors}, we
establish some properties of these random vectors.  These properties,
together with the Convex Gaussian Min-Max Theorem (Gordon's
inequality) implies the sequential Gordon inequality, which we state
and prove in~\Cref{sec:seq-gordon}.
In~\Cref{sec:properties-aux-obj,sec:local-stability}, we establish
several additional properties of the auxiliary objective.
In~\Cref{sec:inductive-step-proof}, we combine these properties with
the sequential Gordon inequality to prove the induction step
(cf.~\Cref{lem:inductive-step}).


\subsection{The auxiliary objective}

Let $\bxi_h \sim \normal(0,\id_n)$ and $\bxi_g \sim \normal(0,\id_p)$
be independent of each other, and all other randomness in the problem.
Introducing the shorthand $\bT_g \defn \bG_{k-1}^{\PO} \bK_g^\ddagger
\bU_{k-1}^{\PO\top}$ and $\bT_h \defn \bH_{k-1}^{\PO} \bK_h^\ddagger
\bV_{k-1}^{\PO\top}$, we define the functions
\begin{align}
\label{eq:gordon-gh-matrix}
\bg(\bu) & \defn \bT_g\bu + \big\| \proj_{\bU_{k-1}^{\PO}}^\perp \bu
\big\|\bxi_g, \quad \mbox{and} \\
\bh(\bv) & \defn \bT_h\bv + \| \proj_{\bV_{k-1}^{\PO}}^\perp \bv
\|\bxi_h.
\end{align}
The proof of~\Cref{lem:inductive-step} relies on a careful analysis of
an auxiliary saddle point problem defined by the objective function
\begin{equation}
\label{eq:def-AuxObj}
    \AuxObj_k(\bu;v_0,\bv) \defn - \< \bg(\bu) , \bv \> + \< \bh(\bv)
    , \bu \> + \phi_k(\bu;v_0,\bv).
\end{equation}
Our first step in analyzing the auxiliary objective involves finding
an approximate saddle point $(\bu_k^{\AO};\nu_{k,0},\bv_k^{\AO})$.  We
first construct the approximate saddle point and then study its
properties.
\begin{lemma}[Convexity and concavity of $\AuxObj_k$]
\label{lem:AuxObj-convex-concave}
    For $\< \bxi_g , \bv \> \geq 0$, the function $\bu \mapsto
    \AuxObj_k(\bu;v_0,\bv)$ is concave.  For $\< \bxi_h , \bu \> \geq
    0$, the function $(v_0,\bv) \mapsto \AuxObj_k(\bu;v_0,\bv)$ is
    convex.
\end{lemma}

\begin{proof}
From their definitions~\eqref{eq:objectives}, the function
$\phi_k(\bu;v_0,\bv)$ are convex in $\bu$ and concave in $(v_0,\bv)$.
By inspection, the function $(\bu,\bv) \mapsto \< \bh(\bv) , \bu \>$
is linear in $\bu$ and, for $\< \bxi_h , \bu \> \geq 0$, convex in
$\bv$.  Similarly, the function $(\bu,\bv) \mapsto -\< \bg(\bu) , \bv
\>$ is linear in $\bv$ and, for $\< \bxi_g , \bv \> \geq 0$, concave
in $\bu$.
\end{proof}

\begin{remark}
An equivalent, and sometimes easier to work with, representation of
the functions $\bg$ and $\bh$ is given by
\begin{equation}
   \label{eq:gordon-gh-explicit}
   \begin{gathered}
     \bg(\bu)  = \sum_{\ell=1}^{k-1} \bg_\ell^{\PO,\perp} \<
     \bu_\ell^{\PO,\perp} , \bu \> + \big\|
     \proj_{\bU_{k-1}^{\PO}}^\perp \bu \big\|\bxi_g, \\
     \bh(\bu) = \sum_{\ell=1}^{k-1} \bh_\ell^{\PO,\perp} \<
     \bv_\ell^{\PO,\perp} , \bv \> + \big\|
     \proj_{\bV_{k-1}^{\PO}}^\perp \bv \big\|\bxi_h.
   \end{gathered}
\end{equation}
\end{remark}


\subsection{The auxiliary vectors}
\label{sec:aux-vectors}

In this section, we construct an approximate saddle point of
$\AuxObj_k$.  We defer the study of most of its properties (and, in
particular, the sense in which it is an approximate saddle point) to
later sections.

Recalling the definitions~\eqref{eq:perp-from-orig} of
$\bg_\ell^{\PO,\perp}$ and $\bh_\ell^{\PO,\perp}$, we define
\begin{equation}
\label{eq:GHk-ao}
    \bG_k^{\AO,\perp}
        =
        \begin{pmatrix}
            \vert & & \vert & \vert \\[5pt]
            \bg_1^{\PO,\perp} & \cdots & \bg_{k-1}^{\PO,\perp} & \bxi_g\\[5pt]
            \vert & & \vert & \vert 
        \end{pmatrix},
    \qquad 
    \bH_k^{\AO,\perp}
        =
        \begin{pmatrix}
            \vert & & \vert & \vert \\[5pt]
            \bh_1^{\PO,\perp} & \cdots & \bh_{k-1}^{\PO,\perp} & \bxi_h\\[5pt]
            \vert & & \vert & \vert 
        \end{pmatrix},
\end{equation}
and let 
\begin{equation}
\label{eq:GHk-ao-perp}
    \bG_k^{\AO}
        =
        \bG_k^{\AO,\perp} \bL_{g,k}^\top,
    \qquad 
    \bH_k^{\AO}
        =
        \bH_k^{\AO,\perp} \bL_{h,k}^\top.
\end{equation}
Here ``\textsf{ao}'' stands for ``auxiliary objective.''  Using
equation~\eqref{eq:perp-from-orig} and the fact that $k = 5,6$ are
innovative with respect to both $\bK_g$ and $\bK_h$
by~\Cref{lem:fixed-pt-bound} (so that $L_{g,kk}^\ddagger > 0$ and
$L_{h,kk}^\ddagger > 0$), we may also write $\bG_k^{\AO,\perp}$ and
$\bH_k^{\AO,\perp}$ in terms of $\bG_k^{\AO}$ and $\bH_k^{\AO}$:
\begin{equation}
\label{eq:GHk-ao-perp-from-std}
\bG_k^{\AO,\perp} = \bG_k^{\AO} \bL_{g,k}^{\ddagger\top}, \qquad
\bH_k^{\AO,\perp} = \bH_k^{\AO} \bL_{h,k}^{\ddagger\top}.
\end{equation}
These definitions also imply that
\begin{equation}
  \begin{gathered}
    \text{for $\ell < k$,} \quad \bg_\ell^{\AO} = \bg_\ell^{\PO},
    \quad \bh_\ell^{\AO} = \bh_\ell^{\PO}, \\ \bg_k^{\AO} = \sum_{\ell
      = 1}^{k-1} L_{g,k\ell} \bg_\ell^{\PO,\perp} + L_{g,kk} \bxi_g,
    \quad \bh_k^{\AO} = \sum_{\ell = 1}^{k-1} L_{h,k\ell}
    \bh_\ell^{\PO,\perp} + L_{h,kk} \bxi_h.
\end{gathered}
\end{equation}
We further define $\bu_\ell^{\AO} = \bu_\ell^{\PO}$ and
$\bv_\ell^{\AO} = \bv_\ell^{\PO}$ for $\ell < k$, and define
$\bu_k^{\AO}$ and $\bv_k^{\AO}$ as in equation~\eqref{eq:SE-opt} with
$\bH_k^{\AO}$ in place of $\bH_k^{\SE}$ and $\bG_k^{\AO}$ in place of
$\bG_k^{\SE}$.  Explicitly,
\begin{equation}
\label{eq:AO-separated-opt}
\begin{aligned}
    \bu_k^{\AO} &= \argmin_{\bu} \Big\{ \frac{\zeta_{kk}^u}2\|\bu\|^2
    + \sum_{\ell=1}^{k-1} \zeta_{k\ell}^u \<\bu_\ell^{\AO},\bu\> -
    \<\bh_k^{\AO},\bu\> +
    \phi_{k,u}(\bu;\bH_{k-1}^{\AO};\beps_1,\beps_1',\beps_2) \Big\},
    \\ \bv_k^{\AO} &= \argmin_{\bv} \Big\{
    \frac{\zeta_{kk}^v}2\|\bv\|^2 + \sum_{\ell=1}^{k-1}
    \zeta_{k\ell}^v \<\bv_\ell^{\AO},\bv\> - \<\bg_k^{\AO},\bv\> +
    \phi_{k,v}(\bv;\bG_{k-1}^{\AO}) \Big\}.
\end{aligned}
\end{equation}
Finally, define
\begin{equation}
\label{eq:UVk-ao-perp}
    \bU_k^{\AO,\perp} = \bU_k^{\AO} \bL_{g,k}^{\ddagger\top}, \qquad
    \bV_k^{\AO,\perp} = \bV_k^{\AO} \bL_{h,k}^{\ddagger\top}.
\end{equation}
Under this definition, we also have
\begin{equation}
\label{eq:UVk-ao}
    \bU_k^{\AO} = \bU_k^{\AO,\perp} \bL_{g,k}^{\top}, \qquad
    \bV_k^{\AO} = \bV_k^{\AO,\perp} \bL_{h,k}^{\top}.
\end{equation}
Indeed, by~\Cref{lem:cholesky-inverse-identities}, $\bU_k^{\AO,\perp}
\bL_{g,k}^\top = \bU_k^{\AO}\bL_{g,k}^{\ddagger\top}\bL_{g,k}^\top =
\bU_k^{\AO} \id_{\bK_{g,k}}^\top$ and $\bV_k^{\AO,\perp}
\bL_{h,k}^\top = \bV_k^{\AO} \id_{\bK_{h,k}}^\top$.  By the definition
of $\id_{\bK_{g,k}}$, $\bU_k^{\AO} \id_{\bK_{g,k}}^\top$ is equal to
$\bU_k^{\AO}$ provided the rows of $\bU_k^{\AO}$ are in
$\range(\bK_{g,k})$.  By~\Cref{lem:fixed-pt-bound}, $\range(\bK_{g,k})
= \spn\{\be_\ell : 2 < \ell \leq k\}$, and by
equation~\eqref{eq:AO-separated-opt}, $\bu_\ell^{\AO} = \bu_\ell^{\PO}
= \bzero$ for $\ell = 1,2$, whence indeed the rows of $\bU_k^{\AO}$
are in $\range(\bK_{g,k})$.  Similarly, $\bV_k^{\AO}
\id_{\bK_{h,k}}^\top$ is equal to $\bV_k^{\AO}$ provided the rows of
$\bV_k^{\AO}$ are in $\range(\bK_{h,k})$.
By~\Cref{lem:fixed-pt-bound}, $\ell = 5,6$ are innovative with respect
to $\bK_{h,k}$, whence the standard bases $\be_5,\be_5 \in
\range(\bK_{h,k})$.  Also, $\bv_3^{\AO} = \bv_4^{\AO} = \bzero$,
whence we only need to check that the rows of $\bV_2^{\AO}$ are in
$\range(\bK_{h,2})$.  Recall $\bV_2^{\AO} = \bV_2^{\PO} = \bTheta$ and
$\bK_{h,2} = \llangle \bTheta \rrangle$ by~\Cref{lem:fixed-pt-bound},
whence the rows of $\bV_2^{\AO}$ are indeed in $\range(\bK_{h,2})$.
We have thus established equation~\eqref{eq:UVk-ao}.

Comparing equation~\eqref{eq:UVk-ao-perp} with
equation~\eqref{eq:perp-from-orig}, we also see that $\bU_{k-1}^{\AO,
  \perp} = \bU_{k-1}^{\PO, \perp}$ and $\bV_{k-1}^{\AO, \perp} =
\bV_{k-1}^{\PO, \perp}$, and sometime interchange these in our
calculations.


\subsection{Concentration of statistics of auxiliary vectors}
\label{sec:conc-aux-vectors}

We first show concentration of order-2 pseudo-Lipschitz functions of
the auxiliary vectors.

\begin{lemma}
\label{lem:hypothesis-k-ao}
For $k = 5,6$, Hypothesis (k-1) implies that for $\phi:
(\reals^p)^{2k} \rightarrow \reals$ (in the first line) or for $\phi:
\reals^{2k+1} \rightarrow \reals$ (in second line) which is
pseudo-Lipschitz of order 2,
\begin{equation}
\label{eq:ao-conc}
\begin{gathered}
  \phi\Big( \bV_k^{\AO}, \bG_k^{\AO} \Big) \mydoteq \E\Big[ \phi\Big(
    \bV_k^{\SE}, \bG_k^{\SE} \Big) \Big], \\ \frac{1}{n} \sum_{i=1}^n
  \phi\Big( (nu_{\ell,i}^{\AO})_{\ell=1}^K,
  (h_{\ell,i}^{\AO})_{\ell=1}^K, \eps_{2,i} \Big) \mydoteq \E\Big[
    \phi\Big( (nu_{\ell,i}^{\SE})_{\ell=1}^K,
    (h_{\ell,i}^{\SE})_{\ell=1}^K, \eps_{2,i}^{\SE} \Big) \Big].
    \end{gathered}
    \end{equation}
\end{lemma}

\begin{proof}[Proof of~\Cref{lem:hypothesis-k-ao}]
  We prove the two lines of equation~\eqref{eq:ao-conc} one at at
  time.  Throughout the proof, we freely invoke the bounds on the
  problem parameters or fixed point parameters (e.g., $L_{g,k\ell}$,
  $L_{h,k\ell}$, etc.) from~\Cref{lem:fixed-pt-bound} and Assumption
  A1 without citing the lemma or assumption each time.  \\

  \noindent \textbf{First line of equation~\eqref{eq:ao-conc}.}  Denote the
  functions which maps $\bG_{k-1}^{\PO,\perp}$, $\bxi_g$ to
  $\bG_k^{\AO}$ via equation~\eqref{eq:GHk-ao-perp} and the function
  which maps $\bV_{k-1}^{\PO,\perp}$, $\bG_{k-1}^{\PO,\perp}$,
  $\bxi_g$ to $\bV_k^{\AO}$ via equations~\eqref{eq:AO-separated-opt}
  and \eqref{eq:UVk-ao} by
  \begin{equation}
    \label{eq:sG-func}
    \bG_k^{\AO} = \sG_k \Big( \frac{1}{\sqrt{n}}
    \bG_{k-1}^{\AO,\perp}, \frac{1}{\sqrt{n}} \bxi_g \Big), \qquad
    \bV_k^{\AO} = \sV_k \Big( \bV_{k-1}^{\AO,\perp},
    \frac{1}{\sqrt{n}} \bG_{k-1}^{\AO,\perp}, \frac{1}{\sqrt{n}} \bxi_g
    \Big),
  \end{equation}
  respectively.  The function $\sG_k$ is $C$-Lipschitz in its
  arguments because $L_{g,\ell\ell'} \lesssim 1 / \sqrt{n}$ and
  $L_{h,\ell\ell'} \lesssim 1$.  Because $L_{h,\ell\ell'} \lesssim 1$,
  equation~\eqref{eq:UVk-ao} implies that for $\ell \leq k-1$,
  $\bv_{\ell}^{\AO}$ is $C$-Lipschitz in
  $\big(\bv_{\ell'}^{\AO,\perp}\big)_{\ell'=1}^\ell$.  Because
  $\phi_{k,v}$ does not depend on $\bG_{k-1}^{\AO}$ (see
  equation~\eqref{eq:SE-penalties}), because proximal operators are
  1-Lipschitz \cite[Proposition 12.27]{bauschke2011convex}, and
  because $\zeta_{k\ell}^v \lesssim 1$ for $\ell \leq k$ and
  $\zeta_{kk}^v \asymp 1$, the definition of $\bv_k^{\AO}$
  (equation~\eqref{eq:AO-separated-opt}) implies it is $C$-Lipschitz
  in $\bV_{k-1}^{\AO,\perp}$, $\bG_{k-1}^{\AO,\perp}/\sqrt{n}$, and
  $\bxi_g/\sqrt{n}$.  Thus, $\sV_k$ is $C$-Lipschitz in its arguments.

These Lipschitz properties imply, by Hypothesis (k-1), that
\begin{equation}
  \label{eq:v-ao-conc}
  \begin{aligned}
    \phi\Big( \bV_k^{\AO}, \bG_k^{\AO} \Big) &= \phi\Big( \sV_k \Big(
    \bV_{k-1}^{\AO,\perp}, \frac{1}{\sqrt{n}} \bG_{k-1}^{\AO,\perp},
    \frac{1}{\sqrt{n}} \bxi_g \Big), \sG_k \Big( \frac{1}{\sqrt{n}}
    \bG_{k-1}^{\AO,\perp}, \frac{1}{\sqrt{n}} \bxi_g \Big) \Big)
    \\ &\mydoteq \E\Big[ \phi\Big( \sV_k \Big( \bV_{k-1}^{\AO,\perp},
      \frac{1}{\sqrt{n}} \bG_{k-1}^{\AO,\perp}, \frac{1}{\sqrt{n}}
      \bxi_g \Big), \sG_k \Big( \frac{1}{\sqrt{n}}
      \bG_{k-1}^{\AO,\perp}, \frac{1}{\sqrt{n}} \bxi_g \Big) \Big)
      \Bigm| \bV_{k-1}^{\AO,\perp},\bG_{k-1}^{\AO,\perp} \Big]
    \\ &\mydoteq \E\Big[ \phi\Big( \sV_k \Big( \bV_{k-1}^{\SE,\perp},
      \frac{1}{\sqrt{n}} \bG_{k-1}^{\SE,\perp}, \frac{1}{\sqrt{n}}
      \bxi_g \Big), \sG_k \Big( \frac{1}{\sqrt{n}}
      \bG_{k-1}^{\SE,\perp}, \frac{1}{\sqrt{n}} \bxi_g \Big) \Big)
      \Big].
    \end{aligned}
    \end{equation}
In the first approximate equality we use that conditionally on
$\bV_{k-1}^{\AO,\perp}$ and $\bG_{k-1}^{\AO,\perp}$, the quantity
$\phi\Big(\bV_k^{\AO},\bG_k^{\AO}/\sqrt{n}\Big)$ is conditionally
sub-Gamma with variance parameter $C(\|\bV_{k-1}^{\AO}\|_{\sF}^2 + \|
\bG_{k-1}^{\AO}\|_{\sF}^2)/n$ and scale parameter $C/n$ (see, for
example, \cite[Proposition G.5]{miolane2021}), and that
$\|\bV_{k-1}^{\AO}\|_{\sF}^2 + \| \bG_{k-1}^{\AO}\|_{\sF}^2 \mydoteq
\E[\|\bV_{k-1}^{\SE}\|_{\sF}^2 + \| \bG_{k-1}^{\SE}\|_{\sF}^2]
\lessdot C$ by Hypothesis (k-1).  The second approximate equality
holds by Hypothesis (k-1).

Then note that
\begin{equation}
  \begin{gathered}
    \sG_k \Big( \frac{1}{\sqrt{n}} \bG_{k-1}^{\SE,\perp},
    \frac{1}{\sqrt{n}} \bxi_g \Big) \underset{\emph{(i)}}\eqnd \sG_k
    \Big( \frac{1}{\sqrt{n}} \bG_{k-1}^{\SE,\perp}, \frac{1}{\sqrt{n}}
    \bg_k^{\SE,\perp} \Big) \underset{\emph{(ii)}}= \bG_k^{\SE}.
    \end{gathered}
    \end{equation}
Recall $\bg_k^{\SE,\perp} = \bG_k^{\SE} (\bL_{g,k}^\ddagger)^\top$
where $\bL_{g,k}^\ddagger$ is the Cholesky pseudo-inverse of
$\bK_{g,k}$ and $\bG_k^{\SE} \sim \normal(0,\bK_{g,k}\otimes \id_p)$.
Thus, distributional equality \emph{(i)} in the first line holds
because $\bxi_g \eqnd \bg_k^{\SE,\perp} \sim \normal(0,\id_p)$
independent of everything else, where we have used that index $k$ is
innovative with respect to $\bK_{g,k}$.  Distributional equality
\emph{(i)} in the second line holds similarly.  Equality \emph{(ii)}
holds by comparing equations~\eqref{eq:orig-from-perp} and
\eqref{eq:GHk-ao-perp}.  Similarly,
\begin{equation}
  \begin{gathered}
    \sV_k \Big( \bV_{k-1}^{\SE,\perp}, \frac{1}{\sqrt{n}}
    \bG_{k-1}^{\SE,\perp}, \frac{1}{\sqrt{n}} \bxi_g \Big) \eqnd \sV_k
    \Big( \bV_{k-1}^{\SE,\perp}, \frac{1}{\sqrt{n}}
    \bG_{k-1}^{\SE,\perp}, \frac{1}{\sqrt{n}} \bg_k^{\SE,\perp} \Big) =
    \bV_k^{\SE}
\end{gathered}
\end{equation}
The distributional equality holds by the same reason as before.  The
equality holds by comparing equations~\eqref{eq:SE-opt}
and~\eqref{eq:orig-from-perp} with
equations~\eqref{eq:AO-separated-opt}, \eqref{eq:UVk-ao},
and~\eqref{eq:GHk-ao-perp}.  Thus, equation~\eqref{eq:v-ao-conc}
implies the first line of equation~\eqref{eq:ao-conc}.  \\

\noindent \textbf{Second line of equation~\eqref{eq:ao-conc}.}  To
establish the second line of equation~\eqref{eq:ao-conc}, we use
explicit expressions for for $\bu_k^{\AO}$ in the case that $k =
5$ and $k = 6$, treating the two cases separately.

First, we establish a fact to be used in both cases.  We denote the
functions which map $(h_{\ell,i}^{\AO})_{\ell=1}^{k-1}$, $\xi_{h,i}$
to $(h_{\ell,i}^{\AO})_{\ell=1}^k$ via equation~\eqref{eq:GHk-ao-perp}
and $(u_{\ell,i}^{\AO,\perp})_{\ell=1}^{k-1}$ to
$(u_{\ell,i}^{\AO})_{\ell=1}^{k-1}$ via
equation~\eqref{eq:UVk-ao-perp} by
\begin{equation}
  (h_{\ell,i}^{\AO})_{k=1}^k = \sh_k \big(
  (h_{\ell,i}^{\AO,\perp})_{\ell=1}^{k-1}, \xi_{h,i} \big), \qquad
  (u_{\ell,i}^{\AO})_{k=1}^{k-1} = \su_{k-1} \big(
  (nu_{\ell,i}^{\AO,\perp})_{\ell=1}^{k-1} \big),
\end{equation}
respectively.  The functions $\sh_k$ and $n\su_{k-1}$ are
$C$-Lipschitz in their arguments by equation~\eqref{eq:GHk-ao-perp}
because $L_{h, \ell \ell'} \lesssim 1$ and $L_{g, \ell \ell'} \lesssim
1/\sqrt{n}$.

Now we specialize to the case $k = 6$.  Rearranging
equation~\eqref{eq:fenchel-legendre} and using $-nu_{4,i}^{\AO} =
y_{1,i} \in \{0,1\}$ and $\zeta_{6\ell}^u = 0$ for $\ell \leq 5$,
\begin{equation}
  \label{eq:u6ao-explicit}
  u_{6i}^{\AO} = - \frac{\nu_{6,0} + \nu_{6,\sx} + h_{6,i}^{\AO} -
    h_{2,i}^{\AO} - \eps_{2,i}}{\zeta_{66}^u + n
    w^{-1}(h_{1,i}^{\AO})} nu_{4,i}^{\AO}.
\end{equation}
This function is not Lipschitz in $h_{1,i}^{\AO}$, $h_{2,i}^{\AO}$,
$h_{6,i}^{\AO}$, $\eps_{2,i}$, $nu_{4,i}^{\AO}$ because the derivative
with respect to $nu_{4,i}^{\AO}$ and with respect to $h_{1,i}^{\AO}$
diverges as $h_{2,i}^{\AO}$, $h_{6,i}^{\AO}$, $\eps_{2,i}$ diverge.
We thus must resort to a truncation argument.  In particular, for $M >
1$, define $\Trunc^M(x) = (x \wedge M) \vee (-M)$ and define
$\subar_6^M$ by
\begin{equation}
  \label{eq:ubar-fnc}
  \subar_6^M (h_{1,i},h_{2,i},h_{6,i},nu_{4,i},\eps_{2,i}) \defn (-n
  u_{4,i}\wedge 1)\,\frac{\Trunc^M(\nu_{6,0} + \nu_{6,\sx} - h_{2,i} -
    \eps_{2,i} + h_{6,i})}{\zeta_{66}^u + n/w(h_{1,i})}.
    \end{equation}
We have that $n\subar_6^M$ is $CM$-Lipschitz because $w(\,\cdot\,)$ is
bounded above by $C$ and below by $c$ and has derivative bounded by
$C$ and $|\nu_{6,0}|$, $|\nu_{6,\sx}|$, $\zeta_{66}^u/n$ are bounded
above by $C$.  Denote $\ubar_6^{\AO}(M) \defn \subar_6^M (h_{1,i}^\AO,
h_{2,i}^\AO, h_{6,i}^\AO, n u_{4,i}^\AO, \eps_{2,i})$.  Then
$\phi/M^2$ is $C$ order-$2$ pseudo-Lipschitz, so by the exact same
logic we used in equation~\eqref{eq:v-ao-conc}, we have
\begin{equation}
  \label{eq:u-ao-conc}
  \begin{aligned}
    &\frac{1}{nM^2} \sum_{i=1}^n
    \phi\big((nu_{\ell,i}^{\AO})_{\ell=1}^5,n\ubar_6^{\AO}(M),(h_{\ell,i}^{\AO})_{\ell
      = 1}^6\big) \\ &\qquad\qquad= \frac{1}{nM^2} \sum_{i=1}^n
    \phi\big( \su_5 \big( (nu_{\ell,i}^{\AO,\perp})_{\ell=1}^5 \big),
    \subar_6^M
    (h_{1,i}^{\AO},h_{2,i}^{\AO},h_{6,i}^{\AO},nu_{4,i}^{\AO},\eps_{2,i}),
    \sh_k \big( (h_{\ell,i}^{\AO,\perp})_{\ell=1}^5, \xi_{h,i} \big)
    \big) \\
    & \qquad\qquad\mydoteq \frac{1}{M^2} \E\Big[ \phi\big( \su_5 \big(
      (nu_{\ell,i}^{\AO,\perp})_{\ell=1}^5 \big), \subar_6^M
      (h_{1,i}^{\AO},h_{2,i}^{\AO},h_{6,i}^{\AO},nu_{4,i}^{\AO},\eps_{2,i}),
      \sh_k \big( (h_{\ell,i}^{\AO,\perp})_{\ell=1}^5, \xi_{h,i} \big)
      \big) \Bigm| (u_{\ell,i}^{\AO})_{\ell=1}^5,
      (h_{\ell,i}^{\AO})_{\ell=1}^5 \Big] \\
    & \qquad \qquad \mydoteq \frac{1}{M^2} \E\Big[ \phi\big( \su_5
      \big( (nu_{\ell,i}^{\SE,\perp})_{\ell=1}^5 \big), \subar_6^M
      (h_{1,i}^{\SE},h_{2,i}^{\SE},h_{6,i}^{\SE},nu_{4,i}^{\SE},\eps_{2,i}),
      \sh_k \big( (h_{\ell,i}^{\SE,\perp})_{\ell=1}^5, \xi_{h,i} \big)
      \big) \Big].
  \end{aligned}
\end{equation}
If we denote $\ubar_6^{\SE}(M)\defn\subar_6^M
(h_{1,i}^\SE,h_{2,i}^\SE,h_{6,i}^\SE,nu_{4,i}^\SE,\eps_{2,i})$, then
we can recognize the right-hand side of the preceding display as $
\frac{1}{M^2} \E \Big[ \phi\big( (nu_{\ell,i}^{\SE,\perp})_{\ell=1}^5,
  \ubar_{6,i}^{\SE}(M), (h_{\ell,i}^{\SE})_{\ell=1}^6 \big) \Big].  $
For $nu_{4,i} \in \{-1,0\}$, $\big|\subar_6^M (h_{1,i}, h_{2,i},
h_{6,i}, nu_{4,i},\eps_{2,i}) - \subar_6 (h_{1,i}, h_{2,i}, h_{6,i},
nu_{4,i},\eps_{2,i})\big| \leq C(1+|h_{2,i}| + |h_{6,i}| +
|\eps_{2,i}| - M)_+$, where $\subar_6$ is defined as $\subar_6^M$ with
the minimum with 1 and $M$ removed.  Thus, standard Gaussian tail
bounds give
\begin{equation}
  \Big| \frac{1}{M^2} \E\Big[ \phi\big(
    (nu_{\ell,i}^{\SE,\perp})_{\ell=1}^5, \ubar_{6,i}^{\SE}(M),
    (h_{\ell,i}^{\SE})_{\ell=1}^6 \big) \Big] - \frac{1}{M^2} \E\Big[
    \phi\big( (nu_{\ell,i}^{\SE,\perp})_{\ell=1}^5, u_{6,i}^{\SE},
    (h_{\ell,i}^{\SE})_{\ell=1}^6 \big) \Big] \Big| \leq Ce^{-cM^2},
\end{equation}
 and
 \begin{equation}
   \begin{aligned}
 &\Big| \frac{1}{nM^2} \sum_{i=1}^n \phi
     \big((nu_{\ell,i}^{\AO})_{\ell=1}^5,n\ubar_{6,i}^{\AO}(M),(h_{\ell,i}^{\AO})_{\ell
       = 1}^6\big) - \frac{1}{nM^2} \sum_{i=1}^n
     \phi\big((nu_{\ell,i}^{\AO})_{\ell=1}^5,nu_{6,i}^{\AO},(h_{\ell,i}^{\AO})_{\ell
       = 1}^6\big) \Big| \\ &\qquad\qquad \leq \frac{C}{nM^2}
     \sum_{i=1}^n \Big( 1 + \sum_{\ell=1}^5|nu_{\ell,i}^\AO| +
     \sum_{\ell=1}^6|h_{\ell,i}^\AO| + |\eps_{2,i}| \Big) \Big( 1 +
     |h_{2,i}^{\AO}| + |h_{6,i}^\AO| + |\eps_{2,i}| - M \Big)_+ \\
& \qquad\qquad \mydoteq \frac{C}{M^2} \E\Big[ \Big( 1 +
       \sum_{\ell=1}^5 |nu_{\ell,i}^\SE| + \sum_{\ell=1}^6
       |h_{\ell,i}^\SE| + |\eps_{2,i}| \Big) \Big( 1 + |h_{2,i}^{\SE}|
       + |h_{6,i}^\SE| + |\eps_{2,i}^{\SE}| - M \Big)_+ \Big] \leq C
     e^{-c n M^2}.
   \end{aligned}
 \end{equation}
 Combining the previous three displays,
 \begin{equation}
   \begin{aligned}
     &\Big| \frac{1}{n M^2} \sum_{i=1}^n
     \phi\big((nu_{\ell,i}^{\AO})_{\ell=1}^5,nu_k^{\AO},(h_{\ell,i}^{\AO})_{\ell
       = 1}^6\big) - \frac{1}{M^2} \E\Big[ \phi\big(
       (nu_{\ell,i}^{\SE,\perp})_{\ell=1}^5, u_{6,i}^{\SE},
       (h_{\ell,i}^{\SE})_{\ell=1}^6 \big) \Big] \Big| \lessdot
     Ce^{-cM^2}.
    \end{aligned}
    \end{equation}
    That is, for all $\epsilon > 0$, 
    the probability that the left-hand side exceeds $\epsilon + Ce^{-cM^2}$ is bounded above by $C'e^{-c'n\epsilon^r}$ for some $C,c,C',c',r>0$ depending only on $\cPmodel$.
    Taking $M = \sqrt{\log(1/\epsilon)}$,
    \begin{equation}
            \frac{1}{n}
            \sum_{i=1}^n 
            \phi\big((nu_{\ell,i}^{\AO})_{\ell=1}^5,nu_k^{\AO},(h_{\ell,i}^{\AO})_{\ell = 1}^6\big)
            \mydoteq
            \E\Big[
                \phi\big(
                    (nu_{\ell,i}^{\SE,\perp})_{\ell=1}^5,
                    u_{6,i}^{\SE},
                    (h_{\ell,i}^{\SE})_{\ell=1}^6
                \big)
            \Big].
    \end{equation}
    The proof of the second line of equation~\eqref{eq:ao-conc} in the case $k = 6$ is complete.

    Now we turn to the case $k = 5$.
    By the definition of $\bu_5^{\AO}$ (equation~\eqref{eq:AO-separated-opt}),
    \begin{equation}
        u_{5,i}^{\AO}
            =
            -
            \prox
            \Big[
                \frac{\ell_{\prop}^*(n\,\cdot\,;1)}{n\zeta_{55}^u}
            \Big]
            \Big(
                \frac{\nu_{5,0} + \nu_{5,\sx} + h_{5,i}^{\AO}}{\zeta_{55}^u}
            \Big)
            nu_{4,i}^{\AO}
            +
            \prox
            \Big[
                \frac{\ell_{\prop}^*(n\,\cdot\,;0)}{n\zeta_{55}^u}
            \Big]
            \Big(
                \frac{\nu_{5,0} + \nu_{5,\sx} + h_{5,i}^{\AO}}{\zeta_{55}^u}
            \Big)
            (1+nu_{4,i}^{\AO}).
    \end{equation}
    Because $\zeta_{55}^u \gtrsim n$ and proximal operators are 1-Lipschitz \cite[Proposition 12.27]{bauschke2011convex},
    $nu_{5,i}^{\AO}$ is $C$-Lipschitz in $h_{5,i}^{\AO}$.
    The derivative with respect to $nu_{4,i}^{\AO}$ may diverge linearly in $h_{4,i}^{\AO}$.
    Thus, as in the case $k = 6$, we must resort to a truncation argument.
    We do this by replacing $\nu_{5,0} + \nu_{5,\sx} + h_{5,i}^{\AO}$ with $\Trunc^M(\nu_{5,0} + \nu_{5,\sx} + h_{5,i}^{\AO})$ in the preceding display,
    and following, line for line, the argument used for $k = 6$.

    Thus, the proof of the lemma is complete.
\end{proof}
\noindent \Cref{lem:hypothesis-k-ao} has several useful corollaries.

\begin{corollary}[Concentration of auxiliary second moments]
\label{cor:conc-aux-2m}
    For $k = 5,6$,
    under Hypothesis (k-1), 
    \begin{equation}
    \label{eq:second-moment-conc}
    \begin{aligned}
        \llangle \bG_k^{\AO} \rrangle
            &\mydoteq
            p\bK_{g,k},
        \quad&\quad
        \frac{1}{n}\llangle \bH_k^{\AO} \rrangle
            &\mydoteq
            \bK_{h,k},
        \\
        n\llangle \bU_k^{\AO} \rrangle 
            &\mydoteq
            n\bK_{g,k},
        \quad&\quad
        \llangle \bV_k^{\AO} \rrangle
            &\mydoteq
            \bK_{h,k},
        \\
        \llangle \bG_k^{\AO} , \bV_k^{\AO} \rrangle
            &\mydoteq
            \bK_{g,k}\bZ_{u,k}^\top,
        \quad&\quad
        \llangle \bH_k^{\AO} , \bU_k^{\AO} \rrangle
            &\mydoteq 
            \bK_{h,k} \bZ_{v,k}^\top.
    \end{aligned}
    \end{equation}
\end{corollary}

\begin{proof}[Proof of~\Cref{cor:conc-aux-2m}]
    The entries of the second-moment functions $\llangle \cdot , \cdot
    \rrangle$ and $\llangle \cdot \rrangle$ are the empirical averages
    order-2 pseudo-Lipschitz functions applied to each coordinate $i$.
    Thus, this is a consequence of~\Cref{lem:hypothesis-k-ao} and the
    fixed point equations~\eqref{eq:fixpt-general}.
\end{proof}

\begin{corollary}[Concentration of Cholesky decomposition]
\label{cor:chol-conc}
    For $k =5,6$, under Hypothesis (k-1)
    \begin{equation}
    \begin{aligned}
        \sL\big(n\llangle \bU_k^{\AO} \rrangle\big)
            &\mydoteq
            \sqrt{n} \bL_{g,k},
        \qquad
        &\sL^{\ddagger}\big(n\llangle \bU_k^{\AO} \rrangle\big)
            &\mydoteq
            \bL_{g,k}^{\ddagger}/\sqrt{n},
        \\
        \sL\big(\llangle \bV_k^{\AO} \rrangle\big)
            &\mydoteq
            \bL_{h,k},
        &\qquad
        \sL^{\ddagger}\big(\llangle \bV_k^{\AO} \rrangle\big)
            &\mydoteq
            \bL_{h,k}^{\ddagger}.
    \end{aligned}
    \end{equation}
    We also have
    \begin{equation}
    \begin{aligned}
        \sK^\ddagger\big(n\llangle \bU_k^{\AO} \rrangle\big)
            &\mydoteq
            \bK_{g,k}^\ddagger/n,
        \qquad
        &\sK^\ddagger\big(\llangle \bV_k^{\AO} \rrangle\big)
            &\mydoteq
            \bK_{h,k}^\ddagger.
    \end{aligned}
    \end{equation}
\end{corollary}

\begin{proof}[Proof of \Cref{cor:chol-conc}]
    We have $n \llangle \bU_4^{\AO} \rrangle_2 \mydoteq n \bK_{g,4} = \bzero$ and $\llangle \bV_4^{\AO} \rrangle_2 = \bK_{h,2} = \llangle \bTheta \rrangle$ using equation~\eqref{eq:K-form} and the explicit expression $(\bu_1^\AO,\bu_2^\AO,\bu_3^\AO,\bu_4^\AO) = (\bzero,\bzero,-\ones/n,-\by_1/n)$ and $(\bv_1^\AO,\bv_2^\AO,\bv_3^\AO,\bv_4^\AO) = (\btheta_1,\btheta_2,\bzero,\bzero)$.
    Thus, we have equality of the upper-left $4 \times 4$ sub-matrices in the lemma.
    
We check the high-probability approximation for entries $(\ell,\ell')$
with $\ell' \leq \ell$ and $\ell \leq k$.  We induct on $\ell$.  We
use the explicit expressions for the Cholesky decomposition
in~\Cref{lem:cholesky} and the bounds on the fixed point parameters
in~\Cref{lem:fixed-pt-bound} without explicitly citing them each time.

Assume the result has been shown for the upper-left
$(\ell-1)\times(\ell-1)$ submatrices.  For $\ell' < \ell$, the formula
$\sL(\bK)_{\ell \ell'} = \sum_{\ell''=1}^{\ell'}
\sL^\ddagger(\bK)_{\ell' \ell''} K_{\ell\ell''}$ for both $\bK = n
\llangle \bU_k^{\AO} \rrangle$ and $\bK = \llangle \bV_k^{\AO}
\rrangle$ implies $\sL(n \llangle \bU_k^{\AO} \rrangle)_{\ell \ell'}
\mydoteq \sqrt{n} L_{g,\ell\ell'}$ and $\sL(\llangle \bV_k^{\AO}
\rrangle)_{\ell \ell'} \mydoteq L_{h,\ell\ell'}$ because, by the
inductive hypothesis and \Cref{cor:conc-aux-2m}, we have
$\sL(n \llangle \bU_k^{\AO}\rrangle)_{\ell'\ell''}^\ddagger \mydoteq
L_{g,\ell'\ell''}^\ddagger/\sqrt{n}$, $n \llangle
\bU_k^{\AO}\rrangle_{\ell\ell''}^\ddagger \mydoteq n K_{g,\ell\ell'}$,
$\sL^\ddagger( \llangle \bV_k^{\AO}\rrangle)_{\ell'\ell''} \mydoteq
L_{h,\ell'\ell''}^\ddagger$, and $\llangle
\bV_k^{\AO}\rrangle_{\ell\ell''}^\ddagger \mydoteq K_{h,\ell\ell'}$,
and because $L_{g,\ell'\ell''}^\ddagger/\sqrt{n}$, $n
K_{g,\ell\ell'}$, $L_{h,\ell'\ell''}^\ddagger$, and $K_{h,\ell\ell'}$
are all $\lesssim 1$.  Similarly, the formula $\sL(\bK)_{\ell \ell} =
\sqrt{K_{\ell\ell} - \sum_{j=1}^{\ell - 1} \sL(\bK)^2_{\ell \ell'}}$
for both $\bK = n \llangle \bU_k^{\AO} \rrangle$ and $\bK = \llangle
\bV_k^{\AO} \rrangle$ implies $\sL(n \llangle \bU_k^{\AO}
\rrangle)_{\ell \ell} \mydoteq \sqrt{n} L_{g,\ell\ell}$ and
$\sL(\llangle \bV_k^{\AO} \rrangle)_{\ell \ell} \mydoteq
L_{h,\ell\ell}$.
    
Because $\ell > 4$, $\ell$ is innovative with respect to both $n\bK_g$
and $ \bK_h $.  Moreover, $\sqrt{n} L_{g,\ell\ell} \gtrsim 1$ and
$L_{h,\ell \ell} \gtrsim 1$, whence, by the high-probability
approximation in the previous paragraph, with exponentially high
probability (depending only on $\cPmodel$), $\ell$ is innovative with
respect to both $n \llangle \bU_k^{\AO} \rrangle$ and $\llangle
\bV_k^{\AO} \rrangle$.  Using the identities
$\sL^{\ddagger}(\bK)_{\ell\ell'} = -\frac{1}{\sL(\bK)_{\ell\ell}}
\sum_{\ell''=\ell'}^{\ell-1} \sL(\bK)_{\ell
  \ell''}\sL^\ddagger(\bK)_{\ell''\ell'}$, $\ell' < \ell$ and
$\sL^{\ddagger}(\bK)_{\ell\ell} = -\frac{1}{\sL(\bK)_{\ell\ell}}$, which
hold on this high-probability event, the inductive hypothesis and
bounds on the fixed point parameters establish the desired
high-probability approximation for $\sL^\ddagger(n \llangle
\bU_k^{\AO} \rrangle)$ and $\sL^\ddagger(\llangle \bV_k^{\AO}
\rrangle)$.
    
The second display in the lemma is a consequence of the first.
\end{proof}

\begin{corollary}[Bounds on auxiliary objective vectors]
\label{cor:AO-bounds}
    For $k =5,6$,
    under Hypothesis (k-1),
    for $\ell,\ell' \leq k$,
    \begin{equation}
    \begin{aligned}
        \sqrt{n} \| \bu_\ell^{\AO} \| 
            &\lessdot
            C,
        \qquad
        &\| \bu_\ell^{\AO,\perp} \| 
            &\lessdot
            C,
        \qquad
        &
        \| \bg_\ell^{\AO} \| 
            &\lessdot
            C,
        \qquad
        &\| \bg_\ell^{\AO,\perp} \|/\sqrt{n} 
            &\lessdot
            C,
        \\
        \| \bv_\ell^{\AO} \|
            &\lessdot
            C,
        \qquad
        &\| \bv_\ell^{\AO,\perp} \|
            &\lessdot
            C,
        \qquad
        &\| \bh_\ell^{\AO} \| / \sqrt{n}
            &\lessdot
            C,
        \qquad
        &\| \bh_\ell^{\AO,\perp} \| / \sqrt{n}
            &\lessdot
            C,
        \\
        &&
        \< \bg_k^{\AO,\perp} , \bv_k^{\AO} \> / \sqrt{n}
            &\gtrdot 
            c,
        &
        \< \bh_k^{\AO,\perp} , \bu_k^{\AO} \>
            &\gtrdot 
            c.
    \end{aligned}
    \end{equation}
\end{corollary}

\begin{proof}[Proof of~\Cref{cor:AO-bounds}]
    The bounds on $\sqrt{n} \| \bu_\ell^{\AO} \|$, $\| \bv_\ell^{\AO}
    \|$, $\| \bg_\ell^{\AO} \| $, and $\| \bh_\ell^{\AO,\perp} \| /
    \sqrt{n}$ hold by the first-two lines of~\Cref{cor:conc-aux-2m}
    and the bounds on $\bK_g$ and $\bK_h$
    in~\Cref{lem:fixed-pt-bound}, and using that $p/n < C$.  We then
    get the bounds $\| \bu_\ell^{\AO,\perp} \|$, $\|
    \bg_\ell^{\AO,\perp} \|/\sqrt{n}$, $\| \bv_\ell^{\AO,\perp} \|$,
    and $\| \bh_\ell^{\AO,\perp} \| / \sqrt{n}$ as consequences, using
    their definition in equations~\eqref{eq:GHk-ao-perp-from-std} and
    \eqref{eq:UVk-ao-perp} and the bounds on the entries of
    $\bL_{g,k}^{\ddagger\top}/\sqrt{n}$ and $\bL_{h,k}^{\ddagger\top}$
    in \Cref{lem:fixed-pt-bound}.  We get the lower bound on $\<
    \bg_k^{\AO,\perp} , \bv_k^{\AO} \> / \sqrt{n}$ by observing
    \begin{equation}
    \begin{aligned}
        \< \bg_k^{\AO,\perp} , \bv_k^{\AO} \> / \sqrt{n} 
            &\underset{\emph{(i)}}= 
            \llangle \bG_k^{\AO} \bL_{g,k}^{\ddagger\top}, \bV_k^{\AO} \rrangle_{kk} / \sqrt{n}
            = 
            [\bL_{g,k}^\ddagger\llangle \bG_k^{\AO} , \bV_k^{\AO} \rrangle]_{kk}/ \sqrt{n} 
        \\
            &\underset{\emph{(ii)}}\mydoteq 
            [\bL_{g,k}^\ddagger\bK_{g,k}\bZ_{u,k}^\top]_{kk}/\sqrt{n} 
            \underset{\emph{(iii)}}=
            [\bL_{g,k}^\top\bZ_{u,k}^\top]_{kk} / \sqrt{n} 
            =
            L_{g,kk} \zeta_{kk}^u / \sqrt{n}
            \underset{\emph{(iv)}}\asymp
            1,
    \end{aligned}
    \end{equation}
    where equality \emph{(i)} holds by
    equation~\eqref{eq:GHk-ao-perp-from-std}, approximate equality
    \emph{(ii)} holds by~\Cref{cor:conc-aux-2m} and the bounds on
    $\bL_{g,k}^\ddagger$ and $\bZ_{u,k}$ in~\Cref{lem:fixed-pt-bound},
    equality \emph{(iii)} holds
    by~\Cref{lem:cholesky-inverse-identities}, and the approximation
    \emph{(iv)} holds by~\Cref{lem:fixed-pt-bound}.  The lower bound
    on $\< \bh_k^{\AO,\perp},\bu_k^{\AO} \>$ holds similarly.
\end{proof}

\begin{corollary}
\label{cor:emp-innovation}
For $k \geq 3$, under Hypothesis (k-1) we have
\begin{equation}
  \label{eq:uk-vk-perp-approx}
  \frac{\proj_{\bU_{k-1}^{\AO}}^\perp
    \bu_k^{\AO}}{\|\proj_{\bU_{k-1}^{\AO}}^\perp \bu_k^{\AO}\|}
  \mydoteq \bu_k^{\AO,\perp}, \qquad
  \frac{\proj_{\bV_{k-1}^{\AO}}^\perp
    \bv_k^{\AO}}{\|\proj_{\bV_{k-1}^{\AO}}^\perp \bv_k^{\AO}\|}
  \mydoteq \bv_k^{\AO,\perp},
    \end{equation}
where, recall, $\bu_k^{\AO,\perp}$ and $\bv_k^{\AO,\perp}$ are defined
by equations~\eqref{eq:UVk-ao-perp}.
\end{corollary}

\begin{proof}[Proof of~\Cref{cor:emp-innovation}]
  By equation~\eqref{eq:UVk-ao-perp},
    \begin{equation}
      \begin{gathered}
        \frac{\proj_{\bU_{k-1}^{\AO}}^\perp
          \bu_k^{\AO}}{\|\proj_{\bU_{k-1}^{\AO}}^\perp \bu_k^{\AO}\|}
        = \Big[ \sqrt{n} \bU_k^{\AO} \big(\sL^\ddagger(n \llangle
          \bU_k^{\AO} \rrangle)\big)^\top \Big]_{\,\cdot\,,k} \mydoteq
        \Big[ \sqrt{n} \bU_k^{\AO}
          \big(\bL_{g,k}^\ddagger/\sqrt{n}\big)^\top
          \Big]_{\,\cdot\,,k} = \bu_k^{\AO,\perp},
        \\ \frac{\proj_{\bV_{k-1}^{\AO}}^\perp
          \bv_k^{\AO}}{\|\proj_{\bV_{k-1}^{\AO}}^\perp \bv_k^{\AO}\|}
        = \Big[ \bV_k^{\AO} \big(\sL^\ddagger(\llangle \bV_k^{\AO}
          \rrangle)\big)^\top \Big]_{\,\cdot\,,k} \mydoteq \Big[
          \bV_k^{\AO} \big(\bL_{h,k}^\ddagger\big)^\top
          \Big]_{\,\cdot\,,k} = \bv_k^{\AO,\perp},
    \end{gathered}
    \end{equation}
    where in each line the approximate equality
    uses~\Cref{cor:chol-conc} and the fact that $\sqrt{n}
    \|\bu_\ell^{\AO}\| \lessdot C$ and $\| \bv_\ell^{\AO} \| \lessdot
    C$ for all $\ell$ and some $\cPmodel$-dependent $C > 0$
    by~\Cref{cor:conc-aux-2m}.
\end{proof}

\subsection{The sequential Gordon inequality}
\label{sec:seq-gordon}

\noindent The key tool for showing Hypothesis (k-1) implies Hypothesis
(k) is the sequential Gordon inequality.

We call the objective in equation~\eqref{eq:min-max} the \emph{primary
objective}, and denote it by $\PriObj_k(\bu;v_0,\bv)$.

\begin{lemma}[Sequential Gordon inequality]
\label{lem:seq-gordon}
Assume Hypothesis (k-1).  Let $S_u =
S_u(\bU_{k-1}^{\PO},\bH_{k-1}^{\PO},\beps_1,\beps_2)$ be a random,
compact subset of $\reals^n$ which is, with probability 1, contained
in a ball $\ball_2(\bzero,C/\sqrt{n})$, and let $S_v =
S_v(\bV_{k-1}^{\PO},\bG_{k-1}^{\PO})$ be a random, compact subset of
$\reals^{p + 1}$ that is, with probability 1, contained in the ball
$\ball_2(\bzero,C)$, where $C$ is any $\cPmodel$-dependent constant.
Then
\begin{equation}
  \min_{(v_0,\bv) \in S_v} \max_{\bu \in S_u} \;
  \PriObj_k(\bu;v_0,\bv) \gtrdot \min_{(v_0,\bv) \in S_v} \max_{\bu
    \in S_u} \; \AuxObj_k(\bu;v_0,\bv),
    \end{equation}
with the reverse inequality if the minimization and maximization
are reversed.  If $S_u$ and $S_v$ are almost-surely convex, we
also have
\begin{equation}
  \min_{\bv \in S_v} \max_{\bu \in S_u} \;
  \PriObj_k(\bu;v_0,\bv) \lessdot \min_{\bv \in S_v} \max_{\bu
    \in S_u} \; \AuxObj_k(\bu;v_0,\bv),
\end{equation}
with the reverse inequality if the minimization and maximization are
reversed.
\end{lemma}

\begin{proof}[Proof of~\Cref{lem:seq-gordon}]
Let $\bA^\| \defn \bA - \proj_{\bU_{k-1}^{\PO}}^\perp \bA
\proj_{\bV_{k-1}^{\PO}}^\perp$.  The proof of~\Cref{lem:seq-gordon} is
based on following two facts, to be proved in the sequel:
\begin{enumerate}
\item We have the following equality of distributions:
  \begin{equation}
    \label{eq:X-to-Xtilde}
            (\bU_{k-1}^{\PO},\;\bV_{k-1}^{\PO},\;\bG_{k-1}^{\PO},\;\bH_{k-1}^{\PO},\;\bA)
    \eqnd
    (\bU_{k-1}^{\PO},\;\bV_{k-1}^{\PO},\;\bG_{k-1}^{\PO},\;\bH_{k-1}^{\PO},\;\bA^\|
    + \proj_{\bU_{k-1}^{\PO}}^\perp \bAtilde
    \proj_{\bV_{k-1}^{\PO}}^\perp),
  \end{equation}
  where $\bAtilde$ is independent of everything else.
\item We have the approximation
  \begin{equation}
    \label{eq:approx-gordon-terms}
    \frac{1}{\sqrt{n}} \big\| \bA^\| - (-\bT_g^\top + \bT_h)
    \big\|_{\op} \mydoteq 0.
  \end{equation}
\end{enumerate}

First we show how these two facts imply the result.  Applying the
standard Gordon inequality (see, for example, Theorem 5.1 in the
paper~\cite{miolane2021} or Theorem 3 in the
paper~\cite{pmlr-v40-Thrampoulidis15}) to the matrix $\bAtilde$
conditionally on $\bA^\|, \bU_{k-1}^{\PO}, \bV_{k-1}^{\PO},
\bG_{k-1}^{\PO}, \bH_{k-1}^{\PO}$, and then marginalizing over
$\bA^\|, \bU_{k-1}^{\PO}, \bV_{k-1}^{\PO}, \bG_{k-1}^{\PO},
\bH_{k-1}^{\PO}$, we have
\begin{equation}
  \label{eq:seq-gordon-exact}
  \begin{aligned}
    &\P\Big( \min_{\bv \in S_v} \max_{\bu \in S_u} \bu^\top \bA \bv +
    \phi(\bu;v_0,\bv) \leq t \Big) \\ &\qquad\leq 2\P\Big( \min_{\bv
      \in S_v} \max_{\bu \in S_u} \bu^\top \bA^\| \bv - \|
    \proj_{\bU_{k-1}^{\PO}}^\perp \bu \| \< \bxi_g,
    \proj_{\bV_{k-1}^{\PO}}^\perp \bv \> + \|
    \proj_{\bV_{k-1}^{\PO}}^\perp \bv \| \< \bxi_h,
    \proj_{\bU_{k-1}^{\PO}}^\perp \bu \> + \phi(\bu;v_0,\bv) \leq t
    \Big).
  \end{aligned}
\end{equation}
Because $\|\bv \| \leq C$ for all $(v_0,\bv) \in S_v$ and $\| \bu \|
\leq C / \sqrt{n}$ for all $\bu \in S_u$,
\begin{equation}
  \begin{gathered}
    \sup_{\substack{\bv \in S_v\\ \bu \in S_u}} \big| \bu^\top \bA^\|
    \bv - \bu^\top(-\bT_g^\top + \bT_h)\bv \big| \leq
    \frac{C}{\sqrt{n}} \big\| \bA^\| - (-\bT_g^\top + \bT_h)
    \big\|_{\op} \mydoteq 0, \\ \sup_{\substack{\bv \in S_v \\ \bu \in
        S_u}} \big| \| \proj_{\bU_{k-1}^{\PO}}^\perp \bu \| \< \bxi_g,
    \proj_{\bV_{k-1}^{\PO}}^\perp \bv \> - \|
    \proj_{\bU_{k-1}^{\PO}}^\perp \bu \| \< \bxi_g, \bv \> \big| \leq
    \frac{C}{\sqrt{n}} \|\proj_{\bV_{k-1}^{\PO}} \bxi_g \| \mydoteq 0,
    \\
\sup_{\substack{\bv \in S_v \\ \bu \in S_u}} \big| \|
    \proj_{\bV_{k-1}^{\PO}}^\perp \bv \| \< \bxi_h,
    \proj_{\bU_{k-1}^{\PO}}^\perp \bu \> - \|
    \proj_{\bV_{k-1}^{\PO}}^\perp \bv \| \< \bxi_h, \bu \> \big| \leq
    \frac{C}{\sqrt{n}} \|\proj_{\bU_{k-1}^{\PO}} \bxi_h \| \mydoteq 0,
\end{gathered}
\end{equation}
where the final two lines use the distributional relations
$\|\proj_{\bV_{k-1}^{\PO}} \bxi_g \|^2 \sim
\chi_{\rank(\bV_{k-1}^{\PO})}^2$ and $\|\proj_{\bU_{k-1}^{\PO}} \bxi_h
\|^2 \sim \chi_{\rank(\bU_{k-1}^{\PO})}^2$, and $\bV_{k-1}^{\PO}$ and
$\bU_{k-1}^{\PO}$ have rank bounded by $k-1$, where $k$ is a finite
constant.  The first display in the lemma follows from the previous
two displays.  The second display in the lemma follows by an analogous
argument, except that when $S_v$ and $S_u$ are convex, the standard
Gordon inequality (see, for example, Theorem 5.1 in the
paper~\cite{miolane2021} or Theorem 3 in the
paper~\cite{pmlr-v40-Thrampoulidis15}) yields the
claim~\eqref{eq:seq-gordon-exact} with $\leq t$ replaced by $\geq t$.

We now turn to proving the claims~\eqref{eq:X-to-Xtilde}
and~\eqref{eq:approx-gordon-terms}.  Beginning with the
claim~\eqref{eq:X-to-Xtilde}, consider any deterministic quantities
$\bU_{k-1}, \allowbreak \bV_{k-1}, \allowbreak \{v_{0,\ell}\}_{\ell
  \leq k-1},, \allowbreak \bM_{k-1},\allowbreak \bN_{k-1}, \allowbreak
\be_1, \allowbreak \be_1'$ and $\be_2$ such that
\begin{equation}
    \label{eq:deterministic-subgradient-identity}
    \begin{gathered}
      \ba_\ell \in \partial_{\bv} \phi_k(\bu_\ell; v_{0,\ell},
      \bv_\ell; \be_1, \be_1', \be_2), \quad \bb_\ell \in
      \partial_{\bu}\Big(-\phi_k(\bu_\ell; v_{0,\ell}, \bv_\ell;
      \be_1, \be_1', \be_2)\Big), \quad 0 \in \partial_{v_0}
      \phi_k(\bu_\ell;v_{0,\ell},\bv_\ell;\be_1,\be_1',\be_2)
    \end{gathered}
\end{equation}
for all $\ell \leq k - 1$.  Then the event that
\begin{equation}
  \label{eq:soln-event}
  \cE_1 \defn \Bigg\{
  \begin{gathered}
    \bU_{k-1}^{\PO} = \bU_{k-1},\quad \bA^\top \bU_{k-1}^{\PO} =
    \bM_{k-1},\quad \bV_{k-1}^{\PO} = \bV_{k-1},\quad \bA
    \bV_{k-1}^{\PO} = \bN_{k-1}, \\ \beps_1 = \be_1,\quad
    \beps_1' = \be_1',\quad \beps_2 = \be_2,\quad
    v_{0,\ell}^{\PO} = v_{0,\ell} \text{ for $\ell \leq k-1$},
  \end{gathered}
  \Bigg\}
\end{equation}
is equivalent to the event that
\begin{equation}
    \label{eq:lin-constraint-event}
    \cE_2 \defn \Big\{ \bA^\top \bU_{k-1} = \bM_{k-1},\quad \bA
    \bV_{k-1} = \bN_{k-1},\quad \beps_1 = \be_1,\quad \beps_1' =
    \be_1',\quad \beps_2 = \be_2 \Big\}.
\end{equation}
Indeed, the inclusion of $\cE_1$ within $\cE_2$ is immediate.  The
opposite inclusion of $\cE_2$ within $\cE_w$ follows by the KKT
conditions for the saddle point problem~\eqref{eq:min-max} and the
identities~\eqref{eq:deterministic-subgradient-identity}.  Writing
\begin{align*}
\bA^\| = \bA \proj_{\bV_{k-1}^{\PO}} + \proj_{\bU_{k-1}^{\PO}} \bA -
\proj_{\bU_{k-1}^{\PO}} \bA \proj_{\bV_{k-1}^{\PO}},
\end{align*}
we see that $\bA^\|$ is a function of $\bU_{k-1}^{\PO}$, $\bA^\top
\bU_{k-1}^{\PO}$, $\bV_{k-1}^{\PO}$, $\bA \bV_{k-1}^{\PO}$, which we
denote by $f$.  Then $ \bA = f\big(\bU_{k-1}^{\PO},\allowbreak
\bA^\top \bU_{k-1}^{\PO},\allowbreak \bV_{k-1}^{\PO},\allowbreak \bA
\bV_{k-1}^{\PO}\big) + \proj_{\bU_{k-1}^{\PO}}^\perp \bA
\proj_{\bV_{k-1}^{\PO}}^\perp $, so that
    \begin{equation}
    \begin{aligned}
        \Law(\bA \bigm| \cE_1)
            &=
            \Law(\bA \bigm| \cE_2)
            =
            \Law(f(\bU_{k-1},\bM_{k-1},\bV_{k-1},\bN_{k-1}) + \proj_{\bU_{k-1}}^\perp \bA \proj_{\bV_{k-1}}^\perp \bigm| \cE_2)
        \\
            &=
            \Law(f(\bU_{k-1},\bM_{k-1},\bV_{k-1},\bN_{k-1}) + \proj_{\bU_{k-1}}^\perp \bAtilde \proj_{\bV_{k-1}}^\perp ),
    \end{aligned}
    \end{equation}
    where the final equality uses that $\proj_{\bU_{k-1}}^\perp \bA
    \proj_{\bV_{k-1}}^\perp$ is independent of $\bA^\top \bU_{k-1}$
    and $\bA \bV_{k-1}$ (for fixed $\bU_{k-1}$, $\bV_{k-1}$), which is
    easily checked because $\bA$ has iid standard Gaussian entries.
    Note that for all realizations of the matrix $\bA$ and the noise
    $\beps_1,\beps_2$, there is some collection of quantities
    $\bU_{k-1}\allowbreak\bV_{k-1}\allowbreak\{v_{0,\ell}\}_{\ell \leq
      k-1},\allowbreak\bM_{k-1}\allowbreak\bN_{k-1}\allowbreak\be_1\allowbreak\be_2$
    for which the event $\cE_1$ has occurred.  Thus, marginalizing the
    previous display over the realization of the matrix $\bA$ and
    noise $\beps_1,\beps_1',\beps_2$ gives~\cref{eq:X-to-Xtilde}.

Next, we prove~\cref{eq:approx-gordon-terms}.
By~\cref{eq:loo-noise-def}, we have $\bG_{k-1}^{\PO} = \bV_{k-1}^{\PO}
\bZ_{v,k-1}^\top - \bA^\top \bU_{k-1}^{\PO}$ and $\bH_{k-1}^{\PO} =
\bU_{k-1}^{\PO} \bZ_{u,k-1}^\top + \bA \bV_{k-1}^{\PO}$.  Thus,
\begin{equation}
  \begin{aligned}
    -\bT_g^\top + \bT_h &= - \bU_{k-1}^{\PO} \bK_{g,k-1}^\ddagger
    \bZ_{v,k-1}\bV_{k-1}^{\PO\top} + \bU_{k-1}^{\PO}
    \bK_{g,k-1}^\ddagger \bU_{k-1}^{\PO\top} \bA \\ &\qquad\qquad+
    \bU_{k-1}^{\PO} \bZ_{u,k-1}^\top \bK_{h,k-1}^\ddagger
    \bV_{k-1}^{\PO\top} + \bA \bV_{k-1}^{\PO}
    \bK_{h,k-1}^{\ddagger}\bV_{k-1}^{\PO\top}.
  \end{aligned}
\end{equation}
We have
\begin{equation}
  \begin{aligned}
    \frac{1}{\sqrt{n}}\bU_{k-1}^{\PO}
    \bZ_{u,k-1}^\top\bK_{h,k-1}^\ddagger&\bV_{k-1}^{\PO\top}
    \underset{\emph{(i)}}= \frac{1}{\sqrt{n}}\bU_{k-1}^{\PO}
    \bK_{g,k-1}^\ddagger \bK_{g,k-1} \bZ_{u,k-1}^\top
    \bK_{h,k-1}^\ddagger\bV_{k-1}^{\PO\top}
    \\ &\underset{\emph{(ii)}}\mydoteq
    \frac{1}{\sqrt{n}}\bU_{k-1}^{\PO} \bK_{g,k-1}^\ddagger
    \llangle \bG_{k-1}^{\PO} , \bV_{k-1}^{\PO} \rrangle
    \bK_{h,k-1}^\ddagger\bV_{k-1}^{\PO\top}
    \\ &\underset{\emph{(iii)}}= \frac{1}{\sqrt{n}}\bU_{k-1}^{\PO}
    \bK_{g,k-1}^\ddagger \llangle \bV_{k-1}^{\PO} \bZ_{v,k-1}^\top
    - \bA^\top \bU_{k-1}^{\PO} , \bV_{k-1}^{\PO} \rrangle
    \bK_{h,k-1}^\ddagger\bV_{k-1}^{\PO\top}
    \\ &\underset{\emph{(iv)}}\mydoteq \frac{1}{\sqrt{n}} \Big(
    \bU_{k-1}^{\PO} \bK_{g,k-1}^\ddagger
    \bZ_{v,k-1}\bV_{k-1}^{\PO\top} - \bU_{k-1}^{\PO}
    \bK_{g,k-1}^\ddagger \bU_{k-1}^{\PO\top} \bA \bV_{k-1}^{\PO}
        \bK_{h,k-1}^\ddagger \bV_{k-1}^{\PO\top} \Big).
  \end{aligned}
\end{equation}
Equation \emph{(i)} holds because $\bK_{g,k-1}^\ddagger \bK_{g,k-1}
\bZ_{u,k-1}^\top = \id_{\bK_{g,k-1}}^\top \bZ_{u,k-1}^\top =
\bZ_{u,k-1}^\top$ by~\Cref{lem:cholesky-inverse-identities} and the
innovation compatibility of $\bZ_{u,k-1}$ with $\bK_{g,k-1}$ (the
columns of $\bZ_{u,k-1}^\top$ lie in the span of the standard basis
vectors corresponding to innovative indices).  Approximation
\emph{(ii)} holds by~\Cref{cor:conc-aux-2m}, the bounds on
$\bU_{k-1}^{\PO}$, $\bV_{k-1}^{\PO}$ in~\Cref{cor:AO-bounds}, and the
bounds on $\bK_{g,k-1}^\ddagger$, $\bK_{h,k-1}^\ddagger$
in~\Cref{lem:fixed-pt-bound}.  Equation \emph{(iii)} holds by the
definition of $\bG_{k-1}^{\PO}$.  Approximation \emph{(iv)} holds
because, by~\Cref{cor:conc-aux-2m}, $\llangle \bV_{k-1}^{\PO} \rrangle
\mydoteq \bK_{h,k-1}$, and,
using~\Cref{lem:cholesky-inverse-identities} and the innovation
compatibility of $\bZ_{v,k-1}$ with respect to $\bK_{h,k-1}$,
$\bZ_{v,k-1} \bK_{h,k-1} \bK_{h,k-1}^\ddagger = \bZ_{v,k-1}
\id_{\bK_{h,k-1}} = \bZ_{v,k-1}$.  Combining the previous two
displays, we get
\begin{equation}
  \frac{1}{\sqrt{n}}(-\bT_g^\top + \bT_h) \mydoteq \frac{1}{\sqrt{n}}
  \Big( \bU_{k-1}^{\PO} \bK_{g,k-1}^\ddagger \bU_{k-1}^{\PO\top} \bA +
  \bA \bV_{k-1}^{\PO} \bK_{h,k-1}^{\ddagger}\bV_{k-1}^{\PO\top} -
  \bU_{k-1}^{\PO} \bK_{g,k-1}^\ddagger \bU_{k-1}^{\PO\top} \bA
  \bV_{k-1}^{\PO} \bK_{h,k-1}^\ddagger \bV_{k-1}^{\PO\top} \Big).
\end{equation}
Now, using~\Cref{cor:chol-conc}, we have $\bU_{k-1}^{\PO}
\bK_{g,k-1}^\ddagger \bU_{k-1}^{\PO\top} \mydoteq \bU_{k-1}^{\PO}
\sK^\ddagger(\llangle \bU_{k-1}^{\PO}\rrangle) \bU_{k-1}^{\PO\top} =
\proj_{\bU_{k-1}^{\PO}}$ and likewise, $\bV_{k-1}^{\PO}
\bK_{h,k-1}^\ddagger \bV_{k-1}^{\PO\top} \mydoteq \bV_{k-1}^{\PO}
\sK^\ddagger(\llangle \bV_{k-1}^{\PO}\rrangle) \bV_{k-1}^{\PO\top} =
\proj_{\bV_{k-1}^{\PO}}$.  Moreover, $\| \bA \|_{\op} /\sqrt{n}
\lessdot C$ \cite[Corollary 5.35]{vershynin_2012}, whence
\begin{equation}
\frac{1}{\sqrt{n}}(-\bT_g^\top + \bT_h) \mydoteq
\frac{1}{\sqrt{n}}(\proj_{\bU_{k-1}^{\PO}} \bA + \bA
\proj_{\bV_{k-1}^{\PO}} - \proj_{\bU_{k-1}^{\PO}} \bA
\proj_{\bV_{k-1}^{\PO}}) = \frac{1}{\sqrt{n}} \bA^\|,
\end{equation}
as desired.
\end{proof}


\subsection{Properties of the auxiliary objective}
\label{sec:properties-aux-obj}

Throughout the proofs in this section, we freely invoke the bounds on
the problem parameters or fixed point parameters
from~\Cref{lem:fixed-pt-bound} and Assumption A1 without citing the
lemma or assumption each time.

\begin{lemma}
\label{lem:lk*-conc}
For $k = 5,6$, under Hypothesis $(k-1)$, we have
\begin{equation}
  \ell_k^*(\bu_k^{\AO}; \bw, \by_1, \by_2) \mydoteq
  \E[\ell_k^*(\bu_k^{\SE}; \bw^{\SE}, \by_1^{\SE}, \by_2^{\SE})].
\end{equation}
\end{lemma}

\begin{proof}[Proof of~\Cref{lem:lk*-conc}]
We handle $k = 5$ and $k = 6$ separately.  First, we study $k = 5$.
Because $\ell_\prop(\eta;\action)$ is strongly convex and smooth
(Assumption A1) in $\eta$, so too is $\ell_\prop^*(nu;\action)$ in
$nu$.  Moreover, $\ell_\prop^*(0;\action) = - \sup_\eta \{
\ell_\prop(\eta;\action) \}$ and
$\partial_\eta\ell_\prop^*(\eta;\action)|_{\eta=0} = \argmin_\eta \{
\ell_\prop(\eta;\action)$.  From Assumption A1, we can conclude that
all of these quantities are bounded above in absolute value by $C$.
This implies that $|\ell_\prop^*(\eta;1) - \ell_\prop^*(\eta;0)| \leq
C(1+|\eta|)$.  In particular, $\ell_5^*(\bu_6;\bw,-n\bu_5,\by_2)$ is a
$C$ order-2 pseudo-Lipschitz function of $n\bu_6$ and $n\bu_5$, whence
the result for $k = 5$ holds by~\Cref{lem:hypothesis-k-ao}.

Next, we study $k = 6$.  Recalling the expression for
$u_{6,i}^\AO$ in equation~\eqref{eq:u6ao-explicit}, we may write
\begin{equation}
  \begin{gathered}
  \ell_6^*(\bu_6^{\AO};\bw,\by_1,\by_2) = \<\bu_6^{\AO},\bh_2^{\AO} +
  \beps_2\> - nu_{4,i}^{\PO}\,\frac{(\nu_{6,0} + \nu_{6,\sx} -
    h_{2,i}^{\PO} - \eps_{2,i} +
    h_{6,i}^{\PO})^2}{\Big(\frac{\zeta_{66}^u}{n}
    \sqrt{w(h_{1,i}^{\PO})} +
    \frac{1}{\sqrt{w(h_{1,i}^{\PO})}}\Big)^2},
    \end{gathered}
    \end{equation}
where we use that $-nu_{4,i}^{\AO} = y_{1,i}^{\AO}$ and that
$y_{1,i}^{\AO} \in \{0,1\}$.  By~\Cref{lem:hypothesis-k-ao}, we have
that $\< \bu_6^{\AO} , \bh_2^{\AO} + \beps_2 \> \mydoteq \<
\bu_6^{\SE} , \bh_2^{\SE} + \beps_2^{\SE} \>_{\Ltwo}$.  Because the
derivatives with respect to $h_{1,i}^{\AO}$ and $nu_{4_i}^{\AO}$ grow
quadratically in $h_{2,i}^{\AO}$, $\eps_{6,i}$, and $h_{6,i}^{\AO}$,
its concentration requires a truncation argument.  For
\begin{equation}
  (h_{1,i}^{\AO},h_{2,i}^{\AO},h_{6,i}^{\AO},nu_{4,i}^{\AO},\eps_i)
  \mapsto (-nu_{4,i}^{\AO}\wedge 1)\,\frac{(\nu_{6,0} + \nu_{6,\sx} -
    h_{2,i}^{\AO} - \eps_{2,i} + h_{6,i}^{\AO})^2\wedge
    M^2}{\Big(\frac{\zeta_{66}^u}{n} \sqrt{w(h_{1,i}^{\AO})} +
    \frac{1}{\sqrt{w(h_{1,i}^{\AO})}}\Big)^2}
\end{equation}
is $CM^2$-Lipschitz, because $c < w(\,\cdot\,) < C$ has derivative
bounded by $C$ and $|\nu_{6,0}|$, $|\nu_{6,\sx}|$,
$\zeta_{66}^u/n$ are bounded above by $C$.
By~\Cref{lem:hypothesis-k-ao},
\begin{equation}
  \begin{aligned}
    &\frac{1}{2nM^2} \sum_{i=1}^n (-nu_{4,i}^{\AO}\wedge
    1)\,\frac{(\nu_{6,0} + \nu_{6,\sx} - h_{2,i}^{\AO} -
      \eps_{2,i} + h_{6,i}^{\AO})^2\wedge
      M^2}{\Big(\frac{\zeta_{66}^u}{n} \sqrt{w(h_{1,i}^{\AO})} +
      \frac{1}{\sqrt{w(h_{1,i}^{\AO})}}\Big)^2}
    \\ &\qquad\qquad\mydoteq \frac{1}{2M^2} \E\Big[
      (-nu_{4,i}^{\SE}\vee 1)\,\frac{(\nu_{6,0} + \nu_{6,\sx} -
        h_{2,i}^{\SE} - \eps_{2,i} + h_{6,i}^{\SE})^2\wedge
        M^2}{\Big(\frac{\zeta_{66}^u}{n} \sqrt{w(h_{1,i}^{\SE})} +
        \frac{1}{\sqrt{w(h_{1,i}^{\SE})}}\Big)^2} \Big]
    \\ &\qquad\qquad= \frac{1}{2M^2} \E\Big[
      -nu_{4,i}^{\SE}\,\frac{(\nu_{6,0} + \nu_{6,\sx} -
        h_{2,i}^{\SE} - \eps_{2,i} +
        h_{6,i}^{\SE})^2}{\Big(\frac{\zeta_{66}^u}{n}
        \sqrt{w(h_{1,i}^{\SE})} +
        \frac{1}{\sqrt{w(h_{1,i}^{\SE})}}\Big)^2} \Big] +
    \epsilon(M),
  \end{aligned}
\end{equation}
where $|\epsilon(M)|\leq Ce^{-cM^2}$ by standard Gaussian tail bounds
and the fact that $w(h_{1,i}^{\SE})$ is bounded above and below and
$-nu_{4,i}^{\SE} = \in \{0,1\}$.  Moreover,
\begin{equation}
  \begin{aligned}
    &\Big| \frac{1}{2nM^2} \sum_{i=1}^n (-nu_{4,i}^{\AO}\wedge
    1)\,\frac{(\nu_{6,0} + \nu_{6,\sx} - h_{2,i}^{\AO} - \eps_{2,i} +
      h_{6,i}^{\AO})^2\wedge M^2}{\Big(\frac{\zeta_{66}^u}{n}
      \sqrt{w(h_{1,i}^{\AO})} +
      \frac{1}{\sqrt{w(h_{1,i}^{\AO})}}\Big)^2}
    \\ &\qquad\qquad\qquad\qquad\qquad\qquad- \frac{1}{2nM^2}
    \sum_{i=1}^n (-nu_{4,i}^{\AO})\,\frac{(\nu_{6,0} + \nu_{6,\sx} -
      h_{2,i}^{\AO} - \eps_{2,i} +
      h_{6,i}^{\AO})^2}{\Big(\frac{\zeta_{66}^u}{n}
      \sqrt{w(h_{1,i}^{\AO})} +
      \frac{1}{\sqrt{w(h_{1,i}^{\AO})}}\Big)^2} \Big|
    \\ &\qquad\qquad\leq \frac{C}{2nM^2} \sum_{i=1}^n (C +
    |h_{2,i}^{\AO}| + |\eps_{2,i}| + |h_{6,i}^{\AO}| - M)_+^2
    \\ &\qquad\qquad\mydoteq \frac{C}{2M^2} \E\big[(C +
      |h_{2,i}^{\SE}| + |\eps_{2,i}| + |h_{6,i}^{\SE}| - M)_+^2\big]
    \Big) \leq Ce^{-cM^2},
    \end{aligned}
\end{equation}
where in the first inequality we have
applied~\Cref{lem:fixed-pt-bound}, in the approximate equality we have
applied~\Cref{lem:hypothesis-k-ao}, and in the second inequality we
have applied standard Gaussian tail bounds.  Combining the previous
two displays, we get
\begin{equation}
  \begin{aligned}
    &\Big| \frac{1}{2n(M \vee M^2)} \sum_{i=1}^n
    (-nu_{4,i}^{\PO})\,\frac{(\nu_{6,0} + \nu_{6,\sx} - h_{2,i}^{\PO}
      - \eps_{2,i} + h_{6,i}^{\PO})^2}{\Big(\frac{\zeta_{66}^u}{n}
      \sqrt{w(h_{1,i}^{\PO})} +
      \frac{1}{\sqrt{w(h_{1,i}^{\PO})}}\Big)^2}
    \\
& \qquad\qquad\qquad\qquad\qquad- \frac{1}{2(M \vee M^2)} \E\Big[
      (-nu_{4,i}^{\SE})\,\frac{(\nu_{6,0} + \nu_{6,\sx} -
        h_{2,i}^{\SE} - \eps_{2,i} +
        h_{6,i}^{\SE})^2}{\Big(\frac{\zeta_{66}^u}{n}
        \sqrt{w(h_{1,i}^{\SE})} +
        \frac{1}{\sqrt{w(h_{1,i}^{\SE})}}\Big)^2} \Big] \Big| \lessdot
    Ce^{-cM^2}.
  \end{aligned}
\end{equation}
That is, for all $\epsilon > 0$, the probability that the left-hand
side exceeds $\epsilon + Ce^{-cM^2}$ is bounded above by
$C'e^{-c'n\epsilon^r}$ for some $C,c,C',c',r > 0$ depending only on
$\cPmodel$.  Then, taking $M = \sqrt{\log(1/\epsilon)}$, we get
\begin{equation}
  \frac{1}{2n} \sum_{i=1}^n (-nu_{4,i}^{\PO})\,\frac{(\nu_{6,0} +
    \nu_{6,\sx} - h_{2,i}^{\PO} - \eps_{2,i} +
    h_{6,i}^{\PO})^2}{\Big(\frac{\zeta_{66}^u}{n}
    \sqrt{w(h_{1,i}^{\PO})} +
    \frac{1}{\sqrt{w(h_{1,i}^{\PO})}}\Big)^2} \mydoteq \E\Big[
    (-nu_{4,i}^{\SE})\,\frac{(\nu_{6,0} + \nu_{6,\sx} - h_{2,i}^{\SE}
      - \eps_{2,i} + h_{6,i}^{\SE})^2}{\Big(\frac{\zeta_{66}^u}{n}
      \sqrt{w(h_{1,i}^{\SE})} +
      \frac{1}{\sqrt{w(h_{1,i}^{\SE})}}\Big)^2} \Big].
    \end{equation}
Because also $\< \bu_6^{\AO} , \bh_2^{\AO} + \beps_2 \> \mydoteq \<
\bu_6^{\SE} , \bh_2^{\SE} + \beps_2^{\SE} \>_{\Ltwo}$, the proof of
the lemma in the case $k = 6$ is complete.
\end{proof}

\begin{lemma}[Value of $\AuxObj_k$ at $(\bu_k^{\AO};\nu_{k,0},\bv_k^{\AO})$]
\label{lem:AuxObj-conc}
For $k = 5,6$, under Hypothesis (k-1) we have
\begin{equation}
  \AuxObj_k(\bu_k^{\AO};\nu_{k,0},\bv_k^{\AO}) \mydoteq \ell_k,
\end{equation}
where
\begin{equation}
  \ell_k \defn - \< \bg_k^{\SE} , \bv_k^{\SE} \>_{\Ltwo} + \<
  \bh_k^{\SE} , \bu_k^{\SE} \>_{\Ltwo} -
  \E[\ell_k^*(\bu_k^{\SE};\bw^{\SE},\by_1^{\SE},\by_2^{\SE})] +
  \E[\Omega_k(\bv_k^{\SE})].
    \end{equation}
\end{lemma}

\begin{proof}[Proof of~\Cref{lem:AuxObj-conc}]
By the definition of $\phi_k$ and $\AuxObj_k$
(equations~\eqref{eq:objectives} and \eqref{eq:def-AuxObj}),
\begin{equation}
\begin{aligned}
\AuxObj_k(\bu_k^{\AO};\nu_{k,0},\bv_k^{\AO}) &= - \< \bg(\bu_k^{\AO})
, \bv_k^{\AO} \> + \< \bh(\bv_k^{\AO}) , \bu_k^{\AO} \> + \<
\bu_k^{\AO} , \ones \>(\nu_{k,0} + \< \bmu_{\sx} , \bv_k^{\AO} \>)
\\ &\qquad- \ell_k^*(\bu_k^{\AO};\bw,\by_1,\by_2) +
\Omega_k(\bv_k^{\AO}).
\end{aligned}
\end{equation}
Using the definition of $\bg(\,\cdot\,)$ (see
equation~\eqref{eq:gordon-gh-explicit}), we have
\begin{equation}
\begin{aligned}
\< \bg(\bu_k^{\AO}) , \bv_k^{\AO} \> &= \sum_{\ell=1}^{k-1} \<
\bu_\ell^{\AO,\perp},\bu_k^{\AO}\>\<\bg_\ell^{\PO,\perp},\bv_k^{\AO}\>
+ \| \proj_{\bU_{k-1}^{\AO}}^\perp \bu_k^{\AO} \| \< \bxi_g ,
\bv_k^{\AO} \> \\ &\mydoteq \sum_{\ell=1}^k \<
\bu_\ell^{\AO,\perp},\bv_k^{\AO}\>\<\bg_\ell^{\PO,\perp},\bv_k^{\AO}\>
= \llangle \bu_k^{\AO} , \bU_k^{\AO} \bL_{g,k}^{\ddagger\top} \rrangle
\llangle \bG_k^{\AO} \bL_{g,k}^{\ddagger\top} , \bv_k^{\AO} \rrangle
\\ &\mydoteq (\bK_{g,k})_{k,1:k} \bK_{g,k}^\ddagger
(\bK_{g,k}\bZ_{u,k}^\top)_{1:k,k} = (\id_{\bK_{g,k}} \bK_{g,k}
\bZ_{u,k}^\top)_{kk} = (\bK_{g,k} \bZ_{u,k}^\top)_{kk} = \<
\bg_k^{\SE} , \bv_k^{\SE} \>_{\Ltwo},
\end{aligned}
\end{equation}
where the approximate equality in the second line holds
by~\Cref{cor:emp-innovation} and the high-probability bounds on $\|
\proj_{\bU_{k-1}^{\AO}}^\perp \bu_k^{\AO} \|$ and $ \< \bxi_g ,
\bu_k^{\AO} \>$ implied by~\Cref{cor:conc-aux-2m,cor:chol-conc}; the
equality in the second line holds by equations~\eqref{eq:GHk-ao-perp}
and~\eqref{eq:UVk-ao-perp}; the approximate equality in the third line
holds by~\Cref{cor:conc-aux-2m}, the fixed point
relation~\eqref{eq:fixpt-general}; and the remaining equalities in the
third line use equation~\eqref{eq:K-pseudo-inverse}, the definition of
$\id_{\bK}$ from~\Cref{lem:cholesky-inverse-identities}, and the fixed
point equations~\eqref{eq:fixpt-general}.  We have used bounds on the
fixed point parameters throughout.  By a completely analogous argument
we can show that $\< \bh(\bv_k^{\AO}) , \bu_k^{\AO} \>_{\Ltwo}
\mydoteq \< \bh_k^{\SE} , \bu_k^{\SE} \>_{\Ltwo}$.  Combining these
facts with~\Cref{lem:lk*-conc} and the fact $\Omega_k(\bv_k^{\AO}) \mydoteq \E[\Omega_k(\bv_k^{\SE})]$ by ~\Cref{lem:hypothesis-k-ao}
gives~\Cref{lem:AuxObj-conc}.
\end{proof}

\begin{lemma}[Approximate stationarity of $\AuxObj_k$ at $(\bu_k^{\AO};\nu_{k,0},\bv_k^{\AO})$]
\label{lem:approximate-stationarity}
For $k = 5,6$, under Hypothesis (k-1) we have
\begin{equation}
  \begin{gathered}
    \frac{\de\,}{\de v_0}\AuxObj_k(\bu_k^{\AO};\nu_{k,0},\bv_k^{\AO})
    \mydoteq 0, \qquad \frac{\de\,}{\de
      \bv}\AuxObj_k(\bu_k^{\AO};\nu_{k,0},\bv_k^{\AO}) \mydoteq
    \bzero.
    \end{gathered}
\end{equation}
Moreover, there exists a random vector $\bdelta_u$ such that
\begin{equation}
  \bdelta_u \in \partial_{\bu} \Big(
  -\AuxObj_k(\bu_k^{\AO};\nu_{k,0},\bv_k^{\AO}) \Big) \text{
    almost surely, and } \frac{1}{\sqrt{n}}\bdelta_u \mydoteq
      \bzero.
\end{equation}
\end{lemma}
\noindent The proof of~ \Cref{lem:approximate-stationarity} relies on
an approximation for the derivative of each term in $\AuxObj_k$.  The
next two lemmas give these term-by-term approximations.  We first
state these two lemmas, then show that they
imply~\Cref{lem:approximate-stationarity}, and then prove the two
lemmas.

\begin{lemma}[Derivatives of Gordon terms]
\label{lem:gordon-term-deriv}
    For $k =5,6$, under Hypothesis (k-1) we have
    \begin{equation}
    \begin{aligned}
        \frac{1}{\sqrt{n}}\frac{\de}{\de \bu}\< \bg(\bu_k^{\AO}),
        \bv_k^{\AO} \> &\mydoteq \frac{1}{\sqrt{n}} \sum_{\ell}
        \zeta_{k\ell}^u \bu_\ell^{\AO}, \;\; & \;\; \frac{\de}{\de
          \bv}\< \bg(\bu_k^{\AO}), \bv_k^{\AO} \> &\mydoteq
        \bg_k^{\AO}, \\ \frac{1}{\sqrt{n}}\frac{\de}{\de \bu}\<
        \bh(\bv_k^{\AO}), \bu_k^{\AO} \> &\mydoteq
        \frac{1}{\sqrt{n}}\bh_k^{\AO}, \;\; & \;\; \frac{\de}{\de
          \bv}\< \bh(\bv_k^{\AO}), \bu_k^{\AO} \> &\mydoteq
        \sum_{\ell} \zeta_{k\ell}^v \bv_\ell^{\AO}.
    \end{aligned}
    \end{equation}
\end{lemma}

\begin{lemma}[Derivatives of penalty terms]
\label{lem:deriv-penalty}
    For $k \geq 5,6$, under Hypothesis (k-1), we have the relations
    \begin{equation}
      \begin{gathered}
        \frac{\de\,}{\de v_0}
        \phi_k(\bu_k^{\AO};\nu_{k,0},\bv_k^{\AO}) \mydoteq 0, \qquad
        \frac{\de\,}{\de\bv} \phi_k(\bu_k^{\AO};\nu_{k,0},\bv_k^{\AO})
        \mydoteq \frac{\de\,}{\de\bv}
        \phi_{k,v}(\bv_k^{\AO};\bG_{k-1}^{\AO}),
    \end{gathered}
    \end{equation}
    provided that $\phi_{k,v}$ is defined using solutions to the fixed
    point equations~\eqref{eq:fixpt-general}.

    Moreover, there exists a random vector $\bdelta \in \reals^n$ such
    that $\bdelta/\sqrt{n} \mydoteq \bzero$ and almost surely the
    following occurs: for all $\bdelta \in \partial_{\bu}
    \phi_{k,u}(\bu_k^{\AO};\bH_{k-1}^{\AO};\beps_1,\beps_1',\beps_2)$,
    also $\bdelta + \bdelta \in \partial_{\bu} \big( -\phi_k
    (\bu_k^{\AO}; \nu_{k,0}, \bv_k^{\AO}) \big)$.  (This holds
    provided $\phi_{k,v}$ is defined using the parameters that solve
    the fixed point equations~\eqref{eq:fixpt-general}).
\end{lemma}

\begin{remark}
  We have a more complicated derivative guarantee for the derivative
  with respect to $\bu$ because, when $y_{1,i} = 0$, $\ell_6^*$ is the
  convex indicator of $u_{6,i} = 0$, which is not differentiable.
\end{remark}

\begin{proof}[Proof of~\Cref{lem:approximate-stationarity}]
    By~\Cref{lem:gordon-term-deriv,lem:deriv-penalty} and the
    definition of $\AuxObj_k$ (see~\cref{eq:def-AuxObj}), we have
\begin{align*}
    \frac{\de\,}{\de \bv}\AuxObj_k(\bu_k^{\AO};\nu_{k,0},\bv_k^{\AO})
    \mydoteq -\bg_k^{\AO} + \sum_{\ell=1}^k \zeta_{k\ell}^u
    \bv_\ell^{\AO} + \frac{\de\,}{\de \bv}
    \phi_{k,v}(\bv_k^{\AO};\bG_{k-1}^{\AO}) = \bzero,
\end{align*}
where the final equality holds by the KKT conditions for the
problem~\eqref{eq:AO-separated-opt}.  Because $-\< \bg(\bu),\bv\>$ and
$\< \bh(\bv),\bu\>$ have no dependence on $v_0$,
\Cref{lem:deriv-penalty} and the definition of $\AuxObj_k$ imply that
\begin{align*}
\frac{\de\,}{\de v_0}\AuxObj_k(\bu_k^{\AO};\nu_{k,0},\bv_k^{\AO}) =
\frac{\de\,}{\de v_0} \phi_k(\bu_k^{\AO};\nu_{k,0},\bv_k^{\AO})
\mydoteq 0.
\end{align*}
By the KKT conditions for the problem~\eqref{eq:AO-separated-opt},
we have
\begin{align*}
\bh_k^{\AO} - \sum_{\ell=1}^k \zeta_{k\ell}^u \bu_\ell^{\AO} \in
\partial_{\bu}\phi_{k,u}(\bu_k^{\AO};\bH_{k-1}^{\AO};(\beps_j)_j).
\end{align*}
Taking $\bdelta$ as in~Cref{lem:deriv-penalty} yields
\begin{align*}
\bh_k^{\AO} - \sum_{\ell=1}^k \zeta_{k\ell}^u \bu_\ell^{\AO} + \bdelta
\in \partial_{\bu} \Big( -\phi_k(\bu_k^{\AO};\nu_{k,0},\bv_k^{\AO})
\Big).
\end{align*}
Then, by~\Cref{lem:gordon-term-deriv,lem:deriv-penalty} and the
definition of $\AuxObj_k$, we have
\begin{equation}
  \bzero \mydoteq \frac{1}{\sqrt{n}} \Big( \frac{\de}{\de \bu}\<
  \bg(\bu_k^{\AO}), \bv_k^{\AO} \> - \frac{\de}{\de \bu}\<
  \bh(\bv_k^{\AO}), \bu_k^{\AO} \> + \bh_k^{\AO} - \sum_{\ell=1}^k
  \zeta_{k\ell}^u \bu_\ell^{\AO} + \bdelta \Big) \in \partial_{\bu}
  \Big( -\AuxObj_k(\bu_k^{\AO};\nu_{k,0},\bv_k^{\AO}) \Big),
    \end{equation}
which completes the proof.
\end{proof}

\begin{proof}[Proof of~\Cref{lem:gordon-term-deriv}]
Because $\sqrt{n} L_{g,kk} \gtrsim 1$ and $L_{h,kk} \gtrsim 1$,
\Cref{cor:chol-conc} implies that with exponentially high probability
$\proj_{\bU_{k-1}^{\AO}}^\perp \bu_k^{\AO} \neq \bzero$ and
$\proj_{\bV_{k-1}^{\AO}}^\perp \bv_k^{\AO} \neq \bzero$.  On this
event, both $\bg(\bu)$ and $\bh(\bv)$, as defined in
equation~\eqref{eq:gordon-gh-explicit}, are differentiable at
$\bu_k^{\AO}$ and $\bv_k^{\AO}$.  We compute
\begin{equation}
\begin{aligned}
 \frac{1}{\sqrt{n}}\frac{\de}{\de \bu}\< \bg(\bu_k^{\AO}), \bv_k^{\AO}
 \> &= \frac{1}{\sqrt{n}}\sum_{\ell=1}^{k-1} \bu_\ell^{\AO,\perp} \<
 \bg_\ell^{\AO,\perp} , \bv_k^{\AO} \> +
 \frac{1}{\sqrt{n}}\frac{\proj_{\bU_{k-1}^{\AO}}^\perp
   \bu_k^{\AO}}{\|\proj_{\bU_{k-1}^{\AO}}^\perp \bu_k^{\AO}\|}
 \<\bg_k^{\AO,\perp},\bv_k^{\AO}\> \\ &\underset{\emph{(i)}}\mydoteq
 \frac{1}{\sqrt{n}}\sum_{\ell=1}^k \bu_\ell^{\AO,\perp} \<
 \bg_\ell^{\AO,\perp} , \bv_k^{\AO} \> \underset{\emph{(ii)}}=
 \frac{1}{\sqrt{n}}\bU_k^{\AO} \bK_{g,k}^{\ddagger} (\bG_k^{\AO})^\top
 \bv_k^{\AO} \\ &\underset{\emph{(iii)}}\mydoteq
 \frac{1}{\sqrt{n}}\bU_k^{\AO} \bK_{g,k}^{\ddagger}
      [\bK_{g,k}\bZ_{u,k}^\top]_{\,\cdot\,,k} \underset{\emph{(iv)}}=
      \frac{1}{\sqrt{n}}[\bU_k^{\AO} \id_{\bK_{g,k}}^\top
        \bZ_{u,k}^\top]_{\,\cdot\,,k} \underset{\emph{(v)}}=
      \frac{1}{\sqrt{n}}[\bU_k^{\AO}\bZ_{u,k}^\top]_{\,\cdot\,,k}.
    \end{aligned}
    \end{equation}
Approximate equality \emph{(i)} uses that $\<
\bg_k^{\AO,\perp},\bv_k^{\AO}\>/\sqrt{n} \lessdot C$ by
Cauchy--Schwartz and that $\|\bg_\ell^{\AO,\perp}\|/\sqrt{n} \lessdot
C$ and $\| \bv_k^{\AO} \| \lessdot C$ by~\Cref{cor:AO-bounds}.
Equality \emph{(ii)} uses equations~\eqref{eq:GHk-ao-perp} and
\eqref{eq:UVk-ao-perp}.  Approximate equality \emph{(iii)} uses
\Cref{cor:conc-aux-2m} and the fact that $\sqrt{n}\| \bH_k^{\AO}
\|_{\op} \lessdot C$ and $\|\bK_{g,k}^\ddagger\|_{\op}/n \lessdot C$
by~\Cref{cor:AO-bounds}.  Equality \emph{(iv)}
uses~\Cref{lem:cholesky-inverse-identities}.  Equality \emph{(v)} uses
the definition of $\id_{\bK_{g,k}}^\top$ and the innovation
compatibility of $\bZ_{u,k}$ with $\bK_{g,k}$ (the columns of
$\bZ_{u,k}^\top$ lie in the span of the standard basis vectors
corresponding to innovative indices).  Bounds are the fixed point
parameters are used throughout.  This gives the first approximate
equality in the lemma.
    
Again using the expression for $\bg$ and $\bh$ in
equation~\eqref{eq:gordon-gh-explicit}, we compute $ \frac{\de}{\de
  \bv}\< \bg(\bu_k^{\AO}), \bv_k^{\AO} \> = \sum_{\ell=1}^{k-1}
\bg_\ell^{\AO,\perp} \< \bu_\ell^{\AO,\perp} , \bu_k^{\AO} \> + \|
\proj_{\bU_{k-1}^{\AO}}^\perp \bu_k^{\AO} \| \bxi_g.  $ We have $
\sqrt{n}\< \bu_\ell^{\AO,\perp} , \bu_k^{\AO} \> =
\sqrt{n}\big[\bL_{g,k}^\ddagger\bU_k^{\AO\top}\bu_k^{\AO}\big]_\ell
\mydoteq \sqrt{n}\big[\bL_{g,k}^\ddagger\bK_{g,k}\big]_{\ell k} =
\sqrt{n}L_{g,k\ell}, $ where the first equality holds by
equation~\eqref{eq:UVk-ao-perp}; the approximate equality holds
by~\Cref{cor:conc-aux-2m}, and the second equality holds
by~\Cref{lem:cholesky-inverse-identities} (in particular, that
$\bL_{g,k}^\ddagger \bL_{g,k} \bL_{g,k}^\top = \id_{\bK_{g,k}}^\perp
\bL_{g,k}^\top = (\id_{\bK_{g,k}}^\perp)^\top\bL_{g,k}^\top =
\bL_{g,k}^\top$).  By~\Cref{cor:chol-conc}, $ \sqrt{n} \|
\proj_{\bU_{k-1}^{\AO}}^\perp \bu_k^{\AO} \| = \sL(n \llangle
\bU_k^{\AO} \rrangle)_{kk} \mydoteq \sqrt{n} L_{g,kk}.  $ Because
$\|\bg_\ell^{\AO,\perp}\|/\sqrt{n} \lessdot C$ and
$\|\bxi_g\|/\sqrt{n} \lessdot C$ by~\Cref{cor:AO-bounds}, these
computations imply that \mbox{$\frac{\de}{\de \bv}\< \bg(\bu_k^{\AO}),
  \bv_k^{\AO} \> \mydoteq \sum_{\ell=1}^k
  L_{g,k\ell}\bg_\ell^{\AO,\perp} = \bg_k^{\AO}$,} where the equality
uses equation~\eqref{eq:GHk-ao-perp}.  This gives the second equation
in the first line of the lemma.

  The second line of the lemma follows similarly.  For completeness,
  we write out the computations, but do not explain the justification
  for each line, as they are analogous to the justifications above.
  We have
  \begin{equation}
    \begin{aligned}
      \frac{\de}{\de \bv}\< \bh(\bv_k^{\AO}), \bu_k^{\AO} \>
      &\mydoteq
      \sum_{\ell=1}^k\bv_\ell^{\AO,\perp}\<\bh_\ell^{\AO,\perp},\bu_k^{\AO}\>
      =
      \bV_k^{\AO}\bK_{h,k}^\ddagger\bH_k^{\AO\top}\bu_k^{\AO}
      \mydoteq
      \bV_k^{\AO}\bK_{h,k}^\ddagger[\bK_{h,k}\bZ_v^\top]_{\,\cdot\,,k}
      =
      [\bV_k^{\AO}\bZ_v^\top]_{\,\cdot\,,k},
    \end{aligned}
  \end{equation}
  which gives the second equation in the second line of the lemma.
  Further, we have
  \begin{align*}
\frac{1}{\sqrt{n}}\frac{\de}{\de \bu}\< \bh(\bv_k^{\AO}), \bu_k^{\AO}
\> =
\sum_{\ell=1}^k\bh_\ell^{\AO,\perp}\<\bv_\ell^{\AO,\perp},\bv_k^{\AO}\>
+ \|\proj_{\bV_{k-1}^{\AO}}^\perp \bv_k^{\AO}\|\bxi_h.
  \end{align*}
Because $\< \bv_\ell^{\AO,\perp},\bv_k^{\AO}\> \mydoteq L_{h,k\ell}$
and $\|\proj_{\bV_{k-1}^{\AO}}^\perp \bv_k^{\AO} \| \mydoteq
L_{h,kk}$, we get the first equation in the second line of the lemma.
\end{proof}

\begin{proof}[Proof of~\Cref{lem:deriv-penalty}]
By explicit differentiation of equation~\eqref{eq:objectives} for $k =
3,4$, $ \frac{\de\,}{\de v_0}
\phi_k(\bu_k^{\AO};\nu_{k,0},\bv_k^{\AO}) = \< \bu_k^{\AO} , \ones \>
\mydoteq \< \bu_k^{\SE} , \ones \>_{\Ltwo} = 0, $ where the
approximate equality holds by~\Cref{lem:hypothesis-k-ao} and the final
equality holds by the fixed point equations~\eqref{eq:fixpt-general}.
We conclude the first approximate equality in the first display of the
lemma.

By Assumption A1, the function $\Omega_k$ is differentiable, whence by
explicit differentiation of equation~\eqref{eq:objectives} for $k =
3,4$, $ \frac{\de}{\de \bv} \phi_k(\bu_k^{\AO};\nu_{k,0},\bv_k^{\AO})
= \< \bu_k^{\AO} , \ones \> \bmu_{\sx} + \nabla \Omega_k(\bv_k^{\AO})
\mydoteq \nabla \Omega_k(\bv_k^{\AO}), $ where the approximate
equality uses that $\< \bu_k^{\AO} , \ones \> \mydoteq 0$ and $\|
\bmu_{\sx} \| \lesssim C$.  By the fixed point equations, $\nu_{k,\su}
= 0$, whence $\frac{\de\,}{\de\bv}
\phi_{k,v}(\bv_k^{\AO};\bG_{k-1}^{\AO}) = \nabla
\Omega_k(\bv_k^{\AO})$.  We conclude the second approximate equality
in the first display of the lemma.

Recall from equations~\eqref{eq:model}
and~\eqref{eq:prop-functional-form} that $y_{1,i} = \indic\{
\eps_{1,i}' = 1\} + \indic\{\eps_{1,i}' = \circ\} \indic\{\eps_{1,i}
\leq \theta_{1,0} + \bX \btheta_1 \}$, and $\by_2 = \theta_{2,0} \ones
+ \bX \btheta_2 + \beps_2$.  Recalling that $\bX = \bA + \ones
\bmu_{\sx}^\top$, $\bw = w(\theta_{1,0}\ones + \bX \btheta_1)$
(equation~\eqref{eq:outcome-fit}), and, by
equation~\eqref{eq:gh-1to4}, that $\bh_1^{\AO} = \bh_1^{\PO} = \bA
\btheta_1$ and $\bh_2^{\AO} = \bh_2^{\PO} = \bA \btheta_2$, we see
that
\begin{equation}
  \begin{gathered}
    \bw = w((\theta_{1,0} + \< \bmu_{\sx},\btheta_1\>) + \bh_1^{\AO}),
    \\ y_{1,i} = \indic\{ \eps_{1,i}' = 1\} + \indic\{\eps_{1,i}' =
    \circ\} \indic\{\eps_{1,i} \leq \theta_{1,0} + \bX \btheta_1 \},
    \quad \by_2 = (\mu_2 + \< \bmu_{\sx},\btheta_2\>)\ones +
    \bh_2^{\AO} + \beps_2.
    \end{gathered}
    \end{equation}
By explicit computation using equation~\eqref{eq:objectives}, we see
that $\bdelta \in
\partial_{\bu}\big(-\phi_k(\bu_k^{\AO};\nu_{k,0},\bv_k^{\AO})\big)$ if
and only if
    \begin{equation}
      (\nu_{k,0} + \< \bmu_{\sx} , \bv_k^{\AO}\> )\ones - \bdelta
      \in
      \partial_{\bu} \ell_k^*\big(n\bu_k^{\AO}; \bw,\by_1,\by_2 \big).
    \end{equation}
    By comparison of the second-to-last display with
    equations~\eqref{eq:SE-penalties} and~\eqref{eq:yw-func-of-Heps},
    $\bdelta' \in
    \partial_{\bu}\phi_{k,u}(\bu_k^{\AO};\bH_{k-1}^{\AO},\beps_1,\beps_2)$
    if and only if
    \begin{equation}
        (\nu_{k,0} + \nu_{k,\sx} )\ones - \bdelta' \in \partial_{\bu}
      \ell_k^*\big(n\bu_k^{\AO}; \bw,\by_1,\by_2 \big).
    \end{equation}
    Because the sub-differential on the right-hand side of the
    previous two displays are the same, we see that we can take
    $\bdelta = (\< \bmu_{\sx} , \bv_k^{\AO}\> - \nu_{k,\sx} )\ones$.
    By~\Cref{lem:hypothesis-k-ao} and because $\<
    \bmu_{\sx},\bv_k^{\SE}\>_{\Ltwo} = \nu_{k,\sx}$ by the fixed point
    equations \eqref{eq:fixpt-general}, we get that $\< \bmu_{\sx} ,
    \bv_k^{\AO}\> - \nu_{k,\sx} \mydoteq 0$.  Because $\| \ones \| /
    \sqrt{n} = 1$, we conclude that $\bdelta / \sqrt{n} \mydoteq
    \bzero$, as desired.
\end{proof}

\begin{lemma}[Curvature of auxiliary objective]
\label{lem:AuxObj-curvature}
For $k = 5,6$, under Hypothesis (k-1), with exponentially high
probability $\bu \mapsto \AuxObj_k(\bu;\nu_{0,k},\bv_k^{\AO})$ is
$cn$-strongly concave and $\bv \mapsto
\AuxObj_k(\bu_k^{\AO};\nu_{0,k},\bv)$ is $c$-strongly convex.
\end{lemma}

\begin{proof}[Proof of~\Cref{lem:AuxObj-curvature}]
    By Assumption A1 $\ell_{\prop}$ is $C$-strongly smooth, whence by
    Fenchel-Legendre duality, $\ell_{\prop}^*$ is $1/C$-strongly
    convex.  By the definition of $\phi_5$ and $\ell_5^*$ in
    equations~\eqref{eq:objectives} and \eqref{eq:ellk*-def},
    $\phi_5(\bu;v_0,\bv)$ is $cn$-strongly concave in $\bu$.
    Moreover, $\ell_6^*$ is $cn$-strongly concave in $\bu$ by
    assumption, using that $w_i$ is bounded below by $c > 0$ due to
    Assumption A1.  By the definition of $\bg(\bu)$ and $\bh(\bv)$ in
    equation~\eqref{eq:gordon-gh-explicit}, we see that
    $-\<\bg(\bu),\bv_k^{\AO}\> + \< \bh(\bv_k^{\AO}),\bu\>$ is concave
    in $\bu$ provided $\<\bxi_g,\bv_k^{\AO}\> \geq 0$.  This occurs
    with exponentially high-probability by~\Cref{cor:AO-bounds}
    (recalling that $\bxi_g = \bg_k^{\AO,\perp}$, see
    equation~\eqref{eq:GHk-ao}).  Thus with exponentially
    high-probability, $\bu \mapsto
    \AuxObj_k(\bu;\nu_{0,k},\bv_k^{\AO})$ is $cn$-strongly concave.

    Likewise, by Assumption A1, $\Omega_k$ is $c$-strong-convex,
    whence by the definition of $\phi_k$ in
    equation~\eqref{eq:objectives}, $\phi_k$ is $c$-strongly convex in
    $\bv$.  By the definition of $\bg(\bu)$ and $\bh(\bv)$ in
    equation~\eqref{eq:gordon-gh-explicit}, we see that
    $-\<\bg(\bu_k^{\AO}),\bv\> + \< \bh(\bv),\bu_k^{\AO}\>$ is convex
    in $\bv$ provided $\<\bxi_h,\bu_k^{\AO}\> \geq 0$.  This occurs
    with exponentially high-probability by~\Cref{cor:AO-bounds},
    recalling that $\bxi_h = \bh_k^{\AO,\perp}$, see
    equation~\eqref{eq:GHk-ao}.  Thus with exponentially
    high-probability, $\bv \mapsto
    \AuxObj_k(\bu_k^{\AO};\nu_{0,k},\bv)$ is $c$-strongly concave.
\end{proof}

\subsection{Local stability of Gordon's objective}
\label{sec:local-stability}

Let $\phi_u : \reals^{2k+1} \rightarrow \reals$ and $\phi_v:
(\reals^p)^{2k}$ be any order-2 pseudo-Lipschitz functions.  Define
the random sets (which depend on $\bU_{k-1}^{\AO}$, $\bH_{k-1}^{\AO}$,
$\beps_2$, $\bV_{k-1}^{\AO}$, and $\bG_{k-1}^{\AO}$)
\begin{equation}
\begin{gathered}
  E_u(\epsilon) \defn \Big\{ \bu \in \reals^n : \Big| \frac{1}{n}
  \sum_{i=1}^n \phi\Big( (nu_{\ell,i}^{\AO,\perp})_{\ell=1}^{k-1},
  nu_i, (h_{\ell,i}^{\AO,\perp})_{\ell=1}^{k-1}, \eps_{2,i} \Big) -
  \E\Big[ \phi\Big( (nu_{\ell,i}^{\SE,\perp})_{\ell=1}^{k-1}, nu_i,
    (h_{\ell,i}^{\SE,\perp})_{\ell=1}^{k-1}, \eps_{2,i}^{\SE} \Big)
    \Big] \Big| < \epsilon \Big\}, \\ E_v(\epsilon) \defn \Big\{ \bv
  \in \reals^p : \Big| \phi_v
  \Big(\bV_{k-1}^{\AO,\perp},\bv,\frac{1}{\sqrt{n}}\bG_{k-1}^{\AO,\perp}\Big)
  - \E
  \Big[\phi\Big(\bV_{k-1}^{\SE,\perp},\bv,\frac{1}{\sqrt{n}}\bG_{k-1}^{\SE,\perp}\Big)\Big]
  \Big| < \epsilon \Big\}.
\end{gathered}
\end{equation}

\begin{lemma}[Local stability]
\label{lem:local-stability}
Define $\ell_k$ as in \Cref{lem:AuxObj-conc}.  For $k \geq 5,6$,
under Hypothesis (k-1), we have for sufficiently large constant $C$
depending only on $\cPmodel$ that
\begin{equation}
  \label{eq:max-min-AO-conc}
  \max_{\|\bu\| \leq C/\sqrt{n}} \; \min_{\substack{|v_0| \leq C
      \\ \|\bv\| \leq C}} \; \AuxObj_k(\bu;v_0,\bv) \mydoteq \ell_k,
\end{equation}
and
\begin{equation}
  \label{eq:min-max-AO-conc}
  \min_{\substack{|v_0| \leq C \\ \|\bv\| \leq C}} \max_{\|\bu\| \leq
    C/\sqrt{n}} \; \AuxObj_k(\bu;v_0,\bv) \mydoteq \min_{\|\bv\| \leq
    C} \max_{\|\bu\| \leq C/\sqrt{n}} \; \AuxObj_k(\bu;\nu_{k,0},\bv)
  \mydoteq \ell_k.
\end{equation}
Further, for sufficiently large $C$ and appropriately chosen $c,c'
> 0$ depending only on $\cPmodel$, and any $\epsilon < c'$
\begin{equation}
  \label{eq:AO-local-stability}
  \max_{\substack{\bu \in E_u^c(\epsilon)\\ \| \bu \| \leq
      C/\sqrt{n}}}\; \min_{\substack{|v_0| \leq C \\ \|\bv\| \leq C}}
  \; \AuxObj_k(\bu;v_0,\bv) \lessdot \ell_k - c\epsilon^2, \qquad
  \min_{\substack{\bv \in E_v^c(\epsilon) \\ |v_0| \leq C \\ \| \bv \|
      \leq C }} \; \max_{\|\bu\| \leq C/\sqrt{n}} \;
  \AuxObj_k(\bu;\nu_{k,0},\bv) \gtrdot \ell_k + c\epsilon^2.
\end{equation}
\end{lemma}

\begin{proof}[Proof of~\Cref{lem:local-stability}]
First, let $\bdelta_u$ be as in~\Cref{lem:approximate-stationarity}.
Then there exists $C > 0$ such that
\begin{equation}
    \begin{aligned}
      &\max_{\|\bu\| \leq C/\sqrt{n}} \; \min_{\substack{|v_0| \leq C
          \\ \|\bv\| \leq C}} \; \AuxObj_k(\bu;v_0,\bv)
      \underset{\emph{(i)}}\lessdot \max_{\|\bu\| \leq C/\sqrt{n}} \;
      \AuxObj_k(\bu;\nu_{k,0},\bv_k^{\AO}) \\ &\qquad
      \underset{\emph{(ii)}}\lessdot \max_{\|\bu\| \leq C/\sqrt{n}}
      \left\{ \AuxObj_k(\bu_k^{\AO};\nu_{k,0},\bv_k^{\AO}) +
      \bdelta_u^\top(\bu - \bu_k^{\AO}) \right\} \\ &\qquad \leq
      \AuxObj_k(\bu_k^{\AO};\nu_{k,0},\bv_k^{\AO}) +
      \frac{C}{\sqrt{n}}\|\bdelta_u\| \underset{\emph{(iii)}}\lessdot
      \ell_k.
    \end{aligned}
\end{equation}
    Inequality $\emph{(i)}$ holds because on the events $|\nu_{0,k}|
    \leq C$ and $\| \bv_k^{\AO} \| < C$, which~\Cref{cor:AO-bounds}
    shows have exponentially high probability, the inequality holds
    exactly (that is, with ``$\lessdot$'' replaced by ``$\leq$'').
    Inequality \emph{(ii)} holds because $\bdelta_u \in
    \partial_{\bu}\big(-\AuxObj_k(\bu_k^{\AO};\nu_{k,0},\bv_k^{\AO})\big)$
    and by~\Cref{lem:AuxObj-convex-concave}, on the event $\< \bxi_g ,
    \bv_k^{\AO} \> \geq 0$, $\AuxObj_k$ is concave in $\bu$, and
    by~\Cref{cor:AO-bounds}, $\< \bxi_g , \bv_k^{\AO} \> \geq 0$ with
    exponentially high probability (recall $\bxi_g =
    \bg_k^{\AO,\perp}$).  Inequality \emph{(iii)} holds
    by~\Cref{lem:AuxObj-conc} and because $\| \bdelta_u \| / \sqrt{n}
    \mydoteq 0$ by~\Cref{lem:approximate-stationarity}.

 To establish the reverse inequality, let $\bdelta_v =
 \frac{\de\,}{\de \bv}\AuxObj_k(\bu_k^{\AO};\nu_{k,0},\bv_k^{\AO})$
 and $\delta_0 = \frac{\de\,}{\de
   v_0}\AuxObj_k(\bu_k^{\AO};\nu_{k,0},\bv_k^{\AO})$.  We follow an
 very similar argument as above
    \begin{equation}
    \begin{aligned}
      &\max_{\|\bu\| \leq C/\sqrt{n}} \; \min_{\substack{|v_0| \leq C
          \\ \|\bv\| \leq C}} \; \AuxObj_k(\bu;v_0,\bv)
      \underset{\emph{(i)}}\gtrdot \min_{\substack{|v_0| \leq C
          \\ \|\bv\| \leq C}} \; \AuxObj_k(\bu_k^{\AO};v_0,\bv)
      \\ &\qquad \underset{\emph{(ii)}}\gtrdot \min_{\substack{|v_0|
          \leq C \\ \|\bv\| \leq C}} \Big\{
      \AuxObj_k(\bu_k^{\AO};\nu_{k,0},\bv_k^{\AO}) +
      \bdelta_v^\top(\bv - \bv_k^{\AO}) + \delta_0(v_0 - \nu_{k,0})
      \Big\} \\ &\qquad\gtrdot
      \AuxObj_k(\bu_k^{\AO};\nu_{k,0},\bv_k^{\AO}) - C\|\bdelta_v\| -
      C|\delta_0| \underset{\emph{(iii)}}\gtrdot \ell_k.
    \end{aligned}
    \end{equation}
    Inequality $\emph{(i)}$ holds because on the event
    $\|\bu_k^{\AO}\| \leq C/\sqrt{n}$, which~\Cref{lem:fixed-pt-bound}
    and~\Cref{cor:AO-bounds} show have exponentially high probability,
    the inequality holds exactly (that is, with ``$\gtrdot$'' replaced
    by ``$\geq$'').  Inequality \emph{(ii)} holds because
    by~\Cref{lem:AuxObj-convex-concave}, on the event $\< \bxi_h ,
    \bu_k^{\AO} \> \geq 0$, $\AuxObj_k$ is convex in $(v_0,\bv)$, and
    by~\Cref{cor:AO-bounds}, $\< \bxi_h , \bu_k^{\AO} \> \geq 0$ with
    exponentially high probability (recall $\bxi_g =
    \bg_k^{\AO,\perp}$).  Inequality \emph{(iii)} holds
    by~\Cref{lem:AuxObj-conc} and because $\|\bdelta_v\| \mydoteq 0$
    and $|\delta_0| \mydoteq 0$
    by~\Cref{lem:approximate-stationarity}.

    Equations \eqref{eq:min-max-AO-conc} follows by a completely
    analogous argument, which we present for completeness.  There
    exists $C$ such that
    \begin{equation}
    \begin{aligned}
        &\min_{\|\bv\| \leq C} \; \max_{\|\bu\| \leq C/\sqrt{n}} \;
      \AuxObj_k(\bu;\nu_{k,0},\bv) \geq \min_{\substack{|v_0| \leq C
          \\ \|\bv\| \leq C}} \; \max_{\|\bu\| \leq C/\sqrt{n}} \;
      \AuxObj_k(\bu;v_0,\bv) \geq \min_{\substack{|v_0| \leq C
          \\ \|\bv\| \leq C}} \; \AuxObj_k(\bu_k^{\AO};v_0,\bv)
      \\ &\qquad\gtrdot \min_{\substack{|v_0| \leq C \\ \|\bv\| \leq
          C}} \Big\{ \AuxObj_k(\bu_k^{\AO};v_0,\bv) + \bdelta_v^\top
      (\bv - \bv_k^{\AO}) + \delta_0(v_0 - \nu_{k,0}) \Big\}
      \\ &\qquad\gtrdot \AuxObj_k(\bu_k^{\AO};\nu_{k,0},\bv_k^{\AO}) -
      C\|\bdelta_v\| - C|\delta_0| \gtrdot \ell_k,
    \end{aligned}
    \end{equation}
    where the inequalities and approximate inequalities hold by similar
    arguments as before.  Similarly,
    \begin{equation}
    \begin{aligned}
        &\min_{\substack{|v_0| \leq C \\ \|\bv\| \leq C}} \;
      \max_{\|\bu\| \leq C/\sqrt{n}} \; \AuxObj_k(\bu;v_0,\bv) \leq
      \min_{\|\bv\| \leq C} \; \max_{\|\bu\| \leq C/\sqrt{n}} \;
      \AuxObj_k(\bu;\nu_{k,0},\bv) \leq \max_{\|\bu\| \leq C/\sqrt{n}}
      \; \AuxObj_k(\bu;\nu_{k,0},\bv_k^{\AO}) \\ &\qquad\lessdot
      \max_{\|\bu\| \leq C/\sqrt{n}} \left\{
      \AuxObj_k(\bu_k^{\AO};\nu_{k,0},\bv_k^{\AO}) + \bdelta_u^\top
      (\bu - \bu_k^{\AO}) \right\} \leq
      \AuxObj_k(\bu_k^{\AO};\nu_{k,0},\bv_k^{\AO}) +
      \frac{C}{\sqrt{n}}\| \bdelta_u \| \lessdot \ell_k,
    \end{aligned}
    \end{equation}
    where the inequalities and approximate inequalities hold by similar
    arguments as before. Combining the previous two displays gives
    equation~\eqref{eq:min-max-AO-conc}.

 Now we prove equation~\eqref{eq:AO-local-stability}.
 By~\Cref{lem:hypothesis-k-ao}, for sufficiently small $c' > 0$ and
 any fixed $\epsilon < c'$, $\bu_k^{\AO} \in E_u(\epsilon/2)$ with
 probability at least $1 - Ce^{-cn\epsilon^r}$.  When $\| \bu_k^{\AO}
 \| \leq C/\sqrt{n}$, which occurs with exponentially high
 probability, we have for every $\bu \in E_u^c(\epsilon)$ that $\| \bu
 - \bu_k^{\AO} \| \geq \epsilon/(2\sqrt{n})$.  Thus,
    \begin{equation}
    \begin{aligned}
        \max_{\substack{\bu \in E_u^c(\epsilon)\\ \| \bu \| \leq
            C/\sqrt{n}}}\; \min_{\substack{|v_0| \leq C \\ \|\bv\|
            \leq C}} \; \AuxObj_k&(\bu;v_0,\bv) \leq
        \max_{\substack{\bu \in E_u^c(\epsilon)\\ \|\bu\|\leq
            C/\sqrt{n}}} \AuxObj_k(\bu;\nu_{k,0},\bv_k^{\AO})
        \\ &\underset{\emph{(i)}}\lessdot \max_{\substack{\bu \in
            E_u^c(\epsilon)\\ \| \bu \| \leq C/\sqrt{n}}} \Big\{
        \AuxObj_k(\bu_k^{\AO};\nu_{k,0},\bv_k^{\AO}) +
        \|\bdelta_u\|\,\|\bu - \bu_k^{\AO}\| - \frac{cn}2 \| \bu -
        \bu_k^{\AO} \|^2 \Big\} \\ &\underset{\emph{(ii)}}\lessdot
        \max_{\epsilon/2 \leq x \leq C} \Big\{ \ell_k^* +
        \frac{c\epsilon}{8}\,x - \frac{c}2 x^2 \Big\} = \ell_k^* -
        \frac{3c\epsilon^2}{16}.
    \end{aligned}
    \end{equation}
    Inequality $\emph{(i)}$ holds because on the event
    $\|\bu_k^{\AO}\| \leq C/\sqrt{n}$, which~\Cref{cor:AO-bounds}
    shows has exponentially high probability, the inequality holds
    exactly (that is, with ``$\gtrdot$'' replaced by ``$\geq$''), and
    because $\bu \mapsto \AuxObj_k(\bu;\nu_{k,0},\bv_k^{\AO})$ is
    $cn$-strongly concave with exponentially high-probability
    by~\Cref{lem:AuxObj-curvature}.  Inequality \emph{(ii)} holds
    because by~\Cref{lem:approximate-stationarity} and the argument
    preceding the display, with exponentially high probability, $\|
    \bdelta_u \| /\sqrt{n} \leq c\epsilon/8$ and $\| \bu - \bu_k^{\AO}
    \| \geq \epsilon/(2 \sqrt{n})$ for all $\bu_k^{\AO} \in
    E_u(\epsilon/2)$.  The first bound in
    equation~\eqref{eq:AO-local-stability} follows.

    The second bound in equation~\eqref{eq:AO-local-stability} holds
    by a similar argument.  Indeed, by Hypothesis (k-1), for
    sufficiently small $c' > 0$ and any $\epsilon < c'$, $\bv_k^{\AO}
    \in E_v(\epsilon/2)$ with exponentially high probability.  On this
    event, for every $\bv \in E_v^c(\epsilon)$ we have $\| \bv -
    \bv_k^{\AO} \| \geq \epsilon/2$.  We thus have
    \begin{equation}
    \begin{aligned}
        \min_{\substack{\bv \in E_v^c(\epsilon)\\|v_0| \leq C \\ \|
            \bv \| \leq C}}\; \max_{\|\bu \| \leq C / \sqrt{n}}\; \;
        &\AuxObj_k(\bu;v_0,\bv) \geq \min_{\substack{\bv \in
            E_v^c(\epsilon)\\|v_0| \leq C \\ \| \bv \| \leq C}}\;
        \AuxObj_k(\bu;v_0,\bv_k^{\AO})
        \\ &\underset{\emph{(i)}}\gtrdot \min_{\substack{\bv \in
            E_v^c(\epsilon)\\ |v_0| \leq C \\ \| \bv \| \leq C }}\;
        \Big\{ \AuxObj_k(\bu_k^{\AO};\nu_{k,0},\bv_k^{\AO}) -
        \|\bdelta_v\|\,\|\bv - \bv_k^{\AO}\| - |\delta_0|\,|v_0 -
        \nu_{0,k}| + \frac{c}2 \| \bv - \bv_k^{\AO} \|^2 \Big\}
        \\ &\underset{\emph{(ii)}}\gtrdot \max_{\epsilon/2 \leq x \leq
          C} \Big\{ \ell_k^* - \frac{c\epsilon}{8}\,x -
        \frac{c\epsilon^2}{16} + \frac{c}2 x^2 \Big\} = \ell_k^* +
        \frac{c\epsilon^2}{8}.
    \end{aligned}
    \end{equation}
    Inequality $\emph{(i)}$ holds because on the event
    $\|\bv_k^{\AO}\| \leq C$, which~\Cref{cor:AO-bounds} shows has
    exponentially high probability, the inequality holds exactly (that
    is, with ``$\gtrdot$'' replaced by ``$\geq$''), and because $\bv
    \mapsto \AuxObj_k(\bu_k^{\AO};\nu_{k,0},\bv)$ is $c$-strongly
    convex with exponentially high-probability
    by~\Cref{lem:AuxObj-curvature}.  Inequality \emph{(ii)} holds
    because by~\Cref{lem:approximate-stationarity} and the argument
    preceding the display, with exponentially high probability, $\|
    \bdelta_v \| \leq c\epsilon/8$, $|\delta_0| \leq
    c\epsilon^2/(16C)$, and $\| \bv - \bv_k^{\AO} \| \geq \epsilon/2$
    for all $\bv_k^{\AO} \in E_v(\epsilon/2)$.  The first bound in
    equation~\eqref{eq:AO-local-stability} follows.
\end{proof}

\subsection{Proof of~\Cref{lem:inductive-step}}
\label{sec:inductive-step-proof}

Recall that $\PriObj_k(\bu;v_0,\bv) \defn \bu^\top \bA \bv +
\phi_k(\bu;v_0,\bv)$.


\subsubsection{Crude bounds on primary optimization saddle point}

\begin{lemma}
\label{lem:PriObj-properties}
For $k = 5,6$, the primary objective has the following properties.
\begin{enumerate}[(a)]
\item For any fixed $v_0,\bv$, the function $\bu \mapsto
  \PriObj_k(\bu;v_0,\bv)$ is $cn$-strongly concave in $\bu$.  Thus
  also $\bu \mapsto \min_{(v_0,\bv) \in S}\PriObj_k(\bu;v_0,\bv)$ is
  $cn$-strongly concave in $\bu$ for any set $S$.

\item For any constant $M$ depending on $\cPmodel$, with exponentially
  high-probability the function $(v_0,\bv) \mapsto \max_{\bu \in
    \reals^n}\PriObj_k(\bu;v_0,\bv)$ is $c$-strongly convex on
  $\max\{|v_0|,\| \bv \|\} \leq M$, where $c$ depends only on
  $\cPmodel$ and $M$.

\item With exponentially high probability, $\| \bu_k^{\PO} \| \leq C /
  \sqrt{n}$, $|v_{k,0}^{\PO}| \leq C$, $\| \bv_k^{\PO} \| \leq C$.  On
  this event,
  \begin{equation}
          \PriObj_k(\bu_k^{\PO};v_{k,0}^{\PO},\bv_k^{\PO}) =
          \min_{\substack{v_0 \in \reals\\\bv \in \reals^p}} \;
          \max_{\bu \in \reals^n} \PriObj_k(\bu;v_0,\bv) =
          \min_{\substack{|v_0| \leq C \\ \|\bv\| \leq C}} \;
          \max_{\|\bu\| \leq C/\sqrt{n}} \PriObj_k(\bu;v_0,\bv),
  \end{equation}
  and in both places the order of minimization and maximization may be
  exchanged.

\item For any $C > 0$ depending only on $\cPmodel$, there exists $C' >
  0$ depending only on $\cPmodel$ and $C$ such that with exponentially
  high probability
  \begin{equation}
    \begin{aligned}
      \text{for any $|v_0| \leq C$,}\quad &\min_{\substack{\bv
          \in \reals^p}} \; \max_{\bu \in \reals^n}
      \PriObj_k(\bu;v_0,\bv) = \min_{\substack{\|\bv\| \leq C'}}
      \; \max_{\|\bu\| \leq C'/\sqrt{n}} \PriObj_k(\bu;v_0,\bv).
    \end{aligned}
  \end{equation} 

\end{enumerate}
\end{lemma}

\begin{proof}[Proof of~\Cref{lem:PriObj-properties}]
We prove each item, one at a time.
\begin{enumerate}[(a)]
\item We showed that $\phi_k$ is $cn$-strongly concave in $\bu$ in the
  proof of~\Cref{lem:AuxObj-curvature}.  The function $\bu \mapsto
  \bu^\top \bA \bv$ is linear, whence the result follows.
\item By~\cite[Corollary 5.35]{vershynin_2012}, there exists $M > 1$
  such that $\| \bA \|_{\op} / \sqrt{n} \leq M$ with exponentially
  high-probability.  Throughout the proof of item 2, the value of $M$
  stays constant but $C',c,c' > 0$ may change at each appearance and
  depend on $M$ and $\cPmodel$.  We will show that for any
  $|v_0|,\|\bv\| \leq C'$ and $\delta_0^2 + \| \bdelta \|^2 = 1$, the
  function $t \mapsto \max_{\bu \in \reals^n}\PriObj_k(\bu;v_0 + t
  \delta_0,\bv + t\bdelta)$ is $c$-strongly convex for $t \in [0,1]$.

  Note that $\max_{\bu \in \reals^n}\PriObj_k(\bu;v_0,\bv)$ are
  exactly the objectives given in equations~\eqref{eq:propensity-fit}
  (for $k = 5$) and~\eqref{eq:outcome-fit} (for $k = 6$).  With
  exponentially high-probability, there is a set $S_1 \subset [n]$,
  $|S_1| \geq c'n$ such that for all $i \in S_1$, $\eta \mapsto
  \action_i w(\< \bx_i , \btheta_{\prop} \>)(y_i-\eta)^2/2$ is
  $c$-strongly convex.  Indeed, because the weight $w$ is
  lower-bounded by A1, the former holds whenever $\action_i = 1$,
  which occurs in at least $c'n$ coordinates with exponentially high
  probability because $\action_i \iid \Bernoulli(\E[\pi(\mu_{\prop} +
    \< \bx_i,\btheta_{\prop}\>)])$ and $\E[\pi(\mu_{\prop} + \<
    \bx_i,\btheta_{\prop}\>)] > c$ by Assumption A1.
  
  Now, to show the function $t \mapsto \max_{\bu \in
    \reals^n}\PriObj_k(\bu;v_0 + t \delta_0,\bv + t\bdelta)$ is
  $c$-strongly convex for $t \in [0,1]$, consider first the case
  $|\delta_0| \geq 2M \| \bdelta \|$, or equivalently, $|\delta_0|^2
  \geq 4M^2/(1 + 4M^2)$.  On the event $\| \bA \|_{\op} / \sqrt{n}
  \leq M$, we have $C' \geq \|\delta_0 \ones + \bA \bdelta\|/\sqrt{n}
  \geq |\delta_0| - \| \bA \bdelta \| / \sqrt{n} \geq (1 -
  1/(2M))|\delta_0| \geq c > 0$.  Using Markov's inequality across
  coordinates, there is a set $S_2 \subset [n]$, $|S_2| \geq
  (1-c'/2)n$ (using the same $c'$ as above) such that for all $i \in
  S_2$ both $|\delta_0 + \< \bx_i , \bdelta \>| \geq c > 0$ and $|v_0
  + t\delta_0 + \< \bx_i , \bv + t \bdelta \> | \leq C'$ for all $t
  \in [0,1]$.  Note that $|S_1 \cap S_2| \geq c'/2$, whence $t \mapsto
  \frac{1}{n}\sum_{i \in S_1 \cap S_2}\action_i w(\<
  \bx_i,\btheta_{\prop}\>)(y_i - v_0 - t\delta_0 - \< \bx_i , \bv + t
  \bdelta \>)^2$ (in the case $k = 5$) and $t \mapsto
  \frac{1}{n}\sum_{i \in S_1 \cap S_2}\ell_{\prop}(v_0 + t\delta_0 +
  \< \bx_i , \bv + t \bdelta \> ; \action_i)$ (in the case $k =6$) is
  $c$-strongly convex for $t \in [0,1]$.  Because the remaining terms
  in equations~\eqref{eq:propensity-fit} and \eqref{eq:outcome-fit} are are
  convex, we conclude $t \mapsto \max_{\bu \in
    \reals^n}\PriObj_k(\bu;v_0 + t \delta_0,\bv + t\bdelta)$ is
  $c$-strongly convex for $t \in [0,1]$.

  Consider, alternatively, the case that $|\delta_0| < 2M \| \bdelta
  \|$, or equivalently, $\|\bdelta\|^2 \geq 1/(1 + 4M^2)$.  Then,
  because $\Omega_{\prop}$ is $c$-strongly convex by Assumption A1, $t
  \mapsto \Omega_\prop(\bv + t \bdelta)$ is $c$-strongly convex.
  Thus, $t \mapsto \max_{\bu \in \reals^n}\PriObj_k(\bu;v_0 + t
  \delta_0,\bv + t\bdelta)$ is $c$-strongly convex for $t \in [0,1]$
  in this case as well.
\item Let $\bv^*$ be the minimizer of $\Omega_k$, which has bounded
  $\ell_2$-norm by Assumption A1.  Moreover, because $\eta \mapsto
  \ell_{\prop}(\eta;\action_i)$ and $\eta\mapsto \action_i w(\< \bx_i
  , \btheta_{\prop}\>)(y_i - \eta)^2$ have $C$-Lipschitz gradient
  which is bounded by $C(1 + |y_i|)$ at $\eta = 0$, we conclude that
  $(v_0,\bv) \mapsto \max_{\bu \in \reals^n}\PriObj_k(\bu;v_0,\bv)$
  has gradient bounded by $C$ with exponentially high probability at
  $v_0 = 0, \bv = \bv^*$.  The bounds on $|v_{k,0}^{\PO}|$ and
  $\|\bv_k^{\PO}\|$ then hold by item 2.  The bounds on $\| \bu_k
  ^{\PO} \|$ then hold by equations~\eqref{eq:uPO-identity} and the
  Lipschitz continuity of the gradient of loss in Assumption A1.  Then
  the display holds by \cite[Lemma 36.2]{rockafellar-1970a}.

\item This holds by the same argument as in the previous item, where
  we do not optimize over $v_0$ in equations~\eqref{eq:propensity-fit} and
  \eqref{eq:outcome-fit}.
\end{enumerate}
\end{proof}


\subsubsection{Refined bounds on primary optimization saddle point}

By~\Cref{lem:PriObj-properties}, the sequential Gordon inequality
(\Cref{lem:seq-gordon}), and~\Cref{lem:local-stability}, for
sufficiently large $C > 0$ depending on $\cPmodel$
\begin{equation}
\label{eq:PO-saddle-value-conc}
    \PriObj_k(\bu_k^{\PO};v_{k,0}^{\PO},\bv_k^{\PO}) =
    \min_{\substack{v_0 \in \reals\\\bv \in \reals^p}} \; \max_{\bu
      \in \reals^n} \PriObj_k(\bu;v_0,\bv) = \min_{\substack{|v_0|
        \leq C \\ \|\bv\| \leq C}} \; \max_{\|\bu\| \leq C/\sqrt{n}}
    \PriObj_k(\bu;v_0,\bv) \mydoteq \ell_k,
\end{equation}
and
\begin{equation}
    \min_{\bv \in \reals^p} \; \max_{\bu \in \reals^n}
    \PriObj_k(\bu;\nu_{0,k},\bv) = \min_{\|\bv\| \leq C} \;
    \max_{\|\bu\| \leq C/\sqrt{n}} \PriObj_k(\bu;\nu_{0,k},\bv) \mydoteq
    \ell_k,
\end{equation}
where we have used item (c) of~\Cref{lem:PriObj-properties} in the
first display, and item (d) of~\Cref{lem:PriObj-properties} in the
second display.  In particular, for any $\epsilon < c'$, with
probability at least $1 - Ce^{-cn\epsilon^r}$, $\min_{\bv \in
  \reals^p} \max_{\bu \in \reals^n} \PriObj_k(\bu; \nu_{0,k},\bv)$ is
within $\epsilon$ of the minimum of the function $v_0 \mapsto
\min_{\bv \in \reals^p} \max_{\bu \in \reals^n}
\PriObj_k(\bu;v_0,\bv)$ over $[-C,C]$.
By~\Cref{lem:PriObj-properties}(b), the function $(v_0, \bv) \mapsto
\max_{\bu \in \reals^n}\PriObj_k(\bu;v_0,\bv)$ is $c$-strongly convex
over the set defined by the constraints $|v_0|\leq C$ and $\| \bv \|
\leq C$ with exponentially high probability.  When this occurs, the
interval $I \subset [-C,C]$ on which $\min_{\bv \in \reals^p}
\allowbreak \max_{\bu \in \reals^n} \allowbreak \PriObj_k(\bu;
\allowbreak v_0,\bv)$ is within $\epsilon$ of its minimum has length
at most $2 \sqrt{2\epsilon/c}$.  Thus, with probability at least $1 -
Ce^{-c n \epsilon^r}$, $|v_{k,0}^{\PO} - \nu_{0,k}| \leq
2\sqrt{2\epsilon/c}$.  That is, $v_{k,0}^{\PO} \mydoteq \nu_{0,k}$, as
desired.

Next, taking $C$ large enough so that $|v_{k,0}^{\PO} | \leq C$ and
$\| \bv_k^{\PO} \| \leq C/2$ (which is possible
by~\Cref{lem:PriObj-properties}, item 3) and such
that~\cref{eq:AO-local-stability} holds, we have
\begin{equation}
  \min_{\substack{\bv \in E_v^c(\epsilon) \\ |v_0| \leq C \\ \| \bv \|
      \leq C }} \; \max_{ \bu \in \reals^n } \PriObj_k(\bu;v_0,\bv)
  \geq \min_{\substack{\bv \in E_v^c(\epsilon) \\ |v_0| \leq C \\ \|
      \bv \| \leq C }} \; \max_{\|\bu\| \leq C'/\sqrt{n}} \;
  \PriObj_k(\bu;v_0,\bv) \gtrdot \ell_k + c\epsilon^2.
\end{equation}
Likewise, taking $C$ large enough so that $|\bu_k^{\PO} | \leq
C/(2\sqrt{n})$ (which is possible by~\Cref{lem:PriObj-properties},
item 3) and such that~\cref{eq:AO-local-stability} holds, we have
\begin{equation}
    \max_{ \substack{ \bu \in E_u^c(\epsilon) \\ \| \bu \| \leq
        C/\sqrt{n}} }\; \min_{\substack{v_0 \in \reals \\ \bv \in
        \reals^p }} \; \PriObj_k(\bu;v_0,\bv) \leq \max_{ \substack{
        \bu \in E_u^c(\epsilon) \\ \| \bu \| \leq C/\sqrt{n}} }\;
    \min_{\substack{|v_0| \leq C' \\ \|\bv\| \leq C' }} \;
    \PriObj_k(\bu;v_0,\bv) \lessdot \ell_k - c\epsilon^2,
\end{equation}
where recall that we may exchange the order of minimization and
maximization in equation~\eqref{eq:PO-saddle-value-conc}.  Combined with
the relation~\eqref{eq:PO-saddle-value-conc}, we find that probability
at least $1 - Ce^{-cn\epsilon^r}$, $\bu_k^{\PO} \in E_u(\epsilon)$ and
$\bv_k^{\PO} \in E_v(\epsilon)$.  That is,
\begin{equation}
\label{eq:conc-without-gh}
\begin{gathered}
    \frac{1}{n} \sum_{i=1}^n \phi_u\Big(
    (nu_{\ell,i}^{\PO,\perp})_{\ell=1}^{k-1}, nu_{k,i}^{\PO},
    (h_{\ell,i}^{\PO,\perp})_{\ell=1}^{k-1}, \eps_{2,i} \Big) \mydoteq
    \E\Big[ \E\Big[ \phi_u\Big(
        (nu_{\ell,i}^{\SE,\perp})_{\ell=1}^{k-1}, nu_{k,i}^{\SE},
        (h_{\ell,i}^{\SE,\perp})_{\ell=1}^{k-1}, \eps_{2,i}^{\SE}
        \Big) \Big],
      \\ \phi_v\Big(\bV_{k-1}^{\PO,\perp},\bv_k^{\PO},\frac{1}{\sqrt{n}}\bG_{k-1}^{\PO,\perp}\Big)
      \mydoteq
      \E\Big[\phi_v\Big(\bV_{k-1}^{\SE,\perp},\bv_k^{\SE},\frac{1}{\sqrt{n}}\bG_{k-1}^{\SE,\perp}\Big)\Big].
\end{gathered}
\end{equation}
By the KKT conditions for the problem~\eqref{eq:min-max}, the
definitions of $\phi_k$ in equation~

\eqref{eq:objectives}, and the definition of $\bg_k^{\PO}$ (see
equation~\eqref{eq:loo-noise-def}) we have $\< \bu_k^{\PO},\ones\> = 0$ and
\begin{equation}
    \bg_k^{\PO} = \sum_{\ell=1}^k \zeta_{k\ell}^v \bv_\ell^{\PO} -
    \bA^\top \bu_k^{\PO} = \sum_{\ell=1}^k \zeta_{k\ell}^v
    \bv_\ell^{\PO} + \< \bu_k^{\PO},\ones\> \bmu_{\sx} + \nabla
    \Omega_k(\bv_k^{\PO}) = \sum_{\ell=1}^k \zeta_{k\ell}^v
    \bv_\ell^{\PO} + \nabla \Omega_k(\bv_k^{\PO}).
\end{equation}
Because $\nabla \Omega_k$ is $C$-Lipschitz by Assumption A1, the
second line of equation~\eqref{eq:conc-without-gh} holds also if we allow
$\phi_v$ to also be a function of $\bg_k^{\PO}$.

Now we show that, in the case $k = 5$, the first line
of~\cref{eq:conc-without-gh} holds also if we allow $\phi_u$ to also
be a function of $h_{5,i}^{\PO}$.
Observe that because $\ell_\prop$ is $c$-strongly convex and twice differentiable,
$\nabla_{\bu} \ell_5^*(\bu_k^{\PO};\bw,\by_1,\by_2)$ exists and is $C$-Lipschitz (see its definition in equation~\eqref{eq:ellk*-def}).
Thus, using the KKT conditions for equation~\eqref{eq:min-max} and the definition of $\bh_5^{\PO}$ (equation~\eqref{eq:loo-noise-def}),
we have
\begin{equation}
    \bh_5^{\PO}
        =    
        \sum_{\ell=1}^5 \zeta_{5\ell}^u \bu_\ell^{\PO} + \bA \bv_5^{\PO}
        =
        \sum_{\ell=1}^5 \zeta_{5\ell}^u \bu_\ell^{\PO}
        -
        \ones(v_{0,5}^{\PO} + \< \bmu_{\sx} , \bv_5^{\PO} \>)
        +
        \nabla_{\bu}\ell_5^*(\bu_5^{\PO};\bw,\by_1,\by_2),
\end{equation}
whence equation~\eqref{eq:conc-without-gh} holds if we allow $\phi_u$ to also be a function of $h_{5,i}^{\PO}$.

equation~\eqref{eq:se-conc} holds because, for $k = 5,6$,
$\bv_k^{\PO,\perp}$, $\bg_k^{\PO,\perp}$ are $C$-Lipschitz function of
$\bV_{k-1}^{\PO}$, $\bG_{k-1}^{\PO}$, and $\bv_k^{\PO}$; and
equation~\eqref{eq:se-conc-u} holds because $nu_k^{\PO,\perp}$ and (in the
case $k = 5$) $h_{\ell,k}^{\PO,\perp}$ are $C$-Lipschitz functions of
$(n u_{\ell,i}^{\PO,\perp})_{\ell=1}^{k-1}$,
$(h_{\ell,i}^{\PO,\perp})_{\ell=1}^{k-1}$, and $nu_{k,i}^{\PO}$
(see~\cref{eq:perp-from-orig}).


\section{Analysis of fixed point parameters}
\label{sec:fixed-point-parameter-analysis}

This section is devoted to the proofs
of~\Cref{lem:fixed-pt-soln,lem:fixed-pt-bound},
in~\Cref{sec:fix-pt-exist-proof,sec:fix-pt-bound-proof}, respectively.


\subsection{Proof of~\Cref{lem:fixed-pt-soln}}
\label{sec:fix-pt-exist-proof}

Our proof proceeds via induction on $k$, based on the following
induction hypothesis:
\begin{description}
 \item[Hypothesis \FP{k}] There exists a unique solution $\bK_{g,k}$,
   $\bK_{h,k}$, $\bZ_{u,k}$, $\bZ_{v,k}$, $\{\nu_{\ell,0}\}_{\ell \leq
   k}$, $\{\nu_{\ell,\sx}\}_{\ell \leq k}$ to the equations
    \begin{equation}
    \label{eq:fixpt-inductive}
    \tag{SE-fixpt-$k$}
    \begin{gathered}
        \bK_{g,k} = \llangle \bU_k^{\SE} \rrangle_{\Ltwo}, \qquad
        \bK_{h,k} = \llangle \bV_k^{\SE} \rrangle_{\Ltwo},
        \\ \bK_{h,k} \bZ_{v,k}^\top = \llangle \bH_k^{\SE} ,
        \bU_k^{\SE} \rrangle_{\Ltwo}, \qquad \bK_{g,k} \bZ_{u,k}^\top
        = \llangle \bG_k^{\SE} , \bV_k^{\SE} \rrangle_{\Ltwo},
        \\
        \text{if $k \geq 5$,} \quad \nu_{\ell,\sx} =
        \<\bmu_{\sx},\bv_\ell^{\SE}\>_{\Ltwo}, \qquad \nu_{\ell,\su} =
        \<\ones,\bu_\ell^{\SE}\>_{\Ltwo}, \qquad
        \<\ones,\bu_\ell^{\SE}\>_{\Ltwo} = 0, \quad \text{for $5 \leq
          \ell \leq k$.}
    \end{gathered}
    \end{equation}
 where $\bZ_{v,k}$ is lower-triangular and innovation-compatible with
 $\bK_{h,k}$ and $\bZ_{u,k}$ is lower-triangular and innovation
 compatible with $\bK_{g,k}$.
\end{description}

\noindent Equation~\eqref{eq:fixpt-inductive} is a subset of the fixed
point equations~\eqref{eq:fixpt-general}, and involves a subset of the
fixed point parameters.  Note that~\Cref{lem:fixed-pt-soln} is
equivalent to Hypothesis \FP{6}.  Our strategy is to first check
Hypotheis \FP{4} directly, and then prove Hypothesis \FP{6} by
induction.

\subsubsection{Base case: $\mathbf{k = 4}$}
\label{sec:fixpt-existence-base-case}

Using equations~\eqref{eq:SE-opt} and~\eqref{eq:SE-penalties}, we find
that the equalities
\begin{align*}
  \bu_1^{\SE} = \bu_2^{\SE} = \bzero, \quad \bu_3^{\SE} = -\ones/n,
  \quad \bu_4^{\SE} = -\by_1^{\SE}/n, \quad \mbox{and} \\ \bv_1^{\SE}
  = \btheta_1, \quad \bv_2^{\SE} = \btheta_2, \quad \bv_3^{\SE} = \bzero,
 \quad \bv_4^{\SE} = \bzero
\end{align*}
all hold regardless of the choice of $\bK_g$, $\bK_h$, $\bZ_u$, and
$\bZ_v$.  Thus, the first line of equation~\eqref{eq:fixpt-inductive}
takes the form
\begin{equation}
  \bK_{g,4} =
        \begin{pmatrix}
            \bzero_{2 \times 2} & \bzero_{2 \times 2} \\ \bzero_{2
              \times 2} &
                \begin{matrix}
                    1/n & \barpi/n \\ \barpi/n & \barpi/n
                \end{matrix}
        \end{pmatrix},
    \qquad \bK_{h,4} =
        \begin{pmatrix}
            \llangle \bTheta \rrangle & \bzero_{2 \times 2}
            \\ \bzero_{2 \times 2} & \bzero_{2 \times 2}
        \end{pmatrix},
\end{equation}
where $\bTheta \in \reals^{p \times 2}$ is the matrix with columns
$\btheta_1$ and $\btheta_2$.  Here, in computing the entries of
$\bK_{g,k}$, we have used that $\P(y_{1,i}^{\SE} = 1|h_{1,i}^{\SE}) =
\pi(\theta_{1,0} + \< \bmu_{\sx},\btheta_1\> + h_{1,i}^{\SE})$ and
that $h_{1,i}^{\SE} \sim \normal(0,\|\btheta_1\|^2)$.  Thus, the first
line of the display~\eqref{eq:fixpt-inductive} uniquely determines
$\bK_{g,k}$ and $\bK_{h,k}$.  Moreover, this implies that $\bg_1^{\SE}
= \bg_2^{\SE} = \bzero$ and $\bh_3^{\SE} = \bh_4^{\SE} = 0$.  Then
\begin{equation}
    \llangle \bH_4^\SE,\bU_4^\SE\rrangle_{\Ltwo} = \bK_{h,4}
        \begin{pmatrix}
                \begin{matrix}
                    \vert \\ \bzero_{4\times3} \\ \vert
                \end{matrix}
                & 
                \begin{matrix} -\barpi \alpha_1 \\ 0 \\ 0 \\ 0\end{matrix} 
        \end{pmatrix},
    \qquad \llangle \bG_4^{\SE} , \bH_4^{\SE} \rrangle_{\Ltwo} =
    \bzero_{4\times 4}.
\end{equation}
Using~\Cref{lem:cholesky} and because $\barpi \neq 1$, we have the
equivalence $\id_{\bK_{g,4}} = \diag(0,0,1,1)$ and
\begin{align*}
  \id_{\bK_{h,4}} =
\begin{pmatrix}
  \indic_{\btheta_1 \neq \bzero} & 0
  \\ \frac{\<\btheta_2,\btheta_1\>}{\|\btheta_1\|^2}\indic_{\btheta_1
    \neq \bzero,\btheta_2 \propto \btheta_1} & \indic_{\btheta_2 \not
    \propto \btheta_1}
\end{pmatrix}.
\end{align*}
Multiplying the left and right-hand sides of the second line of the
display~\eqref{eq:fixpt-inductive} by the matrices
$\bK_{h,4}^\ddagger$ and $\bK_{g,4}^\ddagger$, respectively, gives
\begin{equation}
    \id_{\bK_{h,4}}^\top\bZ_{v,4}^\top =
        \begin{pmatrix}
            \begin{matrix}
                \vert \\ \bzero_{4\times3} \\ \vert
            \end{matrix}
            &
            \begin{matrix} -\barpi \alpha_1 \indic_{\btheta_1 \neq \bzero} \\ 0 \\ 0 \\ 0\end{matrix} 
        \end{pmatrix},
    \qquad \id_{\bK_{g,4}}^\top \bZ_{u,4}^\top = \bzero_{4\times4}.
\end{equation}
By innovation compatibility, $\id_{\bK_{h,4}}^\top \bZ_{v,4}^\top =
\bZ_{v,4}^\top$ and $\id_{\bK_{g,4}}^\top \bZ_{u,4}^\top =
\bZ_{u,4}^\top$.

The proof of the base case is complete.

\subsubsection{Induction step: Hilbert space problem}

Assume that Hypothesis \FP{k-1} has been established, and let
$\bK_{g,k-1}$, $\bK_{h,k-1}$, $\bZ_{v,k-1}$, $\bZ_{u,k-1}$,
$\{\nu_{\ell,\sx}\}_{5 \leq \ell \leq k-1}$, $\{\nu_{\ell,0}\}_{5 \leq
  \ell \leq k-1}$ be the unique solution to
equation~\eqref{eq:fixpt-inductive} under the appropriate
innovation-compatibility constraints.  Because the fixed point
equations \eqref{eq:fixpt-inductive} at index $k-1$ are a subset of
those at index $k$, Hypothesis \FP{k-1} implies that the fixed point
equations \eqref{eq:fixpt-inductive} at index $k$ uniquely determines
$\bK_{g,k-1}$, $\bK_{h,k-1}$, $\bZ_{v,k-1}$, $\bZ_{u,k-1}$,
$\{\nu_{\ell,\sx}\}_{5 \leq \ell \leq k-1}$, $\{\nu_{\ell,0}\}_{5 \leq
  \ell \leq k-1}$.  It is our task to show that the fixed point
equations \eqref{eq:fixpt-inductive} at index $k$ also uniquely
determine the final row of $\bK_{g,k}$, $\bK_{h,k}$, $\bZ_{u,k}$,
$\bZ_{v,k}$, and $\nu_{k,0}$ and $\nu_{k,\sx}$.  Our strategy for
doing so is to establish a correspondence between the solutions to
equation~\eqref{eq:fixpt-inductive} and the saddle points of a certain
convex-concave saddle point problem on an infinite dimensional Hilbert
space.  Via this connection, we can reduce the proof of Hypothesis
\FP{k} to a proof of the existence and uniqueness of such saddle
points, for which we can draw on techniques from convex analysis.  The
purpose of this subsection is to introduce the infinite dimensional
saddle point problem.

The saddle point objective that we define closely resembles the saddle
point objective~\eqref{eq:def-AuxObj} in the proof
of~\Cref{lem:inductive-step}.  Whereas the proof
of~\Cref{lem:inductive-step} involves a random objective $\AuxObj_k :
\reals^n \times \reals \times \reals^p$, the proof of the induction
step involves a deterministic objective $\AuxObj_k^{\Ltwo} : L_{\su}^2
\times \reals \times L_{\sv}^2 \rightarrow \reals$, where $L_{\su}^2$
and $L_{\sv}^2$ are infinite-dimensional Hilbert spaces of random $n$-
and $p$-dimensional vectors, respectively.

In particular, consider a probability space $P_{\su}$ containing the
random vectors $(\bh_1^{\SE},\ldots,\bh_{k-1}^{\SE}) \sim
\normal(\bzero,\bK_{h,k-1} \otimes \id_n)$, the independent noise
random vectors $(\beps_1^{\SE},\beps_2^{\SE})$, the first $k-1$
solutions $(\bu_1^{\SE},\ldots,\bu_{k-1}^{\SE})$ to~\cref{eq:SE-opt}
(defined with parameters solving to the fixed point equations up to
iteration $k-1$), and auxiliary Gaussian noise $\bxi_h \sim
\normal(\bzero,\id_n)$ independent of everything else.  Let
$L_{\su}^2$ be the space of square-integrable random $n$-dimensional
vectors $\bu$ defined on $P_{\su}$.  Similarly, consider a probability
space $P_{\sv}$ containing the random vectors
$(\bg_1^{\SE},\ldots,\bg_{k-1}^{\SE}) \sim \normal(\bzero,\bK_{g,k-1}
\otimes \id_n)$, the first $k-1$ solutions
$(\bv_1^{\SE},\ldots,\bv_{k-1}^{\SE})$ to equation~\eqref{eq:SE-opt}
(defined with parameters solving to the fixed point equations up to
iteration $k-1$), and auxiliary Gaussian noise $\bxi_g \sim
\normal(\bzero,\id_n)$ independent of everything else.  Let
$L_{\sv}^2$ be the space of square-integrable random $p$-dimensional
vectors $\bv$ defined on $P_{\sv}$.

In the sequel, we consider functions $f:L_\su^2 \rightarrow \reals$
(or $L_\sv^2 \rightarrow \reals$), and $f(\bu)$ will denote $f$
applied to the random vector $\bu \in L_\su^2$.  (This should not be
confused with evaluating a function $\tilde f: \reals^n \rightarrow
\reals$ at the realization of the random variable $\bu$. To avoid
confusion, we sometimes write the latter as $\tilde f(\bu(\omega))$
where $\omega$ denotes an element of the sample space $P_\su$.)

Define a mappings $\bg^{\Ltwo}: L_{\su}^2 \rightarrow L_{\sv}^2$,
$\bh^{\Ltwo}: L_{\sv}^2 \rightarrow L_{\su}^2$ by
\begin{equation}
\label{eq:g-L2}
    \bg^{\Ltwo}(\bu) \defn \sum_{\ell=1}^{k-1} \bg_\ell^\perp
    \<\bu_\ell^\perp,\bu\>_{\Ltwo} + \big\|
    \proj_{\bU_{k-1}^\SE}^\perp \bu \big\|_{\Ltwo} \bxi_g, \qquad
    \bh^{\Ltwo}(\bv) \defn \sum_{\ell = 1}^{k-1} \bh_\ell^\perp
    \<\bv_\ell^\perp,\bv\>_{\Ltwo} + \big\|
    \proj_{\bV_{k-1}^\SE}^\perp \bv \big\|_{\Ltwo}\bxi_h,
\end{equation}
where $\proj_{\bU_{k-1}^\SE}^\perp$ is the projection in $L_{\su}^2$
onto the space orthogonal to the span of $\bu_1,\ldots,\bu_{k-1}$ in
$L_{\su}^2$ and likewise for $\proj_{\bV_{k-1}^\SE}^\perp$; that is,
\begin{equation*}
\proj_{\bU_{k-1}^\SE}^\perp \bu = \bu - \sum_{\ell = 1}^{k-1}
\bu_\ell^\perp \<\bu_\ell^\perp,\bu\>_{\Ltwo}, \qquad
\proj_{\bV_{k-1}^\SE}^\perp \bv = \bv - \sum_{\ell = 1}^{k-1}
\bv_\ell^\perp \<\bv_\ell^\perp,\bv\>_{\Ltwo}.
\end{equation*}
For $k = 5, 6$, define the functions $\phi_k^{\Ltwo}: L_{\su}^2 \times
\reals \times L_{\sv}^2 \rightarrow \reals$ by
\begin{equation}
\label{eq:phi-L2}
\begin{gathered}
 \phi_5^{\Ltwo}(\bu;v_0,\bv) \defn \<\bu,\ones\>_{\Ltwo}(v_0 +
 \<\bmu_{\sx},\bv\>_{\Ltwo}) - \frac{1}{n}\sum_{i=1}^n
 \E\big[\ell_{\prop}^*(nu_i;y_{1,i}^{\SE})\big] +
 \E\big[\Omega_{\prop}(\bv)\big], \\
\phi_6^{\Ltwo}(\bu;v_0,\bv) \defn \<\bu,\by_2^{\SE}\>_{\Ltwo} +
\<\bu,\ones\>_{\Ltwo}(v_0 + \<\bmu_{\sx},\bv\>_{\Ltwo}) -
\frac{n}{2}\sum_{i=1}^n \E\Big[ (w_i^{\SE})^{-1}
  \frac{u_i^2}{y_{1,i}^{\SE}} \Big] + \E\big[\Omega_{\out}(\bv)\big],
\end{gathered}
\end{equation}
where we recall the convention that $u \mapsto u^2/0$ is the convex
indicator function equal to zero when $u=0$ and $\infty$ otherwise.
Because the function $u_i \mapsto \ell_\prop^*(n u_i;y_{1,i}^{\SE})$
is convex, and $(w_i^{\SE})^{-1}u_i^2/y_{1,i}^{\SE}$ is non-negative,
the expectations of these quantities are well-defined but possibly
infinite for $\bu \in L_\su^2$.  Finally, define a function
$\AuxObj_k^{\Ltwo}: L_{\su}^2 \times \reals \times L_{\sv}^2
\rightarrow \reals$ via
\begin{equation}
\label{eq:def-AuxObj-L2}
    \AuxObj_k^{\Ltwo}(\bu;v_0,\bv) \defn
    -\<\bg^{\Ltwo}(\bu),\bv\>_{\Ltwo} +
    \<\bh^{\Ltwo}(\bv),\bu\>_{\Ltwo} + \phi_k^{\Ltwo}(\bu;v_0,\bv).
\end{equation}
The definition $\bg^{\Ltwo}(\bu)$ and $\bh^{\Ltwo}(\bv)$ resembles the
definition of $\bg(\bu)$ and $\bh(\bv)$
in~\cref{eq:gordon-gh-explicit}, and the definition of
$\phi_k^{\Ltwo}$ resembles to the definition of $\phi_k$
in~\cref{eq:objectives}.  Nevertheless, we emphasize that arguments to
$\bg^{\Ltwo}$, $\bh^{\Ltwo}$, $\phi_k^{\Ltwo}$, $\AuxObj_k^{\Ltwo}$
are on completely different spaces than the arguments to $\bg$, $\bh$,
$\phi_k$, $\AuxObj_k$.  The former take arguments on a Hilbert space
of random vectors; the latter take arguments on Euclidean space.

In the sequel, we show that the solutions to the fixed point
equations~\eqref{eq:fixpt-inductive} are related to the KKT conditions
for the saddle point problem
\begin{equation}
\label{eq:min-max-L2}
\tag{Aux-$\Ltwo$} \min_{\substack{v_0 \in
    \reals\\ \<\bv,\bxi_g\>_{\Ltwo} \geq 0}}
\max_{\<\bu,\bxi_h\>_{\Ltwo} \geq 0} \AuxObj_k^{\Ltwo}(\bu;v_0,\bv).
\end{equation}
The sets $\{ \bv \in L_{\sv}^2 : \< \bv , \bxi_g \>_{\Ltwo} \geq 0\}$
and $\{ \bu \in L_{\su}^2 : \< \bu , \bxi_h \>_{\Ltwo} \geq 0\}$ are
closed and convex.  It is easy to check that the mapping $(v_0,\bv)
\mapsto \AuxObj_k^{\Ltwo}(\bu;v_0,\bv)$ is convex if $\< \bxi_h ,
\bu\>_{\Ltwo} \geq 0$ (note, in particular, that the function $\bv
\mapsto \| \proj_{\bV_{k-1}^{\SE}}^\perp \bv \|_{\Ltwo} \< \bxi_h ,
\bu \>_{\Ltwo}$ is then convex, which is not the case for $\<
\bxi_h,\bu\>_{\Ltwo} < 0$).  Likewise, the function $\bu \mapsto
\AuxObj_k^{\Ltwo}(\bu;v_0,\bv)$ is concave if $\< \bxi_g ,
\bv\>_{\Ltwo} \geq 0$, Thus, equation~\eqref{eq:min-max-L2} is a
convex-concave saddle point problem.

\subsubsection{Outline of proof of induction step}

Our proof of the existence and uniqueness of the fixed point
solutions~\eqref{eq:fixpt-inductive} follows three steps:
\begin{description}
\item [Step 1.] Given a fixed point~\eqref{eq:fixpt-inductive}, we
  show how to construct a saddle point for the
  problem~\eqref{eq:min-max-L2}.
\item[Step 2a.] We show that the problem~\eqref{eq:min-max-L2} has at
  least one saddle point $(\buhat;\vhat_0,\bvhat)$, and the values of
  $\buhat$ and $\bvhat$ are unique.
    \item[Step 2b.] Given a saddle point for the
      problem~\eqref{eq:min-max-L2}, we show how to construct a
      fixed point~\eqref{eq:fixpt-inductive}.

    \item[Step 3.] We show that the previous steps imply that
      equation~\eqref{eq:fixpt-inductive} has a unique solution.
\end{description}

\subsubsection{Derivatives of the auxiliary objective on Hilbert space}
\label{sec:derivative-aux-Hilbert}

The connection between the fixed point
equation~\eqref{eq:fixpt-inductive} and the KKT conditions of the
objective~\eqref{eq:min-max-L2} rely on several subgradient
identities, which we state here as lemmas.  Their proofs are
straightforward, and so omitted.  We adopt the convention that
$\bdelta / \| \bdelta \|_{\Ltwo} = \bzero$ if $\bdelta = \bzero$.  \\

\begin{lemma}[Derivatives of Gordon terms]
\label{lem:L2-deriv-gordon}
If $\< \bxi_g , \bv \>_{\Ltwo} \geq 0$, then the function $f(\bu, \bv)
\defn \< \bg^{\Ltwo}(\bu) , \bv \>_{\Ltwo}$ is convex in $\bu$ and
linear in $\bv$.  Moreover,
\begin{equation}
  \label{eq:L2-deriv-gordon-u}
  \begin{gathered}
    \sum_{\ell=1}^{k-1} \bu_\ell^{\SE,\perp}\< \bg_\ell^{\SE,\perp} ,
    \bv \>_{\Ltwo} + \frac{\proj_{\bU_{k-1}^{\SE}}^\perp
      \bu}{\|\proj_{\bU_{k-1}^{\SE}}^\perp \bu\|_{\Ltwo}} \< \bxi_g ,
    \bv \>_{\Ltwo} \in \partial_{\bu} f(\bu, \bv), \\ \partial_{\bv}
    f(\bu, \bv) = \bg^{\Ltwo}(\bu).
    \end{gathered}
    \end{equation}
 If either $\< \bxi_g,\bv \>_{\Ltwo} = 0$ or
 $\proj_{\bU_{k-1}^{\SE}}^\perp \bu \neq \bzero$, then $f$ is
 differentiable with respect to $\bu$ at $(\bu,\bv)$, with gradient
 given by the expression in the first line of the previous display.
 Likewise,
    \begin{equation}
    \begin{gathered}
        \sum_{\ell=1}^{k-1} \bv_\ell^{\SE,\perp}\<
        \bh_\ell^{\SE,\perp} , \bu \>_{\Ltwo} +
        \frac{\proj_{\bV_{k-1}^{\SE}}^\perp
          \bv}{\|\proj_{\bV_{k-1}^{\SE}}^\perp \bv\|_{\Ltwo}} \<
        \bxi_h , \bu \>_{\Ltwo} \in
        \partial_{\bv}\big\<\bh^{\Ltwo}(\bv),\bu\big\>_{\Ltwo},
        \\ \partial_{\bu}\big\<\bh^{\Ltwo}(\bv),\bu\big\>_{\Ltwo} =
        \bh^{\Ltwo}(\bv),
    \end{gathered}
    \end{equation}
    and if either $\< \bxi_h,\bu \>_{\Ltwo} = 0$ or
    $\proj_{\bV_{k-1}^{\SE}}^\perp \bv \neq \bzero$, then in fact
    $\big\<\bh^{\Ltwo}(\bv),\by\big\>_{\Ltwo}$ is differentiable with
    respect to $\bv$ at $(\bu,\bv)$, with gradient given by the
    expression in the first line of the previous display.
\end{lemma}
\noindent The fixed point equations~\eqref{eq:fixpt-inductive} depend
implicitly on the optimizations in display~\eqref{eq:SE-opt}, which
depend, via $\phi_{k,u}$, implicitly on the parameters $\nu_{k,0}$ and
$\nu_{k,\sx}$.  It is useful to make this dependence explicit by
writing $\phi_{k,u}(\bu; \bH_{k-1}^{\SE}; \nu_{k,0},\nu_{k,\sx})$.

\begin{lemma}[Derivatives of penalty terms]
\label{lem:L2-deriv-pen}
We have
\begin{subequations}
\begin{equation}
  \label{eq:L2-pen-deriv-u}
  \bdelta \in \partial_{\bu}\big(-\phi_k^{\Ltwo}(\bu;v_0,\bv)\big)
  \;\; \text{$\Longleftrightarrow$ for almost all $\omega$,} \;\;
  \bdelta(\omega) \in
  \partial_{\bu}\phi_{k,u}\big(\bu(\omega);\bH_{k-1}^{\SE}(\omega);v_0,\<\bmu_{\sx},\bv\>_{\Ltwo}\big),
\end{equation}
and likewise,
\begin{equation}
  \label{eq:L2-pen-deriv-v}
  \bdelta \in \partial_{\bv}\phi_k^{\Ltwo}(\bu;v_0,\bv) \;\;
  \text{$\Longleftrightarrow$ for almost all $\omega$,} \;\;
  \bdelta(\omega) \in
  \partial_{\bv}\phi_{k,v}\big(\bv(\omega);\bG_{k-1}^{\SE}(\omega);\<\ones,\bu\>_{\Ltwo}\big).
\end{equation}
\end{subequations}
We emphasize that in each line, the first subgradient occurs on
infinite dimensional Hilbert space, and the second on
finite-dimensional Euclidean space.  Finally, we have
\begin{align*}
  \partial_{v_0} \phi_k^{\Ltwo}(\bu;v_0,\bv) = \< \ones , \bu
  \>_{\Ltwo}.
\end{align*}
\end{lemma}

\subsubsection{Step 1: From fixed points to saddle points}

We show how to construct a saddle point for the optimization
problem~\eqref{eq:min-max-L2} based on a solution to the fixed point
equation~\eqref{eq:fixpt-inductive}.

We first claim that any solution to the fixed point
equations~\eqref{eq:fixpt-inductive} must satisfy $\zeta_{kk}^v \geq
0$ and $\zeta_{kk}^u \geq 0$.  First consider $\zeta_{kk}^v$.  Recall
that we may write $\bh_k^{\SE} = \sum_{\ell = 1}^k L_{g,k\ell}
\bh_\ell^{\SE,\perp}$, whence we can write the objective in the first
line in~\cref{eq:SE-opt} as $\sR(\bu)- L_{g,kk}\< \bh_k^{\SE,\perp} ,
\bu \>$, where $\sR(\bu)$ depends implicitly on $\{ \bu_\ell^{\SE}
\}_{\ell \leq k-1}$, $\{ \bh_\ell^{\SE} \}_{\ell \leq k-1}$, but does
not depend on $\bh_k^{\SE,\perp}$.  Consider the optimization
problem~\eqref{eq:SE-opt} for two values of $\bh_k^{\SE,\perp}$ and
$\bhtilde_k^{\SE,\perp}$, holding everything else constant, and denote
the respective optimizers $\bu_k^\SE$ and $\butilde_k^{\SE}$.  By the
optimality of $\bu_k^\SE$ and $\butilde_k^{\SE}$,
\begin{equation}
\begin{aligned}
    \sR(\bu_k^{\SE})
        -
        L_{g,kk}\<\bh_k^{\SE,\perp},\bu_k^{\SE}\>
        &\leq 
        \sR(\butilde_k^{\SE})
        -
        L_{g,kk}\<\bh_k^{\SE,\perp},\butilde_k^{\SE}\>
        =
        \sR(\butilde_k^{\SE})
        -
        L_{g,kk}\<\bhtilde_k^{\SE,\perp},\butilde_k^{\SE}\>
        +
        L_{g,kk}\<\bhtilde_k^{\SE,\perp} - \bh_k^{\SE,\perp}, \butilde_k^{\SE}\>
    \\
        &\leq
        \sR(\bu_k^{\SE})
        -
        L_{g,kk}\<\bhtilde_k^{\SE,\perp},\bu_k^{\SE}\>
        +
        L_{g,kk}\<\bhtilde_k^{\SE,\perp} - \bh_k^{\SE,\perp}, \butilde_k^{\SE}\>.
\end{aligned}
\end{equation}
We see that $L_{g,kk}\<\bhtilde_k^{\SE,\perp} - \bh_k^{\SE,\perp},
\butilde_k^{\SE} - \bu_k^{\SE}\> \geq 0$.  Because
$\bhtilde_k^{\SE,\perp}$ is independent of $\bu_k^{\SE}$,
$\bh_k^{\SE,\perp}$ is independent of $\butilde_k^{\SE}$, and $\<
\bhtilde_k^{\SE,\perp} , \butilde_k^{\SE} \>_{\Ltwo} = \<
\bh_k^{\SE,\perp} , \bu_k^{\SE} \>_{\Ltwo}$, taking expectations gives
$L_{g,kk}\< \bh_k^{\SE,\perp} , \bu_k^{\SE} \>_{\Ltwo} \geq 0$.  If
$k$ is innovative with respect to $\bK_{g,k}$, then $L_{g,kk} > 0$
(see~\Cref{lem:cholesky}), whence $\< \bh_k^{\SE,\perp} , \bu_k^{\SE}
\>_{\Ltwo}\geq 0$.  Otherwise, if $k$ is predictable with respect to
$\bK_{g,k}$, then $\bh_k^{\SE,\perp} = \bzero$, whence $\<
\bh_k^{\SE,\perp} , \bu_k^{\SE} \>_{\Ltwo}\geq 0$ in this case as
well.  Multiplying the first equation in the second line of
equation~\eqref{eq:fixpt-inductive} by $\bL_{g,k}^\ddagger$, and
recalling that $\bH_k^{\SE,\perp} = \bH_k^{\SE}
(\bL_g^\ddagger)^\top$, we get that $\bL_{g,k}^\top \bZ_{v,k}^\top =
\llangle \bH_k^{\SE,\perp} , \bU_k^{\SE} \rrangle_{\Ltwo}$.  Because
$\bL_{g,k}$ is lower-diagonal, the $k^\text{th}$ row and column of
this equation is $L_{g,kk}\zeta_{kk}^v = \< \bh_k^{\SE,\perp} ,
\bu_k^{\SE} \>_{\Ltwo} \geq 0$.  If $k$ is innovative with respect to
$\bK_{g,k}$, then $L_{g,kk} > 0$ (see~\Cref{lem:cholesky}), whence
$\zeta_{kk}^v \geq 0$.  If $k$ is predictable with respect to
$\bK_{g,k}$, then $\zeta_{kk}^v = 0$ by innovation compatibility.
Thus, in all cases, $\zeta_{kk}^v \geq 0$.  The inequality
$\zeta_{kk}^u \geq 0$ follows by exactly the same argument, applied to
the second optimization in~\Cref{eq:SE-opt} and the second equation in
the second line of~\Cref{eq:fixpt-inductive}.

Recall the Hilbert space $L_{\su}^2$ already comes endowed with random
variables $\bh_1^{\SE},\ldots,\bh_{k-1}^{\SE}$ and
$\bu_1^{\SE},\ldots,\bu_{k-1}^{\SE}$ and $\beps_1^{\SE}$,
$\beps_2^{\SE}$ with joint distribution given by the state evolution
up to iteration $k-1$.  We then define $\bh_k^{\SE} = \sum_{\ell =
  1}^{k-1} L_{h,k\ell} \bh_\ell^{\SE,\perp} + L_{h,kk} \bxi_h$, so
that $\bH_k^{\SE} \sim \normal(\bzero,\bK_{h,k} \otimes \id_n)$ is
then embedded in the Hilbert space $L_{\su}^2$.  We then define
$\bu_k^{\SE}$ via equation~\eqref{eq:SE-opt} with this choice of
$\bH_k^{\SE}$, so that $\bU_k^{\SE}$ is also embedded on the Hilbert
space $L_{\su}^2$ with distribution given by the state evolution.  We
similarly define $\bg_k^{\SE} = \sum_{\ell = 1}^{k-1} L_{g,k\ell}
\bg_\ell^{\SE,\perp} + L_{g,kk} \bxi_g$ and $\bv_k^{\SE}$ via
equation~\eqref{eq:SE-opt}, so that $\bG_k^{\SE} \sim
\normal(\bzero,\bK_{g,k} \otimes \id_p)$ and $\bV_k^{\SE}$ are embedded
on the Hilbert space $L_{\sv}^2$ with distribution given by the state
evolution.  We claim that $(\bu_k^{\SE};\nu_{k,0},\bv_k^{\SE})$ is a
saddle point of equation~\eqref{eq:min-max-L2}, which we now show.

First, the point $(\bu_k^{\SE};\nu_{k,0},\bv_k^{\SE})$ is feasible.
Note that $\bh_k^{\SE,\perp} = L_{g,kk} \bxi_h$.  We showed above that
$\< \bh_k^{\SE,\perp} , \bu_k^{\SE} \>_{\Ltwo} \geq 0$.  In the case
that $L_{g,kk} > 0$, this implies that $\< \bxi_h , \bu_k^{\SE}
\>_{\Ltwo} \geq 0$.  In the case that $L_{g,kk} = 0$, we see from
equation~\eqref{eq:SE-opt} that $\bxi_h$ is independent of $\bu_k^{\SE}$,
whence $\< \bxi_h , \bu_k^{\SE} \>_{\Ltwo} = 0$.  By an analogous
argument, $\< \bxi_g , \bv_k^{\SE} \>_{\Ltwo} \geq 0$.  Thus,
$(\bu_k^{\SE};\nu_{k,0},\bv_k^{\SE})$ is feasible, as claimed.  In
particular, this implies $(v_0,\bv) \mapsto
\AuxObj_k^{\Ltwo}(\bu_k^{\SE};v_0,\bv)$ is convex, and $\bu \mapsto
\AuxObj_k^{\Ltwo}(\bu;\nu_{k,0},\bv_k^{\SE})$ is concave (see
discussion following equation~\eqref{eq:min-max-L2}).

Next, by equation~\eqref{eq:SE-opt}, there exists random vector $\bdelta$
such that for almost every $\omega$,
\begin{equation}
    \bh_k^{\SE}(\omega) - \sum_{\ell=1}^k \zeta_{k\ell}^u \bu_\ell^{\SE}(\omega) 
        \in 
        \partial_{\bu} \phi_{k,u}(\bu_k^{\SE}(\omega);\bH_{k-1}^{\SE}(\omega);\nu_{k,0},\nu_{k,\sx})
        =
        \partial_{\bu} \phi_{k,u}(\bu_k^{\SE}(\omega);\bH_{k-1}^{\SE}(\omega);\nu_{k,0},\< \bmu_{\sx} , \bv_k^{\SE} \>_{\Ltwo}),
\end{equation}
where the equality holds by the last line of the fixed point equations
\eqref{eq:fixpt-inductive}.  Thus, by equation~\eqref{eq:L2-pen-deriv-u},
\begin{equation}
    \bh_k^{\SE} - \sum_{\ell = 1}^k \zeta_{k\ell}^u \bu_\ell^{\SE} \in \partial_{\bu}\big(- \phi_k^{\Ltwo}(\bu_k^{\SE};\nu_{k,0},\bv_k^{\SE})\big).
\end{equation}
By the first equation in the first line 
and the second equation in the second line of equation~\eqref{eq:fixpt-inductive},
\begin{equation}
\label{eq:gor-deriv-to-se-deriv}
\begin{aligned}
    \sum_{\ell=1}^{k-1} \bu_\ell^{\SE,\perp}\< \bg_\ell^{\SE,\perp} , \bv_k^{\SE} \>_{\Ltwo}
        +
        \frac{\proj_{\bU_{k-1}^{\SE}}^\perp \bu_k^{\SE}}{\|\proj_{\bU_{k-1}^{\SE}}^\perp \bu_k^{\SE}\|_{\Ltwo}}
        \< \bxi_g , \bv \>_{\Ltwo}
        &=
        \bU_k^{\SE} (\bL_{u,k}^\ddagger)^\top\E[\bL_{u,k}^\ddagger \bG_k^{\SE,\top} \bv_k^{\SE}]
    \\
        &=
        \bU_k^{\SE} \bK_{g,k}^\ddagger [\bK_{g,k} \bZ_{u,k}^\top]_{\,\cdot\,,k}
        = 
        \sum_{\ell=1}^k \zeta_{k\ell}^u \bu_\ell^{\SE},
\end{aligned}
\end{equation}
where the last equation uses $\bK_{g,k}^\ddagger \bK_{g,k} = \id_{\bK_{g,k}}^\top$ and the innovation compatibility of $\bZ_{u,k}$ and $\bK_{g,k}$ (see \Cref{lem:cholesky-inverse-identities}).
Moreover, by the second equation in the first line of equation~\eqref{eq:fixpt-inductive},
$\bh^{\Ltwo}(\bv_k^{\SE}) = \bh_k^{\SE}$.
Thus, by equation~\eqref{eq:L2-deriv-gordon-u},
\begin{equation}
    \sum_{\ell=1}^k \zeta_{k\ell}^u \bu_\ell^{\SE} - \bh_k^{\SE}
        \in
        \partial_{\bu}
        \Big(
            \< \bg^{\Ltwo}(\bu_k^{\SE}) , \bv_k^{\SE} \>_{\Ltwo}
            -
            \< \bh^{\Ltwo}(\bv_k^{\SE}) , \bu_k^{\SE} \>_{\Ltwo}
        \Big).
\end{equation}
Combining this display with the third-to-last display shows that $\bzero \in \partial_{\bu}\big(- \AuxObj_k^{\Ltwo}(\bu_k;\nu_{k,0},\bv_k^{\SE}) \big) $.
An analogous argument shows that $\bzero \in \partial_{\bv} \AuxObj_k^{\Ltwo}(\bu_k;\nu_{k,0},\bv_k^{\SE})$.
Further, $\partial_{v_0} \AuxObj_k^{\Ltwo}(\bu_k;\nu_{k,0},\bv_k^{\SE}) = \< \bu_k^{\SE}, \ones \>_{\Ltwo}$ (see equation~\eqref{eq:phi-L2}),
which, by the last line of the fixed point equations \eqref{eq:fixpt-inductive} is equal to 0.
Thus $(\bu_k^{\SE};\nu_{k,0},\bv_k^{\SE})$ is a saddle point of equation~\eqref{eq:min-max-L2}, as claimed.

\subsubsection{Step 2a: Existence/uniqueness of saddle points}
\label{sec:L2-saddle-exist-unique}

Now we show that equation~\eqref{eq:min-max-L2} has at least one
saddle point $(\buhat;\vhat_0,\bvhat)$, and the value of $\buhat$ and
$\bvhat$ is unique.  We isolate the dependence of $\AuxObj_k^{\Ltwo}$
on $v_0$ by decomposing
\begin{equation}
    \AuxObj_k^{\Ltwo}(\bu;v_0,\bv)
        =
        \<\bu,\ones\>_{\Ltwo}v_0
        +
        \overline{\AuxObj}_k^{\Ltwo}(\bu;\bv),
\end{equation}
where this equation defines $\overline{\AuxObj}_k^{\Ltwo}$.
We complete Step 2a in several steps.

\begin{enumerate}

    \item
    \emph{When restricted to the convex domains~\eqref{eq:min-max-L2},
    the function $\overline{\AuxObj}_k^{\Ltwo}$ (hence
    $\AuxObj_k^{\Ltwo}$) is strongly convex and lower semi-continuous
    in $\bv$ and strongly concave and upper semi-continuous in $\bu$.}

    Weak convexity-concavity and lower/upper semi-continuity are
    clear (see discussion following equation~\eqref{eq:min-max-L2}).
    Strong convexity is due to the terms $\E[\Omega_k(\bv)]$.  Strong
    concavity is due to the terms $-\frac{1}{n}\sum_{i=1}^n
    \E\big[\ell_{\prop}^*(nu_i;y_{1,i}^{\SE})\big]$ and
    $-\frac{n}{2}\sum_{i=1}^n\E\Big[(w_i^{\SE})^{-1}\frac{u_i^2}{y_{1,i}^{\SE}}\Big]$
    for $k = 5,6$, respectively.  In making these assertions, we use
    the strong-convexity of $\Omega_k$, the strong-smoothness of
    $\ell_\prop$, and the upper bound on the weight function $w$
    (cf. Assumption A1).

    \item 
    \emph{The function $(v_0,\bv) \mapsto \max_{\<\bu,\bxi_h\>_{\Ltwo}
      \geq 0} \AuxObj_k^{\Ltwo}(\bu;v_0,\bv)$ is lower semi-continuous
    and convex, and is strongly convex in $\bv$.}
    
    This property follows from its definition as a supremum of lower
    semi-continuous and convex functions.

    \item 
    \emph{The function $(v_0,\bv) \mapsto \max_{\<\bu,\bxi_h\>_{\Ltwo}
      \geq 0} \AuxObj_k^{\Ltwo}(\bu;v_0,\bv)$ is coercive; that is, it
    diverges to infinity if either $|v_0| \rightarrow \infty$ or $\|
    \bv \|_{\Ltwo} \rightarrow \infty$.}

    To establish this property, we define two elements $\bu^+,\bu^-
    \in L_{\su}^2$ satisfying $\<\bu^+,\ones\>_{\Ltwo} \geq 0$ and
    $\<\bu^-,\ones\>_{\Ltwo} \leq 0$ and
    \begin{equation}
        \lim_{\substack{v_0 \rightarrow
            \infty\\\|\bv\|_{\Ltwo}\rightarrow
            \infty}}\AuxObj_k^{\Ltwo}(\bu^+;v_0,\bv) = \infty \quad
        \text{and} \quad \lim_{\substack{v_0 \rightarrow
            -\infty\\\|\bv\|_{\Ltwo}\rightarrow
            \infty}}\AuxObj_k^{\Ltwo}(\bu^-;v_0,\bv) = \infty.
    \end{equation}
    The coercivity of $(v_0,\bv) \mapsto \max_{\<\bu,\bxi_h\>_{\Ltwo}
      \geq 0} \AuxObj_k^{\Ltwo}(\bu;v_0,\bv)$ then follows because
    \begin{equation}
        \max_{\<\bu,\bxi_h\>_{\Ltwo} \geq 0} \AuxObj_k^{\Ltwo}(\bu;v_0,\bv)
            \geq 
            \max\Big\{
                \AuxObj_k^{\Ltwo}(\bu^+;v_0,\bv),
                \AuxObj_k^{\Ltwo}(\bu^-;v_0,\bv)
            \Big\}.
    \end{equation}
    Define $u_i^+(\omega) =\ones\{y_{1,i}^{\SE}(\omega) \neq 0\}$, and
    note that $\bu^+$ is in $L_{\su}^2$ because it is bounded.
    Because $\bu^+$ is independent of $\bxi_h$, $\< \bu^+ , \bxi_h
    \>_{\Ltwo} = 0$, whence it is in the domain of optimization.
    Further, we have $\AuxObj_k^{\Ltwo}(\bu^+;v_0,\bv) > -\infty$.
    Indeed, for $k = 5$, this is clear because $\AuxObj_k^{\Ltwo}$ is
    finite everywhere.
    For $k = 6$, the only term which might be infinite is
    $\frac{n}{2}\sum_{i=1}^n\E\Big[(w_i^{\SE})^{-1}\frac{u_i^{{+}2}}{y_{1,i}^{\SE}}\Big]$,
    but we have constructed $\bu^+$ so that this term is bounded by
    $\frac{n}{2}\sum_{i=1}^n\E\big[(w_i^{\SE})^{-1}\big] < \infty$ (by
    the lower bound on the weight function, Assumption A1).
    Finally, $\<\bu^+,\ones\>_{\Ltwo} v_0 = v_0 \sum_{i=1}^n
    \P(y_{1,i}^{\SE} \neq 0)$.  Because $\P(y_{1,i}^{\SE} \neq 0) >
    0$, we get that
    \begin{equation}
        \AuxObj_k^{\Ltwo}(\bu^+;v_0,\bv) 
            \geq 
            v_0 \sum_{i=1}^n \P(y_{1,i}^{\SE} \neq 0)
            +
            \overline{\AuxObj}_k^{\Ltwo}(\bu;\bv).
    \end{equation}
    Because $\overline{\AuxObj}_k^{\Ltwo}$ is strongly convex in $\bv$,
    we see that this diverges if $v_0 \rightarrow \infty$ or $\| \bv \|_{\Ltwo} \rightarrow \infty$.
    Setting $\bu^- = - \bu^+$ and following the same argument, we get that $\AuxObj_k^{\Ltwo}(\bu^+;v_0,\bv) $ diverges if $v_0 \rightarrow -\infty$ or $\| \bv \|_{\Ltwo} \rightarrow \infty$.
    Thus, coercivity is established.
    
    \item 
    \emph{There exists a saddle point of $\AuxObj_k^{\Ltwo}$
    restricted to the domains~\eqref{eq:min-max-L2}.}

    Using the lower semi-continuity, convexity, and coercivity of
    $(v_0,\bv) \mapsto \max_{\<\bu,\bxi_h\>_{\Ltwo} \geq 0}
    \AuxObj_k^{\Ltwo}(\bu;v_0,\bv)$, Theorem 11.9 in the
    book~\cite{bauschke2011convex} implies that this function has a
    minimizer on the closed domain $\< \bv , \bxi_g \>_{\Ltwo} \geq
    0$.  Let $(\vhat_0,\bvhat_0)$ be one such minimizer.  Because $\bu
    \mapsto \AuxObj_k^{\Ltwo}(\bu; \vhat_0,\bvhat)$ is strongly
    concave in $\bu$ (see above), this function is maximized by a
    unique $\bu$, which we call $\buhat$.  Then, $(\buhat;
    \vhat_0,\bvhat)$ is a saddle point.

    \item 
    \emph{The value of $\buhat$ and $\bvhat$ at the saddle point of $\AuxObj_k^{\Ltwo}$ is unique.}

Consider any saddle point $(\buhat; \vhat_0, \bvhat)$.  Because
$\AuxObj_k^{\Ltwo}$ has a saddle point, we may exchange the order of
minimization and maximization, so that $\buhat$ is a maximizer of $\bu
\mapsto \allowbreak \min_{v_0 \in \reals, \<\bv,\bxi_g\>_{\Ltwo} \geq
  0} \AuxObj_k^{\Ltwo}(\bu;v_0,\bv)$.  Because
$\AuxObj_k^{\Ltwo}(\bu;v_0,\bv)$ is strongly-concave in $\bu$, so too
is this function, whence $\buhat$ is unique.

Next, by the KKT conditions, we must have at any saddle point
$(\buhat;\vhat_0,\bvhat)$
\begin{equation}
  \begin{gathered}
    \vhat_0 \ones \in -\partial_{\bu}
    \overline{\AuxObj}_k^{\Ltwo}(\buhat;\bvhat) +
    \cN_{\{\<\bu,\bxi_h\>_{\Ltwo}\geq 0\}}(\buhat), \\ \bzero \in
    \partial_{\bv} \overline{\AuxObj}_k^{\Ltwo}(\buhat;\bvhat) +
    \cN_{\{\<\bv,\bxi_g\>_{\Ltwo}\geq 0\}}(\bvhat),
    \end{gathered}
    \end{equation}
where $\cN_{\{\<\bu,\bxi_h\>_{\Ltwo}\geq 0\}}(\buhat)$ denotes the
normal cone to the set $\{\<\bu,\bxi_h\>_{\Ltwo}\geq 0\}$ at $\buhat$,
and analogously for $\cN_{\{\<\bv,\bxi_g\>_{\Ltwo}\geq 0\}}(\bvhat)$.
Because $\overline{\AuxObj}_k^{\Ltwo}(\bu;\bv)$ is strongly-concave in
$\bu$ and strongly convex in $\bv$, the correspondence on the
right-hand side of the preceding display is strongly monotone (see
Definition 22.1 in the book~\cite{bauschke2011convex}).  Thus, for
some $c > 0$ and any two saddle points $(\buhat; \vhat_0, \bvhat)$,
$(\buhat; \vhat_0', \bvhat')$, we must have
\begin{equation}
  \<\vhat_0' \ones - \vhat_0 \ones,\buhat-\buhat\>_{\Ltwo} + \< \bzero
  - \bzero, \bvhat' - \bvhat \>_{\Ltwo} \geq
  c(\|\buhat-\buhat\|_{\Ltwo}^2 + \| \bvhat' - \bvhat \|_{\Ltwo}^2),
\end{equation}
whence $\bvhat' = \bvhat$.
\end{enumerate}

This completes Step 2a.


\subsubsection{Step 2b: From saddle points to fixed points}

Consider a saddle point $(\buhat, \vhat_0,\bvhat)$ of
problem~\eqref{eq:min-max-L2}.  Now define $\bu_k^{\SE} = \buhat$,
$\bv_k^{\SE} = \bvhat$, $\nu_{k,0} = \vhat_0$, $\bg_k^{\SE} =
\bg^{\Ltwo}(\buhat)$ and $\bh_k^{\SE} = \bh^{\Ltwo}(\bvhat)$.  Define
the parameters $\bK_{g,k}$, $\bK_{h,k}$, $\{\nu_{\ell,\sx}\}_{\ell
  \leq k}$ via equation~\eqref{eq:fixpt-inductive} with these choices of
$\bu_k^{\SE}$, $\bv_k^{\SE}$, $\bg_k^{\SE}$, and $\bh_k^{\SE}$.  Then
define $\bZ_{v,k}^\top = \bK_{h,k}^\ddagger \llangle \bH_k^{\SE} ,
\bU_k^{\SE}\rrangle_{\Ltwo}$ and $\bK_{g,k}^\ddagger \llangle
\bG_k^{\SE} \bV_k^{\SE} \rrangle$.  In order to show that this gives a
solution to the fixed point equations, we must show that \emph{(1)}
under these choices of $\bg_k^{\SE}$ and $\bh_k^{\SE}$, we have
$\bG_k^{\SE} \sim \normal(\bzero,\bK_{g,k} \otimes \id_p)$ and
$\bH_k^{\SE} \sim \normal(\bzero,\bK_{h,k} \otimes \id_n)$, \emph{(2)}
$\< \ones , \bu_k^{\SE} \>_{\Ltwo} = 0$, \emph{(3)} the second line of
equation~\eqref{eq:fixpt-inductive} holds, and $\bZ_{v,k}$, $\bZ_{u,k}$
satisfy the appropriate innovation compatibility constraints, and
\emph{(4)} for almost all $\omega$, equation~\eqref{eq:SE-opt} is satisfied
with these choices of $\bu_k^{\SE}$, $\bv_k^{\SE}$, $\bg_k^{\SE}$, and
$\bh_k^{\SE}$.

The distributional properties $\bG_k^{\SE} \sim
\normal(\bzero,\bK_{g,k} \otimes \id_p)$ and $\bH_k^{\SE} \sim
\normal(\bzero,\bK_{h,k} \otimes \id_n)$ follow from the definition of
$\bg^{\Ltwo}$ and $\bh^{\Ltwo}$.  That $\< \ones , \bu_k^{\SE}
\>_{\Ltwo} = 0$ holds from the KKT conditions for
equation~\eqref{eq:min-max-L2} because $\partial_{v_0}
\AuxObj_k^{\Ltwo}(\buhat;\vhat_0,\bvhat) = \< \buhat , \ones
\>_{\Ltwo}$.  Note that by our definition of $\bZ_{v,k}^\top$, we have
$\bK_{h,k}\bZ_{v,k}^\top = \id_{\bK_{h,k}} \llangle \bH_k^{\SE} ,
\bU_k^{\SE} \rrangle_{\Ltwo}$, where we have used
equation~\eqref{eq:K-pseudo-inverse}.  Because $\bH_k^{\SE} \sim
\normal(\bK_{h,k} \otimes \id_n)$, we have that $\range(\llangle
\bH_k^{\SE} , \bU_k^{\SE} \rrangle_{\Ltwo}) \subset
\range(\bK_{h,k})$, so that $\id_{\bK_{h,k}} \llangle \bH_k^{\SE} ,
\bU_k^{\SE} \rrangle_{\Ltwo} = \llangle \bH_k^{\SE} , \bU_k^{\SE}
\rrangle_{\Ltwo}$.  Thus, the first equation in the second line of
equation~\eqref{eq:fixpt-inductive} is satisfied, and the second equation
in the second line follows similarly.  Moreover, because
$\bL_{h,k}^\ddagger$ is innovation compatible with respect to
$\bK_{h,k}$, the $\ell^\text{th}$ row of $\bK_{h,k}^\ddagger =
(\bL_{h,k}^\ddagger)^\top \bL_{h,k}^\ddagger$ is 0 for predictable
$\ell$.  Thus, $\bZ_{v,k}$ is innovation compatible with $\bK_{h,k}$.
Similarly, $\bZ_{u,k}$ is innovation compatible with $\bK_{g,k}$.

The remainder of our argument is devoted to proving item \textit{(4)};
the proof consists of two steps.
\begin{enumerate}
\item \emph{If either $\< \bxi_g , \bvhat \>_{\Ltwo} = 0$ or
$\proj_{\bU_{k-1}^{\SE}}^\perp \buhat \neq \bzero$, then the first
line of equation~\eqref{eq:SE-opt} is satisfied.  Likewise, if either $\<
\bxi_h , \buhat \>_{\Ltwo} = 0$ or $\proj_{\bV_{k-1}^{\SE}}^\perp
\bvhat \neq \bzero$, then the second line of equation~\eqref{eq:SE-opt} is
satisfied.}

Consider that either $\< \bxi_g , \bvhat \>_{\Ltwo} = 0$ or
$\proj_{\bU_{k-1}^{\SE}}^\perp \buhat \neq \bzero$.
\Cref{lem:L2-deriv-gordon} states $\bu \mapsto \<
\bg^{\Ltwo}(\bu),\bvhat\>_{\Ltwo}$ and $\bu \mapsto \<
\bh^{\Ltwo}(\bvhat) , \bu \>_{\Ltwo}$ are differentiable at $\bu =
\buhat$ and gives expressions for the derivative.  These expressions
can be simplified using the same calculation we performed in
equation~\eqref{eq:gor-deriv-to-se-deriv}.  Combined
with~\Cref{lem:L2-deriv-pen} and the KKT conditions for
equation~\eqref{eq:min-max-L2}, we must have for some $\lambda_{\su} \geq
0$ and almost every $\omega$
\begin{equation}
  \label{eq:u-min}
  \lambda_{\su} \bxi_h(\omega) + \bh_k^{\SE}(\omega) - \sum_{\ell=1}^k
  \zeta_{k\ell}^u \bu_\ell^{\SE} \in \partial_{\bu}
  \phi_{k,u}(\bu_k^{\SE}(\omega),\bH_{k-1}^{\SE}(\omega);\nu_{k,0},\nu_{k,\sx}),
    \end{equation}
where $\lambda_{\su} \< \bxi_h , \bu_k^{\SE} \>_{\Ltwo} = 0$.
Thus
\begin{equation}
  \begin{aligned}
    \bu_k^{\SE}(\omega) = \argmin_{\bu} \Big\{ &\frac{\zeta_{kk}^u}2
    \| \bu \|^2 + \sum_{\ell=1}^{k-1} \zeta_{k\ell}^u \<
    \bu_\ell^{\SE}(\omega) , \bu \> \\ &\qquad- \< \bh_k^{\SE}(\omega)
    , \bu \> - \lambda_{\su} \<\bxi_h(\omega),\bu\> +
    \phi_{k,u}(\bu,\bH_{k-1}^{\SE}(\omega);\nu_{k,0},\nu_{k,\sx})
    \Big\}.
  \end{aligned}
\end{equation}
We need only show that $\lambda_{\su} = 0$.  Recall that $\bh_k^{\SE}
= \sum_{\ell=1}^k L_{h,k\ell} \bh_\ell^{\SE,\perp}$, where
$\bh_{\ell}^{\SE,\perp}$ is independent of $\bxi_h$ for $\ell \leq
k-1$ and $L_{h,kk} \geq 0$.  If $\< \bxi_h,\bu_k^{\SE} \>_{\Ltwo} >
0$, then $\lambda_{\su} = 0$ by complementary slackness.
    
On the other hand, if $\< \bxi_h,\bu_k^{\SE} \>_{\Ltwo} = 0$, we must
have that $L_{h,kk} + \lambda_{\su} = 0$.  Indeed, assume otherwise,
that $L_{h,kk} + \lambda_{\su} > 0$.  The only dependence of the
objective in the previous display on $\bxi_h(\omega)$ is given by
$(L_{h,kk} + \lambda_u)\< \bxi_h(\omega),\bu\>$ (the term $L_{h,kk}\<
\bxi_h(\omega),\bu\>$ comes from the expansion of $\<
\bh_k^{\SE}(\omega) , \bu \>$).  Thus, using that the minimizer of the
objective is unique, we have for two $\omega,\omega'$ such that
$\bxi_h(\omega) \neq \bxi_h(\omega;)$ and all other random variables
are constant across $\omega,\omega'$, that $\< \bxi_h(\omega') -
\bxi_h(\omega) , \bu_k^{\SE}(\omega') - \bu_k^{\SE}(\omega) \>\geq 0$
with equality if and only if $\bu_k^{\SE}(\omega') =
\bu_k^{\SE}(\omega)$.  Because $\bxi_h$ is independent of everything
else and, conditional on everything else, $\bu_k^{\SE}(\omega)$ is not
conditionally constant, we have that $\< \bxi_h , \bu_k^{\SE}
\>_{\Ltwo} > 0$, a contradiction.  Therefore, we conclude that in this
case $L_{h,kk} + \lambda_{\su} = 0$.  Because $L_{h,kk} \geq 0$, this
implies that $\lambda_{\su} = 0$.  Thus we have established the first
line of the claim~\eqref{eq:SE-opt}.

If either $\< \bxi_h , \buhat \>_{\Ltwo} = 0$ or
$\proj_{\bV_{k-1}^{\SE}}^\perp \bvhat \neq \bzero$, then the second
line of equation~\eqref{eq:SE-opt} is satisfied by an equivalent argument.
\item \emph{Either $\< \bxi_g , \bvhat \>_{\Ltwo} = 0$ or
$\proj_{\bU_{k-1}^{\SE}}^\perp \buhat \neq \bzero$. Likewise, either
$\< \bxi_h , \buhat \>_{\Ltwo} = 0$ or $\proj_{\bV_{k-1}^{\SE}}^\perp
\bvhat \neq \bzero$.}

Assume to the contrary that both $\< \bxi_g , \bvhat \>_{\Ltwo} > 0$
and $\proj_{\bU_{k-1}^{\SE}}^\perp \buhat = \bzero$.  Then
$\proj_{\bV_{k-1}^{\SE}}^\perp \bvhat \neq \bzero$ because $\bxi_g$ is
independent and hence orthogonal to $\bV_{k-1}^{\SE}$.  Then, by item
1 above, the second line of equation~\eqref{eq:SE-opt} is satisfied.  Note,
however, that because $\| \proj_{\bU_{k-1}^\SE} \buhat \|_{\Ltwo} =
0$, we have $\bg_k^{\SE} = \bg^{\Ltwo}(\buhat)$ is independent of
$\bxi_g$.  Thus the objective in the second line of
equation~\eqref{eq:SE-opt} is independent of $\bxi_g$, so that $\bv_k^{\SE}
= \bvhat$ is independent of $\bxi_g$.  We conclude that $\< \bxi_g ,
\bvhat \>_{\Ltwo} = 0$, contradicting that $\< \bxi_g,
\bvhat\>_{\Ltwo}$.  Thus, either $\< \bxi_g , \bvhat \>_{\Ltwo} = 0$
or $\proj_{\bU_{k-1}^{\SE}}^\perp \buhat \neq \bzero$.

Either $\< \bxi_h , \buhat \>_{\Ltwo} = 0$ or
$\proj_{\bV_{k-1}^{\SE}}^\perp \bvhat \neq \bzero$ by an
equivalent argument.
\end{enumerate}

Combining these two steps establishes the claim~\eqref{eq:SE-opt}.


\subsubsection{Step 3: Existence/uniqueness of fixed points}

By Steps 2a and 2b above, there exists a solution to the fixed point
equations~\eqref{eq:fixpt-inductive}.  We now establish uniqueness.

Consider two solutions $\bK_{g,k}$, $\bK_{h,k}$, $\bZ_{u,k}$,
$\bZ_{v,k}$, $\{\nu_{\ell,0}\}_{\ell \leq k}$,
$\{\nu_{\ell,\sx}\}_{\ell \leq k}$ and $\bK_{g,k}'$, $\bK_{h,k}'$,
$\bZ_{u,k}'$, $\bZ_{v,k}'$, $\{\nu_{\ell,0}'\}_{\ell \leq k}$,
$\{\nu_{\ell,\sx}'\}_{\ell \leq k}$, $\{\nu_{\ell,\su}'\}_{\ell \leq
  k}$.  The construction in Step 1 allows us to embed the
corresponding random variables $\bU_k^{\SE}$, $\bH_k^{\SE}$ and
${\bU_k^{\SE}}'$, ${\bH_k^{\SE}}'$ into $L_{\su}$ and $\bV_k^{\SE}$,
$\bG_k^{\SE}$ and ${\bV_k^{\SE}}'$, ${\bG_k^{\SE}}'$ into $L_{\sv}$
such that $(\bu_k^{\SE};\nu_{k,0},\bv_k^{\SE})$ and
$({\bu_k^{\SE}}';\nu_{k,0}',{\bv_k^{\SE}}')$ are both saddle points of
the problem~\eqref{eq:min-max-L2}.  By Step 2, we must have
$\bu_k^{\SE} = {\bu_k^{\SE}}'$ and $\bv_k^{\SE} = {\bv_k^{\SE}}'$.  We
also have $\bU_{k-1}^{\SE} = \bU_{k-1}^{\SE'}$ and $\bV_{k-1}^{\SE} =
\bV_{k-1}^{\SE'}$ by construction.  Thus, by the first line
of~\cref{eq:fixpt-inductive}, $\bK_{g,k} = \bK_{g,k}'$ and $\bK_{h,k}
= \bK_{h,k}'$.  By the third line of~\cref{eq:fixpt-inductive},
$\nu_{\ell,\sx} = \nu_{\ell,\sx}'$ and $\nu_{\ell,\su} =
\nu_{\ell,\su}'$ for $\ell \leq k$.

By~\cref{eq:SE-opt}, the joint distribution of $\bU_k^{\SE}$ and
$\bH_k^{\SE}$ is a function only of $\bK_{h,k}$, and likwise, the
joint distribution of $\bV_k^{\SE}$ and $\bG_k^{\SE}$ is a function
only of $\bK_{g,k}$.  Thus, we have $\llangle \bH_k^{\SE} ,
\bU_k^{\SE} \rrangle = \llangle {\bH_k^{\SE}}' , {\bU_k^{\SE}}'
\rrangle$ and $\llangle \bG_k^{\SE} , \bV_k^{\SE} \rrangle = \llangle
        {\bG_k^{\SE}}' , {\bV_k^{\SE}}' \rrangle$.  Multiplying the
        equations in the second line of
        display~\eqref{eq:fixpt-inductive} by $\bK_{h,k}^\ddagger$ and
        $\bK_{h,k}^\ddagger$, respectively, gives
\begin{equation}
\begin{gathered}
  \id_{\bK_{h,k}}^\top \bZ_{v,k}^\top = \bK_{h,k}^\ddagger \llangle
  \bH_k^{\SE} , \bU_k^{\SE} \rrangle, \\ \id_{\bK_{g,k}}^\top
  \bZ_{u,k}^\top = \bK_{g,k}^\ddagger \llangle \bG_k^{\SE} ,
  \bV_k^{\SE} \rrangle.
\end{gathered}
\end{equation}
Because $\bZ_{v,k}$ is innovation compatible with $\bK_{h,k}$,
$\bZ_{v,k}^\top$ is $0$ in rows with predictable index, so that the
first row in the preceding display uniquely determines
$\bZ_{v,k}^\top$ under the innovation compatibility constraint.  Thus,
$\bZ_{v,k} = \bZ_{v,k}'$.  Likewise, we have the equivalence
$\bZ_{u,k} = \bZ_{u,k}'$.

It remains to establish uniqueness of $\nu_{k,0}$.  Assume for the sake
of contradiction that $\nu_{k,0} \neq \nu_{k,0}'$, and consider the
optimization in the first line of the display~\eqref{eq:SE-opt} at
$\nu_{k,0}$ and $\nu_{k,0}'$.  Because $\bK_{h,k} = \bK_{h,k}'$ and
$\bZ_{u,k} = \bZ_{u,k}'$, consider two problems defined using the same
random variables $\bH_k^{\SE}$, $\bU_{k-1}^{\SE}$ and parameters
$\bZ_{u,k}$, $\nu_{k,\sx}$, but possibly different values of the
parameters $\nu_{k,0}$ and $\nu_{k,0}'$.  Denote objectives for the
two problems by $\Phi_{k,u}(\,\cdot\,;\bH_{k-1}^{\SE})$ and
$\Phi_{k,u}'(\,\cdot\,;\bH_{k-1}^{\SE})$.  Then
$\Phi_{k,u}'(\bu;\bH_{k-1}^{\SE}) = \Phi_{k,u}(\bu;\bH_{k-1}^{\SE})
-\langle \bu,\ones\>(\nu_{k,0}' - \nu_{k,0})$ (see
equation~\eqref{eq:SE-penalties}).  By the optimality of $\bu_k^{\SE}$ and
${\bu_k^{\SE}}'$, we have
\begin{align*}
  \Phi_{k,u} (\bu_k^{\SE};\bH_{k-1}^{\SE}) & \leq
  \Phi_{k,u}({\bu_k^{\SE}}';\bH_{k-1}^{\SE}) =
  \Phi_{k,u}'({\bu_k^{\SE}}';\bH_{k-1}^{\SE}) +
  \<{\bu_k^{\SE}}',\ones\>(\nu_{k,0}' - \nu_{k,0}) \\
  & \leq
  \Phi_{k,u}'({\bu_k^{\SE}};\bH_{k-1}^{\SE}) +
  \<{\bu_k^{\SE}}',\ones\>(\nu_{k,0}' - \nu_{k,0}) =
  \Phi_{k,u}({\bu_k^{\SE}};\bH_{k-1}^{\SE}) + \<{\bu_k^{\SE}}' -
  \bu_k^{\SE},\ones\>(\nu_{k,0}' - \nu_{k,0}),
\end{align*}
where equality holds in the first inequality if and only if
${\bu_k^{\SE}}'$ is also a minimizer of $\Phi_{k,u}$.  Because
$\Phi_{k,u}$ is strongly convex, this occurs if and only if
${\bu_k^{\SE}}' = \bu_k^{\SE}$.  In particular, $\<{\bu_k^{\SE}}' -
\bu_k^{\SE},\ones\>(\nu_{k,0}' - \nu_{k,0}) \geq 0$, with equality if
and only if ${\bu_k^{\SE}}' = \bu_k^{\SE}$.  By the final line
of~\cref{eq:fixpt-inductive}, $\E[\<{\bu_k^{\SE}}' -
  \bu_k^{\SE},\ones\>(\nu_{k,0}' - \nu_{k,0})] = 0$, whence the
equality conditions must hold almost surely.  That is, $\bu_k^{\SE} =
{\bu_k^{\SE}}'$ almost surely.  Now using that $\bu_k^{\SE} =
{\bu_k^{\SE}}'$, we have that $\bzero \in \partial_{\bu}
\Phi_{k,u}(\bu_k^{\SE};\bH_{k-1}^{\SE})$ and $\bzero \in
\partial_{\bu} \Phi_{k,u}'(\bu_k^{\SE};\bH_{k-1}^{\SE})$, whence
$(\nu_{k,0}' - \nu_{k,0}) \ones \in \partial_{\bu}
\Phi_{k,u}(\bu_k^{\SE};\bH_{k-1}^{\SE})$ almost surely.  Note that
because the function $\ell_\prop$ is strongly convex, its conjugate
dual $\ell_{\prop}^*$ must be smooth.  This fact combined with
~\cref{eq:SE-penalties} ensures that the function $\Phi_{5,u}$ is
smooth.

Further, the function $\Phi_{6,u}$ is smooth in those coordinates for
which $y_{1,i}^{\SE} \neq 0$, and with positive probability
$y_{1,i}^{\SE} \neq 0$ for some $i$.  Thus, with positive probability,
we cannot have both $\bzero$ and $c\ones$ in $\partial_{\bu}
\Phi_{k,u}(\bu_k^{\SE};\bH_{k-1}^{\SE})$ for some $c \neq 0$.  We
conclude that $\nu_{k,0}' - \nu_{k,0} = 0$, so that uniqueness of
$\nu_{k,0}$ is established.

\subsection{Proof of~\Cref{lem:fixed-pt-bound}}
\label{sec:fix-pt-bound-proof}

\begin{proof}[Proof of \Cref{lem:fixed-pt-bound}]
    Without loss of generality, assume that $\Omega_k(\bzero) = 0$ for $k = 5,6$.
    \\

    \noindent \textbf{Locations of zeros in $\bK_g$, $\bK_h$, $\bZ_v$, and $\bZ_u$.}
    We established that the upper-left $4\times4$ blocks of $\bZ_v$, $\bZ_u$, $\bK_g$, and $\bK_h$ are of the claimed form when proving the base case in the proof of \Cref{lem:fixed-pt-soln}.
    This implies the lower-left and upper-right $2 \times 2$ block of $\bK_g$ middle $2\times 2$ block in the last two rows and the middle $2\times 2$ block in the last two columns of $\bK_h$ must be $\bzero_{2 \times 2}$.
    This confirms the location of zeroes in $\bK_g,\bK_h$ asserted by the lemma.
    The location of all zeroes in $\bZ_v$, $\bZ_u$ asserted by the lemma,
    except for $\zeta_{65}^v = 0$ and $\zeta_{65}^u = 0$,
    are confirmed by \Cref{lem:eff-reg-explicit}(b).
    We show $\zeta_{65}^v = 0$ and $\zeta_{65}^u = 0$ below.
    \\

    \noindent \textbf{Bound on $\nu_{0,k}$.}
    Because $\Omega_k$ is $c$-strongly convex, has minimizer with $\ell_2$-norm bounded by $c$,
    and $\Omega_k(\bzero) = 0$,
    we have that $\Omega_k(\bv) \geq - C + \frac{c}{2} \| \bv \|^2$.
    Below we construct $\bu_+ \in L_\su^2$ such that $\| \bu_+ \|_{\Ltwo} \leq C/\sqrt{n}$, $\< \ones , \bu_+ \>_{\Ltwo} \geq c > 0$,
    and $\E[\ell_k^*(n\bu_+;\bw^{\SE},\by_1^{\SE},\by_2^{\SE})] \leq C$.
    Assuming we find such a $\bu_+$,
    we get
    \begin{equation}
    \begin{aligned}
        &\min_{\< \bv , \bxi_g \>_{\Ltwo} \geq 0}
            \;
            \max_{\< \bu , \bxi_h \>_{\Ltwo} \geq 0}
            \AuxObj_k^{\Ltwo}(\bu;v_0,\bv)
            \geq
            \min_{\bv \in L_{\sv}^2}
            \AuxObj_k^{\Ltwo}(\bu_+;v_0,\bv)
        \\
            &=
            \min_{\bv \in L_{\sv}^2}
            -\big\< \bg^{\Ltwo}\big(\bu_+\big) , \bv \big\>_{\Ltwo}
            +\big\< \bh^{\Ltwo}(\bv) , \bu_+ \big\>_{\Ltwo}
            +\< \bu_+ , \ones \>_{\Ltwo} (v_0 + \< \bmu_{\sx} , \bv \>_{\Ltwo})
            -\E[\ell_k^*(n\bu_+;\bw^{\SE},\by_1^{\SE},\by_2^{\SE})]
            +\E[\Omega_k(\bv)]
        \\
            &\geq
            \min_{\bv \in L_{\sv}^2}
            -C\|\bv\|_{\Ltwo}
            + \< \bu_+,\ones\>_{\Ltwo} v_0
            -C+ \frac{c}{2} \| \bv \|_{\Ltwo}^2
        \\
            &\geq
            \< \bu_+ , \ones \>_{\Ltwo} v_0
            -C,
    \end{aligned}
    \end{equation}
    where in the first inequality we have used the bounds on $\Omega_k$ and the properties of $\bu_+$ given above,
    as well as the fact that 
    $\big\|\bg^{\Ltwo}(\bu_+)\big\|_{\Ltwo} = \sqrt{p}\|\bu_+\|_{\Ltwo} \leq C$,
    $\big\|\bh^{\Ltwo}(\bv)\big\|_{\Ltwo} \leq \sqrt{n}\|\bv\|_{\Ltwo}$,
    and
    $\|\bmu_{\sx}\| \leq C$.
    We further have that
    \begin{equation}
    \begin{aligned}
        \min_{\substack{v_0 \in \reals \\ \< \bv , \bxi_g \>_{\Ltwo} \geq 0}}
            \;
            \max_{\< \bu , \bxi_h \>_{\Ltwo} \geq 0}
            \AuxObj_k^{\Ltwo}(\bu;v_0,\bv)
            &\leq
            \max_{\< \bu , \bxi_h \>_{\Ltwo} \geq 0}
            \AuxObj_k^{\Ltwo}(\bu;0,\bzero)
            =
            \max_{\< \bu , \bxi_h \>_{\Ltwo} \geq 0}
            -\E[\ell_k^*(n \bu ; \bw^{\SE},\by_1^{\SE},\by_2^{\SE})]
        \\
            &\leq
            \frac{C}{n} + \frac{c}{2n} \| \by_2^{\SE} \|_{\Ltwo}^2
            \leq C,
    \end{aligned}
    \end{equation}
    where the final inequality holds by the following logic in the case $k = 5$ and $k = 6$.
    For $k = 5$, we use that, by equation~\eqref{eq:ellk*-def},
    $\ell_5^*(n\bu;\bw^{\SE},\by_1^{\SE},\by_2^{\SE}) = \frac{1}{n} \sum_{i=1}^n \sup_{\eta \in \reals} \{n\eta u_i - \ell_\prop(\eta;y_{1,i}^{\SE})\} \geq -\frac{1}{n} \sum_{i=1}^n \ell_\prop(0;y_{1,i}^{\SE}) \geq -C$ by Assumption A1.
    For $k = 6$,
    we use that, by equation~\eqref{eq:ellk*-def}, and using that $w(h)$ is bounded above by $C$ by Assumption A1,
    we have that $\ell_6^*(n\bu;\bw,\by_1,\by_2) \geq - \| \by_2 \| \| \bu \| - \frac{C}{n} + \frac{cn}{2} \| \bu \|^2$.
    We have that $\|\by_2^{\SE}\|_{\Ltwo} \leq C$ using equation~\eqref{eq:yw-func-of-Heps} and that $\| \bh_2 \|_{\Ltwo}^2/n = \| \btheta_2 \|^2 \leq C$, $\|\beps_2\|_{\Ltwo}^2/n = \sigma^2 < C$,
    and $|\theta_{2,0} + \< \bmu_{\sx},\btheta_2 \>_{\Ltwo}| \leq C$.
    Combining the previous two displays, and using that $\< \bu_+ , \ones \>_{\Ltwo} > c > 0$ gives that the minimum of $\min_{\< \bv , \bxi_g \>_{\Ltwo} \geq 0}\max_{\< \bu , \bxi_h \>_{\Ltwo} \geq 0}\AuxObj_k^{\Ltwo}(\bu;v_0,\bv)$ over $v_0$ must be bounded above by $C$.
    If we replace $\bu_+$ by $\bu_- \in L_\su^2$ which satisfies $\| \bu_- \|_{\Ltwo} \leq C/\sqrt{n}$, $\< \ones , \bu_- \>_{\Ltwo} \geq c > 0$,
    and $\E[\ell_k^*(n\bu_-;\bw,\by_1,\by_2)] \leq C$,
    we could by the same argument show that $v_0$ is bounded below by $-C$.
    
    All that remains is to construct $\bu_+$ and $\bu_-$.
    First consider $k = 5$.
    Note that 
    $\ell_\prop^*(\partial_\eta\ell_\prop(0;0);0) = -\ell_\prop(0;0) \leq 0$,
    $\ell_\prop^*(0;0) = \sup_{\eta \in \reals} \{ -\ell_\prop(\eta;0) \} \leq 0$,
    $\ell_\prop^*(\partial_\eta\ell_\prop(0;1);1) = -\ell_\prop(0;1)$,
    and
    $\ell_\prop^*(0;1) = \sup_{\eta \in \reals} \{ -\ell_\prop(\eta;1) \}$.
    Set $u_{+,i} = \partial_\eta\ell_\prop(0;0) \indic\{y_{1,i}^{\SE} = 0\}/n$
    and $u_{-,i} = \partial_\eta\ell_\prop(0;1) \indic\{y_{1,i}^{\SE} = 0\}/n$.
    Then, by Assumption A1 (which bounds $|\partial_\eta\ell_\prop(0;0)|,|\partial_\eta\ell_\prop(0;1)|<C$),
    we have $\| \bu_+ \|_{\Ltwo},\|\bu_-\|_{\Ltwo} \leq C/\sqrt{n}$.
    We moreover have that, by Assumption A1,
    $\< \ones , \bu_+ \>_{\Ltwo} = \partial_\eta\ell_\prop(0;0) \P(y_{1,i}^{\SE} = 0) = (1-\barpi) \partial_\eta\ell_\prop(0;0) > c > 0$,
    and likewise, $\< \ones , \bu_- \>_{\Ltwo} = \partial_\eta\ell_\prop(0;1) \P(y_{1,i}^{\SE} = 1) = \barpi \partial_\eta\ell_\prop(0;0) < -c < 0$.
    Then,
    we have $\ell_5^*(n\bu_+;\bw^{\SE},\by_1^{\SE},\by_2^{\SE}) = \frac{1}{n} \sum_{i=1}^n \ell_\prop^*(\partial_\eta\ell_\prop(0,0)\indic\{y_{1,i}^{\SE}=0\};y_{1,i}^{\SE}) \leq 0$.
    where we use $\ell_\prop^*(\partial_\eta\ell_\prop(0;0);0) = -\ell_\prop(0;0) \leq 0$ and
    $\ell_\prop^*(0;0) = \sup_{\eta \in \reals} \{ -\ell_\prop(\eta;0) \} \leq 0$.
    Verifying that $\ell_5^*(n\bu_+;\bw^{\SE},\by_1^{\SE},\by_2^{\SE})\leq 0$ follows similarly,
    but instead uses 
    $\ell_\prop^*(\partial_\eta\ell_\prop(0;1);1) = -\ell_\prop(0;1)$
    and
    $\ell_\prop^*(0;1) = \sup_{\eta \in \reals} \{ -\ell_\prop(\eta;1) \}$.

    Next, consider $k = 6$.  In this case, we set $\bu_+ =
    \by_1^{\SE}/n$ and $\bu_- = -\by_1^{\SE}/n$.  Because the entries
    of $\by_1^{\SE}$ are either $0$ or $1$, it follows that $\| \bu_+
    \|_{\Ltwo}, \| \bu_- \|_{\Ltwo} \leq C/\sqrt{n}$.  Moreover, $\<
    \ones ,\bu_+ \>_{\Ltwo} = -\< \ones , \bu_-\>_{\Ltwo} = \barpi > c
    > 0$.  Moreover, because $(y_{1,i}^{\SE})^2/y_{1,i} =
    y_{1,i}^{\SE}$ (when $y_{1,i}^{\SE} = 0$, this is 0/0, which we
    have by convention set equal to 0, and when $y_{1,i}^{\SE} = 1$,
    this is 1), we have
    $\E[\ell_6^*(n\bu_+;\bw^{\SE},\by_1^{\SE},\by_2^{\SE})] = - \<
    \by_1^{\SE} , \by_2^{\SE} \>_{\Ltwo}/n + (n/2) \sum_{i=1}^n
    (w_i^{\SE})^{-1}y_{1,i}^{\SE}$.  Because $w(h)$ is bounded below
    by $c > 0$ by assumption, and $\| \by_1^{\SE} \|_{\Ltwo}, \|
    \by_2^{\SE}\|_{\Ltwo} \leq C\sqrt{n}$, we get
    $\E[\ell_6^*(n\bu_+;\bw^{\SE},\by_1^{\SE},\by_2^{\SE})] \leq C$.
    Showing $\E[\ell_6^*(n\bu_-;\bw^{\SE},\by_1^{\SE},\by_2^{\SE})]
    \leq C$ holds similarly.  \\

    \noindent \textbf{Upper bound on $|L_{g,k\ell}|$, $|L_{h,k\ell}|$, $|K_{g,k\ell}|$.}
    By Assumption A1,
    $\|\bv_1^{\SE}\|_{\Ltwo} = \|\btheta_1\|$,
    $\|\bv_2^{\SE}\|_{\Ltwo} = \|\btheta_2\|$,
    $\|\bv_3^{\SE}\|_{\Ltwo} = 0$,
    and
    $\|\bv_4^{\SE}\|_{\Ltwo} = 0$
    are all upper bounded by $C$,
    and
    $\|\bu_1^{\SE}\|_{\Ltwo} = 0$,
    $\|\bu_2^{\SE}\|_{\Ltwo} = 0$,
    $\|\bu_3^{\SE}\|_{\Ltwo} = \| \ones \|/n = 1/\sqrt{n}$,
    $\|\bu_4^{\SE}\|_{\Ltwo} = \|\by_1^{\SE}\|_{\Ltwo}/n \leq 1/\sqrt{n}$
    are all upper bounded by $C/\sqrt{n}$.
    Upper bounds on 
    $\| \bv_5^{\SE} \|_{\Ltwo}$,
    $\| \bv_6^{\SE} \|_{\Ltwo}$,
    $\| \bu_5^{\SE} \|_{\Ltwo}$,
    and
    $\| \bu_6^{\SE} \|_{\Ltwo}$
    require some more work.

    By optimality of $\bv_k^{\SE}$, 
    we have
    \begin{equation}
    \begin{aligned}
        &\AuxObj_k^{\Ltwo}(\bu_k^{\SE};\nu_{0,k},\bv_k^{\SE})
            =
            \max_{\< \bu , \bxi_h \>_{\Ltwo} \geq 0}
            \AuxObj_k^{\Ltwo}(\bu;\nu_{0,k},\bv_k^{\SE})
        \\
            &\qquad\geq
            \AuxObj_k^{\Ltwo}(\bzero;\nu_{0,k},\bv_k^{\SE})
            =
            -\E[\ell_k^*(\bzero;\bw^{\SE},\by_1^{\SE},\by_2^{\SE})]
            +\E[\Omega_k(\bv_k^{\SE})]
            \geq
            -C + \frac{c}{2}\| \bv_k^{\SE} \|_{\Ltwo}^2,
    \end{aligned}
    \end{equation}
    where we use that $\Omega_k(\bv) \geq - C + \frac{c}{2} \| \bv \|^2$ (as we argued above);
    $\ell_5^*(\bzero;\bw^{\SE},\by_1^{\SE},\by_2^{\SE}) = \frac{1}{n} \sum_{i=1}^n \sup_{\eta \in \reals}\{ - \ell_\prop^*(\eta;y_{1,i}^{\SE})\} \geq -C$ because $\ell_\prop^*(0;0),\ell_\prop^*(0;1) \leq $ $C$ by Assumption A1;
    and $\ell_6^*(\bzero;\bw^{\SE},\by_1^{\SE},\by_2^{\SE}) = 0$ by equation~\eqref{eq:ellk*-def}.
    By optimality of $\bu_k^{\SE}$,
    we have
    \begin{equation}
    \begin{aligned}
            &\AuxObj_k^{\Ltwo}(\bu_k^{\SE};\nu_{0,k},\bv_k^{\SE})
            =
            \min_{\< \bv , \bxi_g \>_{\Ltwo} \geq 0}
            \AuxObj_k^{\Ltwo}(\bu_k^{\SE};\nu_{0,k},\bv)
        \\
            &\qquad\leq
            \AuxObj_k^{\Ltwo}(\bu_k^{\SE};\nu_{0,k},\bzero)
            =
            -\E[\ell_k^*(\bu_k^{\SE};\bw^{\SE},\by_1^{\SE},\by_2^{\SE})]
            +\E[\Omega_k(\bzero)]
            \leq
            C - \frac{cn}{2}\| \bu_k^{\SE} \|_{\Ltwo}^2,
    \end{aligned}
    \end{equation}
    where we use that $\Omega_k(\bzero) = 0$ (which we assumed at the beginning of this section without loss of generality);
    that, $\ell_k^*$ is $cn$-strongly convex in $\bu$ because $\ell_k$ are strongly smooth in $\bmeta$ (for $k = 5$, this is by assumption, and for $k = 6$, we use that the weight function $w(h)$ is bounded above).
    Chaining the previous displays together,
    we conclude both that $\| \bv_k^{\SE} \|_{\Ltwo}^2 \leq C$ and $\| \bu_k^{\SE} \|_{\Ltwo}^2 \leq C/n$ for $k = 5,6$.

    Having bounded $\| \bv_k^{\SE} \|_{\Ltwo}^2 \leq C$ for $k = 1,\ldots,6$,
    by the fixed point equations (equation~\eqref{eq:fixpt-general}) we get that $|K_{h,k\ell}| \leq C$ for all $1 \leq k,\ell \leq 6$.
    Having bounded $\| \bu_k^{\SE} \|_{\Ltwo}^2 \leq C/n$ for $k = 1,\ldots,6$,
    by the fixed point equations (equation~\eqref{eq:fixpt-general}) we get that $|K_{g,k\ell}| \leq C/n$ for all $1 \leq k,\ell \leq 6$.
    Because $\sum_{\ell=1}^k L_{g,k\ell}^2 = K_{g,kk}$, we also have $|L_{g,k\ell}| \leq C/\sqrt{n}$,
    and because $\sum_{\ell = 1}^k L_{h,k\ell}^2 = K_{h,kk}$, we also have $|L_{h,k\ell}| \leq C$,
    for $1 \leq \ell \leq k \leq 6$.
    \\

    \noindent \textbf{Upper bound on $|\nu_{k,\sx}|$.}
    The fixed point equations \eqref{eq:fixpt-general} state $\nu_{k,\sx} = \< \bmu_{\sx} , \bv_k^{\SE} \>_{\Ltwo}$.
    Thus, the upper bound on $|\nu_{k,\sx}|$ holds by Cauchy-Schwartz and because $\|\bmu_{\sx}\| \leq C$ (by Assumption A1) and because $\| \bv_k^{\SE} \|_{\Ltwo} \leq C$ (as we have just shown).
    \\

    \noindent \textbf{Lower bound on $L_{g,kk}$ and $L_{h,kk}$.}
    For $k=5,6$,
    the KKT conditions for the first line of \eqref{eq:SE-opt} is
    \begin{equation}
    \begin{gathered}
        (\nu_{k,0} + \nu_{k,\sx})\ones
            +
            \bh_k^{\SE}
            -
            \sum_{\ell = 1}^{k-1} \zeta_{k\ell}^u \bu_\ell^{\SE}
            -
            \zeta_{kk}^u \bu_k^{\SE} 
            \in
            \partial_{\bu}\ell_k^*(\bu_k^{\SE};\bw^{\SE},\by_1^{\SE},\by_2^{\SE}).
    \end{gathered}
    \end{equation}
    By Fenchel-Legendre duality,
    the first of these is equivalent to
    \begin{equation}
    \label{eq:ukse-is-ellk-deriv}
        n\bu_k^{\SE}
            =
            \nabla \ell_k\Big(
                (\nu_{k,0} + \nu_{k,\sx})\ones
                +
                \bh_k^{\SE}
                -
                \sum_{\ell = 1}^{k-1} \zeta_{k\ell}^u \bu_\ell^{\SE}
                -
                \zeta_{kk}^u \bu_k^{\SE}
                ;
                \bw^{\SE},\by_1^{\SE},\by_2^{\SE}
            \Big),
    \end{equation}
    where $\ell_5(\bmeta;\bw,\by_1,\by_2) \defn \sum_{i=1}^n
    \ell_\prop(\eta_i;y_{1,i})$ and $\ell_6(\bmeta;\bw,\by_1,\by_2)
    \defn \frac{1}2 \sum_{i=1}^n y_{1,i} w_i(y_{2,i} - \eta_i)^2$.  Let
    $\proj_{\bU_{k-1}^{\SE}}^\perp$ denote the orthogonal projection
    in Hilbert space: $\proj_{\bU_{k-1}^{\SE}}^\perp\bu = \bu -
    \sum_{\ell=1}^{k-1} \bu_\ell^{\SE,\perp} \< \bu ,
    \bu_\ell^{\SE,\perp} \>_{\Ltwo}/\|\bu_\ell^{\SE,\perp}\|_{\Ltwo}^2$,
    where we adopt the convention that $\bzero/\bzero = \bzero$.  Multiplying both
    sides of the second equation in the second line of
    equation~\eqref{eq:fixpt-general} by $\bL_g^\dagger$ gives
    $\bL_g^{\ddagger\top}\bZ_u^\top = \llangle \bG^{\SE,\perp} ,
    \bV^{\SE} \rrangle_{\Ltwo}$.  Using that $\bL_g^\ddagger$ is
    lower-triangular and that $L_{g,kk} = \|
    \proj_{\bU_{k-1}^{\SE}}^\perp \bu_k^{\SE} \|_{\Ltwo}$ by the first
    line of equation~\eqref{eq:fixpt-general}, we get $\zeta_{kk}^u = \<
    \bg_k^{\SE,\perp} , \bv_k^{\SE} \>_{\Ltwo} / \|
    \proj_{\bU_{k-1}^{\SE}}^\perp \bu_k^{\SE} \|_{\Ltwo}$ (where we use
    the convention $0/0 = 0$).  Thus, $\| \zeta_{kk}^u \bu_k^{\SE} -
    \zeta_{kk}^u \proj_{\bU_{k-1}^{\SE}} \bu_k^{\SE} \|_{\Ltwo} = |\<
    \bg_k^{\SE,\perp},\bv_k^{\SE}\>_{\Ltwo}| \leq \sqrt{p} \|
    \proj_{\bV_{k-1}^{\SE}}^\perp \bv_k^{\SE} \|_{\Ltwo}$, because $\|
    \bg_k^{\SE,\perp} \|_{\Ltwo} \leq \sqrt{p}$.  Moreover, by the
    second equation in the first line of equation~\eqref{eq:fixpt-general},
    $\|\proj_{\bH_{k-1}^{\SE}}^\perp \bh_k^{\SE}\|_{\Ltwo} = \sqrt{n} \|
    \proj_{\bV_{k-1}^{\SE}}^\perp \bv_k^{\SE} \|_{\Ltwo}$.  Thus,
    because $\nabla \ell_k$ are $C$-Lipschitz by Assumption
    A1 (for $ k = 5$, this is by assumption; for $k = 6$, we
    use that the weights $w(h)$ are bounded above by $C$), we have
    \begin{equation}
    \label{eq:ukse-approx-grad}
    \begin{gathered}
        \Big\|
            n\bu_k^{\SE}
            -
            \nabla \ell_k\Big(
                (\nu_{k,0} + \nu_{k,\sx})\ones
                +
                \proj_{\bH_{k-1}^{\SE}} \bh_k^{\SE}
                -
                \sum_{\ell = 1}^{k-1} \zeta_{k\ell}^u \bu_\ell^{\SE}
                -
                \zeta_{kk}^u \proj_{\bU_{k-1}^{\SE}} \bu_k^{\SE}
                ;
                \bw^{\SE},\by_1^{\SE},\by_2^{\SE}
            \Big)
        \Big\|_{\Ltwo}
        \leq
        C\sqrt{n}\| \proj_{\bV_{k-1}^{\SE}}^\perp \bv_k^{\SE} \|_{\Ltwo}.
    \end{gathered}
    \end{equation}
    Note that $ \nabla \ell_k\big( (\nu_{k,0} + \nu_{k,\sx})\ones +
    \proj_{\bH_{k-1}^{\SE}} \bh_k^{\SE} - \sum_{\ell = 1}^{k-1}
    \zeta_{k\ell}^u \bu_\ell^{\SE} - \zeta_{kk}^u
    \proj_{\bU_{k-1}^{\SE}} \bu_k^{\SE} ; \bw,\by_1,\by_2 \big) $ is
    independent of $\bh_k^{\SE,\perp}$ because
    $\proj_{\bH_{k-1}^{\SE}} \bh_k^{\SE}$ is in the span of
    $\bH_{k-1}^{\SE}$, $\proj_{\bU_{k-1}^{\SE}} \bu_k^{\SE}$ is in the
    span of $\bU_{k-1}^{\SE}$ which is a function of $\bH_{k-1}^{\SE}$
    and the auxiliary noise vectors $\beps_1^{\SE},\beps_2^{\SE}$, and
    $\bw^{\SE},\by_1^{\SE},\by_2^{\SE}$ are functions of $\bH_2^{\SE}$
    and the auxiliary noise vectors.  Thus, their inner product in
    $L_\su^2$ is 0.  Then, the previous display and the bound $\|
    \zeta_{kk}^u \bu_k^{\SE} - \zeta_{kk}^u \proj_{\bU_{k-1}^{\SE}}
    \bu_k^{\SE} \|_{\Ltwo} \leq \sqrt{p} \|
    \proj_{\bV_{k-1}^{\SE}}^\perp \bv_k^{\SE} \|_{\Ltwo}$ imply $|\<
    \bh_k^{\SE,\perp} , \bu_k^{\SE} \>_{\Ltwo}| \leq C\|
    \proj_{\bV_{k-1}^{\SE}}^\perp \bv_k^{\SE} \|_{\Ltwo}$.

    For $k=5,6$,
    the KKT conditions for the second line of \eqref{eq:SE-opt} is
    \begin{equation}
    \begin{gathered}
        \bg_k^{\SE}
            -
            \sum_{\ell = 1}^{k-1} \zeta_{k\ell}^v \bv_\ell^{\SE}
            -
            \zeta_{kk}^v \bv_k^{\SE}
            -
            \nabla\Omega_k(\bv_k^{\SE})
            =
            \bzero.
    \end{gathered}
    \end{equation}
    Similarly to how we showed $\zeta_{kk}^u = \< \bg_k^{\SE,\perp} , \bv_k^{\SE} \>_{\Ltwo} / \| \proj_{\bU_{k-1}^{\SE}}^\perp \bu_k^{\SE} \|_{\Ltwo}$,
    we have that $\zeta_{kk}^v = \< \bh_k^{\SE,\perp} , \bu_k^{\SE} \>_{\Ltwo} / \| \proj_{\bV_{k-1}^{\SE}}^\perp \bv_k^{\SE} \|_{\Ltwo}$,
    whence
    $\| \zeta_{kk}^v \bv_k^{\SE} - \zeta_{kk}^v \proj_{\bV_{k-1}^{\SE}} \bv_k^{\SE} \|_{\Ltwo} = |\< \bh_k^{\SE,\perp},\bu_k^{\SE}\>_{\Ltwo}| \leq C \| \proj_{\bV_{k-1}^{\SE}}^\perp \bv_k^{\SE} \|_{\Ltwo}$,
    where in the last inequality we have used the bound established at the end of the previous paragraph.
    Using this bound and the fact that $\nabla \Omega_k$ is $C$-Lipschitz by Assumption A1,
    we have
    \begin{equation}
        \Big\|
            \proj_{\bG_{k-1}^{\SE}}^\perp\bg_k^{\SE}
            +
            \proj_{\bG_{k-1}^{\SE}}\bg_k^{\SE}
            -
            \sum_{\ell = 1}^{k-1} \zeta_{k\ell}^v \bv_\ell^{\SE}
            -
            \zeta_{kk}^v \proj_{\bV_{k-1}^{\SE}}\bv_k^{\SE}
            -
            \nabla\Omega_k(\proj_{\bV_{k-1}^{\SE}}\bv_k^{\SE})
        \Big\|_{\Ltwo}
        \leq
        C\| \proj_{\bV_{k-1}^{\SE}}^\perp \bv_k^{\SE} \|_{\Ltwo}.
    \end{equation}
    Because $\proj_{\bG_{k-1}^{\SE}}^\perp\bg_k^{\SE}$ is independent of
    $
    \proj_{\bG_{k-1}^{\SE}}\bg_k^{\SE}
    -
    \sum_{\ell = 1}^{k-1} \zeta_{k\ell}^v \bv_\ell^{\SE}
    -
    \zeta_{kk}^v \proj_{\bV_{k-1}^{\SE}}\bv_k^{\SE}
    -
    \nabla\Omega_k(\proj_{\bV_{k-1}^{\SE}}\bv_k^{\SE}),
    $
    their inner product in $L_\sv^2$ is 0.
    Thus, the left-hand side of the previous display is lower bounded by $\big\| \proj_{\bG_{k-1}^{\SE}}^\perp \bg_k^{\SE}\big\|_{\Ltwo} = \sqrt{p} \| \proj_{\bU_{k-1}^{\SE}}^\perp \bu_k^{\SE}\|_{\Ltwo}$,
    with the equality holding by the first equation in the first line of equation~\eqref{eq:fixpt-general}.
    Because $n/p \leq C$,
    we thus have $C\|\proj_{\bV_{k-1}^{\SE}}^\perp \bv_k^{\SE}\|_{\Ltwo} \geq C\sqrt{n}\| \proj_{\bU_{k-1}^{\SE}}^\perp \bu_k^{\SE} \|_{\Ltwo}$.
    Moreover,
    \begin{equation}
    \begin{aligned}
        &\Big\|
            \proj_{\bU_{k-1}^{\SE}}^\perp
            \nabla \ell_k\big(
                (\nu_{k,0} + \nu_{k,\sx})\ones
                +
                \proj_{\bH_{k-1}^{\SE}} \bh_k^{\SE}
                -
                \sum_{\ell = 1}^{k-1} \zeta_{k\ell}^u \bu_\ell^{\SE}
                -
                \zeta_{kk}^u \proj_{\bU_{k-1}^{\SE}} \bu_k^{\SE}
                ;
                \bw,\by_1,\by_2
            \big)
        \Big\|_{\Ltwo}
    \\
        &\geq
        \sqrt{
            \E\Big[
            \Tr\Big(
            \Var\Big(
                \nabla \ell_k\big(
                    (\nu_{k,0} + \nu_{k,\sx})\ones
                    +
                    \proj_{\bH_{k-1}^{\SE}} \bh_k^{\SE}
                    -
                    \sum_{\ell = 1}^{k-1} \zeta_{k\ell}^u \bu_\ell^{\SE}
                    -
                    \zeta_{kk}^u \proj_{\bU_{k-1}^{\SE}} \bu_k^{\SE}
                    ;
                    \bw,\by_1,\by_2
                \big)
                \Bigm| \bU_{k-1}^{\SE}, \bH_{k-1}^{\SE}
            \Big)
            \Big)  
            \Big]  
        }.
    \end{aligned}
    \end{equation}
    Note that the argument $(\nu_{k,0} + \nu_{k,\sx})\ones +
    \proj_{\bH_{k-1}^{\SE}} \bh_k^{\SE} - \sum_{\ell = 1}^{k-1}
    \zeta_{k\ell}^u \bu_\ell^{\SE} - \zeta_{kk}^u
    \proj_{\bU_{k-1}^{\SE}} \bu_k^{\SE}$ is a function of
    $\bU_{k-1}^{\SE}$, $\bH_{k-1}^{\SE}$.  Consider $k = 5$.  Then
    $\nabla \ell_k(\bmeta;\bw,\by_1,\by_2) = \nabla_{\bmeta}
    \ell_\prop(\bmeta;\by_1)$. Recall by Assumption A1,
    $\partial_\eta \ell_\prop(\eta;0) - \partial_\eta
    \ell_\prop(\eta;1) \geq c$.  Further, $y_{1,i} \mid
    \bU_{k-1}^{\SE}, \bH_{k-1}^{\SE} \sim \Bernoulli(\pi(\mu_\prop +
    h_{5,i}^{\SE}))$.  By the decay and growth conditions on $\pi$ in
    Assumption A1 and the upper bound on $L_{h,55}$ proved
    above, we have that the right-hand side of the previous display is
    bounded below by $C\sqrt{n}$ when $k = 5$.  When $k = 6$, we have
    $\nabla \ell_k(\bmeta;\bw,\by_1,\by_2) = \bw \odot \by_1 \odot
    (\by_2 - \bmeta)$.  Using the upper and lower bound $w$ and the
    fact that $\barpi > c > 0$ by Assumption A1, and because
    $\by_2 = \mu_\out + \bh_2 + \beps_2$, and $\beps_2$ is independent
    of $\bU_{5}^{\SE}$, $\bH_{5}^{\SE}$ and $\bw,\ba$, we have that
    the right-hand side of the previous display is bounded below by
    $C\sqrt{n}$ when $k = 6$ as well.

    Thus, by equation~\eqref{eq:ukse-approx-grad},
    \begin{equation}
        \big\|
            \proj_{\bU_{k-1}^{\SE}}^\perp\bu_k^{\SE}
        \big\|_{\Ltwo}
            \geq
            C/\sqrt{n} - C\| \proj_{\bV_{k-1}^{\SE}}^\perp \bv_k^{\SE} \|_{\Ltwo}/\sqrt{n}.
    \end{equation}
    Combining this with $C\|\proj_{\bV_{k-1}^{\SE}}^\perp \bv_k^{\SE}\|_{\Ltwo} \geq C\sqrt{n}\| \proj_{\bU_{k-1}^{\SE}}^\perp \bu_k^{\SE} \|_{\Ltwo}$ implies $\|\proj_{\bV_{k-1}^{\SE}}^\perp \bv_k^{\SE}\|_{\Ltwo} \geq c > 0$.
    By the second equation in the first line of equation~\eqref{eq:fixpt-general},
    this implies that $L_{h,kk} \geq c > 0$.

    Next, observe that the KKT condition for equation~\eqref{eq:SE-opt} is equivalent to
    \begin{equation}
        \bv_k 
            =
            \nabla \Omega_k^*\Big(
                \bg_k^{\SE} - \sum_{\ell=1}^{k-1}\zeta_{k\ell}^v \bv_\ell^{\SE} - \zeta_{kk}^v \bv_k^{\SE}
            \Big).
    \end{equation}
    By Assumption A1 (strong convexity of $\Omega_k$),
    we have that $\nabla \Omega_k^*$ is $C$-Lipschitz.
    Thus,
    using that $\| \zeta_{kk}^v \bv_k^{\SE} - \zeta_{kk}^v \proj_{\bV_{k-1}^{\SE}} \bv_k^{\SE} \|_{\Ltwo} \leq C \| \proj_{\bV_{k-1}^{\SE}}^\perp \bv_k^{\SE} \|_{\Ltwo}$,
    (which we have established above)
    and that $\|\proj_{\bG_{k-1}^{\SE}}^\perp \bg_k^{\SE}\|_{\Ltwo} = \sqrt{p}\|\proj_{\bU_{k-1}^{\SE}}^\perp \bu_k^{\SE}\|_{\Ltwo}$ (by the first equation in the first line of \eqref{eq:fixpt-general}),
    \begin{equation}
        \big\|
            \bv_k
            -
            \nabla \Omega_k^*
            \big(
                \proj_{\bG_{k-1}^{\SE}} \bg_k^{\SE}
                -
                \sum_{\ell=1}^{k-1} \zeta_{k\ell}^v \bv_\ell^{\SE}
                -
                \zeta_{kk}^v \proj_{\bV_{k-1}^{\SE}}\bv_k^{\SE}
            \big)
        \big\|_{\Ltwo}
        \leq
        C\sqrt{n}\| \proj_{\bU_{k-1}^{\SE}}^\perp \bu_k^{\SE} \|_{\Ltwo}.
    \end{equation}
    Because
    $
    \nabla \Omega_k^*
    \big(
        \proj_{\bG_{k-1}^{\SE}} \bg_k^{\SE}
        -
        \sum_{\ell=1}^{k-1} \zeta_{k\ell}^v \bv_\ell^{\SE}
        -
        \zeta_{kk}^v \proj_{\bV_{k-1}^{\SE}}\bv_k^{\SE}
    \big)
    $
    is independent of $\bg_k^{\SE,\perp}$,
    their inner product in $L_\sv^2$ is 0.
    Thus, the previous display and the fact that $\|\bg_k^{\SE,\perp}\|_{\Ltwo}$ is either 0 or $\sqrt{p}$ implies $\< \bg_k^{\SE,\perp}, \bv_k \>_{\Ltwo} \leq Cn \| \proj_{\bU_{k-1}^{\SE}}^\perp \bu_k^{\SE} \|_{\Ltwo}$.
    Recalling that $\zeta_{kk}^u = \< \bg_k^{\SE,\perp},\bv_k^{\SE}\>_{\Ltwo}/\| \proj_{\bU_{k-1}^{\SE}}^\perp \bu_k^{\SE} \|_{\Ltwo}$,
    we conclude that 
    $\| \zeta_{kk}^u \bu_k^{\SE} - \zeta_{kk}^u \proj_{\bU_{k-1}^{\SE}} \bu_k^{\SE} \|_{\Ltwo} = |\< \bg_k^{\SE,\perp},\bv_k^{\SE}\>_{\Ltwo}| \leq C n\| \proj_{\bU_{k-1}^{\SE}}^\perp \bu_k^{\SE} \|_{\Ltwo}$.
    Using equation~\eqref{eq:ukse-is-ellk-deriv} and the fact that $\nabla \ell_k$ is $C$-Lipschitz,
    we have that
    \begin{equation}
        \Big\|
            n\proj_{\bU_{k-1}^{\SE}}\bu_k^{\SE}
            -
            \nabla \ell_k\big(
                (\nu_{k,0} + \nu_{k,\sx})\ones
                +
                \bh_k^{\SE}
                -
                \sum_{\ell = 1}^{k-1} \zeta_{k\ell}^u \bu_\ell^{\SE}
                -
                \zeta_{kk}^u \proj_{\bU_{k-1}^{\SE}} \bu_k^{\SE}
                ;
                \bw,\by_1,\by_2
            \big)
        \Big\|_{\Ltwo}
            \leq
            Cn \| \proj_{\bU_{k-1}^{\SE}}^\perp \bu_k^{\SE} \|_{\Ltwo}.
    \end{equation}
        Moreover,
    \begin{equation}
    \begin{aligned}
        &\Big\|
            n\proj_{\bU_{k-1}^{\SE}}\bu_k^{\SE}
            -
            \nabla \ell_k\big(
                (\nu_{k,0} + \nu_{k,\sx})\ones
                +
                \bh_k^{\SE}
                -
                \sum_{\ell = 1}^{k-1} \zeta_{k\ell}^u \bu_\ell^{\SE}
                -
                \zeta_{kk}^u \proj_{\bU_{k-1}^{\SE}} \bu_k^{\SE}
                ;
                \bw,\by_1,\by_2
            \big)
        \Big\|_{\Ltwo}
    \\
        &\geq
        \sqrt{
            \E\Big[
            \Tr\Big(
            \Var\Big(
                n\proj_{\bU_{k-1}^{\SE}}\bu_k^{\SE}
            -
            \nabla \ell_k\big(
                    (\nu_{k,0} + \nu_{k,\sx})\ones
                    +
                    \bh_k^{\SE}
                    -
                    \sum_{\ell = 1}^{k-1} \zeta_{k\ell}^u \bu_\ell^{\SE}
                    -
                    \zeta_{kk}^u \proj_{\bU_{k-1}^{\SE}} \bu_k^{\SE}
                    ;
                    \bw,\by_1,\by_2
                \big)
                \Bigm| \bU_{k-1}^{\SE}, \bH_{k-1}^{\SE}
            \Big)
            \Big)  
            \Big]  
        }.
    \end{aligned}
    \end{equation}
    For both $k =5$ and $6$, Assumption A1 gives $\nabla_{\eta_i} \ell_{k,i}(\eta_i;w_i,y_{1,i},y_{2,i}) > c > 0$.
    Because $\bh_k^{\SE} - \proj_{\bH_{k-1}^{\SE}} \bh_k^{\SE}$ is independent of $\bU_{k-1}^{\SE},\bH_{k-1}^{\SE}$, and everything else appearing inside the conditional variance in the preceding display,
    and because its coordinate-wise variance is given by $\| \proj_{\bV_{k-1}^{\SE}} \bv_k \|_{\Ltwo}$ by equation~\eqref{eq:fixpt-general},
    we have that the conditional variance in the preceding display is, with probability 1, coordinate-wise larger than $c\| \proj_{\bV_{k-1}^{\SE}} \bv_k \|_{\Ltwo}^2$.
    Thus, the right-hand side of the previous display is bounded below by $\sqrt{n} \| \proj_{\bV_{k-1}^{\SE}}^\perp \bv_k \|_{\Ltwo}$.
    Combining the previous two displays thus gives $\| \proj_{\bU_{k-1}^{\SE}}^\perp \bu_k^{\SE} \|_{\Ltwo} \geq \| \proj_{\bV_{k-1}^{\SE}}^\perp \bv_k^{\SE} \|_{\Ltwo}/\sqrt{n}$.
    We showed above that $\| \proj_{\bV_{k-1}^{\SE}}^\perp \bv_k^{\SE} \|_{\Ltwo} \geq c > 0$.
    Thus, $\| \proj_{\bU_{k-1}^{\SE}}^\perp \bu_k^{\SE} \|_{\Ltwo} \geq c/\sqrt{n}$.
    By the first equation in the first line of equation~\eqref{eq:fixpt-general},
    this implies $L_{g,kk} \geq c/\sqrt{n} > 0$.
    \\

    \noindent \textbf{Upper bounds on $\zeta_{kk}^v$ and $\zeta_{kk}^u$.}
    For $k \leq 4$, we have shown that $\zeta_{kk}^v = \zeta_{kk}^u = 0$.
    For $k = 5,6$,
    we use that
    $\zeta_{kk}^u = \< \bg_k^{\SE,\perp} , \bv_k^{\SE} \>_{\Ltwo} / \| \proj_{\bU_{k-1}^{\SE}}^\perp \bu_k^{\SE} \|_{\Ltwo}$
    and
    $\zeta_{kk}^v = \< \bh_k^{\SE,\perp} , \bu_k^{\SE} \>_{\Ltwo} / \| \proj_{\bV_{k-1}^{\SE}}^\perp \bv_k^{\SE} \|_{\Ltwo}$,
    which we showed above.
    Then $|\zeta_{kk}^u| \leq Cn$ and $|\zeta_{kk}^v| \leq C$ by the upper bounds $\|\bv_k\|_{\Ltwo} \leq C$, $\|\bu_k\|_{\Ltwo} \leq C/\sqrt{n}$ and the lower bounds 
    $\| \proj_{\bU_{k-1}^\SE}^\perp \bu_k^{\SE}\|{\Ltwo} \geq C/\sqrt{n}$,
    $\| \proj_{\bV_{k-1}^\SE}^\perp \bv_k^{\SE}\|{\Ltwo} \geq C$ we established above.
    \\

    \noindent \textbf{Lower bounds on $\zeta_{kk}^v$ and $\zeta_{kk}^u$.}
    Because $L_{g,kk} > 0$ and $L_{h,kk} > 0$ for $k=5,6$,
    the indices $5$ and $6$ are innovative with respect to both $\bK_g$ and $\bK_h$.
    Thus, by equations~\eqref{eq:zeta56-uv-explicit} and \eqref{eq:zeta56-uv-explicit},
    we have $\zeta_{kk}^u = \hzeta_{kk}^u$ and $\zeta_{kk}^v = \hzeta_{kk}^v$ for $k = 5,6$.
    By equation~\eqref{eq:zeta-hat-56-u}, 
    the upper bounds on $\zeta_{55}^v$ and $\zeta_{66}^v$ we just established, and the fact that $\Omega_5$ and $\Omega_6$ are $C$-strongly convex by Assumption A1,
    we conclude that $\zeta_{kk}^u \geq cn $ for $k = 5,6$.
    By equation~\eqref{eq:zeta-hat-56-v},
    the upper bounds on $\zeta_{kk}^v$, $\zeta_{kk}^u$ we just established, 
    the fact that $w(\,\cdot\,)^{-1}$ is upper bounded by $C$,
    the fact that $\ell_\prop$ is $c$-strongly convex,
    and the growth and decay conditions on $\pi$ (all by Assumption A1),
    we conclude $\zeta_{kk}^v \geq cn$ for $k = 5,6$.
    \\

    \noindent \textbf{No dependence on $\bu_5^{\SE}$ and $\bv_5^{\SE}$
      (i.e.,\ $\zeta_{65}^v = \zeta_{65}^u = 0$).}  Because indices
    $5$ and $6$ are innovative with respect to both $\bK_g$ and
    $\bK_h$.  Thus, by equations~\eqref{eq:zeta56-uv-explicit}
    and~\eqref{eq:zeta56-uv-explicit}, we have $\zeta_{65}^u =
    \hzeta_{65}^u$ and $\zeta_{65}^v = \hzeta_{65}^v$ for $k = 5,6$.
    Combining the expression for $\zeta_{65}^u = \hzeta_{65}^u$ in
    equation~\eqref{eq:zeta56-uv-explicit} with the expression for
    $\zeta_{65}^v = \hzeta_{65}^v$ in
    equation~\eqref{eq:zeta-hat-56-v}, we have that either
    $\zeta_{65}^u = \zeta_{65}^v = 0$ or
    \begin{equation}
    \begin{aligned}
        1
        &=
        \frac{
            n\E\Big[
            \Tr\Big(
                \Big(
                    \id_p + \frac{\nabla^2 \Omega_6(\bv_6^{\SE})}{\zeta_{66}^v}
                \Big)^{-1}
                \Big(
                    \id_p + \frac{\nabla^2 \Omega_5(\bv_5^{\SE})}{\zeta_{55}^v}
                \Big)^{-1}
            \Big)
            \Big]
            \E\Big[
                \frac{\pi(h_{1,i}^{\SE})}{(1 + (\ell_\prop^*)''(nu_{5,i}^{\SE};1)/\zeta_{55}^u)(1 + nw(h_{1,i}^{\SE})^{-1}/\zeta_{66}^u)}
            \Big]
        }{
            \zeta_{55}^u\zeta_{55}^v\zeta_{66}^u\zeta_{66}^v
        }
    \\
        &=
        \frac{
            \E\Big[
            \Tr\Big(
                \Big(
                    \id_p + \frac{\nabla^2 \Omega_6(\bv_6^{\SE})}{\zeta_{66}^v}
                \Big)^{-1}
                \Big(
                    \id_p + \frac{\nabla^2 \Omega_5(\bv_5^{\SE})}{\zeta_{55}^v}
                \Big)^{-1}
            \Big)
            \Big]
            \E\Big[
                \frac{\pi(h_{1,i}^{\SE})}{(1 + (\ell_\prop^*)''(nu_{5,i}^{\SE};1)/\zeta_{55}^u)(1 + nw(h_{1,i}^{\SE})^{-1}/\zeta_{66}^u)}
            \Big]
        }{
            \E\Big[
                \frac{1-\pi(h_{1,i}^{\SE})}{1 + (\ell_\prop^*)''(nu_{5,i}^{\SE};0)/\zeta_{55}^u}
                +
                \frac{\pi(h_{1,i}^{\SE})}{1 + (\ell_\prop^*)''(nu_{5,i}^{\SE};1)/\zeta_{55}^u}
            \Big]
            \E\Big[
                \Tr\Big[
                    \Big(
                        \id_p
                        +
                        \frac{\nabla^2\Omega_6(\bv_6^{\SE})}{\zeta_{66}^v}
                    \Big)^{-1}
                \Big]
            \Big]
        }
    \\
        &<1,
    \end{aligned}
    \end{equation}
    where the last equality follows from the fact that $\Big\|\Big(\id_p + \frac{\nabla^2 \Omega_6(\bv_6^{\SE})}{\zeta_{66}^v}\Big)^{-1}\Big\|_{\op} < 1$ and $\frac{\pi(h_{1,i}^{\SE})}{(1 + (\ell_\prop^*)''(nu_{5,i}^{\SE};1)/\zeta_{55}^u)(1 + nw(h_{1,i}^{\SE})^{-1}/\zeta_{66}^u)} < \frac{1-\pi(h_{1,i}^{\SE})}{1 + (\ell_\prop^*)''(nu_{5,i}^{\SE};0)/\zeta_{55}^u}+\frac{\pi(h_{1,i}^{\SE})}{1 + (\ell_\prop^*)''(nu_{5,i}^{\SE};1)/\zeta_{55}^u}$.
    Because we cannot have $1 < 1$, we conclude that $\zeta_{65}^u = \zeta_{65}^v = 0$.
    \\

    \noindent \textbf{Upper bounds on $\zeta_{51}^v$, $\zeta_{61}^v$.}
    This follows now from the expressions for $\hzeta_{51}^v$ and
    $\hzeta_{61}^v$ in equation~\eqref{eq:zeta-hat-56-v}, together with the
    upper bounds on $\pi'(\,\cdot\,)$, $w'(\cdot)$, and
    $w(\,\cdot\,)^{-1}$ in Assumption A1 as well as the upper and
    lower bounds on the fixed point parameters we have established
    above.  Note that although we have not bounded
    $\|\bu_5^{\SE,1}\|_{\Ltwo}$ and $\|\bu_{5,i}^{\SE,0}\|_{\Ltwo}$
    (we only bounded $\| \bu_5^{\SE} \|_{\Ltwo}$), these are bounded
    by noting that, fixing $y_{1,i}^{\SE} = 1$ or $y_{1,i}^{\SE} = 0$,
    the objective~\eqref{eq:SE-opt} is $cn$-strongly convex (by the
    lower bound on $\zeta_{kk}^u$) and has derivative at $\bu =
    \bzero$ whose expected $\ell_2$-norm is bounded above by
    $C\sqrt{n}$ (by the upper bound on $\| \bh_k^{\SE} \|_{\Ltwo}$),
    the upper bounds on $\nu_{5,0},\nu_{5,\sx}$, and the fact that
    (using Fenchel-Legendre duality) $\nabla_{\bu}
    \ell_5^*(\bzero);\bw,\by_1,\by_2) = \argmin_{\bmeta}
    \ell_\prop(\bmeta;\by_1)$ is bounded by $C\sqrt{n}$ (by the bounds
    on the derivatives of $\ell_\prop$ at $\bzero$ and its strong
    convexity).  \\

    \noindent \textbf{Lower bound on $\zeta_{51}^v$.} By equation~\eqref{eq:zeta-hat-56-v}, the decay condition on $\pi$ in Assumption A1, the bound on $\nu_{5,0}$ and $\nu_{5,\sx}$ we have established above,
     and the fact that $u_{5,i}^{\SE,1} - u_{5,i}^{\SE,0} < -c$, we get that $-\zeta_{51}^v \gtrsim c$.
     (Note that $u_{5,i}^{\SE,1} - u_{5,i}^{\SE,0} < -c$ because by Assumption A1, $\partial_\eta\ell_\prop(\eta;0) - \partial_\eta \ell_\prop(\eta;1) \geq c  >0 $).
\end{proof}

\section{Degrees of freedom adjustment factors}
\label{AppDOF}

As justified in~\Cref{sec:independent-covariates}, we assume without
loss of generality that $\bSigma = \id_p$.

\subsection{Existence and uniqueness of degrees-of-freedom adjustments}

Our first claim asserts that equations~\eqref{eq:dof-emp-def} have
unique solutions.
\begin{lemma}
\label{lem:dof-adjust-soln}
Under Assumption A1, equations~\eqref{eq:dof-emp-def} have unique
solutions $(\hzeta_\out^\theta, \hzeta_\out^\eta)$ and
$(\hzeta_\prop^\theta, \hzeta_\prop^\eta)$ with strictly positive
components.
\end{lemma}
\begin{proof}
    By Assumption A1, the function $\ell_\prop$ is strongly convex,
    whence $\ddot\ell_\prop(\heta_{\prop,i}) > 0$ for all $i$.  Thus,
    we have
\begin{equation}
  \frac{1}{n} \sum_{i=1}^n
  \frac{\ddot\ell_\prop(\heta_{\prop,i})}{\hzeta_{\prop}^\eta
    \ddot\ell_\prop(\heta_{\prop,i})+1} =
  \E\Big[\frac{1}{\hzeta_\prop^\eta + X_\prop}\Big], \qquad
  \text{where}\quad X_\prop \sim \frac{1}{n} \sum_{i=1}^n
  \delta_{1/\ddot\ell_\prop(\heta_{\prop,i})}.
\end{equation}
Similarly,
\begin{equation}
  \frac{1}{n} \Tr\Big(\big(\hzeta_\prop^\theta \id_p  +
  \nabla^2 \Omega_\prop(\hbtheta_\prop)\big)^{-1}\Big) =
  \frac{p}{n}\E\Big[\frac{1}{\hzeta_\prop^\theta + X_\out}\Big],
  \qquad \text{where}\quad X_\out \sim \frac{1}{p} \sum_{i=1}^n
  \delta_{1/\sigma_i^2},
\end{equation}
and $\sigma_i^2$, $i \in [p]$ are the singular values of
$\nabla^2 \Omega_\prop(\hbtheta_\prop) $.
\begin{lemma}
  \label{lem:iterated-stieltjes}
If for some $\gamma,\gamma' > 0$, $X_1,X_2$ are random variables
supported on $[\gamma,\infty)$, then the system of equations
  \begin{equation}
    \label{eq:iterated-stieltjes}
    \zeta_2 = \E\Big[ \frac{1}{\zeta_1 + X_1} \Big], \qquad \zeta_1 =
    \gamma' \E\Big[ \frac{1}{\zeta_2 + X_2} \Big],
        \end{equation}
  has a unique non-negative solution $\zeta_1,\zeta_2$,
  and this solution is in fact positive.
\end{lemma}
\noindent \Cref{lem:dof-adjust-soln} follows
from~\Cref{lem:iterated-stieltjes}, which we prove below.
\end{proof}

\begin{proof}[Proof of~\Cref{lem:iterated-stieltjes}]
    The positivity of any solutions follows from the fact that
    $X_1,X_2$ are supported on $[c,\infty)$. Moreover, at any
      solution, we have
    \begin{equation}
        \zeta_2 \leq \frac{1}{c},
        \qquad
        \zeta_1 \leq \frac{\gamma'}{ c}.
    \end{equation}
    Define functions $f_1,f_2$ by $f_1(\zeta_1) =
    \E\Big[\frac{1}{\zeta_1 + X_1}\Big]$ and $f_2(\zeta_2) =
    \gamma'\E\Big[\frac{1}{\zeta_2 + X_2}\Big]$.  Note that $f_1$ is
    defined in $[0,\infty)$, is positive, continuous, and strictly
      decreasing, with $f_1(0) = \E[1/X_1]$ and $\lim_{\zeta_1
        \rightarrow \infty} f_1(\zeta_1) = 0$.  Thus, $f_1^{-1}:
      (0,\E[1/X_1]]$ is a non-negative, strictly decreasing function
    satisfying $\lim_{\zeta_1 \rightarrow 0} f_1^{-1}(\zeta_1) =
    \infty$ and $f_1^{-1}(\E[1/X_1]) = 0$.  Likewise, $f_2$ is
    positive, continuous, and strictly decreasing, with $f_2(0) =
    \delta^{-1}\E[1/X_2]$ and $f_2(\E[1/X_1]) > 0$.  Thus, by the
    intermediate value theorem on $[\epsilon,\E[1/X_1])$ for
      sufficiently small $\epsilon$, there exists $\zeta_2 \in
      [\epsilon,\E[1/X_1]]$ such that $f_1^{-1}(\zeta_2) =
      f_2(\zeta_2)$.  Setting $\zeta_1 = f_1^{-1}(\zeta_2) =
      f_2(\zeta_2)$ gives a positive solution to the system of
      equations in the lemma.

    We now show uniqueness. At any solution $f_1(\zeta_1) = f_2^{-1}(\zeta_1)$,
    \begin{equation}
    \begin{aligned}
        -f_1'(\zeta_1)
            &=
            \E\Big[
                \frac{1}{(\zeta_1 + X_1)^2}
            \Big]
            <
            \frac{1}{\zeta_1}\E\Big[
                \frac{1}{\zeta_1 + X_1}
            \Big]
            =
            \frac{\zeta_2}{\gamma'\E\big[\frac{1}{\zeta_2 + X_2}\big]}
        \\
            &<
            \frac{\zeta_2^2}{\frac{\gamma'(\zeta_2+c)}{\zeta_2}\E\big[\big(\frac{\zeta_2}{\zeta_2 + X_2}\big)^2\big]}
            =
            \frac{\zeta_2^2}{-\zeta_2(\zeta_2+c)f_2'(\zeta_2)}
            < \frac{1}{\big(1 + \frac{c}{\zeta_2}\big)} \frac{1}{(-f_2'(\zeta_2))}
            \leq \frac{-(f_2^{-1})'(\zeta_2))}{1 + c^2}.
    \end{aligned}
    \end{equation}
    That is, at any solution $\zeta_1$, $f_1$ is decreasing more slowly than $f_2^{-1}$, which implies there can only be one solution.
\end{proof}

\subsection{Concentration of degrees-of-freedom adjustments}

This section is devoted to the proof
of~\Cref{thm:exact-asymptotics}(b).
Define
\begin{equation}
  \begin{gathered}
    f_\eta(\zeta^\eta) \defn \frac{1}{n} \sum_{i=1}^n \E\Big[
      \frac{1}{\zeta_\prop^\eta + \ddot \ell_\prop^*(n
        u_{5,i}^{\SE};y_{1,i}^{\SE})} \Big], \qquad
    f_\theta(\zeta^\theta) \defn \frac{1}{n} \E\Big[ \Tr\Big(
      \big(\zeta_\prop^\theta \id_n + \nabla^2
      \Omega_\prop(\hbtheta_\prop^{\SE})\big)^{-1} \Big) \Big],
    \\ \fhat_\eta(\hzeta_\prop^\eta) \defn \frac{1}{n} \sum_{i=1}^n
    \frac{1}{\hzeta_\prop^\eta + \ddot \ell_\prop^*(n
      u_{5,i}^{\PO};y_{1,i})} , \qquad
    \fhat_\theta(\hzeta_\prop^\theta) \defn \frac{1}{n} \Tr\Big(
    \big(\hzeta_\prop^\theta \id_n + \nabla^2
    \Omega_\prop(\bv_5^{\PO})\big)^{-1} \Big).
  \end{gathered}
\end{equation}
The quantities $\zeta_\prop^\eta,\zeta_\prop^\theta$ solve
$\zeta_\out^\theta = f_\eta(\zeta_\prop^\eta)$ and $\zeta_\prop^\eta =
f_\theta(\zeta_\prop^\theta)$ by \Cref{lem:eff-reg-explicit}.
The quantities $\hzeta_\prop^\eta,\hzeta_\prop^\theta$ solve
$\hzeta_\out^\theta = \fhat_\eta(\hzeta_\prop^\eta)$ and
$\hzeta_\prop^\eta = \fhat_\theta(\hzeta_\prop^\theta)$
by~\cref{eq:dof-emp-def}.  The result follows from the point-wise
concentration of $\fhat_\eta,\fhat_\theta$ on $f_\eta,f_\theta$ and
the monotonicity properties of $f_\eta,f_\theta$, as we now show.

We will find constants $\gamma_2,\gamma_2',\gamma_2'' > 0$, which
are upper bounded by $\cPmodel$-dependent constant $C$, such that
for $\epsilon < \gamma_2''$,
\begin{equation}
  \label{eq:f-theta-eta-bounds}
  f_\theta(\zeta_\prop^\theta + \gamma_2 \epsilon) \geq
  \zeta_\prop^\eta - \epsilon + \gamma_2'\epsilon, \qquad
  f_\eta(\zeta_\prop^\eta - \epsilon) \leq \zeta_\prop^\theta +
  \gamma_2 \epsilon - \gamma_2'\epsilon.
\end{equation}
Assume, for now, we can find such constants.  By Assumption A1, the
terms in the sum in the definition of $\fhat_\eta(\zeta^\eta)$ are
$C$-Lipschitz in $nu_{5,i}^{\PO}$.  Moreover, because
$\bv_5^{\PO}\mapsto \nabla^2 \Omega_\prop(\bv_5^{\PO})$ is
$C\sqrt{n}$-Lipschitz in Frobenius norm and
$\big\|\big(\zeta_\prop^\theta \id_n +
\nabla^2\Omega_\prop(\bv_5^{\PO})\big)^{-1}\big\|_{\op} \leq C$, the
function $\bv_5^{\PO} \mapsto \big(\zeta_\prop^\theta \id_n + \nabla^2
\Omega_\prop(\bv_5^{\PO})\big)^{-1}$ is $C\sqrt{n}$-Lipschitz in
Frobenius norm.  This implies that $\fhat_\theta(\zeta_\prop^\theta)$
is $C$-Lipschitz in $\bv_5^{\PO}$.  By~\Cref{thm:se}, with probability
at least $1 - Ce^{-cn\epsilon^r}$, $ \fhat_\theta(\zeta_\prop^\theta +
\gamma_2\epsilon) \geq \zeta_\prop^\eta - \epsilon $ and $
\fhat_\eta(\zeta_\prop^\eta - \epsilon) \leq \zeta_\prop^\theta +
\gamma_2 \epsilon.  $ On this event, $\fhat_\eta(\zeta_\prop^\eta -
\epsilon) \leq \zeta_\prop^\theta + \gamma_2\epsilon =
\fhat_\theta^{-1}(\fhat_\theta(\zeta_\prop^\theta + \gamma_2\epsilon))
\leq \fhat_\theta^{-1}(\zeta_\prop^\eta - \epsilon)$, where we have
used that $\fhat_\theta^{-1}$ is strictly decreasing (see proof of
\Cref{lem:dof-adjust-soln}).  The reverse inequality holds at
$\zeta^\eta = \frac{1}{n}
\Tr\Big(\nabla^2\Omega_\prop(\hbtheta_\prop^{\PO})^{-1}\Big)$, where
$\fhat_\eta(\zeta^\eta) > 0$ and $\fhat_\theta^{-1}(\zeta^\eta) = 0$.
By continuity (see proof of~\Cref{lem:dof-adjust-soln}), the
intermediate value theorem shows that the solution must satisfy
$\hzeta_\prop^\eta \geq \zeta_\prop^\eta - \epsilon$ on this event.
Likewise, on this event, we have $\fhat_\theta(\zeta_\prop^\theta +
\gamma_2\epsilon) \geq \zeta_\prop^\eta - \epsilon =
\fhat_\eta^{-1}(\fhat_\eta(\zeta_\prop^\eta-\epsilon)) \geq
\fhat_\eta^{-1}(\zeta_\prop^\theta + \gamma_2\epsilon)$.  The reverse
inequality holds of $\zeta^\theta$ near 0: $\fhat_\theta(0) =
\frac{1}{n}
\Tr\Big(\nabla^2\Omega_\prop(\hbtheta_\prop^{\PO})^{-1}\Big) > 0$ and
$\lim_{\zeta^\theta \downarrow 0}\fhat_\eta^{-1}(\zeta^\theta) =
\infty $.  Thus, the solution must satisfy $\hzeta_\prop^\theta \leq
\zeta_\prop^\theta + \gamma_2 \epsilon$ on this event.  By a
completely analogous argument, if we can find constants
$\gamma_2,\gamma_2',\gamma_2''$, upper bounded by $\cPmodel$-dependent
constants, such that for $\epsilon < \gamma_2''$, $
f_\theta(\zeta_\prop^\theta - \epsilon) \leq \zeta_\prop^\eta +
\gamma_2\epsilon - \gamma_2'\epsilon $ and $ f_\eta(\zeta_\prop^\eta +
\gamma_2\epsilon) \leq \zeta_\prop^\theta - \epsilon +
\gamma_2'\epsilon, $ then with probability at least $1 -
Ce^{-cn\epsilon^r}$, $\hzeta_\prop^\eta \leq \zeta_\prop^\eta +
\gamma_2\epsilon$ and $\hzeta_\prop^\theta \leq \zeta_\prop^\theta -
\epsilon$.  We conclude that $\hzeta_\prop^\eta \mydoteq
\zeta_\prop^\eta$ and $\hzeta_\prop^\theta \mydoteq
\zeta_\prop^\theta$.

We now turn to establishing the claim~\eqref{eq:f-theta-eta-bounds}.
By Assumption A1, the functions $\ell_\prop^*$ and $\Omega_\prop$ are
$C_0$-strongly smooth and $c_0$-strongly convex.  For the remainder of
the proof, $C_0,c_0$ are $\cPmodel$-dependent constants whose value
does \textit{not} change at each appearance.  Letting $\delta \defn
n/p$, the equations $\zeta_\out^\theta = f_\eta(\zeta_\prop^\eta)$,
$\zeta_\prop^\eta = f_\theta(\zeta_\prop^\theta)$, $\hzeta_\out^\theta
= \fhat_\eta(\hzeta_\prop^\eta)$, and $\hzeta_\prop^\eta =
\fhat_\theta(\hzeta_\prop^\theta)$ imply $\zeta_\prop^\theta \leq
1/c_0$ and $\zeta_\prop^\eta \leq \delta^{-1}/c_0$.  These
inequalities then imply the upper bounds $\zeta_\prop^\theta \geq
\frac{1}{1/c_0 + C_0} \defn c_1$ and $\zeta_\prop^\eta \geq
\frac{p}{n(1/c_0 + C_0)} = \delta^{-1}c_1$.  The quantity $c_1$ will
not change in future appearances.  We have
\begin{equation}
  \begin{aligned}
    -f_\eta'(\zeta_\prop^\eta) &= \frac{1}{n} \sum_{i=1}^n \E\Big[
      \frac{1}{\big(\zeta_\prop^\eta + \ddot \ell_\prop^*(n
        u_{5,i}^{\SE};y_{1,i}^{\SE})\big)^2} \Big] <
    \frac{1}{n\zeta_\prop^\eta} \sum_{i=1}^n \E\Big[
      \frac{1}{\zeta_\prop^\eta + \ddot \ell_\prop^*(n
        u_{5,i}^{\SE};y_{1,i}^{\SE})} \Big] =
    \frac{\zeta_\prop^\theta}{ \frac{1}{n} \E\Big[ \Tr\Big(
        \big(\zeta_\prop^\theta \id_n + \nabla^2
        \Omega_\prop(\hbtheta_\prop^{\SE})\big)^{-1} \Big) \Big] }
    \\ &< \frac{\zeta_\prop^\theta}{ \frac{\zeta_\prop^\theta+c_0}{n}
      \E\Big[ \Tr\Big( \big(\zeta_\prop^\theta \id_n + \nabla^2
        \Omega_\prop(\hbtheta_\prop^{\SE})\big)^{-2} \Big) \Big] } =
    \frac{\zeta_\prop^\theta}{(\zeta_\prop^\theta +
      c_0)(-f_\theta'(\zeta_\prop^\theta))} < \frac{1}{\big(1 +
      c_0^2\big)} (-f_\theta^{-1})'(\zeta_\prop^\eta)),
  \end{aligned}
    \end{equation}
where in the last inequality we have used that $\zeta_\prop^\theta < 1/c_0$.
We further have that for $\alpha < 1$,
\begin{equation}
  \begin{gathered}
    -f_\eta'(\alpha \zeta_\prop^\eta) = \frac{1}{n} \sum_{i=1}^n
    \E\Big[ \frac{1}{(\alpha \zeta_\prop^\eta + \ddot \ell_\prop^*(n
        u_{5,i}^{\SE};y_{1,i}^{\SE}))^2} \Big] \leq
    \frac{1}{\alpha^2}\big(-f_\eta'(\zeta_\prop^\eta)\big),
  \end{gathered}
\end{equation}
where we use that $1/(\alpha \zeta_\prop^\eta + \ddot \ell_\prop^*(n
u_{5,i}^{\SE};y_{1,i}^{\SE})) \leq 1/(\alpha( \zeta_\prop^\eta + \ddot
\ell_\prop^*(n u_{5,i}^{\SE};y_{1,i}^{\SE})))$.  We also have that for
all $\alpha > 1$,
\begin{equation}
  \begin{gathered}
    -f_\theta'(\alpha \zeta_\prop^\theta) = \frac{1}{n} \E\Big[
      \Tr\Big( \big(\alpha \zeta_\prop^\theta \id_n + \nabla^2
      \Omega_\prop(\hbtheta_\prop^{\SE})\big)^{-2} \Big) \Big] \leq
    \Big(\frac{c_1 + C_0}{\alpha c_1 +
      C_0}\Big)^2\big(-f_\theta'(\zeta_\prop^\theta)\big),
    \end{gathered}
    \end{equation}
where we have used that for any singular value $\sigma_i^2$ of
$\nabla^2 \Omega_\prop(\hbtheta_\prop^{\SE})$, we have
$\frac{1}{\alpha \zeta_\prop^\theta + \sigma_i^2} \leq \frac{c_1 +
  C_0}{\alpha c_1 + C_0} \frac{1}{\zeta_\prop^\theta + \sigma_i^2}$
because $\sigma_i^2 \leq C_0$ and $\zeta_\prop^\theta > c_1$.  We
therefore have that for $\epsilon > 0$
\begin{align*}
f_\eta(\zeta_\prop^\eta - \epsilon) & \leq f_\eta(\zeta_\prop^\eta) +
\frac{\epsilon (-f_\eta'(\zeta_\prop^\eta))}{(1 - \epsilon /
  \zeta_\prop^\eta)^2} = \zeta_\prop^\theta + \frac{\epsilon
  (-f_\eta'(\zeta_\prop^\eta))}{(1 - \epsilon / \zeta_\prop^\eta)^2}
\leq \zeta_\prop^\theta + \frac{\epsilon
  (-f_\eta'(\zeta_\prop^\eta))}{(1 - \delta \epsilon / c_1)^2}, \quad
\mbox{and} \\
f_\theta(\zeta_\prop^\theta + \epsilon) & \geq
f_\theta(\zeta_\prop^\theta) -
\epsilon(-f_\theta'(\zeta_\prop^\theta)) \Big( \frac{c_1 +
  C_0}{(1-\epsilon/\zeta_\prop^\theta)c_1 + C_0}\Big)^2 \geq
\zeta_\prop^\eta - \epsilon \frac{1}{ (1 +
  c_0^2)(-f_\eta'(\zeta_\prop^\eta))} \Big( \frac{c_1 + C_0}{
  (1-\epsilon/c_1) c_1 + C_0} \Big)^2.
\end{align*}
Thus, we can take $\gamma_2 = (1 +
c_0^2/2)(-f_\eta'(\zeta_\prop^\eta))$, $\gamma_2' =
c_0^2(-f_\eta'(\zeta_\prop^\eta))/4$, and $\gamma_2''$ sufficiently
small so that $1/(1-\delta \epsilon/c_1)^2$ and $(c_1 +
C_0)/((1-\epsilon/c_1)c_1 + C_0)$ are sufficiently close to 1 to make
equation~\eqref{eq:f-theta-eta-bounds} true.  Note that
$|f_\eta'(\zeta_\prop^\eta)| \leq 1/c_0^2$, whence these constants are
upper bounded by $C$.

Switching the roles of $f_\theta$ and $f_\eta$, if we can find
constants $\gamma_2,\gamma_2',\gamma_2''$, upper bounded by
$\cPmodel$-dependent constants, such that for $\epsilon < \gamma_2''$,
we have the inequalities
\begin{align*}
f_\theta(\zeta_\prop^\theta - \epsilon) \leq \zeta_\prop^\eta +
\gamma_2\epsilon - \gamma_2'\epsilon, \quad \mbox{and} \quad
f_\eta(\zeta_\prop^\eta + \gamma_2\epsilon) \leq \zeta_\prop^\theta -
\epsilon + \gamma_2'\epsilon.
\end{align*}


\section{Oracle debiasing: the general case}
\label{sec:db-proofs}

\Cref{thm:pop-mean,thm:db-normality} as regards oracle ASCW follow
from a characterization (\Cref{thm:debiasing}) of a general class of
debiasing constructions.  This characterization also will allow us to
characterize the unsuccessful debiasing constructions alluded to
in~\Cref{sec:procedure-definitions}.  As described
in~\Cref{sec:independent-covariates}, we can assume without loss of
generality that $\bSigma = \id_p$.

\subsection{Some consequences of the fixed point equations}

Our proofs require some basic consequences of the fixed point equations \eqref{eq:regr-fixed-pt}, which we state here.

Recall that $\< \ones , \hbi_\out^f \>_{\Ltwo} = 0$ by
equation~\eqref{eq:regr-fixed-pt}.  Using the explicit expressions for
the score in equation~\eqref{eq:score-explicit}, we have
$\E[\pi_{\zeta_\out^\eta}(\eta_{\prop,i}^f)(y_i^f -
  \hzeta_{\out,i}^{f,\loo})] = 0$.  Using Gaussian integration by
parts, this is equivalent to $ 0 =
\E[\pi_{\zeta_\out^\eta}(\eta_{\prop,i}^f)](\mu_\out - \hmu_\out^f) +
\E\big[\pi_{\zeta_\out^\eta}'(\eta_{\prop,i}^f)\big]\< \btheta_\prop ,
\btheta_\out - \hbtheta_\out^f \>_{\Ltwo}.  $ Because, by assumption
A1, $C > w(\eta_{\prop,i}^f) > c > 0$ and $\E[\pi_{\zeta_\out^\eta}] >
c > 0$, and because $\zeta_\out^\theta > c > 0$ by
Lemma~\Cref{lem:regr-fixpt-exist-and-bounds}, we have
$\E[\pi_{\zeta_\out^\eta}(\eta_{\prop,i}^f)] > c
\E[\pi(\eta_{\prop,i}^f)] > c > 0$.  Thus, recalling the definition of
$\alpha_1(\zeta)$ in~\cref{eq:alpha-12-zeta-def}, we have
\begin{equation}
\label{eq:hmu-out-f-bias}
    \hmu_\out^f - \mu_\out = \alpha_1(\zeta_\out^\eta)\<
    \btheta_\prop, \quad \mbox{and} \quad \btheta_\out -
    \hbtheta_\out^f \>_{\Ltwo}.
\end{equation} 
Then, by equations~\eqref{eq:regr-fixed-pt} and
\eqref{eq:score-deriv-explicit},
\begin{equation}
\label{eq:beta-out-prop}
\begin{aligned}
    \beta_{\out\prop}
        &=
        \frac{
            \E\big[\pi_{\zeta_\out^\eta}'(\eta_{\prop,i}^f)(y_i^f - \heta_{\out,i}^{f,\loo})\big]
        }{
            \E[\pi_{\zeta_\out^\eta}(\eta_{\prop,i}^f)]
        }
        =
        \frac{
            \E[\pi_{\zeta_\out^\eta}'(\eta_{\prop,i}^f)](\mu_\out - \hmu_\out^f)
            +
            \E\big[\pi_{\zeta_\out^\eta}''(\eta_{\prop,i}^f)]\<\btheta_\prop,\btheta_\out - \hbtheta_\out^f\>_{\Ltwo}
        }{
            \E[\pi_{\zeta_\out^\eta}(\eta_{\prop,i}^f)]
        }
    \\
        &=
        \left(
            \alpha_2(\zeta_\out^\eta)
            -
            \alpha_1^2(\zeta_\out^\eta)
        \right)
        \< \btheta_\prop,\btheta_\out - \hbtheta_\out^f\>_{\Ltwo}.
\end{aligned}
\end{equation}


\subsection{Proof of~\Cref{thm:debiasing}}

We now prove our general theorem on oracle debiasing.  We consider
debiased estimates based on the modified mean and inverse covariance
matrices
\begin{equation}
 \bm \defn \bmu_{\sx} + \sc_{\bmu} \btheta_\prop, \qquad \bM \defn
 \id_p - \sc_{\bSigma} \btheta_\prop \btheta_\prop^\top,
\end{equation}
for deterministic values $\sc_{\bmu},\sc_{\bSigma} \in \reals$.  In
particular, we consider the debiased estimates
\begin{equation}
\label{eq:general-db-explicit}
\begin{gathered}
    \htheta_{\out,0}^\de = \htheta_{\out,0} - \bm^\top \frac{\bM
      \bX^\top (\ba \odot \bw \odot(\by - \htheta_{\out,0} \ones - \bX
      \hbtheta_\out))}{n\hzeta_\out^\theta}, \\ \hbtheta_\out^\de =
    \hbtheta_\out + \frac{\bM\bX^\top\big(\ba \odot \bw \odot (\by -
      \htheta_{\out,0} \ones - \bX \hbtheta_\out)\big)}{n
      \hzeta_\out^\theta},
\end{gathered}
\end{equation}
where $(\htheta_{\out,0},\hbtheta_\out)$ is fit using weights $w_i$ in
equation~\eqref{eq:outcome-fit} satisfying Assumption A1.  The next
theorem characterizes their behavior.
\begin{theorem}
\label{thm:debiasing}
Define the modified propensity scores
\begin{equation}
  \pi_\zeta(\eta)
  \defn
  \frac{\zeta w(\eta)}{1 + \zeta w(\eta)} \pi(\eta),
    \end{equation}
and define the functions
\begin{equation}
  \label{eq:alpha-12-zeta-def}
  \alpha_1(\zeta) =
  \frac{\E[\pi_\zeta'(\eta_\prop)]}{\E[\pi_\zeta(\eta_\prop)]}, \qquad
  \alpha_2(\zeta) =
  \frac{\E[\pi_\zeta''(\eta_\prop)]}{\E[\pi_\zeta(\eta_\prop)]},
  \qquad \text{where }\eta_\prop \sim
  \normal(\mu_\prop,\|\btheta_\prop\|^2).
\end{equation}
Define also
\begin{equation}
  \label{eq:db-gen-conc}
  \bias_1(\zeta) \defn \alpha_1(\zeta) - \sc_{\bmu}, \qquad
  \bias_2(\zeta) \defn \left( \big(\alpha_2(\zeta) -
  \alpha_1^2(\zeta)\big) (1 - \sc_{\bSigma}\|\btheta_\prop\|^2) -
  \sc_{\bSigma} \right).
    \end{equation}
Then, under Assumption A1 and if $|\sc_{\bmu}|,|\sc_{\bSigma}| \leq C$,
for any $t \in \reals$
\begin{equation}
  \begin{gathered}
    \htheta_{\out,0}^\de - \theta_{\out,0} \mydoteq
    \<\btheta_\prop,\btheta_\out - \hbtheta_\out\> \Big(
    \bias_1(\zeta_\out^\eta) - \bias_2(\zeta_\out^\eta) \<
    \bm,\btheta_\prop\> \Big), \\
    \hmu_\out^\de - \mu_\out \mydoteq \< \btheta_\prop , \btheta_\out
    - \hbtheta_\out \> \Big( \bias_1(\zeta_\out^\eta) - \sc_{\bmu}
    \bias_2(\zeta_\out^\eta) \|\btheta_\prop\|^2 \Big), \\
    \frac{1}{p} \sum_{j=1}^p \indic\Bigg\{ \frac{ \sqrt{n} \big(
      \htheta_{\out,j}^\de - \big( \theta_{\out,j} + \< \btheta_\prop
      , \btheta_\out - \hbtheta_\out \>\bias_2(\zeta_\out^\eta)
      \theta_{\prop,j} \big) \big) }{ \shat_\out } \leq t \Bigg\}
    \mydoteq \P\big(\normal(0,1) \leq t\big),
    \end{gathered}
    \end{equation}
    where
    \begin{equation}
        \shat_\out^2
            \defn
            \frac{\|\ba \odot \bw \odot (\by - \htheta_{\out,0}\ones - \bX \hbtheta_\out)\|^2/n}{(\hzeta_\out^\theta)^2}.
    \end{equation}
\end{theorem}

\begin{proof}[Proof of~\Cref{thm:debiasing}]
By the KKT conditions for the problem~\eqref{eq:outcome-fit}, we have
\begin{equation}
  \frac{1}{n}\bX^\top\big( \ba \odot \bw \odot (\by - \ones
  \htheta_{\out,0} - \bX \hbtheta_\out)\big) = \nabla
  \Omega_\out\big(\hbtheta_\out\big).
\end{equation}
Thus, equation~\eqref{eq:general-db-explicit} implies that
\begin{equation}
  \begin{gathered}
    \htheta_{\out,0}^{\de} = \htheta_{\out,0} - \frac{\bm^\top \bM
      \nabla \Omega_\out(\hbtheta_\out)}{\hzeta_\out^\theta}, \qquad
    \hbtheta_\out^\de = \hbtheta_\out + \frac{\bM\nabla
      \Omega_\out\big(\hbtheta_\out\big)}{\hzeta_\out^\theta}.
    \end{gathered}
\end{equation}
By~\Cref{lem:regr-fixpt-exist-and-bounds}, $\hzeta_\out^\theta
\mydoteq \zeta_\out^\theta > c$.  By Assumption A1 and because
$|\sc_{\bmu}|,|\sc_{\bSigma}| \leq C$, we have $\| \bm \| \leq C$, and
$\| \bM \|_{\op} \leq C$.  By~\Cref{lem:regr-fixpt-exist-and-bounds}
and~\Cref{thm:exact-asymptotics}, $\| \hbtheta_\out \| \mydoteq \|
\hbtheta_\out^f \|_{\Ltwo} < C$.  By Assumption A1,
$\nabla\Omega_\out$ is $C$-Lipschitz and $\Omega_\out$ has minimizer
bounded in $\ell_2$-norm by $C$, whence, because $\| \hbtheta_\out \|
\lessdot C$, we have $\| \nabla \Omega_\out (\hbtheta_\out) \|
\lessdot C$.  Combining these exponential high probability bounds, we
conclude
\begin{equation}
  \begin{gathered}
    \htheta_{\out,0}^{\de} \mydoteq \htheta_{\out,0} - \frac{\bm^\top
      \bM \nabla \Omega_\out(\hbtheta_\out)}{\zeta_\out^\theta},
    \qquad \hbtheta_\out^\de \mydoteq \hbtheta_\out + \frac{\bM\nabla
      \Omega_\out\big(\hbtheta_\out\big)}{\zeta_\out^\theta}.
  \end{gathered}
\end{equation}
The KKT conditions for the problem~\eqref{eq:param-fit-f} imply that
$\nabla \Omega_\out(\hbtheta_\out^f) / \zeta_\out^\theta = \by_\out^f
- \hbtheta_\out^f$.  Thus, \Cref{thm:exact-asymptotics} implies that
\begin{equation}
  \label{eq:db-offset-conc}
  \begin{aligned}
    \htheta_{\out,0}^\de &- \theta_{\out,0} \mydoteq \hmu_\out^f - \<
    \bmu_\sx , \hbtheta_\out^f \>_{\Ltwo} - \E[\< \bm , \by_\out^f -
      \hbtheta_\out^f \>_{\bM}] - \theta_{\out,0} \\ &= \hmu_\out^f -
    \< \bmu_\sx , \hbtheta_\out^f \>_{\Ltwo} - \< \bm , \by_\out^f -
    \hbtheta_\out^f \>_{\Ltwo} + \sc_{\bSigma}\< \bm , \btheta_\prop
    \>_{\Ltwo} \< \btheta_\prop , \by_\out^f - \hbtheta_\out^f
    \>_{\Ltwo} - \theta_{\out,0} \\ &= (\hmu_\out^f- \< \bmu_{\sx} ,
    \btheta_\out \> - \theta_{\out,0}) -
    \sc_{\bmu}\<\btheta_\prop,\btheta_\out - \hbtheta_\out^f\>_{\Ltwo}
    - \beta_{\out\prop}\<\bm,\btheta_\prop\> + \sc_{\bSigma} \< \bm ,
    \btheta_\prop \>_{\Ltwo} \big( \beta_{\out\prop} \| \btheta_\prop
    \|^2 + \< \btheta_\prop , \btheta_\out - \hbtheta_\out^f\>_{\Ltwo}
    \big) \\ &= \left( \alpha_1(\zeta_\out^\eta) - \sc_{\bmu} \right)
    \<\btheta_\prop,\btheta_\out - \hbtheta_\out^f\>_{\Ltwo} - \left(
    \big( \alpha_2(\zeta_\out^\eta) - \alpha_1^2(\zeta_\out^\eta)
    \big) (1 - \sc_{\bSigma} \| \btheta_\prop\|^2) - \sc_{\bSigma}
    \right) \<\btheta_\prop,\btheta_\out - \hbtheta_\out^f\>_{\Ltwo}
    \< \bm , \btheta_\prop \>,
    \end{aligned}
\end{equation}
where in the last equation we use equations~\eqref{eq:hmu-out-f-bias}
and~\eqref{eq:beta-out-prop}.  \Cref{thm:exact-asymptotics} also
implies that $\shat_\out^2 \mydoteq S_{44}$, and
by~\Cref{lem:regr-fixpt-exist-and-bounds}, $S_{44} > c > 0$.  Thus, we
also have
\begin{equation}
  \label{eq:db-param-conc}
    \begin{aligned}
        &\phi\Big( \frac{ \hbtheta_\out^\de - (\btheta_\out +
        \<\btheta_\prop,\btheta_\out - \hbtheta_\out\>
        \bias_2(\zeta_\out^\eta) \btheta_\prop) }{ \shat_\out } \Big)
      \mydoteq \E\Big[ \phi\Big( \frac{ \hbtheta_\out^f +
          \bM(\by_\out^f - \hbtheta_\out^f) - \big( \btheta_\out +
          \<\btheta_\prop,\btheta_\out - \hbtheta_\out^f\>_{\Ltwo}
          \bias_2(\zeta_\out^\eta) \btheta_\prop \big) }{ S_{44}^{1/2}
        } \Big) \Big]
      \\
      &\quad\quad\quad\quad\quad\quad\quad\quad\quad\quad\quad=
      \E\Big[ \phi\Big( \frac{ \by_\out^f - \sc_{\bSigma} \<
          \btheta_\prop, \by_\out^f - \hbtheta_\out^f \>\btheta_\prop
          - \big( \btheta_\out + \<\btheta_\prop,\btheta_\out -
          \hbtheta_\out^f\>_{\Ltwo} \bias_2(\zeta_\out^\eta)
          \btheta_\prop \big) }{ S_{44}^{1/2} } \Big) \Big]
      \\
      &\quad\quad\quad\quad\quad\quad\quad\quad\quad\quad\quad\mydoteq
      \E\Big[ \phi\Big( \frac{ \bg_\out^f + \big( \beta_{\out\prop} -
          (\sc_{\bSigma} + \bias_2(\zeta_\out^\eta)) \< \btheta_\prop,
          \by_\out^f - \hbtheta_\out^f \>_{\Ltwo} \big) \btheta_\prop
        }{ S_{44}^{1/2} } \Big]
      \\ &\quad\quad\quad\quad\quad\quad\quad\quad\quad\quad\quad=
      \E\Big[ \phi\big( \bg_\out^f / S_{44}^{1/2} \big) \Big],
    \end{aligned}
\end{equation}
where in the first line we use~\Cref{thm:exact-asymptotics} first to
approximate $\< \btheta_\prop , \btheta_\out - \hbtheta_\out\>
\mydoteq \< \btheta_\prop , \btheta_\out - \hbtheta_\out^f \>_{\Ltwo}$
and then to approximate the value of $\phi$ applied to the given
parameter estimate; in the second line we have used the definition of
$\bM$; in the third line we have used that $\< \btheta_\prop ,
\by_\out^f - \hbtheta_\out^f \>$ concentrates on $\< \btheta_\prop ,
\by_\out^f - \hbtheta_\out^f\>_{\Ltwo}$ with sub-Gaussian tails and
variance parameter $C/n$ using the fact that $\by_\out^f -
\hbtheta_\out^f$ is $C$-Lipschitz in $\bg_\out^f$
by~\cref{eq:param-fit-f} and the bounds on the variance parameters
in~\Cref{lem:regr-fixpt-exist-and-bounds}; and in the fourth line we
have used the definition of $\bias_2$ and the identity for
$\beta_{\out\prop}$ in equation~\eqref{eq:beta-out-prop}.

Finally, note that we can write $ \hmu_\out^\prop =
\htheta_{\out,0}^\de + \< \hbmu_\sx , \hbtheta_\out^\de \>, $ where
recall $\hbmu_\sx = \frac{1}{\numobs} \sum_{i=1}^n \bx_i$.  Thus,
\begin{equation}
  \begin{aligned}
    \hmu_\out^\de - \mu_\out &= \htheta_{\out,0}^\de - \theta_{\out,0}
    + \< \hbmu_\sx - \bmu_\sx , \hbtheta_\out^\de \> + \< \bmu_\sx ,
    \hbtheta_\out^\de - \btheta_\out \>.
    \end{aligned}
    \end{equation}
By~\Cref{thm:exact-asymptotics},
\begin{equation}
  \begin{aligned}
    \< \hbmu_\sx - \bmu_\sx , \hbtheta_\out^\de \> &\mydoteq \Big\<
    \hbmu_\sx - \bmu_\sx , \hbtheta_\out + \frac{\bM\nabla
      \Omega_\out\big(\hbtheta_\out\big)}{\zeta_\out^\theta} \Big\>
    \mydoteq \big\< \by_\sx^f - \bmu_\sx , \hbtheta_\out^f + \bM
    (\by_\out^f - \hbtheta_\out^f) \big\>_{\Ltwo} \\ &\mydoteq \big\<
    \bg_\sx^f, \by_\out^f - \sc_{\bSigma} \< \btheta_\prop, \by_\out^f
    - \hbtheta_\out^f \>\btheta_\prop \big\>_{\Ltwo} \mydoteq \big\<
    \bg_\sx^f, \by_\out^f - \sc_{\bSigma} \< \btheta_\prop, \by_\out^f
    - \hbtheta_\out^f \>_{\Ltwo}\btheta_\prop \big\>_{\Ltwo} \\ &= \<
    \bg_\sx^f , \bg_\out^f \>_{\Ltwo} = \frac{p}{n} \< \hbi_\sx^f,
    \hbi_\out^f \>_{\Ltwo} = 0,
  \end{aligned}
\end{equation}
where in the second line we have used the fact that $\< \btheta_\prop
, \by_\out^f - \hbtheta_\out^f \>$ concentrates on $\< \btheta_\prop ,
\by_\out^f - \hbtheta_\out^f\>_{\Ltwo}$ with sub-Gaussian tails and
variance parameter $C/n$ as argued above; and in the final line we
have used in the second equality the fixed point
equation~\eqref{eq:regr-fixed-pt} and in the final equality the fact
that $\< \ones , \hbi_\out^f \>_{\Ltwo} = 0$.  We also have
\begin{equation}
  \begin{aligned}
    \< \bmu_\sx , \hbtheta_\out^\de - \btheta_\out \> &\mydoteq \big\<
    \bmu_\sx , \hbtheta_\out^f + \bM (\by_\out^f - \hbtheta_\out^f) -
    \btheta_\out \big\>_{\Ltwo} = \big\< \bmu_\sx , \bg_\out^f +
    \beta_{\out\prop}\btheta_\prop - \sc_{\bSigma} \< \btheta_\prop,
    \by_\out^f - \hbtheta_\out^f \>\btheta_\prop \big\>_{\Ltwo} \\
& \mydoteq \bias_2(\zeta_\out^\eta) \<\btheta_\prop,\btheta_\out -
    \hbtheta_\out^f\>_{\Ltwo} \< \bmu_\sx , \btheta_\prop \>,
  \end{aligned}
\end{equation}
where in the last line we have used that $\< \btheta_\prop ,
\by_\out^f - \hbtheta_\out^f \>$ concentrates on $\< \btheta_\prop ,
\by_\out^f - \hbtheta_\out^f\>_{\Ltwo}$ with sub-Gaussian tails and
variance parameter $C/n$ as argued above, as well as the definition of
$\bias_2$ and the expression of $\beta_{\out\prop}$
in~\cref{eq:beta-out-prop}.

In order to obtain the final line of~\cref{eq:db-gen-conc}, we
use~\cref{eq:db-param-conc} along with Lipschitz approximations to
indicator functions.  In particular, let
$\SoftIndic_{t,\Delta}(\cdot)$ be the function which is $1$ for $x
\leq t$, $0$ for $x \geq t+\Delta$, and linearly interpolates between
the boundaries for $x \in [t,t+\Delta]$.  Note that $\SoftIndic$ is
$1/\Delta$-Lipschitz.  Then, by~\Cref{thm:exact-asymptotics}
and~\cref{eq:db-param-conc}, and using that $\bg_\out^f / S_{44}^{1/2}
\sim \normal(0,\id_p)$, we have
 \begin{equation}
   \frac{\Delta}{\sqrt{p}} \sum_{j=1}^p \SoftIndic_{t,\Delta} \Big(
   \frac{ \htheta_{\out,j}^\de - \big(\theta_{\out,j} + \<
     \btheta_\prop , \btheta_\out - \hbtheta_\out \>
     \bias_2(\zeta_\out^\eta) \theta_{\prop,j}\big) }{ \shat_\out}
   \Big) \mydoteq \Delta\sqrt{p} \E[\SoftIndic_{t,\Delta}(G)],
 \end{equation}
 where $G \sim \normal(0,1)$, and the constants in the exponential
 concentration do not depend on $\Delta$.  Note that $\Phi(t) \leq
 \E[\SoftIndic_{t,\Delta}(G)] \leq \Phi(t+\Delta) \leq \Phi(t) +
 \Delta/\sqrt{2\pi}$, where $\Phi$ is the standard Gaussian cdf and
 $\pi$ here refers to the numerical constant and not the propensity
 function.  Thus, with probability at least $1 - Ce^{-cn\epsilon^r}$
 we have 
 \begin{equation}
   \begin{aligned}
     &\frac{1}{p} \sum_{j=1}^p \indic \Big\{ \frac{ \htheta_{\out,j}^\de -
       \big(\theta_{\out,j} + \< \btheta_\prop , \btheta_\out -
       \hbtheta_\out \> \bias_2(\zeta_\out^\eta) \theta_{\prop,j}\big)
     }{ \shat_\out } \leq t \Big\} \\ &\qquad\qquad\leq \frac{1}{p}
     \sum_{j=1}^p \SoftIndic_{t,\Delta} \Big( \frac{
       \htheta_{\out,j}^\de - \big(\theta_{\out,j} + \< \btheta_\prop
       , \btheta_\out - \hbtheta_\out \> \bias_2(\zeta_\out^\eta)
       \theta_{\prop,j}\big) }{ \shat_\out } \Big) \leq \Phi(t) +
     \frac{\Delta}{\sqrt{2\pi}} + \frac{\epsilon}{\Delta \sqrt{p}}.
   \end{aligned}
 \end{equation}
 If we take $\Delta = \sqrt{\epsilon}$, we get that with probability
 at least $1 - Ce^{-cn\epsilon^r}$ that the left-hand side in the
 previous display is upper bounded by $\Phi(t) + C'\sqrt{\epsilon}$.
 Using $\SoftIndic_{t-\Delta,\Delta}$ in place of
 $\SoftIndic_{t,\Delta}$, we can replace the ``$\leq$'s'' in the
 previous display by ``$\geq$'s'', whence we also have that with
 probability at least $1 - Ce^{-cn\epsilon^r}$ that the left-hand side
 in the previous display is lower bounded by $\Phi(t) -
 C'\sqrt{\epsilon}$.  The final line of~\cref{eq:db-gen-conc} follows.
\end{proof}

\begin{remark}
    Because $\bias_1(\zeta)$ and $\bias_2(\zeta)$ are $C$-Lipschitz,
    by~\Cref{thm:exact-asymptotics} (part (b)), we have that
    $\bias_1(\hzeta_\out^\eta) \mydoteq \bias_1(\zeta_\out^\eta)$ and
    $\bias_2(\hzeta_\out^\eta) \mydoteq \bias_2(\zeta_\out^\eta)$.
    Thus, it is straightforward to show that~\Cref{thm:debiasing}
    holds with $\bias_1(\zeta_\out^\eta)$, $\bias_2(\zeta_\out^\eta)$
    replaced by $\bias_1(\hzeta_\out^\eta)$,
    $\bias_1(\hzeta_\out^\eta)$.  This statement has the benefit that
    the quantities in~\cref{eq:db-gen-conc}, while not fully empirical
    because of $\< \btheta_\prop , \btheta_\out - \hbtheta_\out\>$,
    have no explicit dependence on the solutions to the fixed point
    equations~\eqref{eq:regr-fixed-pt}.
\end{remark}


\section{Proofs of~\Cref{thm:pop-mean,thm:db-normality}}
\label{sec:proof-main-theorems}

In this section, we prove~\Cref{thm:pop-mean,thm:db-normality},
separating our analysis according to the method being analyzed: a
section on oracle ASCW (\Cref{SecOrASCW}) and one on empirical SCA
(\Cref{SecEmpSCA}).

\subsection{Proofs for oracle ASCW}
\label{SecOrASCW}

We prove~\Cref{thm:pop-mean,thm:db-normality} for the oracle ASCW
estimate by applying~\Cref{thm:debiasing} with $w(\eta) = 1$, $\bm =
\bmu_{\sx,\cfd}$, and $\bM = \bSigma_{\cfd}^{-1}$.
By~\cref{eq:cfd-mean-variance} and the Sherman-Morrison-Woodbury
formula, this is equivalent to taking $\sc_{\bmu} = \alpha_1$ and
$\sc_{\bSigma} = \frac{\alpha_2 - \alpha_1^2}{1 + (\alpha_2 -
  \alpha_1^2)\|\btheta_\prop\|^2}$.  In this case,
equation~\eqref{eq:alpha-12-zeta-def} implies that
\begin{equation}
\begin{gathered}
    \pi_\zeta(\eta) \defn \frac{\zeta}{1 + \zeta} \pi(\eta),
    \\ \alpha_1(\zeta) =
    \frac{\E[\pi'(\eta_\prop)]}{\E[\pi(\eta_\prop)]} = \alpha_1,
    \qquad \alpha_2(\zeta) =
    \frac{\E[\pi''(\eta_\prop)]}{\E[\pi(\eta_\prop)]} = \alpha_2,
    \qquad \text{where }\eta_\prop \sim
    \normal(\mu_\prop,\|\btheta_\prop\|_{\bSigma}^2).
\end{gathered}
\end{equation}
Thus, $\bias_1(\hzeta_\out^\eta) = \alpha_1 - \sc_{\bmu} = 0$ and
$\bias_2(\hzeta_\out^\eta) = \Big((\alpha_2 - \alpha_1^2)\Big(1 -
\frac{\alpha_2-\alpha_1^2}{1 + (\alpha_2 -
  \alpha_1^2)\|\btheta_\prop\|^2}\|\btheta_\prop\|^2\Big) -
\frac{\alpha_2-\alpha_1^2}{1 + (\alpha_2 -
  \alpha_1^2)\|\btheta_\prop\|^2}\Big) = 0$.  The result follows.


\subsection{Empirical SCA}
\label{SecEmpSCA}

First assume we have the stronger guarantees $\widehat{\dbAdj}_{01}
\mydoteq \dbAdj_{01}$, $\widehat{\dbAdj}_{02} \mydoteq \dbAdj_{02}$,
$\widehat{\dbAdj}_1 \mydoteq \dbAdj_1$, $\hbeta \mydoteq \beta$, with
$\beta = \alpha_1$ if we use the moment method in
\eqref{eq:prop-db-option} and $\beta = \beta_{\prop\prop}$ if we use
M-estimation.  Indeed, these stronger consistency guarantees will be
established for the estimates
in~\Cref{sec:prop-debiasing,sec:pop-cor-with-adjust}.  To avoid
confusion, we denote
$(\hmu_\out^\de,\htheta_{\out,0}^\de,\hbtheta_\out^\de)$ by
$(\hmu_\out^\ASCW,\htheta_{\out,0}^\ASCW,\hbtheta_\out^\ASCW)$ when
defined using the oracle ASCW bias
estimates~\eqref{eq:orc-ASCW-matrix-form} and by
$(\hmu_\out^\SCA,\htheta_{\out,0}^\SCA,\hbtheta_\out^\SCA)$ when using
the empirical SCA bias estimates~\eqref{eq:SCA-bias-hat}.

Then we have
\begin{align*}
  \htheta_{\out,0}^\SCA = \htheta_{\out,0} - \widehat{\dbAdj}_{01} -
  \widehat{\dbAdj}_{02} \mydoteq \htheta_{\out,0} - \dbAdj_{01} -
  \dbAdj_{02} = \htheta_{\out,0}^\ASCW \mydoteq 0 \theta_{\out,0}
\end{align*}
by~\Cref{thm:db-normality} for the oracle ASCW estimates.  Likewise, $
\hbtheta_\out^\SCA \mydoteq \hbtheta_\out^\ASCW - \dbAdj_1 \cdot
(\hbtheta_\prop^\de - \btheta_\prop)$.  Then, using the calculation in
equation~\eqref{eq:db-param-conc} and the fact that
$\bias_2(\zeta_\out^\eta) = 0$ (see proofs of
\Cref{thm:pop-mean,thm:db-normality} for oracle ASCW), for any
Lipschitz function $\phi: \reals^p \rightarrow \reals$ we have
\begin{equation}
    \phi(\hbtheta_\out^\SCA) \mydoteq \E\Big[ \phi\Big( \btheta_\out +
      \bg_\out^f - \dbAdj_1 \cdot \bg_\circ^f \Big) \Big],
\end{equation}
where $\bg_\circ^f = (\bg_{\sx,\cfd} - \bg_{\sx})/\alpha_1$ if the
moment method from equation~\eqref{eq:prop-db-option} is used to
compute $\hbtheta_\prop^\de$ and $\bg_\circ^f = \bg_\prop^f /
\beta_{\prop\prop}$ if $M$-estimation is used.  Because $\bS =
\llangle \IFhat^f \rrangle_{\Ltwo}$ (equation~\eqref{eq:regr-fixed-pt}),
$\bg_\out^f - \dbAdj_1 \cdot \bg_\circ^f \sim
\normal(\bzero,\tau^2\id_p/n)$ where $\tau^2 = \| \hbi_\out^f
\|_{\Ltwo}^2/n - 2\, \dbAdj_1 \< \hbi_\out^f , \hbi_\circ^f
\>_{\Ltwo}/n + \dbAdj_1^2 \| \hbi_\circ \|_{\Ltwo}^2/n$, where
$\hbi_\circ^f \defn (\ba^f / \barpi - \ones)/\alpha_1$ if the moment
method~\eqref{eq:prop-db-option} is used to compute
$\hbtheta_\prop^\de$ and $\hbi_\circ^f \defn \hbi_\prop^f /
\beta_{\prop\prop}$ if M-estimation is used.  Thus,
by~\Cref{thm:exact-asymptotics}(d), $\htau^2 \mydoteq \tau^2$.  The
concentration of the empirical quantiles in~\Cref{thm:db-normality} is
a consequence of the concentration of Lipschitz functions $\phi$ by
considering Lipschitz approximations to indicator functions exactly as
in the proof of~\Cref{thm:debiasing}.

Finally, we have
\begin{align*}
  \hmu_\out^\SCA = \htheta_{\out,0}^\SCA + \<
  \hbmu_\sx,\hbtheta_\out^\SCA\> \mydoteq \theta_{\out,0} + \langle
  \bmu_\sx + \bg_\sx^f , \btheta_\out + \bg_\out^f -
  \dbAdj_1\cdot\bg_\circ^f\>_{\Ltwo} = \theta_{\out,0} +
  \<\bmu_\sx,\btheta_\out\>_{\Ltwo} = \mu_\out,
\end{align*}
which proves~\Cref{thm:pop-mean} for empirical SCA.  \\

\noindent \textbf{Weaker consistency guarantees.}  If instead we only
have the guarantees $\widehat{\dbAdj}_{01} \gotop \dbAdj_{01}$,
$\widehat{\dbAdj}_{02} \gotop \dbAdj_{02}$, $\widehat{\dbAdj}_1 \gotop
\dbAdj_1$, $\hbeta \gotop \beta$ as $(n,p) \rightarrow \infty$, then
all previous arguments hold except with expressions of the form
``$A\mydoteq B$'' replaced by ``$A-B \gotop 0$ as $(n,p) \rightarrow
\infty$'' in all places.


\section{Empirical adjustments: Proofs of~\Cref{lem:propensity-summary-stats} and~\Cref{prop:adjustment-concentration}}
\label{sec:hbeta-dbAdj-consistency}

\begin{proof}[Proof of~\Cref{lem:propensity-summary-stats}]
The random variables $\action_i$ are independent Bernoullis with
success probability $\barpi$, so that $\hpi \mydoteq \barpi$ by
sub-Gaussian concentration.

By~\Cref{thm:exact-asymptotics}, we have
\begin{align*}
\| \hbmu_\sx \|^2 \mydoteq \|
\by_\sx^f \|_{\Ltwo}^2 = \| \bmu_\sx \|^2 + p S_{11}/n = \| \bmu_\sx
\|^2 + p/n,
\end{align*}
using the fact that $S_{11} = \| \hbi_\sx^f \|_{\Ltwo}^2 / n = \|
\ones \|^2 / n = 1$ from equation~\eqref{eq:regr-fixed-pt}, Thus, we
get $\hgamma_{\bmu} \mydoteq \| \bmu_\sx \|$.

By~\Cref{thm:exact-asymptotics}, $\| \hbmu_{\sx,\cfd} - \hbmu_{\sx}
\|^2 \mydoteq \| \by_\sx^f - \by_{\sx,\cfd}^f \|_{\Ltwo}^2 = \|
\bg_\sx^f - \alpha_1 \btheta_\prop - \bg_{\sx,\cfd}^f \|_{\Ltwo}^2 =
\alpha_1^2 \| \btheta_\prop \|^2 + \| \bg_\sx^f - \bg_{\sx,\cfd}
\|_{\Ltwo}^2 = \alpha_1^2 \| \btheta_\prop \|^2 + p(S_{11} - 2S_{12} +
S_{22})/n$, where we have used the definition of $\by_\sx^f$,
$\by_{\sx,\cfd}^f$ in equation~\eqref{eq:fixed-design-outcomes} and
the covariance of $\bg_\sx^f$, $\bg_{\sx,\cfd}$.  By the fixed point
equations~\eqref{eq:regr-fixed-pt}, we have $S_{11} - 2S_{12} + S_{22}
= \| \hbi_{\sx}^f - \hbi_{\sx,\cfd}^f \|_{\Ltwo}^2/n = \| \ones -
\ba^f/\barpi\|_{\Ltwo}^2/n = \E[(1 - \action_i^f/\barpi)^2] =
(1-\barpi)/\barpi$, where we have used
equation~\eqref{eq:empirical-influence-function}.  Because $p/n < C$
and $\barpi > c$ by Assumption A1, we conclude $(p/n)(1-\hpi)/\hpi
\mydoteq (p/n)(1-\barpi)/\barpi$ and $\| \hbmu_{\sx,\cfd} -
\hbmu_{\sx} \|^2 \mydoteq \alpha_1^2 \| \btheta_\prop \| +
(p/n)(1-\barpi)/\barpi$.  Thus, $\hgamma_{\prop*} \mydoteq \alpha_1 \|
\btheta_\prop \|$.

Next we show that
equation~\eqref{eq:prop-offset-and-signal-strenght-estimates} has a
unique solution.  Define the function\footnote{Here we allow $t < 0$,
and for $t < 0$ interpret the integral in the normal way as $\int_0^t
\pi(t) \de t = -\int_t^0 \pi(t) \de t$.}
\begin{equation}
  F(t) = \int_0^t \pi(t) \de t.
\end{equation}
By the KKT conditions, the pair $(\hmu_\prop,\hgamma_\prop)$ is a
solution to
equation~\eqref{eq:prop-offset-and-signal-strenght-estimates} if and
only if it is a stationary point of the objective $f(v_0, v_1) \defn
\E_{G \sim \normal(0,1)} [F(v_0 + v_1 G)] - \hpi v_0 - \hpi
\hgamma_{\prop*} v_1$.  Assumption A1 ensures that $\pi$ is strictly
increasing, so that $F$ is strictly convex, so that this is a convex
objective.  Moreover, it has Hessian given by
    \begin{equation}
      \nabla^2 f(v_0,v_1) = \E_{G \sim \normal(0,1)} \left[ \pi'(v_0 +
        v_1G)
            \begin{pmatrix}
              1 & G \\ G & G^2
            \end{pmatrix}
        \right].
    \end{equation}
Since $\pi'$ is always positive by Assumption A1, and $\begin{pmatrix}
  1 & G \\ G & G^2
            \end{pmatrix}$
    is positive definite with probability 1 (the only realization of
    $G$ for which is is positive semi-definite are $G = \pm1$), the
    objective $f$ is strictly convex.  Thus, if
    equation~\eqref{eq:prop-offset-and-signal-strenght-estimates} has
    a solution, it must be unique.

Now note that for any $\cPmodel$-dependent $C' > 0$, for all $v_0,v_1
\in [-C',C']$ we have that with probability at least $\P(|G| \leq 1)$
that $v_0 + v_1 G \in [-2C',2C']$.  Thus, for any such $v_0,v_1$, we
have
\begin{equation}
  \nabla^2f(v_0,v_1) \succeq \P(|G| \leq 1) \big(\min_{|\eta|\leq 2C'}
  \pi'(\eta)\big) \id_2 \succeq c\id_2,
    \end{equation}
where we have used the fact that $\E\left[\begin{pmatrix} 1 & G \\ G &
    G^2 \end{pmatrix}\right] = \id_2$ and that, by Assumption A1,
$\min_{|\eta|\leq 2C'} \pi'(\eta) > c$.  Note also that
$\partial_{v_0}\E[F(v_0 +
  v_1G)]|_{(v_0,v_1)=(\mu_\prop,\|\btheta_\prop\|)} = \barpi$ and
$\partial_{v_1}\E[F(v_0 +
  v_1G)]|_{(v_0,v_1)=(\mu_\prop,\|\btheta_\prop\|)} = \barpi \alpha_1
\| \btheta_\prop\|$.  Thus,
\begin{equation}
  \nabla f(\barpi,\|\btheta_\prop\|) =
  \begin{pmatrix}
                \barpi - \hpi \\ \barpi \alpha_1 \| \btheta_\prop \| -
                \hpi \hgamma_{\prop*}
  \end{pmatrix}.
\end{equation}
Therefore, because $\barpi \leq C$ and $\barpi \alpha_1 \|
\btheta_\prop \| \leq C$, the previous display and the fact that $\hpi
\mydoteq \barpi$ and $\hgamma_{\prop*} \mydoteq \alpha_1 \|
\btheta_\prop \|$ implies that with probability at least
$1-Ce^{-cn\epsilon^r}$, the function $f$ has a minimizer satisfying
$\|(\hmu_\prop,\hgamma_\prop) - (\mu_\prop,\|\btheta\|)\|\leq
\epsilon$.  Thus, we have that with exponentially high probability,
equation~\eqref{eq:prop-offset-and-signal-strenght-estimates} has a
unique solution, and it satisfies $\hpi \mydoteq \barpi$ and
$\hgamma_\prop \mydoteq \| \btheta_\prop \|$.

Further, by Assumption A1 (because $\pi$ has second derivative bounded
by $C$ and $\barpi > c$), we have $\frac{\E[\pi'(\hmu_\prop +
    \hgamma_\prop G)]}{\E[\pi(\hmu_\prop + \hgamma_\prop G)]}$ and
$\frac{\E[\pi''(\hmu_\prop + \hgamma_\prop G)]}{\E[\pi(\hmu_\prop +
    \hgamma_\prop G)]}$ are $C$-Lipschitz in
$\hmu_\prop,\hgamma_\prop$ on $|\hmu_\prop - \mu_\prop|$,
$|\hgamma_\prop - \gamma_\prop| \leq c'$ for some constant $c'$.
Thus, because $\hmu_\prop \mydoteq \mu_\prop$ and
$\hgamma_\prop\mydoteq \gamma_\prop$, we have $\halpha_1 \mydoteq
\alpha_1$ and $\halpha_2 \mydoteq \alpha_2$.  The fact that
$\hsc_{\bSigma} \mydoteq \frac{\alpha_2 - \alpha_1^2}{1 + (\alpha_2 -
  \alpha_1^2)\| \btheta_\prop\|^2}$ then follows from the
concentration properties we have established and the boundedness of
the parameters $\alpha_2,\alpha_1,\|\btheta_\prop\|$ by $C$.
\end{proof}

\begin{proof}[Proof of~\Cref{prop:adjustment-concentration}]
    We establish the concentration claims one at at time.  \\

\noindent \textbf{Consistency of $\hbeta$.}  When the moment
method~\eqref{eq:prop-db-option} is used to define
$\hbtheta_\prop^\de$ and $\hbeta = \halpha_1$, $\hbeta \mydoteq
\alpha_1$ by~\Cref{lem:propensity-summary-stats}.
By~\Cref{thm:exact-asymptotics}(b), $\hbmeta_\prop^\loo \mydoteq
\hbmeta_\prop + \zeta_\prop^\eta
\nabla_{\bmeta}\ell_\prop(\hbmeta_\prop;\ba)$.  Then,
by~\Cref{thm:exact-asymptotics}(d) and the KKT conditions for
equation~\eqref{eq:lin-predict-f}, $\bareta_\prop^\loo \mydoteq
\E[(\heta_{\prop,i}^f + \zeta_\prop^\eta
  \ell_\prop'(\heta_{\prop,i}^f;\action_i^f))] =
\E[\eta_{\prop,i}^{f,\loo}] = \hmu_\prop^f$.  Similarly,
$\shat_{\prop\widehat\prop} \mydoteq \E[(\heta_{\prop,i}^{f,\loo} -
  \hmu_\prop^f)\action_i^f]/(\barpi \alpha_1) =
\Cov(\heta_{\prop,i}^{f,\loo},\eta_{\prop,i}^f) = \<
\hbtheta_\prop^f,\btheta_\prop\>_{\Ltwo}$, where we have applied
Gaussian integration by parts and
equation~\eqref{eq:fixed-design-score}.  We also have that
$\|\hbtheta_\prop \|^2 \mydoteq \| \hbtheta_\prop^f \|_{\Ltwo}$ by
\Cref{thm:exact-asymptotics}.  That $\hbeta \mydoteq
\beta_{\prop\prop}$ follows now from the smoothness properties of
$\ell_\prop$ and bounds on its derivatives (Assumption A1), the bounds
on $\zeta_\prop^\eta$ (\Cref{lem:regr-fixpt-exist-and-bounds}), and
the identities \eqref{eq:regr-fixed-pt} and
\eqref{eq:score-deriv-explicit}.  \\

\noindent \textbf{Consistency of $\widehat{\dbAdj}_{01}$.}  Recall
$\bX^\top \big(\ba \odot (\by - \htheta_{\out,0}\ones - \bX
\hbtheta_\out)\big) / n = \nabla \Omega_\out(\hbtheta_\out)$ by the
KKT conditions for equation~\eqref{eq:outcome-fit}.
Then~\Cref{thm:exact-asymptotics}, the bounds on the fixed point
parameters (\Cref{lem:regr-fixpt-exist-and-bounds}), the bounds in
Assumption A1, and the KKT conditions for the fixed-design model
optimization~\eqref{eq:param-fit-f} imply that
\begin{align*}
\widehat{\dbAdj}_{01} - \dbAdj_{01} \mydoteq \<
\bg_{\sx,\cfd}^f,\by_\out^f - \hbtheta_\out^f\>_{\Ltwo}.
\end{align*}
By the fixed point equations`\eqref{eq:regr-fixed-pt} and the fact
that $\hpsi_{\out,i}^f$ is non-zero only if $a_i^f$ is nonzero, we
have $\< \ba^f , \hbi_{\sx,\cfd}^f \>_{\Ltwo} = \< \ones ,
\hbi_{\sx,\cfd}^f \>_{\Ltwo} = 0$.  Thus, again applying the fixed
point equations, $\bg_{\sx,\cfd}^f$ is independent of $\bg_{\out}^f$/
Because both $\by_\out^f$ and $\hbtheta_\out^f$ are functions of
$\bg_\out^f$, we have $\< \bg_{\sx,\cfd}^f,\by_\out^f -
\hbtheta_\out^f\>_{\Ltwo} = 0$, whence $\widehat{\dbAdj}_{01} -
\dbAdj_{01}\mydoteq 0$.  \\

\noindent \textbf{Consistency of $\widehat{\dbAdj}_{02}$.}
By~\Cref{lem:propensity-summary-stats}, $\hsc_{\bSigma} \mydoteq
\sc_{\bSigma}$.  By~\Cref{thm:exact-asymptotics}, $\< \hbmu_{\sx,\cfd}
, \hbtheta_\prop^\de \> \mydoteq \< \bmu_{\sx,\cfd} +
\bg_{\sx,\cfd}^f,\btheta_\prop + \bg_\circ^f\>_{\Ltwo} = \<
\bmu_{\sx,\cfd} , \btheta_\prop \> + \< \bg_{\sx,\cfd}^f ,
\bg_\circ^f\>_{\Ltwo}$, where $\bg_\circ^f$ is defined as
in~\Cref{SecEmpSCA}.  Completely analogous to the argument
in~\Cref{SecEmpSCA}, the fixed point equations
\eqref{eq:regr-fixed-pt} imply $\bg_{\sx,\cfd}^f$ and $\bg_\circ^f$
have iid coordinates with correlation given by $\< \ba^f / \barpi -
\ones , \hbi_\circ^f \>_{\Ltwo}/n^2$, whence, by the definition of
$\shat_{\sx\circ}$ and~\Cref{thm:exact-asymptotics},
$\shat_{\sx\circ}p/n \mydoteq \< \bg_{\sx,\cfd}^f ,
\bg_\circ^f\>_{\Ltwo}$.  We conclude
$\<\hbmu_{\sx,\cfd},\hbtheta_\prop^\de\> - \shat_{\sx\circ}p/n
\mydoteq \< \bmu_{\sx,\cfd},\btheta_\prop\>$.

By a similar argument, and again using $\bX^\top \big(\ba \odot (\by -
\htheta_{\out,0}\ones - \bX \hbtheta_\out)\big) / n = \nabla
\Omega_\out(\hbtheta_\out)$ and the KKT conditions for the
fixed-design model optimization \eqref{eq:param-fit-f}, we have
$\big\< \hbtheta_\prop^\de , \bX^\top\big(\ba \odot (\by -
\htheta_{\out,0}\ones - \bX \hbtheta_\out)\big) \big\> /
(n\hzeta_\out^\theta) \mydoteq \<\btheta_\prop + \bg_\circ^f ,
\by_\out^f - \hbtheta_\out^f \>_{\Ltwo} = \< \btheta_\prop ,
\by_\out^f - \hbtheta_\out^f \>_{\Ltwo} + \< \bg_\circ^f , \by_\out^f
- \hbtheta_\out^f \>_{\Ltwo} \mydoteq \big\< \hbtheta_\prop ,
\bX^\top\big(\ba \odot (\by - \htheta_{\out,0}\ones - \bX
\hbtheta_\out)\big) \big\> / (n\hzeta_\out^\theta) + \< \bg_\circ^f ,
\by_\out^f - \hbtheta_\out^f \>_{\Ltwo}$.  By Gaussian integration by
parts and the last line of the fixed point equations
\eqref{eq:regr-fixed-pt}, $\< \bg_\circ^f , \by_\out^f -
\hbtheta_\out^f \>_{\Ltwo} = (\< \hbi_\circ^f , \hbi_\out^f
\>_{\Ltwo}/n^2)(p-n\zeta_\out^\eta\zeta_\out^\theta )$, whence, by the
definition of $\shat_{\sx\circ}$ and~\Cref{thm:exact-asymptotics},
$\shat_{\out\circ} (p-n\hzeta_\out^\eta\hzeta_\out^\theta)/n \mydoteq
\< \bg_\circ^f , \by_\out^f - \hbtheta_\out^f \>_{\Ltwo}$.  We
conclude
\begin{equation*}
  \frac{\big\< \hbtheta_\prop^\de , \bX^\top\big(\ba \odot (\by -
    \htheta_{\out,0}\ones - \bX \hbtheta_\out)\big) \big\>}
       {n\hzeta_\out^\theta} - \frac{\shat_{\out\circ}(p-n
         \hzeta_\out^\eta \hzeta_\out^\theta)}{n} \mydoteq
       \frac{\big\< \btheta_\prop , \bX^\top\big(\ba \odot (\by -
         \htheta_{\out,0} \ones - \bX \hbtheta_\out)\big) \big\>}
            {n\hzeta_\out^\theta}.
\end{equation*}
Combining the above approximations gives $\widehat{\dbAdj}_{02}
\mydoteq \dbAdj_{02}$.  \\

\noindent \textbf{Consistency of $\widehat{\dbAdj}_1$.}  This follows
from the previous display and the consistency of $\hsc_{\bSigma}$.
\end{proof}

\section{Failure of alternative estimators}
\label{SecInconsistency}

We now establish that several natural constructions of the debiased
plug-ins fail to achieve consistent estimation of $\mu_\out$ in the
sense of~\Cref{thm:pop-mean} or unbiased normality in the sense
of~\Cref{thm:db-normality}.  These construction were briefly described
in the beginning of~\Cref{sec:procedure-definitions}.

\subsection{Failure of na\"{i}ve debiased ridge}
\label{sec:db-prop-agnostic-failure}

Using the bias estimates~\eqref{eq:db-naive} with weights $w_i = 1$
(both in this equation and in the M-estimation
objective~\eqref{eq:outcome-fit}) gives provably inconsistent
estimates when $\bSigma = \id_p$ and ridge regression is used.
\begin{proposition}
\label{prop:naive-db-failure}
Consider the bias estimates~\eqref{eq:db-naive}, obtained from the
$M$-estimate~\eqref{eq:outcome-fit} using $\Omega_\out(\bv) =
\frac\lambda2 \|\bv\|^2$ for $w_i = 1$, some $\lambda > c > 0$.
Assume $\bSigma = \id_p$.

Under Assumption A1, there exists $C_0,c_0,c_1 > 0$, depending only on
$\cPmodel$, such that if $c_0\|\btheta_\prop\|^2 + \<
\btheta_\prop,\btheta_\out \> > \Delta$ or $C_0\|\btheta_\prop\|^2 +
\< \btheta_\prop,\btheta_\out \> < -\Delta$ for $\Delta \geq 0$, then
\begin{equation}
  \begin{gathered}
        \hmu_\out^\de \gtrdot \mu_\out + c_1\Delta, \qquad
        \hmu_\out^\de \lessdot \mu_\out - c_1\Delta,
  \end{gathered}
\end{equation}
respectively.  Moreover, in either of these cases, there exists a
number $\bias$, which depends on the model \eqref{eq:model} and the
choice of $\lambda$, such that
\begin{equation}
  \frac{1}{p} \sum_{j=1}^p \indic\Bigg\{ \frac{ \sqrt{n} \big(
    \htheta_{\out,j}^\de - \big( \theta_{\out,j} + \bias\cdot
    \theta_{\prop,j} \big) \big) }{ \shat_\out
    \bSigma_{j|-j}^{-1/2} } \leq t \Bigg\} \mydoteq
  \P\big(\normal(0,1) \leq t\big),
\end{equation}
where
\begin{equation}
  \label{eq:se-w-zeta}
  \shat_\out^2 = \frac{\frac{1}{n}\|\by - \htheta_{\out,0}\ones
    - \bX \hbtheta_\out\|^2}{(\hzeta_\out^\theta)^2}.
\end{equation}
Moreover, $|\bias| \geq c_1\Delta|\alpha_2-\alpha_1^2|$, where
$\alpha_1 \defn \E[\pi'(\eta_\prop)]/\E[\pi(\eta_\prop)]$ and
$\alpha_2 = \E[\pi''(\eta_\prop)]/\E[\pi(\eta_\prop)]$.
\end{proposition}
The conditions $c_0\|\btheta_\prop\|^2 + \< \btheta_\prop,\btheta_\out
\> > \Delta$ and $C_0\|\btheta_\prop\|^2 + \<
\btheta_\prop,\btheta_\out \> < -\Delta$ have a meaningful
interpretation.  The first case, $c_0\|\btheta_\prop\|^2 + \<
\btheta_\prop,\btheta_\out \> > \Delta$, is one in which both (1) the
propensity model signal strength is bounded below, which is equivalent
to their being non-negligible dependency between the missingness
indicator and the covariates, and (2) the propensity model and outcome
model parameters are sufficiently aligned, so that the missingness
mechanism is non-negligibly confounded with the outcome.  The case
$C_0\|\btheta_\prop\|^2 + \< \btheta_\prop,\btheta_\out \> < -\Delta$
likewise requires that the propensity model signal strength is bounded
below and the propensity and outcome models are sufficiently
anti-aligned.  \Cref{prop:naive-db-failure} also implies that the
coordinates of the debiased estimates have a bias in the direction of
the propensity model parameter.  Note that $\alpha_2$ and $\alpha_1$
depend only on the link $\pi$ and the mean and variance of the linear
predictors $\eta_{\prop,i}$.  Thus, in a proportional asymptotics in
which the signal strength $\| \btheta_\prop \|$ and the linear
predictor mean $\< \bmu_\sx , \btheta_\prop \>$ are held constant, the
$\mathsf{bias}$ term will not, in general, decay to 0.

\begin{proof} [Proof of~\Cref{prop:naive-db-failure}]
We apply~\Cref{thm:debiasing} with $w(\eta) = 1$, $\bm = \bmu_\sx$,
and $\bM = \id_p$.  This corresponds to taking $\sc_{\bmu} =
\sc_{\bSigma} = 0$.  As in the proof
of~\Cref{thm:pop-mean,thm:db-normality} for oracle ASCW, we have that
$\pi_\zeta(\eta) = \frac{\zeta}{1 + \zeta}\pi(\eta)$, whence
$\alpha_1(\zeta) = \alpha_1$ and $\alpha_2(\zeta) = \alpha_2$.  Thus,
$\bias_1(\zeta_\out^\eta) = \alpha_1$ and $\bias_2(\zeta_\out^\eta) =
\alpha_2 - \alpha_1^2$.  By Assumption A1 and because $\pi(\eta) \leq
1$,
\begin{equation}
  \bias_1(\hzeta_\out^\theta) =
  \frac{\E[\pi'(\eta_\prop)]}{\E[\pi(\eta_\prop)]} \geq
  \E[\pi'(\eta_\prop)] > c,
\end{equation}
because $\pi'$ is lower bounded by a constant on compact intervals and
$\eta_\prop$ has mean and variance bounded by $C$.  In the case of
ridge regression with $\Omega_\out(\bv) = \frac{\lambda}{2}
\|\bv\|^2$, equation~\eqref{eq:param-fit-f} gives $ \hbtheta_\out^f =
\frac{\zeta_\out^\theta}{\lambda + \zeta_\out^\theta} \by_\out^f.  $
Thus, by~\Cref{thm:exact-asymptotics}, $ \< \btheta_\prop ,
\btheta_\out - \hbtheta_\out^f \>_{\Ltwo} =
\frac{\zeta_\out^\theta}{\lambda + \zeta_\out^\theta} \big(
\beta_{\out\prop} \|\btheta_\prop\|^2 + \<
\btheta_\prop,\btheta_\out\> \big).  $ Because $C >
\zeta_\out^\theta,\beta_{\out\prop} > c$
by~\Cref{lem:regr-fixpt-exist-and-bounds}, we have that for some $c >
0$, if $ c\|\btheta_\prop\|^2+\< \btheta_\prop,\btheta_\out\> > \Delta
> 0$, the right-hand side of the preceding display is bounded below by
$c''\Delta$, whence $\< \btheta_\prop,\btheta_\out - \hbtheta_\out \>
\gtrdot c''\Delta$.  On the other hand, if $C\|\btheta_\prop\|^2+\<
\btheta_\prop,\btheta_\out\> < -\Delta$, then the right-hand side of
the preceding display is bounded above by $-c''\Delta$, whence $\<
\btheta_\prop,\btheta_\out - \hbtheta_\out \> \lessdot -c''\Delta$.
The two claims of the Lemma now follow directly
from~\Cref{thm:debiasing}, using that $\sc_{\bmu} = 0$.
\end{proof}

\subsection{Failure of debiased ridge with IPW loss}
\label{sec:db-reweighting}

The failure of the construction which uses $\bmu_{\sx} = \E[\bx]$ and
$\bSigma = \Var(\bx)$ in place of $\bmu_{\sx,\cfd}$ and
$\bSigma_{\cfd}$ occurs due to the discrepancy between the
unconditional feature distribution $\bx_i \iid
\normal(\bmu_\sx,\bSigma)$ and the feature distribution conditional on
inclusion in the outcome fit $\bx_i \iid \Law(\bx \mid d= 1)$.  An
alternative and popular approach to addressing this discrepancy is to
reweight the sample so that the reweighted sample has the desired
unconditional
distribution~\cite{Hahn1998,arkhangelsky2021doublerobust}.  In this
section, we study such approaches when they are based on oracle
knowledge of the propensity score.  The failure with inverse
propensity weighted loss suggests to us that it may be difficult to
construct a fully empirical approach based on bias estimates
\eqref{eq:db-naive} with appropriately constructed weights $w_i$,
although we cannot eliminate this possibility.

\begin{proposition}
\label{prop:weighting-failure}
Consider the bias estimates~\eqref{eq:db-naive}, obtained from
equation~\eqref{eq:outcome-fit} with $w_i = 1/\pi(\theta_{\prop,0} +
\< \bx_i , \btheta_\prop\>)$, penalty $\Omega_\out(\bv) =
\frac\lambda2 \|\bv\|^2$ for some $\lambda > c > 0$, and $\bSigma =
\id_p$.  Then, there exists $C_0,c_0,c_1 > 0$, depending only on
$\cPmodel$, such that if $c_0\|\btheta_\prop\|^2 + \<
\btheta_\prop,\btheta_\out \> > \Delta$ or $C_0\|\btheta_\prop\|^2 +
\< \btheta_\prop,\btheta_\out \> < -\Delta$ for $\Delta \geq 0$, then,
with $\hmu_\out^\de$ the G-computation estimate with these debiased
estimates as plug-ins, we have
\begin{equation}
  \begin{gathered}
    \hmu_\out^\de \gtrdot \mu_\out + c_1\Delta, \qquad
    \hmu_\out^\de \lessdot \mu_\out - c_1\Delta,
  \end{gathered}
\end{equation}
respectively.
\end{proposition}
\noindent See the paragraph following~\Cref{prop:naive-db-failure} for
a discussion of the assumptions $c_0\|\btheta_\prop\|^2 + \<
\btheta_\prop,\btheta_\out \> > \Delta$ or $C_0\|\btheta_\prop\|^2 +
\< \btheta_\prop,\btheta_\out \> < -\Delta$ for $\Delta \geq 0$.

\begin{proof}[Proof of~\Cref{prop:weighting-failure}]
    We apply~\Cref{thm:debiasing} with $w(\eta) = 1/\pi(\eta)$, $\bm =
    \bmu_\sx$, and $\bM = \id_p$.  This corresponds to taking
    $\sc_{\bmu} = \sc_{\bSigma} = 0$.  In this case, by
    equation~\eqref{eq:alpha-12-zeta-def}, $ \pi_\zeta(\eta) \defn
    \frac{\zeta \pi(\eta)}{\zeta + \pi(\eta)}.  $ Note that
    $\pi_\zeta'(\eta)/\pi_\zeta(\eta) =
    \frac{\zeta\pi'(\eta)}{(\zeta+\pi(\eta))\pi(\eta)}\geq
    \frac{\zeta\pi'(\eta)}{\zeta + 1}$.  Because by Assumption A1,
    $\pi'(\eta) \geq \sc(M)$ for $\eta \in [-M,M]$, $\eta_\prop$ has
    mean and variance bounded by $C$, and $\hzeta_\prop^\theta
    \mydoteq \zeta_\prop^\theta \geq c$
    by~\Cref{thm:exact-asymptotics}
    and~\Cref{lem:regr-fixpt-exist-and-bounds}, we have that
    $\alpha_1(\hzeta_\prop^\theta)\gtrdot c > 0$.  The result then
    follows by combining this bound with the bounds on $\<
    \btheta_\prop ,\btheta_\out - \hbtheta_\out\>_{\Ltwo}$ in the
    proof of~\Cref{prop:naive-db-failure}.
\end{proof}

\subsection{Debiased ridge with degrees-of-freedom adjusted IPW loss}

In the course of studying debiased ridge with IPW loss, the bias
estimates~\eqref{eq:db-naive} are successful provided a certain
degrees-of-freedom adjusted IPW loss is used in the outcome
fit~\eqref{eq:outcome-fit} and a strict overlap conditions holds.
This method relies on oracle knowledge of the propensity score, so it
is not clear that it is of practical interest: for example, we did not
assess whether it outperforms the Horwitz-Thompson estimator.
Further, it relies on choosing the correct regularization parameter,
and we have not investigated the success of methods for hyperparameter
tuning.  Nevertheless, we find it interesting that the failure of
debiased ridge with IPW loss can, in principle, be resolved by certain
degrees-of-freedom adjusted weights, so we include this result here.
\begin{proposition}
  \label{prop:df-adjusted-weights-db}
In addition to Assumption A1, suppose that $\pi(\eta) \geq \pi_{\min}$
uniformly for some $\pi_{\min} > 0$.  Moreover, consider $\omega \in
(0,\pi_{\min})$.  Then, there exists some $C > c > 0$ depending only
on $\cPmodel$, $\pi_{\min}$ and $\pi_{\min} - \omega$ such that if we
compute $\htheta_{\out,0}^\de,\hbtheta_\out^\de$ as in
equations~\eqref{eq:outcome-fit} with with weights $ w_i =
\frac{1}{\pi(\theta_{\prop,i} + \< \bx_i,\btheta_\prop\>) - \omega}, $
and penalty $\lambda \Omega_\out$ in place of $\Omega_\out$, we have
\begin{align*}
\hmu_\out^\de - \mu_\out \lessdot C \big|\hzeta_\out^\eta/\lambda -
\omega\big|.
\end{align*}
\end{proposition}
\noindent Note that $\big|\hzeta_\out^\eta/\lambda - \omega\big|$ is
an empirical quantity, so that~\Cref{prop:df-adjusted-weights-db}
gives an empirical upper bound on the bias of the population mean
estimate $\hmu_\out^\de$.  This suggests a natural approach to
estimating the population mean with a weighting-based strategy: tune
$\lambda$ so that $\hzeta_\out^\eta/\lambda - \omega = 0$, and then
use the estimates in equation~\eqref{eq:outcome-fit} with this choice
of $\lambda$.  Analyzing this approach rigorously would require
addressing several difficulties.  First, all results we have presented
hold for a fixed penalty $\Omega_\out$, so do not immediately apply to
$\lambda$ chosen adaptively.  The
papers~\cite{miolane2021,celentanoMontanariWei2020} consider adaptive
choices of $\lambda$ using uniform concentration and continuity
results, and it is likely the similar analyses could be carried out in
the present setting.  Second, $\hzeta_\out^\eta$ is itself a function
of $\lambda$, so that it is not clear that a solution to
$\hzeta_\out^\eta/\lambda - \omega = 0$ exists.  Nevertheless, we
suspect that such solutions do exist.  Indeed, recall that for
least-squares and the Lasso with fully observed outcomes,
$\hzeta_\out^\eta = (\dfhat_\out/n)/(1-\dfhat_\out/n)$
(see~\Cref{sec:dof-adjustment}).  Intuitively, the degrees-of-freedom
is large for small $\lambda$ and small for large $\lambda$.  Because
the degrees-of-freedom tends to decrease with $\lambda$ (though it
need not be strictly decreasing in $\lambda$), it is natural to expect
that $\hzeta_\out^\eta/\lambda - \omega = 0$ is solved for some value
of $\lambda$ based on an Intermediate Value Theorem type
argument. Because our primary interest in this paper is to provide
estimates which are empirical in both the outcome and propensity
model, and alternative approaches are known to be consistent when
propensity score are known exactly (e.g., the Horwitz-Thompson
estimator~\cite{Hahn1998,horvitzThompson1952}), we do not rigorously
investigate a tuning-based version
of~\Cref{prop:df-adjusted-weights-db} here.

\begin{remark}[The consistency regime limit]
Rather than tuning $\lambda$, one could also imagine tuning $\omega$.
Based on the intuition described above that $\hzeta_\out^\eta =
(\dfhat_\out/n)/(1-\dfhat_\out/n)$, we should expect that
$\hzeta_\out^\eta \rightarrow 0$ in a consistency regime.  For
example, the consistency regime for the Lasso is $s = o(n/\log p)$.
If this is case, we can take $\omega \rightarrow 0$
in~\Cref{prop:df-adjusted-weights-db}, recovering inverse propensity
weighting.  The details of this statement depend on how $\lambda$
scales in consistency regimes, with the correct scaling in this regime
potentially incompatible with Assumption A1 under
which~\Cref{prop:df-adjusted-weights-db}.  Thus, verifying this
intuition rigorously is beyond the scope of the present paper.
\end{remark}

\begin{proof}[Proof of~\Cref{prop:df-adjusted-weights-db}]
    We apply~\Cref{thm:debiasing} with $w(\eta) =
    \frac{\lambda^{-1}}{\pi(\eta) - \omega}$, , $\bm = \bmu_\sx$, and
    $\bM = \id_p$.  This corresponds to taking $\sc_{\bmu} = 0$ and
    $\sc_{\bSigma} = 0$.  Note that Assumption A1 is satisfied for
    this choice of weights if we allow the constants in the upper and
    lower bounds on $w$ and its Lipschitz constant to depend on an
    upper and lower bound on $\lambda$ and a lower bound on
    $\pi_{\min} - \omega$.  Thus, allowing the constants
    in~\Cref{thm:debiasing} (including those in the high probability
    approximations) to depend on these additional quantities, we may
    apply~\Cref{thm:debiasing} with this choice of weights.

    We compute
    \begin{equation}
      \pi_\zeta(\eta) = \frac{\zeta\lambda^{-1}}{1 +
        (\zeta\lambda^{-1}-\omega)/\pi(\eta)}, \qquad
      \frac{\pi_\zeta'(\eta)}{\pi_\zeta(\eta)} = \frac{
        (\zeta\lambda^{-1}-\omega)\pi'(\eta)/\pi(\eta)^2 }{ 1 +
        (\zeta\lambda^{-1}-\omega)/\pi(\eta) }.
    \end{equation}
    Using the upper bound on $\pi'(\eta)$ from Assumption A1, the
    lower bound on $\pi(\eta) \geq \pi_{\min}$, and the lower bound on
    $\pi_{\min} - \omega$, we see that \mbox{$|\bias_1(\zeta)| =
      |\alpha_1(\zeta)| \leq C |\zeta \lambda^{-1} - \omega|$.}
    \Cref{thm:exact-asymptotics} and the bounds
    in~\Cref{lem:regr-fixpt-exist-and-bounds} imply that $|\<
    \btheta_\prop , \btheta_\out - \hbtheta_\out \>| \lessdot C$.
    Thus, the first claim of the proposition follows
    from~\Cref{thm:debiasing}.
\end{proof}

\end{document}